\documentclass[letterpaper,11pt]{article}  
\usepackage{amsfonts}
\usepackage[pdftex]{graphicx}
\usepackage{amsmath,amssymb,amsthm,mathabx}
\usepackage{natbib}
\usepackage{hyperref}
\usepackage{bm}
\usepackage{fancyhdr}
\usepackage[titletoc,title]{appendix}
\usepackage{sectsty}
\usepackage{stmaryrd}
\usepackage{setspace}
\usepackage{booktabs}
\usepackage{pdfsync}
\usepackage[backgroundcolor=blue!40,linecolor=blue!40]{todonotes}
\usepackage{fancybox}
\usepackage[skip=2pt, font=small]{caption}
\usepackage{tikz}
\usepackage[latin9]{inputenc}
\usepackage{enumitem}
\usepackage{lipsum}
\usepackage{subfigure}
\usepackage[pdftex]{graphicx}
\usepackage{multirow}
\usepackage{pdflscape}
\usepackage{graphicx}
\usepackage{booktabs,caption,fixltx2e}
\usepackage[flushleft]{threeparttable}
\usepackage{geometry}
\usepackage{xr}
\usepackage{pdflscape}
\usepackage{afterpage}
\usepackage{changepage}
\usepackage{etoc}

\definecolor{light-gray}{gray}{0.8}

\newtheoremstyle{myplain}
  {9pt}
  {9pt}
  {\itshape}
  {\parindent}
  {\scshape}
  {:}
  {.5em}
  {}
\newtheoremstyle{mydefinition}
  {9pt}
  {9pt}
  {\itshape}
  {\parindent}
  {\scshape}
  {:}
  {.5em}
  {}
\newtheoremstyle{myremark}
  {9pt}
  {9pt}
  {}
  {\parindent}
  {\scshape}
  {:}
  {.5em}
  {}
\theoremstyle{myplain}
\newtheorem{theorem}{Theorem}[section]

\newtheorem{lemma}{Lemma}[section]

\theoremstyle{mydefinition}
\newtheorem{algorithm}{Algorithm}[section]
\newtheorem{assumption}{Assumption}[section]

\theoremstyle{myremark}

\newtheorem{remark}{Remark}[section]

\renewcommand{\cite}{\citet}

\newcommand{\uni}{\stackrel{u}{\rightarrow}}

\newcommand{\eig}{\operatorname{eig}}

\def\argmax{\mathop{\rm arg\,max}}
\def\argmin{\mathop{\rm arg\,min}}

\newcommand{\G}{\mathbb{G}}
\newcommand{\HH}{\mathbb{Z}}

\newcommand{\N}{\mathbb{N}}

\newcommand{\R}{\mathbb{R}}
\newcommand{\E}{\mathbb{E}}

\newcommand{\cC}{\mathcal{C}}

\newcommand{\cJ}{\mathcal{J}}

\newcommand{\cM}{\mathcal{M}}

\newcommand{\cP}{\mathcal{P}}

\newcommand{\cR}{\mathcal{R}}

\newcommand{\cT}{\mathcal{T}}

\newcommand{\cX}{\mathcal{X}}

\newcommand{\F}{\mathcal{F}}
\newcommand{\EI}{\mathbb{EI}}
\newcommand{\bP}{\mathbb{P}}
\newcommand{\bQ}{\mathbb Q}
\newcommand{\dsM}{\mathcal{M}}
\newcommand{\gcM}{\mathcal{S}}
\newcommand{\larhon}{\frac{\lambda \rho }{\sqrt n}}
\newcommand{\tlarhon}{\tfrac{\lambda \rho }{\sqrt n}}
\newcommand{\thetnprime}{(\theta_n +\rho /\sqrt{n} B^d)\cap \Theta}

\hypersetup{colorlinks=true, linkcolor=blue, citecolor=blue}
\numberwithin{equation}{section}
\numberwithin{figure}{section}

\newenvironment{parindent1}{\begin{adjustwidth}{1cm}{}}{\end{adjustwidth}}

\usetikzlibrary{shapes.geometric, arrows,positioning}
\usetikzlibrary{arrows}

\tikzstyle{lemma} = [rectangle, minimum width=3cm, minimum height=1cm,text centered, draw=black, fill=white!100]
\tikzstyle{comment} = [rectangle, rounded corners, minimum width=3cm, minimum height=1cm,text centered, draw=black, fill=gray!30]
\tikzstyle{line} = [draw, -latex']

\begin{document}

\title{Confidence Intervals for\linebreak
Projections of Partially Identified Parameters\thanks{We are grateful to Elie Tamer and three anonymous reviewers for very useful suggestions
that substantially improved the paper. We thank for their comments Ivan Canay and seminar and conference participants at Amsterdam, Bonn, BC/BU joint workshop, Brown, Cambridge, Chicago, Cologne, Columbia, Cornell, CREST, Duke, ECARES, Harvard/MIT, Kiel, Kobe, Luxembourg, Mannheim, Maryland, Michigan, Michigan State, NUS, NYU, Penn, Penn State, Rochester, Royal Holloway, SMU, Syracuse, Toronto, Toulouse, UCL, UCLA, UCSD, Vanderbilt, Vienna, Yale, Western, and Wisconsin as well as CEME, Cornell-Penn State IO/Econometrics 2015 Conference, ES Asia Meeting 2016, ES European Summer Meeting 2017, ES North American Winter Meeting 2015, ES World Congress 2015, Frontiers of Theoretical Econometrics Conference (Konstanz), KEA-KAEA International Conference, Notre Dame Second Econometrics Workshop, Verein f\"{u}r Socialpolitik Ausschuss f\"{u}r \"{O}konometrie 2017. We are grateful to Undral Byambadalai, Zhonghao Fu, Debi Mohapatra, Sida Peng, Talal Rahim, Matthew Thirkettle, and Yi Zhang for excellent research assistance. A MATLAB package implementing the method proposed in this paper, \cite{KMST_code}, is available at \url{https://molinari.economics.cornell.edu/programs/KMSportable_V3.zip}.  We are especially grateful to Matthew Thirkettle for his contributions to this package. We gratefully acknowledge financial support through NSF grants SES-1230071 and SES-1824344 (Kaido), SES-0922330 and SES-1824375 (Molinari), and SES-1260980 and SES-1824375 (Stoye).}}
\author{Hiroaki Kaido\thanks{Department of Economics, Boston University, hkaido@bu.edu.}
\and Francesca Molinari\thanks{Department of Economics,
Cornell University, fm72@cornell.edu.}
\and J\"{o}rg Stoye\thanks{Department of Economics, Cornell University, stoye@cornell.edu.}}

\maketitle

\vspace{-1cm}
\begin{abstract}
\small
We propose a bootstrap-based \emph{calibrated projection} procedure to build confidence intervals for single components and for smooth functions of a partially identified parameter vector in moment (in)equality models. The method controls asymptotic coverage uniformly over a large class of data generating processes.
The extreme points of the calibrated projection confidence interval are obtained by extremizing the value of the function of interest subject to a proper relaxation of studentized sample analogs of the moment (in)equality conditions. The degree of relaxation, or critical level, is calibrated so that the function of $\theta $, not $\theta $ itself, is uniformly asymptotically covered with prespecified probability. This calibration is based on repeatedly checking feasibility of linear programming problems, rendering it computationally attractive.

Nonetheless, the program defining an extreme point of the confidence interval is generally nonlinear and potentially intricate.
We provide an algorithm, based on the response surface method for global optimization, that approximates the solution rapidly and accurately, and we establish its rate of convergence. The algorithm is of independent interest for optimization problems with simple objectives and complicated constraints.
An empirical application estimating an entry game illustrates the usefulness of the method. Monte Carlo simulations confirm the accuracy of the solution algorithm, the good statistical as well as computational performance of calibrated projection (including in comparison to other methods), and the algorithm's potential to greatly accelerate computation of other confidence intervals.

\medskip
\small

\end{abstract}

\vfill
\thispagestyle{empty}
\pagebreak
\onehalfspacing

\pagenumbering{arabic}

\etocdepthtag.toc{mtchapter}
\etocsettagdepth{mtchapter}{subsection}
\etocsettagdepth{mtappendix}{none}

\section{Introduction}
\label{sec:introduction}
This paper provides novel confidence intervals for projections and smooth functions of a parameter vector $\theta \in \Theta \subset \R^{d}$, $d<\infty$, that is partially or point identified through a finite number of moment (in)equalities. In addition, we develop a new algorithm for computing these confidence intervals and, more generally, for solving optimization problems with ``black box'' constraints, and obtain its rate of convergence.

Until recently, the rich literature on inference for moment (in)equalities focused on confidence sets for the entire vector $\theta$, usually obtained by test inversion as
\begin{align}
	\mathcal C_n(c_{1-\alpha})\equiv \left\{\theta\in\Theta:T_n(\theta )\leq c_{1-\alpha}(\theta )\right\},\label{eq:cs}
\end{align}
where the test statistic $T_n(\theta)$ aggregates violations of the sample analog of the moment (in)equalities and the critical value $c_{1-\alpha}(\theta)$ controls asymptotic coverage, often uniformly over a large class of data generating processes (DGPs). 
However, applied researchers are frequently interested in a specific component (or function) of $\theta$, e.g., the returns to education. Even if not, they may simply want to report separate confidence intervals for components of a vector, as is standard practice in other contexts. Thus, consider inference on the projection $p^{\prime }\theta$, where $p$ is a known unit vector.
To date, it is common to report as confidence set the corresponding projection of $\mathcal C_n(c_{1-\alpha})$ or the interval
\begin{align}
	CI^{proj}_n=\left[\inf_{\theta\in\mathcal C_n(c_{1-\alpha})} p^{\prime}\theta, \sup_{\theta\in\mathcal C_n(c_{1-\alpha})} p^{\prime }\theta \right],\label{eq:def_ci_AS}
\end{align}
which will miss any ``gaps" in a disconnected projection but is much easier to compute. This approach yields asymptotically valid but typically conservative and therefore needlessly large confidence regions. The potential severity of this effect is easily appreciated in a point identified example. Given a $\sqrt{n}$-consistent estimator $\hat{\theta}_n \in \R^d$ with limiting covariance matrix equal to the identity matrix, the usual 95\% confidence interval for $\theta_k$ equals $[\hat{\theta}_{n,k}-1.96,\hat{\theta}_{n,k}+1.96]$. Yet the analogy to $CI^{proj}_n$ would be projection of a 95\% confidence ellipsoid, which with $d=10$ yields $[\hat{\theta}_{n,k}-4.28,\hat{\theta}_{n,k}+4.28]$ and a true coverage of essentially $1$.

Our first contribution is to provide a bootstrap-based \emph{calibrated projection} method to largely anticipate and correct for the conservative effect of projection. The method uses an estimated critical level $\hat{c}_{n,1-\alpha}$ calibrated so that the projection of $\mathcal C_n(\hat{c}_{n,1-\alpha})$ covers $p'\theta$ (but not necessarily $\theta$) with probability at least $1-\alpha$. As a confidence region for the true $p'\theta$, one may report this projection, i.e.
\begin{align}
\{p'\theta:\theta \in \cC_n(\hat{c}_{n,1-\alpha})\}, \label{eq:def_with_gaps}
\end{align}
or, for computational simplicity and presentational convenience, the interval
\begin{align}
	CI_n \equiv \left[\inf_{\theta\in\mathcal C_n(\hat{c}_{n,1-\alpha})} p^{\prime}\theta, \sup_{\theta\in\mathcal C_n(\hat{c}_{n,1-\alpha})} p^{\prime }\theta \right].\label{eq:def_ci}
\end{align}
We prove uniform asymptotic validity of both over a large class of DGPs.

Computationally, calibration of $\hat{c}_{n,1-\alpha}$ is relatively attractive: We linearize all constraints around $\theta $, so that coverage of $p^\prime\theta$ can be calibrated by analyzing many linear programs. Nonetheless, computing the above objects is challenging in moderately high dimension.
This brings us to our second contribution, namely a general method to accurately and rapidly compute confidence intervals whose construction resembles \eqref{eq:def_ci}. Additional applications within partial identification include projection of confidence regions defined in \cite{CHT}, \cite{AS}, or \cite{AShiECMA}, as well as (with minor tweaking; see Appendix \ref{sec:EAM_for_BCS}) the confidence interval proposed in \cite[BCS henceforth]{BCS14_subv} and further discussed later. In an application to a point identified setting, \cite[Supplement Section S.3]{fre:rev17} use our method to construct uniform confidence bands for an unknown function of interest under (nonparametric) shape restrictions. They benchmark it against gridding and find it to be accurate at considerably improved speed. More generally, the method can be broadly used to compute confidence intervals for optimal values of optimization problems with estimated constraints.

Our algorithm (henceforth called E-A-M for Evaluation-Approximation-Maximization) is based on the response surface method, thus it belongs to the family of \emph{expected improvement algorithms} \citep[see e.g.][and references therein]{jonesa2001,jonesefficient1998}. \cite{Bull_Convergence_2011} established convergence of an expected improvement algorithm for unconstrained optimization problems where the objective is a ``black box" function. The rate of convergence that he derives depends on the smoothness of the black box objective function. We substantially extend his results to show convergence, at a slightly slower rate, of our similar algorithm for constrained optimization problems in which the constraints are sufficiently smooth ``black box" functions. Extensive Monte Carlo experiments (see Appendix \ref{sec:MC} and Section 5 of \cite{KMS_2017}) confirm that the E-A-M algorithm is fast and accurate.

\textbf{Relation to existing literature.}
The main alternative inference prodedure for projections -- introduced in \cite{Romano_Shaikh2008aJSPIWP} and significantly advanced in BCS -- is based on profiling out a test statistic. The classes of DGPs for which calibrated projection and the profiling-based method of BCS (BCS-profiling henceforth) can be shown to be uniformly valid are non-nested.\footnote{See \cite[Section 4.2 and Supplemental Appendix F]{KMS_2017} for a comparison of the statistical properties of calibrated projection and BCS-profiling, summarized here at the end of Section \ref{sec:main:conv}.}

Computationally, calibrated projection has the advantage that the bootstrap iterates over linear as opposed to nonlinear programming problems. While the ``outer" optimization problems in \eqref{eq:def_ci} are potentially intricate, our algorithm is geared toward them. Monte Carlo simulations suggest that these two factors give calibrated projection a considerable computational edge over profiling, though profiling can also benefit from the E-A-M algorithm. Indeed, in Appendix \ref{sec:MC} we replicate the Monte Carlo experiment of BCS and find that adapting E-A-M to their method improves computation time by a factor of about $4$, while switching to calibrated projection improves it by a further factor of about $17$.

In an influential paper,
\cite[][PPHI henceforth]{PakesPorterHo2011} also use linearization but, subject to this approximation, directly bootstrap the sample projection. This is valid only under stringent conditions.\footnote{The published version of PPHI, i.e. \cite{PPHI_ECMA}, does not contain the inference part. \cite[Section 4.2]{KMS_2017} show that calibrated projection can be much simplified under the conditions imposed by PPHI.} Other related articles that explicitly consider inference on projections include 
\cite{ber:mol08}, \cite{BontempsMagnacMaurin2012E},
\cite{Kaido12},
and \cite{KT15}.
None of these establish uniform validity of confidence sets. \cite{CCOT} establish uniform validity of MCMC-based confidence intervals for projections, but aim at covering the projection of the entire identified region $\Theta_I(P)$ (defined later) and not just of the true $\theta$. \cite{GMO16} use our insight in the context of set identified spatial VARs.

Regarding computation, previous implementations of projection-based inference \citep[e.g.,][]{CilibertoTamer09,Grieco14,dic:mor16} reported the smallest and largest value of $p^\prime \theta$ among parameter values $\theta \in \mathcal C_n(c_{1-\alpha})$ that were discovered using, e.g., grid-search or simulated annealing with no cooling. This becomes computationally cumbersome as $d$ increases because it typically requires a number of evaluation points that grows exponentially with $d$. In contrast, using a probabilistic model, our method iteratively draws evaluation points from regions that are considered highly relevant for finding the confidence interval's end point. In applications, this tends to substantially reduce the number of evaluation points.

\textbf{Structure of the paper.} Section \ref{sec:setup} sets up notation and describes our approach in detail, including computational implementation of the method and choice of tuning parameters.
 Section \ref{sec:main:res} establishes uniform asymptotic validity of $CI_n$, and Section \ref{sec:main:conv} shows that our algorithm converges at a specific rate which depends on the smoothness of the constraints.
Section \ref{sec:KT15} reports the results of an empirical application that revisits the analysis in \cite[Section 8]{KT15}. Section \ref{sec:conclusion} draws conclusions.
The proof of convergence of our algorithm is in Appendix \ref{app:EAM_conv}. Appendix \ref{sec:EAM_for_BCS} shows that our algorithm can be used to compute BCS-profiling confidence intervals. Appendix \ref{sec:MC} reports the results of Monte Carlo simulations comparing our proposed method with that of BCS. All other proofs, background material for our algorithm, and additional results are in the Online Appendix.\footnote{Appendix \ref{sec:background} provides convergence-related results and background material for our algorithm and describes how to compute $\hat{c}_{n,1-\alpha}(\theta)$.
Appendix \ref{app:all_ass} presents the assumptions under which we prove uniform asymptotic validity of $CI_n$.
Appendix \ref{sec:verify_examples} verifies, for a number of canonical partial identification problems, the assumptions that we invoke to show validity of our inference procedure and for our algorithm.
Appendix \ref{sec:app_A} contains the proof of Theorem \ref{thm:validity}. Appendix \ref{app:Lemma} collects Lemmas supporting this proof.
}

\section{Detailed Explanation of the Method}
\label{sec:setup}
\subsection{Setup and Definition of $CI_n$}
Let $X_i\in\mathcal X\subseteq\mathbb R^{d_X}$ be a random vector with distribution $P$, let $\Theta\subseteq\mathbb R^{d}$ denote the parameter space, and let $m_j:\cX\times\Theta\to\R$ for $j=1,\dots,J_1+J_2$ denote known measurable functions characterizing the model. The true parameter value $\theta$ is assumed to satisfy the moment inequality and equality restrictions
\begin{align}
&E_{P}[m_j(X_i,\theta)]\le 0,~ j=1,..., J_1 \label{eq:Theta_I1}\\
&E_{P}[m_j(X_i,\theta)]= 0,~j=J_1+1,..., J_1+J_2.\label{eq:Theta_I2}
\end{align}
The identification region $\Theta_I(P)$ is the set of parameter values in $\Theta$ satisfying \eqref{eq:Theta_I1}-\eqref{eq:Theta_I2}. For a random sample $\{X_i,i=1,..., n\}$ of observations drawn from $P$, we write
\begin{align}
\bar m_{n,j}(\theta) &\equiv \textstyle{n^{-1}\sum_{i=1}^n m_j(X_i,\theta)},~~j=1,\dots, J_1+J_2 \label{eq:m_bar}\\
\hat{\sigma}_{n,j} &\equiv \textstyle{(n^{-1}\sum_{i=1}^n [m_j(X_i,\theta)]^2-[\bar m_{n,j}(\theta)]^2)^{1/2}},~~j=1,\dots, J_1+J_2\label{eq:sigma_hat}
\end{align}
for the sample moments and the analog estimators of the population moment functions' standard deviations $\sigma_{P,j}$. 
The confidence interval in \eqref{eq:def_ci} then is
\begin{align}
\label{eq:def_ci_KMS}
CI_n= [-s(-p,\mathcal C_n(\hat{c}_{n,1-\alpha})),s(p,\mathcal C_n(\hat{c}_{n,1-\alpha}))]
\end{align}
with
\begin{align}
\label{eq:CI}
s(p,\mathcal C_n(\hat{c}_{n,1-\alpha}))\equiv\sup_{\theta\in\Theta} ~p'\theta \text{   s.t. } \sqrt{n} \frac{\bar{m}_{n,j}(\theta )}{\hat{\sigma}_{n,j}(\theta )}\leq \hat{c}_{n,1-\alpha}(\theta),~j=1,\dots ,J
\end{align}%
and similarly for $(-p)$. Henceforth, to simplify notation, we write $\hat{c}_n$ for $\hat{c}_{n,1-\alpha}$. We also define $J\equiv J_1+2J_2$ moments, where $\bar m_{n,J_1+J_2+k}(\theta)=-\bar m_{J_1+k}(\theta)$ for $k=1,\dots,J_2$. That is, we treat moment equality constraints as two opposing inequality constraints.

For a class of DGPs $\mathcal P$ that we specify below, define the asymptotic size of $CI_n$ by\footnote{Here we focus on the confidence interval $CI_n$ defined in \eqref{eq:def_ci}. See Appendix \ref{cor:math:proj} for the analysis of the confidence region given by the mathematical projection in \eqref{eq:def_with_gaps}.}
\begin{align}
	\liminf_{n\to\infty}\inf_{P\in\mathcal P}\inf_{\theta\in\Theta_I(P)}P(p'\theta\in CI_n). \label{eq:asy_size}
\end{align}
We next explain how to control this size and then how to compute $CI_n$.

\subsection{Calibration of $\hat{c}_n(\theta)$}\label{sec:boot}
Calibration of $\hat{c}_n$ requires careful analysis of the moment restrictions' local behavior at each point in the identification region.
This is because the extent of projection conservatism depends on (i) the asymptotic behavior of the sample moments entering the inequality restrictions, which can change discontinuously depending on whether they bind at $\theta$ or not, and (ii) the local geometry of the identification region at $\theta$, i.e. the shape of the constraint set formed by the moment restrictions. Features (i) and (ii) can be quite different at different points in $\Theta_I(P)$, making uniform inference challenging. In particular, (ii) does not arise if one only considers inference for the entire parameter vector, and hence is a new challenge requiring new methods.

To build an intuition, fix $P\in\mathcal P$ and $\theta\in\Theta_I(P)$. The projection of $\theta$ is covered when
\begin{align}
\begin{Bmatrix} \inf_{\vartheta\in\Theta} p'\vartheta   \\ \text{s.t.} \; \frac{\sqrt n \bar{m}_{n,j}(\vartheta)}{\hat\sigma_{n,j}(\vartheta)}\le  \hat{c}_n(\vartheta), \forall j \end{Bmatrix} & \le   p'\theta  \le  \begin{Bmatrix} \sup_{\vartheta\in\Theta} p'\vartheta   \\ \text{s.t.} \; \frac{\sqrt n \bar{m}_{n,j}(\vartheta)}{\hat\sigma_{n,j}(\vartheta)}\le  \hat{c}_n(\vartheta), \forall j \end{Bmatrix} \notag \\
	 \Longleftrightarrow  \begin{Bmatrix} \inf_{\lambda\in \sqrt n(\Theta-\theta)} p'\lambda  \\ \text{s.t.} \; \frac{\sqrt n\bar{m}_{n,j}\left(\theta+\lambda/\sqrt n\right)}{\hat\sigma_{n,j}\left(\theta+\lambda/\sqrt n\right)}\le  \hat{c}_n\left(\theta+\lambda/\sqrt n\right), \forall j \end{Bmatrix} & \leq 0  \leq \begin{Bmatrix} \sup_{\lambda\in \sqrt n(\Theta-\theta)} p'\lambda   \\ \text{s.t.} \; \frac{\sqrt n\bar{m}_{n,j}\left(\theta+\lambda/\sqrt n\right)}{\hat\sigma_{n,j}\left(\theta+\lambda/\sqrt n\right)}\le  \hat{c}_n\left(\theta+\lambda/\sqrt n\right), \forall j \end{Bmatrix} \notag \\
	  \Longleftarrow  \begin{Bmatrix} \inf_{\lambda\in \sqrt n(\Theta-\theta) \cap \rho B^d} p'\lambda  \\ \text{s.t.} \; \frac{\sqrt n\bar{m}_{n,j}\left(\theta+\lambda/\sqrt n\right)}{\hat\sigma_{n,j}\left(\theta+\lambda/\sqrt n\right)}\le  \hat{c}_n\left(\theta+\lambda/\sqrt n\right), \forall j \end{Bmatrix} & \leq 0  \leq \begin{Bmatrix} \sup_{\lambda\in \sqrt n(\Theta-\theta) \cap \rho B^d} p'\lambda   \\ \text{s.t.} \; \frac{\sqrt n\bar{m}_{n,j}\left(\theta+\lambda/\sqrt n\right)}{\hat\sigma_{n,j}\left(\theta+\lambda/\sqrt n\right)}\le  \hat{c}_n\left(\theta+\lambda/\sqrt n\right), \forall j \end{Bmatrix}.
	 \label{eq:heuristic}
\end{align}
Here, we first substituted $\vartheta=\theta+\lambda/\sqrt n$ and took $\lambda$ to be the choice parameter; intuitively, this localizes around $\theta$ at rate $1/\sqrt{n}$. We then make the event smaller by adding the constraint $\lambda \in \rho B^d$, with $B^d \equiv [-1,1]^d$ and $\rho \geq 0$ a tuning parameter. We motivate this step later.

Our goal is to set the probability of \eqref{eq:heuristic} equal to $1-\alpha$. To ease computation, we approximate \eqref{eq:heuristic} by linear expansion in $\lambda$ of the constraint set. For each $j$, add and subtract $\sqrt n E_P[m_j(X_i,\theta+\lambda/\sqrt n)]/\hat\sigma_{n,j}(\theta+\lambda/\sqrt n)$ and apply the mean value theorem to obtain
\begin{eqnarray}
	\frac{\sqrt n\bar m_{n,j}\left(\theta+\lambda/\sqrt n\right)}{\hat\sigma_{n,j}\left(\theta+\lambda/\sqrt n\right)}
	= \bigl(\mathbb{G}_{n,j}\left(\theta+\lambda/\sqrt n\right)  + D_{P,j}(\bar\theta)\lambda+\sqrt n\gamma_{1,P,j}(\theta)\bigr)\frac{\sigma_{P,j}\left(\theta+\lambda/\sqrt n\right)}{\hat{\sigma}_{n,j}\left(\theta+\lambda/\sqrt n\right)} \label{eq:mean_val_exp}.
\end{eqnarray}
Here $\mathbb G_{n,j}(\cdot) \equiv \sqrt n (\bar m_{n,j}(\cdot)-E_P[m_j(X_i,\cdot)])/\sigma_{P,j}(\cdot)$ is a normalized empirical process indexed by $\theta\in \Theta$,
$D_{P,j}(\cdot)\equiv\nabla_\theta \{E_P[m_j(X_i,\cdot)]/\sigma_{P,j}(\cdot)\}$ is the gradient of the normalized moment, $\gamma_{1,P,j}(\cdot)\equiv E_P(m_j(X_i,\cdot))/\sigma_{P,j}(\cdot)$ is the studentized population moment,
and the mean value $\bar\theta$ lies componentwise between $\theta$ and $\theta+\lambda/\sqrt n$.\footnote{\label{ftn:bar_theta_j}The mean value $\bar\theta$ changes with $j$ but we omit the dependence to ease notation.}

We formally establish that the probability of the last event in \eqref{eq:heuristic} can be approximated by the probability that $0$ lies between the optimal values of two stochastic linear programs. The components that characterize these programs can be estimated. Specifically, we replace $D_{P,j}(\cdot)$ with a uniformly consistent (on compact sets) estimator, $\hat{D}_{n,j}(\cdot)$,\footnote{See Online Appendix \ref{sec:verify_examples} for such estimators in some canonical moment (in)equality examples.} and the process $\mathbb G_{n,j}(\cdot)$ with its simple nonparametric bootstrap analog, $\mathbb G^b_{n,j}(\cdot)\equiv n^{-1/2}\sum_{i=1}^{n} (m_{j}(X_{i}^{b},\cdot)-\bar{m}_{n,j}(\cdot))/\hat{\sigma}_{n,j}(\cdot )$.\footnote{BCS approximate $\G_{n,j}(\cdot)$ by $n^{-1/2}\sum_{i=1}^{n} [(m_{j}(X_{i},\cdot)-\bar{m}_{n,j}(\cdot))/\hat{\sigma}_{n,j}(\cdot )]\chi_i$ with $\{\chi_i\sim N(0,1)\}_{i=1}^n$ i.i.d. This approximation is equally valid in our approach, and can be faster as it avoids repeated evaluation of $m_{j}(X^b_{i},\cdot)$.} Estimation of $\gamma_{1,P,j}(\theta)$ is more subtle because it enters \eqref{eq:mean_val_exp} scaled by $\sqrt{n}$, so that a sample analog estimator will not do. However, this specific issue is well understood in the moment inequalities literature. Following \cite[][AS henceforth]{AS} and others \citep{Bugni2009E,Canay2010JE,Stoye09}, we shrink this sample analog toward zero, leading to conservative (if any) distortion in the limit. Formally, we estimate $\gamma_{1,P,j}(\theta)$ by $\varphi(\hat{\xi}_{n,j}(\theta))$, where $\varphi:\R^J_{[\pm\infty]} \mapsto \R^J_{[\pm\infty]}$ is one of the Generalized Moment Selection (GMS henceforth) functions proposed by AS,
\begin{align}
\label{eq:hat_xi}
\hat{\xi}_{n,j}(\theta)\equiv\begin{cases}
 \kappa_n^{-1}\sqrt n\bar m_{n,j}(\theta)/\hat\sigma_{n,j}(\theta) & j=1,\dots,J_1\\
 0 & j=J_1+1,\dots,J,
 \end{cases}
\end{align}
and $\kappa_n\to \infty$ is a user-specified thresholding sequence.\footnote{\label{fn:eq:zeta}A common choice of $\varphi$ is given component-wise by
	\begin{align*}
		\varphi_j(x)=\begin{cases}
			0&\text{if}~~x\ge -1\\
			-\infty&\text{if}~~x< -1.
		\end{cases}
	\end{align*}
Restrictions on $\varphi$ and the rate at which $\kappa_n$ diverges are imposed in Assumption \ref{as:GMS}.
While for concreteness here we write out the ``hard thresholding" GMS function, Theorem \ref{thm:validity} below applies to all but one of the GMS functions in AS, namely to $\varphi^1-\varphi^4$, all of which depend on $\kappa_n^{-1}\sqrt n\bar m_{n,j}(\theta)/\hat\sigma_{n,j}(\theta)$. We do not consider GMS function $\varphi^5$, which depends also on the covariance matrix of the moment functions.}
In sum, we replace the random constraint set in \eqref{eq:heuristic} with the (bootstrap based) random polyhedral set\footnote{Here, we implicitly assume that $\Theta$ is a polyhedral set. If it is instead defined by smooth convex (in)equalities, these can be linearized too.}
\begin{align}
\Lambda_n^b (\theta, \rho ,c)
\equiv \bigl\{ \lambda \in \sqrt{n}(\Theta-\theta) \cap \rho B^{d}:\mathbb{G}_{n,j}^{b}(\theta )+\hat{D}
_{n,j}(\theta)\lambda +\varphi_j(\hat{\xi}_{n,j}(\theta))\leq c,j=1,\dots,J\bigr\} \label{eq:Lambda_n}.
\end{align}
The critical level $\hat{c}_n(\theta)$ to be used in \eqref{eq:CI} then is
\begin{align}
\hat c_n(\theta)&\equiv \inf\left\{c\in\mathbb R_+:P^*\left(\min_{\lambda \in \Lambda_n^b (\theta ,\rho ,c)}p^{\prime }\lambda \leq 0\leq \max_{\lambda \in \Lambda_n^b (\theta ,\rho ,c)}p^{\prime }\lambda\right)\ge 1-\alpha\right\} \label{eq:event_kms} \\
&= \inf\bigl\{c\in\mathbb R_+:P^*(\Lambda_n^b (\theta ,\rho ,c)\cap \{p^{\prime }\lambda =0\}\neq \emptyset)\ge 1-\alpha\bigr\},\label{eq:def:c_hat}
\end{align}
where $P^*$ denotes the law of the random set $\Lambda_n^b (\theta ,\rho ,c)$ induced by the bootstrap sampling process, i.e. by the distribution of $(X_1^b,\dots,X_n^b)$ conditional on the data. Expression \eqref{eq:def:c_hat} uses convexity of $\Lambda_n^b (\theta, \rho ,c)$ and reveals that the probability inside curly brackets can be assessed by repeatedly checking feasibility of a linear program.\footnote{We implement a program in $\mathbb{R}^d$ for simplicity but, because $p'\lambda=0$,  
one could reduce this to $\mathbb{R}^{d-1}$.} We describe in detail in Online Appendix \ref{sec:bisection} how we compute $\hat{c}_n(\theta)$ through a root-finding algorithm.

We conclude by motivating the ``$\rho$-box constraint" in \eqref{eq:heuristic}, which is a major novel contribution of this paper. The constraint induces conservative bias but has two fundamental benefits: First, it ensures that the linear approximation of the feasible set in \eqref{eq:heuristic} by \eqref{eq:Lambda_n} is used only in a neighborhood of $\theta$, and therefore that it is uniformly accurate. More subtly, it ensures that coverage induced by a given $c$ depends continuously on estimated parameters even in certain intricate cases. This renders calibrated projection valid in cases that other methods must exclude by assumption.\footnote{In \eqref{eq:Lambda_n}, set $(\mathbb{G}_{n,1}^b(\cdot),\mathbb{G}_{n,2}^b(\cdot))\sim N(0,I_2)$, $p=\hat{D}_{n,1}=\hat{D}_{n,2}=(0,1)$, $\varphi_1(\cdot)=\varphi_2(\cdot)=0$, and $\alpha=.05$. Then simple algebra reveals that (with or without $\rho$-box) $\hat{c}_n(\cdot)=\Phi^{-1}(\sqrt{.95}) \approx 1.95$. If $\hat{D}_{n,1}=(0,1-\delta)$ and $\hat{D}_{n,2}=(0,1-\delta)$, then without $\rho$-box we have $\hat{c}_n(\cdot)=\Phi^{-1}(.95)/\sqrt{2}\approx 1.16$ for any small $\delta>0$, and we therefore cannot expect to get $\hat{c}_n(\cdot)$ right if gradients are estimated. With $\rho$-box, $\hat{c}_n(\cdot) \to 1.95$ as $\delta \to 0$, so the problem goes away. This stylized example is relevant because it resembles polyhedral identified sets where one face is near orthogonal to $p$. It violates assumptions in BCS and PPHI.}

\subsection{Computation of $CI_n$ and of Similar Confidence Intervals}
\label{sec:Computation}
Projection based methods as in \eqref{eq:def_ci_AS} and \eqref{eq:def_ci} have nonlinear constraints involving a critical value which in general is an unknown function, with unknown gradient, of $\theta$.
Similar considerations often apply to critical values used to build confidence intervals for optimal values of optimization problems with estimated constraints.
When the dimension of the parameter vector is large, directly solving optimization problems with such constraints can be expensive even if evaluating the critical value at each $\theta$ is cheap.

This concern motivates this paper's second main contribution, namely a novel algorithm for constrained optimization problems of the following form:
\begin{align}
p'\theta^*\equiv \sup_{\theta\in\Theta}&~ p'\theta\notag\\
\text{s.t. }&~	g_j(\theta)\le c(\theta),~j=1,..., J,\label{eq:general_problem}
\end{align}
where $\theta^*$ is an optimal solution of the problem and $g_j(\cdot),j=1,...,J$ as well as $c(\cdot)$ are fixed functions of $\theta$. In our own application, $g_j(\theta)=\sqrt n\bar m_{n,j}(\theta)/\hat \sigma_{n,j}(\theta)$ and, for calibrated projection, $c(\theta)=\hat c_n(\theta)$.\footnote{\label{fn:stochastic}We emphasize that, in analyzing the computational problem, we take the data, including bootstrap data, as given. Thus, while an econometrician would usually think of $\sqrt n\bar m_{n,j}(\theta)/\hat \sigma_{n,j}(\theta)$ and $\hat c_n(\theta)$ as random variables, for this section's purposes they are indeed just functions of $\theta$.}

The key issue is that evaluating $c(\cdot)$ is costly.\footnote{For simplicity and to mirror our motivating application, we suppose that $g_j(\cdot)$ is easy to compute. The algorithm is easily adapted to the case where it is not. Indeed, in Appendix \ref{sec:EAM_for_BCS}, we show how E-A-M can be employed to compute BCS-profiling confidence intervals, where the profiled test statistic itself is costly to compute and is approximated together with the critical value.} Our algorithm does so at relatively few values of $\theta$. Elsewhere, it approximates $c(\cdot)$ through a probabilistic model that gets updated as more values are computed. We use this model to determine the next evaluation point but report as tentative solution the best value of $\theta$ at which $c(\cdot)$ was computed, \textit{not} a value at which it was merely approximated. Under reasonable conditions, the tentative optimal values converge to $p'\theta^*$ at a rate (relative to iterations of the algorithm)  that is formally established in Section \ref{sec:main:conv}.

After drawing an initial set of evaluation points that we set to grow linearly with $d$, the algorithm has three steps called E, A, and M below. 
\medskip

\noindent\underline{\textbf{Initialization:}} Draw randomly (uniformly) over $\Theta$ a set $(\theta^{(1)},...,\theta^{(k)})$ of initial evaluation points. Evaluate $c(\theta^{(\ell)})$ for $\ell=1,...,k-1$. Initialize $L=k$.
\medskip

\noindent\underline{\textbf{E-Step:}} Evaluate $c(\theta^{(L)})$ and record the tentative optimal value
\begin{align}
p'\theta^{*,L}\equiv\max\bigl\{p'\theta^{(\ell)}:\ell \in\{1,..., L\}, \bar g(\theta)\le c(\theta^{(\ell)})\bigr\},\label{eq:E_step_max}
\end{align}
with $\bar g(\theta)=\max_{j=1,...,J}g_j(\theta)$. 
		\medskip

\noindent\underline{\textbf{A-step:}} Approximate $\theta\mapsto c(\theta)$ by a flexible auxiliary model. We use a Gaussian-process regression model (or kriging), which for a mean-zero Gaussian process $\zeta(\cdot)$ indexed by $\theta$ and with constant variance $\varsigma^2$ specifies
	\begin{align}
		\Upsilon^{(\ell)}&=\mu+\zeta(\theta^{(\ell)}), ~\ell=1,..., L,\label{eq:EAM:gauss_prior}\\
		Corr(\zeta(\theta),\zeta(\theta'))&=K_\beta(\theta-\theta'),~ \theta,\theta' \in \Theta\label{eq:EAM:corr},
	\end{align}
	where $\Upsilon^{(\ell)}=c(\theta^{(\ell)})$ and $K_\beta$ is a kernel with parameter vector $\beta \in \bigtimes_{h=1}^d[\underline{\beta}_h,\overline{\beta}_h]\subset \R^d_{++}$; e.g., $K_\beta(\theta-\theta')=\exp(-\sum_{h=1}^d|\theta_h-\theta'_h|^{2}/\beta_h)$. The unknown parameters $(\mu,\varsigma^2)$ can be estimated by running a GLS regression of $\mathbf\Upsilon=(\Upsilon^{(1)},...,\Upsilon^{(L)})'$ on a constant with the given correlation matrix. The unknown parameters $\beta$ can be estimated by a (concentrated) MLE.

	The (best linear) predictor of the critical value and its gradient at $\theta$ are then given by
	\begin{align}
		c_L(\theta)&=\hat\mu+\mathbf{r}_L(\theta)'\mathbf{R}_L^{-1}(\mathbf \Upsilon-\hat\mu \mathbf 1),\label{eq:cLdef}\\
		\nabla_\theta c_L(\theta)&=\hat\mu+\mathbf{Q}_L(\theta)\mathbf{R}_L^{-1}(\mathbf \Upsilon-\hat\mu \mathbf 1),\label{eq:cLgrad}
	\end{align}
	where $\mathbf r_L(\theta)$ is a vector whose $\ell$-th component is $Corr(\zeta(\theta),\zeta(\theta^{(\ell)}))$ as given above with estimated parameters, $\mathbf Q_L(\theta)=\nabla_\theta \mathbf r_L(\theta)'$, and $\mathbf R_L$ is an $L$-by-$L$ matrix whose $(\ell,\ell')$ entry is $Corr(\zeta(\theta^{(\ell)}),\zeta(\theta^{(\ell')}))$ with estimated parameters.
	This surrogate model has the property that its predictor satisfies $c_L(\theta^{(\ell)})=c(\theta^{(\ell)}), \ell=1,..., L$. Hence, it provides an analytical interpolation, with analytical gradient, of evaluation points of $c(\cdot)$.\footnote{See details in \cite{jonesefficient1998}. We use the DACE MATLAB kriging toolbox (\url{http://www2.imm.dtu.dk/projects/dace/}) for this step in our empirical application and Monte Carlo experiments.} The uncertainty left in $c(\cdot)$ is captured by the variance
\begin{align}
\hat\varsigma^2s^2_L(\theta)= \hat\varsigma^2\left(1-\mathbf r_L(\theta)'\mathbf R_L^{-1}\mathbf r_L(\theta)+\frac{(1-\mathbf 1'\mathbf R_L^{-1}\mathbf r_L(\theta))^2}{\mathbf 1'\mathbf R_L^{-1}\mathbf 1}\right). \label{eq:variance}
\end{align}

\noindent\underline{\textbf{M-step:}} With probability $1-\epsilon$, obtain the next evaluation point $\theta^{(L+1)}$ as
	\begin{align}
		\theta^{(L+1)}\in\argmax_{\theta\in\Theta}\EI_L(\theta)=\argmax_{\theta\in\Theta} (p'\theta-p'\theta^{*,L})_+\Big(1-\Phi\Big(\frac{\bar g(\theta)-c_L(\theta)}{\hat\varsigma s_L(\theta)}\Big)\Big),\label{eq:max_ei}
	\end{align}
where $\EI_L(\theta)$ is the \emph{expected improvement function}.\footnote{Heuristically, $\EI_L(\theta)$ is the expected improvement gained from analyzing parameter value $\theta$ for a Bayesian whose current beliefs about $c$ are described by the estimated model. Indeed, for each $\theta$, the maximand in \eqref{eq:max_ei} multiplies improvement from learning that $\theta$ is feasible with this Bayesian's probability that it is.} 
	This step can be implemented by standard nonlinear optimization solvers, e.g. MATLAB's \verb1fmincon1 or \verb1KNITRO1 (see Appendix \ref{sec:Mstep_nlp} for details). With probability $\epsilon$, draw $\theta^{(L+1)}$ randomly from a uniform distribution over $\Theta$. Set $L \leftarrow L+1$ and return to the E-step.
\medskip

The algorithm yields an increasing sequence of tentative optimal values $p'\theta^{*,L},L=k+1,k+2,...$, with $\theta^{*,L}$ satisfying the \emph{true} constraints in \eqref{eq:general_problem} but the sequence of evaluation points leading to it obtained by maximization of expected improvement defined with respect to the \emph{approximated} surface. Once a convergence criterion is met, $p'\theta^{*,L}$ is reported as the end point of $CI_n$. We discuss convergence criteria in Appendix \ref{sec:MC}.

 The advantages of E-A-M are as follows. First, we control the number of points at which we evaluate the critical value; recall that this evaluation is the expensive step. Also, the initial $k$ evaluations can easily be parallelized. For any additional E-step, one needs to evaluate $c(\cdot)$ only at a single point $\theta^{(L+1)}$.
 The M-step is crucial for reducing the number of additional evaluation points.
  To determine the next evaluation point, it trades off ``exploitation'' (i.e. the benefit of drawing a point at which the optimal value is high) against ``exploration'' (i.e. the benefit of drawing a point in a region in which the approximation error of $c$ is currently large) through maximizing expected improvement.\footnote{It is also possible to draw multiple points in each iteration \citep{Schonlau_Global_1998}, as we do in our implementation of the method.} Finally, the algorithm simplifies the M-step by providing constraints and their gradients for program \eqref{eq:max_ei} in closed form, thus greatly aiding fast and stable numerical optimization. The price is the additional approximation step. In the empirical application in Section \ref{sec:KT15} and in the numerical exercises of Appendix \ref{sec:MC}, this price turns out to be low.
\subsection{Choice of Tuning Parameters}
\label{sec:expl_rho}
Practical implementation of calibrated projection and the E-A-M algorithm is detailed in \cite{KMST_code}. It involves setting several tuning parameters, which we now discuss.

Calibration of $\hat{c}_n$ in \eqref{eq:def:c_hat} must be tuned at two points, namely the use of GMS and the choice of $\rho$. The trade-offs in setting these tuning parameters are apparent from inspection of \eqref{eq:Lambda_n}. GMS is parameterized by a shrinkage function $\varphi$ and a sequence $\kappa_n$ that controls the rate of shrinkage. In practice, choice of $\kappa_n$ is more delicate. A smaller $\kappa_n$ will make $\Lambda_n^b$ larger, hence increase bootstrap coverage probability for any given $c$, hence reduce $\hat{c}_n$ and therefore make for shorter confidence intervals -- but the uniform asymptotics will be misleading, and finite sample coverage therefore potentially off target, if $\kappa_n$ is too small. We follow the industry standard set by AS and recommend $\kappa_n=\sqrt{\log n}$.

The trade-off in choosing $\rho$ is similar but reversed. A larger $\rho$ will expand $\Lambda_n^b$ and therefore make for shorter confidence intervals, but (our proof of) uniform validity of inference requires $\rho<\infty$. Indeed, calibrated projection with $\rho=0$ will disregard any projection conservatism and (as is easy to show) exactly recovers projection of the AS confidence set. Intuitively, we then want to choose $\rho$ large but not too large.

To this end, we heuristically calibrate $\rho$ based on how much conservative distortion one is willing to accept in well-behaved cases. This distortion -- denote it $\eta$, for which we suggest a numerical value of $0.01$ -- is compared against a bound on conservative distortion that is itself likely to be conservative but data free and trivial to compute. In particular, we set
\begin{align}
\label{eq:beta_rho}\small
\rho=\Phi^{-1}\left(\tfrac{1}{2}+\tfrac{1}{2}\left(1-\eta/\tbinom{J_1+J_2}{d}\right)^{1/d}\right).
\end{align}
The underlying heuristic is as follows: If all basic solutions (i.e., intersections of exactly $d$ constraints) that potentially define vertices of $\Lambda^b_n$ realize inside the $\rho$-box, then the $\rho$-box cannot affect the values in \eqref{eq:event_kms} and hence not whether coverage obtains in a given bootstrap sample. Conversely, the probability that at least one basic solution realizes outside the $\rho$-box bounds from above the conservative distortion. This probability is, of course, dependent on unknown parameters. Our data free approximation imputes multivariate standard normal distributions for all basic solutions and Bonferroni adjustment to handle their covariation.\footnote{To reproduce the expression, recall that if $a\equiv\tbinom{J_1+J_2}{d}$ random variables in $\mathbb{R}^d$ are individually multivariate standard normal, then a Bonferroni upper bound on the probability that \textit{not} all of them realize inside the $\rho$-box equals
$a\bigl(1-\left(1-2\Phi(-\rho)\right)^d\bigr).$ Also, if Bonferroni is replaced with an independence assumption, the expression changes to $\rho=\Phi^{-1}\bigl(\tfrac{1}{2}+\tfrac{1}{2}(1-\eta)^{1/ad}\bigr)$. The numerical difference is negligible for moderate $J_1+J_2$.}

The E-A-M algorithm also has two tuning parameters.
One is $k$, the initial number of evaluation points.
The other is $\epsilon$, the probability of drawing $\theta^{(L+1)}$ randomly from a uniform distribution on $\Theta$ instead of by maximizing $\mathbb{EI}_L$.
In calibrated projection use of the E-A-M algorithm there is a single ``black box" function, $\hat{c}_n(\theta)$. We therefore suggest setting $k=10d+1$, similarly to the recommendation in \cite[p. 473]{jonesefficient1998}. In our Monte Carlo exercises we experimented with larger values, e.g. $k=20d+1$, and found that the increased number had no noticeable effect on the computed $CI_n$.
If a user applies our E-A-M algorithm to a constrained optimization problem with \emph{many} ``black box" functions to approximate, we suggest using a larger number of initial points.

The role of $\epsilon$ \citep[e.g.,][p. 2889]{Bull_Convergence_2011} is to trade off the greediness of the $\mathbb{EI}_L$ maximization criterion with the overarching goal of global optimization. \cite[pp. 28-29]{sut:bar98} explore the effect of setting $\epsilon=0.1$ and $0.01$ on different optimization problems, and find that for sufficiently large $L$, $\epsilon=0.01$  performs  better.
In our own simulations we have found that drawing \emph{both} a uniform point and computing the value of $\theta$ for each $L$ (thereby sidestepping the choice of $\epsilon$) is fast and accurate, and that is what we recommend doing.

\section{Theoretical Results}
\label{sec:main}
\subsection{Asymptotic Validity of Inference}
\label{sec:main:res}
In this section we establish that $CI_n$ is uniformly asymptotically valid in the sense of ensuring that \eqref{eq:asy_size} equals at least $1-\alpha$. The result applies to: (i) Confidence intervals for one projection; (ii) joint confidence regions for several projections, in particular confidence hyperrectangles for subvectors; (iii) confidence intervals for smooth nonlinear functions $f:\Theta \mapsto \R$. Examples of the latter extension include policy analysis and estimation of partially identified counterfactuals as well as demand extrapolation subject to rationality constraints.\footnote{In Appendix \ref{cor:math:proj}, we show that the result actually applies to the mathematical projection in \eqref{eq:def_with_gaps}.}
\begin{theorem}\label{thm:validity}
Suppose Assumptions \ref{as:momP_AS}, \ref{as:GMS}, \ref{as:correlation}, \ref{as:momP_KMS}, and \ref{as:bcs1} hold. Let $0<\alpha < 1/2$.
\begin{enumerate}[label=(\Roman*)]
\item \label{thm:validity:basic} Let $CI_n$ be as defined in \eqref{eq:def_ci}, with $\hat{c}_n$ as in \eqref{eq:def:c_hat}. Then:
\begin{align}
\liminf_{n\to\infty}\inf_{P\in\mathcal P}\inf_{\theta\in\Theta_I(P)}P(p'\theta\in CI_n)\ge 1-\alpha.\label{eq:coverage_control}
\end{align}
\item \label{cor:thm:validity} Let $p^1,\dots,p^h$ denote unit vectors in $\R^d$, $h \le d$. Then:\begin{align}
\liminf_{n\to\infty}\inf_{P\in\mathcal P}\inf_{\theta\in\Theta_I(P)}P(p^{k \prime}\theta \in CI_{n,k},k=1,\dots,h)\ge 1-\alpha,
\end{align}
where $CI_{n,k}=\left[\inf_{\theta\in\mathcal C_n(\hat{c}^h_n)} p^{k \prime}\theta, \sup_{\theta\in\mathcal C_n(\hat{c}^h_n)} p^{k \prime }\theta \right]$ and $\hat c^h_n(\theta)\equiv \inf\{c\in\mathbb R_+:P^*(\Lambda_n^b (\theta ,\rho ,c)\cap \{\cap_{k=1}^h\{p^{k \prime }\lambda =0\}\}\neq \emptyset)\ge 1-\alpha\}$.
\item \label{thm:nonlinear}	 Let $CI_n^f$ be a confidence interval whose lower and upper points are obtained solving
\begin{align*}
\inf_{\theta \in \Theta} / \sup_{\theta \in \Theta} f(\theta) \text{  s.t. }\sqrt{n}\bar{m}_{n,j}(\theta )/\hat{\sigma}_{n,j}(\theta )\leq
\hat{c}^f_n(\theta),~j=1,... ,J,
\end{align*}
where $\hat c^f_n(\theta)\equiv \inf\{c \geq 0:P^*(\Lambda_n^b (\theta ,\rho ,c)\cap \{\|\nabla_\theta f(\theta )\|^{-1}\nabla_\theta f(\theta )\lambda =0\}\neq \emptyset)\ge 1-\alpha\}$.
Suppose that there exist $\varpi>0$ and $M<\infty$ such that $\inf_{P\in\mathcal P}\inf_{\theta\in\Theta_I(P)}\Vert \nabla f(\theta)\Vert\ge \varpi$ and $\sup_{\theta,\bar{\theta}\in\Theta}\Vert \nabla f(\theta)-\nabla f(\bar{\theta})\Vert\le M \Vert \theta -\bar{\theta} \Vert$, where $\nabla_\theta f(\theta)$ is the gradient of $f(\theta)$.\footnote{Because the function $f$ is known, these conditions can be easily verified in practice (especially if the first one is strengthened to hold over $\Theta$).} Let $0<\alpha < 1/2$. Then:
\begin{align}
\liminf_{n\to\infty}\inf_{P\in\mathcal P}\inf_{\theta\in\Theta_I(P)}P(f(\theta)\in CI_n^f)\ge 1-\alpha.
\end{align}
\end{enumerate}
\end{theorem}
All assumptions can be found in Online Appendix \ref{sec:AssRes}.
Assumptions \ref{as:momP_AS} and \ref{as:bcs1} are mild regularity conditions typical in the literature; see, e.g., Definition 4.2 and the corresponding discussion in BCS.
Assumption \ref{as:GMS} is based on AS and constrains the GMS function $\varphi(\cdot)$ as well as the rate at which $\kappa_n$ diverges.
Assumption \ref{as:momP_KMS} requires normalized population moments to be sufficiently smooth and consistently estimable.
Assumption \ref{as:correlation} is our key departure from the related literature.
In essence, it requires that the correlation matrix of the moment functions corresponding to close-to-binding moment conditions has
eigenvalues uniformly bounded from below.\footnote{Assumption \ref{as:correlation} allows for high correlation among moment inequalities that cannot cross. This covers equality constraints but also entry games as the ones studied in \cite{CilibertoTamer09}.}
Under this condition, we are able to show that in the limit problem corresponding to \eqref{eq:heuristic} --where constraints are replaced with their local linearization using population gradients and Gaussian processes-- the probability of coverage increases continuously in $c$. If such continuity is directly assumed (Assumption \ref{ass:continuity_limit_cov}), Theorem \ref{thm:validity} remains valid (Online Appendix \ref{app:proofs_alt_cont}). While the high level Assumption \ref{ass:continuity_limit_cov} is similar in spirit to a key condition (Assumption A.2) in BCS, we propose Assumption \ref{as:correlation} due to its familiarity and ease of interpretation; a similar condition is required for uniform validity of standard point identified Generalized Method of Moments inference. 
In Online Appendix \ref{sec:verify_3.3} we verify that our assumptions hold in some of the canonical examples in the partial identification literature:  mean with missing data, linear regression and best linear prediction with interval data (and discrete covariates), entry games with multiple equilibria (and discrete covariates), and semi-parametric binary regression models with discrete or interval valued covariates \citep[as in][]{mag:mau08}.

Assumptions \ref{as:momP_AS}-\ref{as:bcs1} define the class of DGPs over which our proposed method  yields uniformly asymptotically valid coverage. This class is non-nested with the class of DGPs over which the profiling-based methods of \cite{Romano_Shaikh2008aJSPIWP} and BCS are uniformly asymptotically valid. \cite[Section 4.2 and Supplemental Appendix F]{KMS_2017} show that in well behaved cases, calibrated projection and BCS-profiling are asymptotically equivalent. They also provide conditions under which calibrated projection has lower probability of false coverage in finite sample, thereby establishing that the two methods' finite sample power properties are non-ranked.

\subsection{Convergence of the E-A-M Algorithm}
\label{sec:main:conv}
We next provide formal conditions under which the sequence $p'\theta^{*,L}$ generated by the E-A-M algorithm converges to the true end point of $CI_n$ as $L\to\infty$ at a rate that we obtain.
Although $p'\theta^{*,L}=\max\{p'\theta^{(\ell)}:\ell \in\{1,..., L\}, \bar g(\theta)\le c(\theta^{(\ell)})\}$, so that $\theta^{*,L}$ satisfies the \emph{true} constraints for each $L$, the sequence of evaluation points $\theta^{(\ell)}$ is mostly obtained through expected improvement maximization (M-Step) with respect to the \emph{approximating} surface $c_L(\cdot)$.
Because of this, a requirement for convergence is that the function $c(\cdot)$ is sufficiently smooth, so that the approximation error in $|c(\theta)-c_L(\theta)|$ vanishes uniformly in $\theta$ as $L\to\infty$.\footnote{As in \cite{Bull_Convergence_2011}, our convergence result accounts for the fact that the parameters of the Gaussian process prior in \eqref{eq:EAM:gauss_prior} are re-estimated for each iteration of the A-step using the ``training data" $\{\theta^\ell,c(\theta^\ell)\}_{\ell=1}^L$.}
We furthermore assume that the constraint set in \eqref{eq:general_problem} satisfies a degeneracy condition introduced to the partial identification literature by \cite[Condition C.3]{CHT}.\footnote{\cite[eq. (4.6)]{CHT} impose the condition on the population identified set.}
In our application, the condition requires that $\cC_n(\hat{c}_n)$ has an interior and that the inequalities in \eqref{eq:CI}, when evaluated at points in a (small) $\tau$-contraction of $\cC_n(\hat{c}_n)$, are satisfied with a slack that is proportional to $\tau$.
Theorem \ref{thm:eiconv} below establishes that these conditions jointly ensure convergence of the E-A-M algorithm at a specific rate. This is a novel contribution to the literature on response surface methods for constrained optimization.

In the formal statement below, the expectation $E_{\bQ}$ is taken with respect to the law of $(\theta^{(1)},...,\theta^{(L)})$ determined by the Initialization step and the M-step but conditioning on the sample. We refer to Appendix \ref{app:EAM_conv} for a precise definition of $E_{\bQ}$ and a proof of the theorem.
\begin{theorem}\label{thm:eiconv}
	Suppose $\Theta \subset \R^d$ is a compact hyperrectangle with nonempty interior, that $\|p\|=1$, and that Assumptions \ref{as:kernel}, \ref{as:smoothness}, and \ref{as:degeneracy} hold.
	Let the evaluation points $(\theta^{(1)},\cdots,\theta^{(L)})$ be drawn according to the Initialization and M-steps. Then
\begin{align}
\|p'\theta^*-p'\theta^{*,L}\|_{L^1_{\mathbb Q}}=O\Big(\Big(\frac{L}{\ln L}\Big)^{-\nu/d}(\ln L)^\delta\Big),\label{eq:eiconv_thm}
\end{align}
where $\|\cdot\|_{L^1_{\mathbb Q}}$ is the $L^1$-norm under $\mathbb Q$, $\delta\ge 1+\chi,$ and the constants $0<\nu\le \infty$ and $0<\chi<\infty$ are defined in Assumption \ref{as:kernel}. If $\nu=\infty$, the statement in \eqref{eq:eiconv_thm} holds for any $\nu<\infty.$
\end{theorem}
The requirement that $\Theta$ is a compact hyperrectangle with nonempty interior can be replaced by a requirement that $\Theta$ belongs to the interior of a closed hyperrectangle in $\R^d$. Assumption \ref{as:kernel} specifies the types of kernel to be used to define the correlation functional in \eqref{eq:EAM:corr}.
Assumption \ref{as:smoothness} collects requirements on differentiability of $g_j(\theta),j=1,\dots,J$, and smoothness of $c(\theta)$.
Assumption \ref{as:degeneracy} is the degeneracy condition discussed above.

To apply Theorem \ref{thm:eiconv} to calibrated projection, we provide low level conditions (Assumption \ref{as:Lipschitz_m_over_sigma} in Online Appendix \ref{app:as:Lipschitz}) under which the map $\theta \mapsto \hat c_n(\theta)$ uniformly stochastically satisfies a Lipschitz-type condition. To get smoothness, we work with a mollified version of $\hat{c}_n$, denoted $\hat c_{n,\tau_n}$ in equation \eqref{eq:c_hat_mollified}, where $\tau_n=o(n^{-1/2})$.\footnote{For a discussion of mollification, see e.g. \cite[Example 7.19]{Rockafellar_Wets2005aBK}.} Theorem \ref{cor:eam_conv} in the Online Appendix shows that $\hat c_n$ and $\hat c_{n,\tau_n}$ can be made uniformly arbitrarily close, and that $\hat c_{n,\tau_n}$ yields valid inference as in \eqref{eq:coverage_control}. In practice, we directly apply the E-A-M steps to $\hat c_n$.

The key condition imposed in Theorem \ref{cor:eam_conv} is Assumption \ref{as:Lipschitz_m_over_sigma}. It requires that the GMS function used is Lipschitz in its argument,\footnote{\label{fn:gms}This requirement rules out the GMS function in footnote \ref{fn:eq:zeta}, but it is satisfied by other GMS functions proposed by AS.} and that the standardized moment functions are Lipschitz in $\theta$. In Online Appendix \ref{sec:verify_ass_EAM} we establish that the latter condition is satisfied by some canonical examples in the moment (in)equality literature: mean with missing data, linear regression and best linear prediction with interval data (and discrete covariates), entry games with multiple equilibria (and discrete covariates), and semi-parametric binary regression models with discrete or interval valued covariates \citep[as in][]{mag:mau08}.\footnote{For these same examples we verify the differentiability requirement in Assumption \ref{as:smoothness} on $g_j(\theta)$.}

\label{sentence:error:negligible}The E-A-M algorithm is proposed as a method to implement our statistical procedure, not as part of the statistical procedure itself. As such, its approximation error is not taken into account in Theorem \ref{thm:validity}. Our comparisons of the confidence intervals obtained through the use of E-A-M as opposed to directly solving problems \eqref{eq:CI} through the use of MATLAB's \texttt{fmincon} in our empirical application in the next section suggest that such error is minimal.

\section{Empirical Illustration: Estimating a Binary Game}
\label{sec:KT15}
We employ our method to revisit the study in \cite[Section 8]{KT15} of ``what explains the decision of an airline to provide service between two airports." We use their data and model specification.\footnote{The data, which pertains to the second quarter of the year 2010, is downloaded from \url{http://qeconomics.org/ojs/index.php/qe/article/downloadSuppFile/371/1173}.} Here we briefly summarize the set-up and refer to \cite{KT15} for a richer discussion.

The study examines entry decisions of two types of firms, namely Low Cost Carriers ($LCC$) versus Other Airlines ($OA$).  A market is defined as a trip between two airports, irrespective of intermediate stops.  The entry decision $Y_{\ell,i}$ of player $\ell\in\{LCC,OA\}$ in market $i$ is recorded as a $1$ if a firm of type $\ell$ serves market $i$ and $0$ otherwise.  Firm $\ell$'s payoff equals $Y_{\ell,i}(Z_{\ell,i}'\vartheta_\ell+\delta_i Y_{-\ell,i}+u_{\ell,i})$, where $Y_{-\ell,i}$ is the opponent's entry decision. Each firm enters if doing so generates non-negative payoffs.
The observable covariates in the vector $Z_{\ell,i}$ include the constant and the variables $W_i^{size}$ and $W_{\ell,i}^{pres}$. The former is market size, a market-specific variable common to all airlines in that market and defined as the population at the endpoints of the trip. The latter is a firm-and-market-specific variable measuring the market presence of firms of type $\ell$ in market $i$ \citep[see][p. 356 for its exact definition]{KT15}. While $W_i^{size}$ enters the payoff function of both firms, $W_{LCC,i}^{pres}$ (respectively, $W_{OA,i}^{pres}$) is excluded from the payoff of firm $OA$ (respectively, $LCC$). 
Each of market size and of the two market presence variables are transformed into binary variables based on whether they realized above or below their respective median.
This leads to a total of 8 market types, hence $J_1=16$ moment inequalities and $J_2=16$ moment equalities.
The unobserved payoff shifters $u_{\ell,i}$ are assumed to be i.i.d. across $i$ and to have a bivariate normal distribution with $E(u_{\ell,i})=0$, $Var(u_{\ell,i})=1$, and $Corr(u_{LCC,i},u_{OA,i})=r$ for each $i$ and $\ell\in\{LCC,OA\}$, where the correlation $r$ is to be estimated. Following \cite{KT15}, we assume that the strategic interaction parameters $\delta_{LCC}$ and $\delta_{OA}$ are negative, that $r\ge 0$, and that the researcher imposes these sign restrictions. To ensure that Assumption \ref{as:momP_KMS} is satisfied,\footnote{This assumption, common in the literature on projection inference, requires that $D_{P,j}(\theta)$ are Lipschitz in $\theta$ and have bounded norm. But $\partial(\{E_P[m_j(X,\cdot)]/\sigma_{P,j}(\cdot)\})/\partial r$ includes a denominator equal to $(1-r^2)^2$. As $r\to 1$, this leads to a violation of the assumption and to numerical instability.} we furthermore assume that $r\le 0.85$ and use this value as its upper bound in the definition of the parameter space.

The results of the analysis are reported in Table \ref{tab:empirical}, which displays $95\%$ nominal confidence intervals (our $CI_n$ as defined in equations \eqref{eq:def_ci_KMS}-\eqref{eq:CI}) for each parameter. The output of the E-A-M algorithm is displayed in the accordingly labeled column. The next column shows a robustness check, namely the output of MATLAB's \texttt{fmincon} function, henceforth labelled ``direct search," that was started at each of a widely spaced set of feasible points that were previously discovered by the E-A-M algorithm. We emphasize that this is a robustness or accuracy check, not a horse race: Direct search mechanically improves on E-A-M because it starts (among other points) at the point reported by E-A-M as optimal feasible. Using the standard \texttt{MultiStart} function in MATLAB instead of the points discovered by E-A-M produces unreliable and extremely slow results. 
In 10 out of 18 optimization problems that we solved, the E-A-M algorithm's solution came within its set tolerance ($0.005$) from the direct search solution. The other optimization problems were solved by E-A-M with a minimal error of less than $5\%$. 

Table \ref{tab:empirical} also reports computational time of the E-A-M algorithm, of the subsequent direct search, and the total time used to compute the confidence intervals. The direct search greatly increases computation time with small or negligible benefit.
Also, computational time varied substantially across components.
We suspect this might be due to the shape of the level sets of $\max_{j=1,\dots,J}\sqrt n \bar m_{n,j}(\theta)/\hat\sigma_{n,j}(\theta)$: By manually searching around the optimal values of the program, we verified that the level sets in specific directions can be extremely thin, rendering search more challenging.

Comparing our findings with those in \cite{KT15}, we see that the results qualitatively agree.
The confidence intervals for the interaction effects ($\delta_{LCC}$ and $\delta_{OA}$) and for the effect of market size on payoffs ($\vartheta^{size}_{LCC}$ and $\vartheta^{size}_{OA}$) are similar to each other across the two types of firms. 
The payoffs of $LCC$ firms seem to be impacted more than those of $OA$ firms by market presence.
On the other hand, monopoly payoffs for $LCC$ firms seem to be smaller than for $OA$ firms.\footnote{Monopoly payoffs are those associated with a market with below-median
size and below-median market presence (i.e., the constant terms).} 
The confidence interval on the correlation coefficient is quite large and includes our upper bound of 0.85.\footnote{Being on the boundary of the parameter space is not a problem for calibrated projection; indeed, it is accounted for in the calibration of $\hat{c}_n$ in equations \eqref{eq:Lambda_n}-\eqref{eq:def:c_hat}.}

For most components, our confidence intervals are narrower than the corresponding 95\% credible sets reported in \cite{KT15}.\footnote{For the interaction parameters $\delta$, Kline and Tamer's upper confidence points are lower than ours; for the correlation coefficient $r$, their lower confidence point is higher than ours.}
However, the intervals are not comparable for at least two reasons: We impose a stricter upper bound on $r$ and we aim to cover the projections of the true parameter value as opposed to the identified set.

Overall, our results suggest that in a reasonably sized, empirically interesting problem, calibrated projection yields informative confidence intervals. Furthermore, the E-A-M algorithm appears to accurately and quickly approximate solutions to complex smooth nonlinear optimization problems.

\section{Conclusion}
\label{sec:conclusion}
This paper proposes a confidence interval for linear functions of parameter vectors that are partially identified through finitely many moment (in)equalities.
The extreme points of our \emph{calibrated projection} confidence interval are obtained by minimizing and maximizing $p^\prime \theta$ subject to properly relaxed sample analogs of the moment conditions.
The relaxation amount, or critical level, is computed to insure uniform asymptotic coverage of $p^\prime \theta$ rather than $\theta$ itself. Its calibration is computationally attractive because it is based on repeatedly checking feasibility of (bootstrap) linear programming problems.
Computation of the extreme points of the confidence intervals is furthermore attractive thanks to an application of the response surface method for global optimization; this is a novel contribution of independent interest.
Indeed, one key result is a convergence rate for this algorithm when applied to constrained optimization problems in which the objective function is easy to evaluate but the constraints are ``black box" functions.
The result is applicable to any instance when the researcher wants to compute confidence intervals for optimal values of constrained optimization problems.
Our empirical application and Monte Carlo analysis show that, in the DGPs that we considered, calibrated projection is fast and accurate, and also that the E-A-M algorithm can greatly improve computation of other confidence intervals.

\ifx\undefined\BySame
\newcommand{\BySame}{\leavevmode\rule[.5ex]{3em}{.5pt}\ }
\fi
\ifx\undefined\textsc
\newcommand{\textsc}[1]{{\sc #1}}
\newcommand{\emph}[1]{{\em #1\/}}
\let\tmpsmall\small
\renewcommand{\small}{\tmpsmall\sc}
\fi

\newpage
\small
\begin{appendix}
\small
\section{Convergence of the E-A-M Algorithm}
\label{app:EAM_conv}
In this appendix, we provide details on the algorithm used to solve the outer maximization problem as described in Section \ref{sec:Computation}.
Below, let $(\Omega,\mathcal F)$ be a measurable space and $\omega$ a generic element of $\Omega$.
Let $L\in \mathbb N$ and let $(\theta^{(1)},...,\theta^{(L)})$ be a measurable map on $(\Omega,\mathcal F)$ whose law is specified below.
The value of the function $c$ in \eqref{eq:general_problem} is unknown ex ante. Once the evaluation points $\theta^{(\ell)},\ell=1,..., L$ realize, the corresponding values of $c$, i.e. $\Upsilon^{(\ell)}\equiv c(\theta^{(\ell)}),\ell=1,..., L$, are known.
We may therefore define the information set
\begin{align}
	\F_L\equiv\sigma(\theta^{(\ell)},\Upsilon^{(\ell)},\ell=1,...,L).
\end{align}
Let $\mathcal C_L\equiv\{\theta^{(\ell)}:\ell \in\{1,\cdots, L\}, g_j(\theta^{(\ell)})\le c(\theta^{(\ell)}),j=1,\cdots,J\}$
be the set of feasible evaluation points. Then $\text{argmax}_{\theta\in \mathcal C_L} p'\theta$ is measurable with respect to $\mathcal F_L$ and we take a measurable selection $\theta^{*,L}$ from it.

Our algorithm iteratively determines evaluation points based on the \emph{expected improvement} criterion \citep{jonesefficient1998}.
For this, we formally introduce a model that describes the  uncertainty associated with the values of $c$ outside the current evaluation points.
Specifically, the unknown function $c$ is modeled as a Gaussian process such that\footnote{We use $\bP$ and $\E$ to denote the probability and expectation for the prior and posterior distributions of $c$ to distinguish them from $P$ and $E$ used for the sampling uncertainty for $X_i$.}
\begin{align}
\E[c(\theta)]=\mu,~\mathbb{C}\text{ov}(c(\theta),c(\theta'))=\varsigma^2 K_{\beta}(\theta-\theta'),\label{eq:muK}
\end{align}
 where $\beta=(\beta_1,...,\beta_d)\in\mathbb R^d$ controls the length-scales of the process. Two values $c(\theta)$ and $c(\theta')$ are highly correlated when $\theta_k-\theta'_k$ is small relative to $\beta_k$.
Throughout,  we assume $\underline{\beta}_k\le \beta_k\le \overline\beta_k$ for some $0<\underline{\beta}_k<\overline{\beta}_k<\infty$ for $k=1,...,d$. We let $\bar\beta=(\bar\beta_1,...,\bar\beta_d)'\in\mathbb R^d$.  Specific suggestions on the forms of $K_\beta$ are given in Appendix \ref{sec:RKHS}.

For a given $(\mu,\varsigma,\beta)$, the posterior distribution of $c$ given $\mathcal F_L$ is then another Gaussian process whose mean $c_L(\cdot)$ and variance $\varsigma^2 s^2_L(\cdot)$ are given as follows \citep[][Section 4.1.3]{Santner:2013aa}:
\begin{align}
	c_L(\theta)&=\mu+\mathbf{r}_L(\theta)'\mathbf{R}_L^{-1}(\mathbf \Upsilon-\mu \mathbf 1)\\
	\varsigma^2s^2_L(\theta)&= \varsigma^2\biggl(1-\mathbf r_L(\theta)'\mathbf R_L^{-1}\mathbf r_L(\theta)+\frac{(1-\mathbf 1'\mathbf R_L^{-1}\mathbf r_L(\theta))^2}{\mathbf 1'\mathbf R_L^{-1}\mathbf 1}\biggr).
\end{align}

Given this, the expected improvement function can be written as
\begin{align}
\EI_L(\theta)&\equiv \E[(p'\theta-p'\theta^{*,L})_+1\{\bar g(\theta)\le c(\theta)\}|\F_L]\notag\\
&=(p'\theta-p'\theta^{*,L})_+\bP(c(\theta)\ge \max_{j=1,...,J}g_j(\theta)|\F_L)\notag\\
&=(p'\theta-p'\theta^{*,L})_+\bP\left(\frac{c(\theta)-c_L(\theta)}{\varsigma s_L(\theta)}\ge \frac{\max_{j=1,...,J}g_j(\theta)- c_L(\theta)}{\varsigma s_L(\theta)}\Big|\F_L\right)\notag\\
&=(p'\theta-p'\theta^{*,L})_+\left(1-\Phi\left(\frac{\bar g(\theta)-c_L(\theta)}{\varsigma s_L(\theta)}\right)\right),\label{eq:eidef}
\end{align}
The evaluation points $(\theta^{(1)},...,\theta^{(L)})$ are then generated according to the following algorithm (\textbf{M-step} in Section \ref{sec:Computation}).
\begin{algorithm}\label{alg:evalpts}
Let $k\in\mathbb N$.

\noindent
Step 1: Initial evaluation points $\theta^{(1)},...,\theta^{(k)}$ are drawn uniformly over $\Theta$  independent of $c$.

\noindent
Step 2: For $L\ge k$,	with probability $1-\epsilon$, let $\theta^{(L+1)}=\text{argmax}_{\theta\in\Theta}\EI_L(\theta).$ With probability $\epsilon$,  draw $\theta^{(L+1)}$  uniformly at random from $\Theta$.
\end{algorithm}
Below, we use  $\bQ$ to denote the law of $(\theta^{(1)},...,\theta^{(L)})$ determined by the algorithm above.
We also note that $\theta^{*,L+1}=\argmax_{\theta\in\mathcal C_{L+1}}p'\theta$ is a function of the evaluation points and therefore is a random variable whose law is governed by $\bQ$.
We let
\begin{align}
\mathcal C&\equiv\{\theta\in\Theta:\bar g(\theta)-c(\theta)\le 0\}. \label{eq:calc}
\end{align}

We require that the kernel used to define the correlation functional for the Gaussian process in \eqref{eq:EAM:corr} satisfies some basic regularity conditions. For this, let $\hat K_\beta=\int e^{-2\pi i x'\xi}K_\beta(x)dx$ denote the Fourier transform of $K_\beta$. Note also that, for real valued functions $f,g$, $f(y)=\mathit\Theta(g(y))$ means $f(y)=O(g(y))$ as $y\to\infty$ and $\liminf_{y\to\infty}f(y)/g(y)>0$.
\begin{assumption}[Kernel Function]\label{as:kernel}
(i) $K_\beta$ is continuous and integrable; (ii) $\hat K_\beta=\hat k_\beta(\|x\|)$ for some nonincreasing function $\hat k_\beta:\mathbb R_+\to\mathbb R_+$; (iii) As $x\to\infty$ either $\hat K_\beta(x)=\mathit{\Theta}(\|x\|^{-2\nu-d})$ for some $\nu>0$ or $\hat K_\beta(x)=O(\|x\|^{-2\nu-d})$ for all $\nu>0$; (iv) $K_\beta$ is $k$-times continuously differentiable for $k=\lfloor 2\nu\rfloor$, and at the origin $K$ has $k$-th order Taylor approximation $P_k$ satisfying $|K(x)-P_k(x)|=O(\|x\|^{2\nu}(-\ln\|x\|)^{2\chi})$ as $x\to 0$, for some $\chi>0.$
\end{assumption}

Assumption \ref{as:kernel} is essentially the same as Assumptions 1-4 in \cite{Bull_Convergence_2011}.
When a kernel satisfies the second condition of Assumption \ref{as:kernel} (iii), i.e. $\hat K_\beta(x)=O(\|x\|^{-2\nu-d}),\forall\nu>0$, we say $\nu=\infty.$
Assumption \ref{as:kernel} is satisfied by popular kernels such as the Mat\'{e}rn kernel (with $0<\nu<\infty$ and $\chi=1/2$) and the Gaussian kernel ($\nu=\infty$ and $\chi=0$). These kernels are discussed in Appendix \ref{sec:RKHS}.

Finally, we require that the functions $g_j$ are differentiable with continuous Lipschitz gradient,\footnote{This requirement holds in the canonical partial identification examples discussed in Online Appendix \ref{sec:verify_examples}, using the same arguments as in Online Appendix \ref{sec:verify_ass_EAM}, provided $\hat{\sigma}_{n,j}(\theta)>0$.} that the function $c$ is smooth, and we impose on the constraint set $\cC$ (which is a confidence set in our application) a degeneracy condition inspired by \cite[Condition C.3]{CHT}.\footnote{\cite{CHT} impose the degeneracy condition on the population identified set.}
Below $\mathcal H_{\beta}(\Theta)$ is the reproducing kernel Hilbert space (RKHS) on $ \Theta\subseteq \mathbb R^d$ determined by the kernel used to define the correlation functional in \eqref{eq:EAM:corr}. The norm on this space is $\|\cdot\|_{\mathcal H_{\beta}}$; see Online Appendix \ref{sec:RKHS} for details.
\begin{assumption}[Continuity and Smoothness]\label{as:smoothness}
(i) For each $j=1,\dots,J$, the function $g_j(\theta)$ is differentiable in $\theta$ with Lipschitz continuous gradient.
(ii) The function $c:\Theta \mapsto \mathbb{R}$ satisfies $\|c\|_{\mathcal H_{\bar\beta}}\le R$ for some $R>0$, where $\bar\beta=(\bar\beta_1,\cdots,\bar\beta_d)'$.
\end{assumption}
\begin{assumption}[Degeneracy]\label{as:degeneracy}
There exist constants $(C_1,M,\tau_1)$ such that for all $\varpi\in [0,\tau_1]$,
\begin{align*}
& \max_{j} g_j(\theta)-c(\theta)\le  -C_1\varpi, \text{ for all } \theta\in \mathcal C^{-\varpi},\\
& d_H(\mathcal C^{-\varpi},\mathcal C)\le M\varpi,
\end{align*}
where $\mathcal C^{-\varpi}\equiv\{\theta\in\mathcal C:d(\theta,\Theta\setminus \mathcal C)\ge \varpi\}$.
\end{assumption}
Assumptions \ref{as:smoothness}-\ref{as:degeneracy} jointly imply a linear minorant property on $\max_j(g_j(\theta)-c(\theta))_+$:
\begin{align}
\exists C_2>0,\tau_2>0: ~\max_j(g_j(\theta)-c(\theta))_+\ge C_2\min\{d(\theta,\mathcal C),\tau_2\}.\label{as:pm}
\end{align}
To see this, define $f_j(\theta)\equiv g_j(\theta)-c(\theta)$, so that the l.h.s. of the above inequality is $\max_j f_j(\theta)$. By Assumptions \ref{as:smoothness}-\ref{as:degeneracy} and compactness of $\Theta$,  $f_j(\cdot)$ is differentiable with Lipschitz continuous gradient. Let $\tilde{D}_j(\cdot)$ denote its gradient and let $\tilde{M}$ denote the corresponding Lipschitz constant. Let $\varepsilon=C_1/(M\tilde{M}J)$, where $(C_1,M)$ are from Assumption \ref{as:degeneracy}. We will show that, for constants $(C_2,\tau_2)$ to be determined, (i) $d(\theta,\cC) \leq \varepsilon \Rightarrow \max_j f_j(\theta) \geq C_2 d(\theta,\cC)$ and (ii) $d(\theta,\cC) \geq \varepsilon \Rightarrow \max_j f_j(\theta) \geq C_2 \tau_2$, so that the minimum between these bounds applies to any $\theta$.

To see (i), write $\theta=\theta^*+r$, where $\theta^*$ is the projection of $\theta$ onto $\cC$. Fix a sequence $\varpi_m \to 0$. By assumption \ref{as:degeneracy}, there exists a corresponding sequence $\theta^*_m \to \theta^*$ with (for $m$ large enough) $\Vert \theta^*_m - \theta^* \Vert \leq M\varpi_m$ but also $\max_j f_j(\theta^*_m) \leq -C_1 \varpi_m$. Let $t_m \equiv (\theta^*_m - \theta^*)/\Vert \theta^*_m-\theta^* \Vert$ be the sequence of corresponding directions. Then for any accumulation point $t$ of $t_m$ and any active constraint $j$ (i.e., $f_j(\theta^*)=0$; such $j$ necessarily exists due to continuity of $f_j(\cdot)$), one has $\tilde{D}_j(\theta^*)t \leq -C_1/M$.
We note for future reference that this finding implies $\Vert \tilde{D}_j(\theta^*) \Vert \geq C_1/M$. It also implies that the Mangasarian-Fromowitz constraint qualification holds at $\theta^*$, hence $r$ (being in the normal cone of $\cC$ at $\theta^*$) is in the positive span of the active constraints' gradients. Thus $j$ can be chosen such that $f_j(\theta^*)=0$ and $\tilde{D}_j(\theta^*)r \geq \Vert \tilde{D}_j(\theta^*)\Vert \Vert r \Vert / J$. For any such $j$, write
\begin{eqnarray*}
f_j(\theta) &=& f_j(\theta^*)+ \int_0^1 \frac{df_j(\theta^*+kr)}{dk}dk \\
&=& 0+ \int_0^1 \tilde{D}_j(\theta^*+kr)r dk \\
&=& \int_0^1 \left( \tilde{D}_j(\theta^*)r + \bigl(\tilde{D}_j(\theta^*+kr) - \tilde{D}_j(\theta^*)\bigr)r\right) dk \\
&\geq & \Vert\tilde{D}_j(\theta^*)\Vert \Vert r \Vert /J + \int_0^1 (-\tilde{M}k\Vert r \Vert)\Vert r \Vert
dk \\
&\geq &\tfrac{C_1}{MJ} \Vert r \Vert - \tilde{M}\Vert r \Vert^2/2 \\
&\geq &\tfrac{C_1}{2MJ} \Vert r \Vert.
\end{eqnarray*}
In the inequality steps, we successively substituted bounds stated before the display, evaluated the integral in $k$, and (in the last step) used $\Vert r \Vert \leq \varepsilon$. This establishes (i), where $C_2=C_1/(2MJ)$.
Next, by continuity of $\max_j f_j(\cdot)$ and compactness of the constraint set, $\tau \equiv \min_\theta \{\max_j f_j(\theta):d(\theta,\cC) \geq \varepsilon\}$ is well-defined and strictly positive. This establishes (ii) with $\tau_2=\tau/C_2$.

\subsection{Proof of Theorem \ref{thm:eiconv}}
For each $L\in\mathbb N$, let
\begin{align}
	r_L\equiv \Big(\frac{L}{\ln L}\Big)^{-\nu/d}( \ln L)^\chi.
\end{align}
\begin{proof}[Proof of Theorem \ref{thm:eiconv}]
	First, note that
\begin{align}
\|p'\theta^*-p'\theta^{*,L}\|_{L^1_{\mathbb Q}}=E_{\bQ}\big[\left|p'\theta^*-p'\theta^{*,L}\right|\big]	=	 E_{\bQ}\big[p'\theta^*-p'\theta^{*,L}\big],
\end{align}
where the last equality follows form $p'\theta^*-p'\theta^{*,L+1}\ge 0, \bQ-a.s.$ Hence, it suffices to show
\begin{align}
	E_{\bQ}\big[p'\theta^*-p'\theta^{*,L}\big]=O\Big(\Big(\frac{L}{\ln L}\Big)^{-\nu/d}( \ln L)^\delta\Big).
\end{align}

Let $(\Omega,\mathcal F)$ be a measurable space. Below, we let $L\ge 2k$.
Let $0<\nu<\infty$. Let $0<\eta<\epsilon$ and $A_L\in\mathcal F$ be the event that at least $\lfloor \eta L\rfloor$ of the points $\theta^{(k+1)},\cdots,\theta^{(L)}$ are drawn independently from a uniform distribution on $\Theta.$ Let $B_L\in\mathcal F$ be the event that one of the points $\theta^{(L+1)},\cdots,\theta^{(2L)}$ is chosen by maximizing the expected improvement. For each $L$, define the mesh norm:
\begin{align}
h_L\equiv \sup_{\theta\in\Theta}\min_{\ell=1,\cdots L}\|\theta-\theta^{(\ell)}\|.\label{eq:eiconv1}
\end{align}
For a given $\bar M>0$, let $C_L\in\mathcal F$ be the event that $h_L\le \bar M(L/\ln L)^{-1/d}$. We then let
\begin{align}
D_L\equiv A_L \cap B_L\cap C_L.	\label{eq:eiconv2}
\end{align}
For each $\omega\in D_L$, let
\begin{align}
\ell(\omega,L)\equiv\inf\{\tilde \ell\in\mathbb N:L\le\tilde\ell\le 2L,\theta^{(\tilde\ell)}\in\argmax_{\theta\in\Theta}\EI_{\tilde\ell-1}(\theta)\}.\label{eq:eiconv2a}
\end{align}
This is a (random) index that is associated with the first maximizer of the expected improvement between $L$ and $2L$.

Let  $\varepsilon_L=(L/\ln L)^{-\nu/d}(\ln L)^\delta$ for $\delta\ge 1+\chi$ and note that $\varepsilon_L$ is a positive sequence such that $\varepsilon_L\to 0$ and $r_L=o(\varepsilon_L)$.
We further define the following events:
\begin{align}
E_{1L}&\equiv\{\omega\in\Omega:0<\bar g(\theta^{(\ell(\omega,L))})-c(\theta^{(\ell(\omega, L))})\le \varepsilon_{\ell(\omega,L)}\}\label{eq:eiconv4a}\\
E_{2L}&\equiv\{\omega\in\Omega:-\varepsilon_{\ell(\omega,L)}\le\bar g(\theta^{(\ell(\omega,L))})-c(\theta^{(\ell(\omega, L))})<0\}\label{eq:eiconv4b}\\
E_{3L}&\equiv\{\omega\in\Omega:|\bar g(\theta^{(\ell(\omega,L))})-c(\theta^{(\ell(\omega, L))})|>\varepsilon_{\ell(\omega,L)}\}.\label{eq:eiconv4c}
\end{align}
Note that $D_L$ can be partitioned into $D_L\cap E_{1L}$, $D_L\cap E_{2L}$, and $D_L\cap E_{3L}.$
By Lemmas \ref{lem:star-current}, \ref{lem:star-current_inner}, and \ref{lem:case3}, there exists a constant $M>0$ such that, respectively,
\begin{align}
&\sup_{\omega\in D_L\cap E_{1L}}|p'\theta^*-p'\theta^{*,\ell(\omega,L)}|/\varepsilon_{\ell(\omega,L)}\le M\label{eq:eiconv5a}\\
&\sup_{\omega\in D_L\cap E_{2L}}|p'\theta^*-p'\theta^{*,\ell(\omega,L)}|/\varepsilon_{\ell(\omega,L)}\le M\label{eq:eiconv5b}\\
&\sup_{\omega\in D_L\cap E_{3L}}|p'\theta^*-p'\theta^{*,\ell(\omega,L)}|/\exp(-M\eta_{\ell(\omega,L)})\le M,\label{eq:eiconv5c}
\end{align}
where $\eta_L\equiv\varepsilon_L/r_L$.  Note that
\begin{align}
	\eta_L=\varepsilon_L/r_L=(\ln L)^{\delta-\chi}.\label{eq:eiconv6}
\end{align}
Hence, by taking $M$ sufficiently large so that $M> \nu/d$,
\begin{align}
\exp(-M\eta_L)=\exp\left( -M(\ln L)^{\delta-\chi}\right)\le \exp\left(-M\ln L\right) =L^{-M}=O(L^{-\nu/d})=O(\varepsilon_L),\label{eq:eiconv6a}
\end{align}
where the inequality follows from $M(\ln L)^{\delta-\chi}\ge M\ln L$ by $\delta\ge 1+\chi$.
By \eqref{eq:eiconv5a}-\eqref{eq:eiconv6a},
\begin{align}
	\sup_{\omega\in D_L}|p'\theta^*-p'\theta^{*,\ell(\omega,L)}|/\varepsilon_{\ell(\omega,L)}\le M,\label{eq:eiconv6b}
\end{align}
for some constant $M>0$ for all $L$ sufficiently large.
Since $L\le \ell(\omega,L) \le 2L$, $p'\theta^{*,L}$ is non-decreasing in $L$, and $\varepsilon_L$ is non-increasing in $L$, we have
\begin{align}
p'\theta^*-p'\theta^{*,2L}	\le M (L/\ln L)^{-\nu/d}(\ln L)^\delta\le M(2L/\ln 2L)^{-\nu/d}(\ln 2L)^\delta \label{eq:eiconv8}
\end{align}
where the last equality follows from $L^{-\nu/d}= 2^{\nu/d} (2L)^{-\nu/d}$ and $\ln L\le \ln 2L$.

Now consider the case $\omega\notin D_L$. By \eqref{eq:eiconv2},
\begin{align}
\bQ(D_L^c)\le \bQ(A^c_L)+\bQ(B^c_L)+\bQ(C^c_L).\label{eq:eiconv9}
\end{align}
Let $Z_\ell$ be a Bernoulli random variable such that $Z_\ell=1$ if $\theta^{(\ell)}$ is randomly drawn from a uniform distribution.
Then, by the Chernoff bounds \citep[see e.g.][p.48]{boucheron2013concentration},
\begin{align}
	\bQ(A^c_L)=\bQ(\sum_{\ell=k+1}^LZ_\ell< \lfloor \eta L\rfloor)\le \exp(-(L-k+1)\epsilon(\epsilon-\eta)^2/2).\label{eq:eiconv10}
\end{align}
Further, by the definition of $B_L$,
\begin{align}
\bQ(B^c_L)	=\epsilon^L,\label{eq:eiconv11}
\end{align}
and finally by taking $\bar M$ large upon defining the event $C_L$ and applying Lemma 12 in \cite{Bull_Convergence_2011},
one has
\begin{align}
\bQ(C^c_L) =O(L^{-\gamma}),	\label{eq:eiconv12}
\end{align}
for any $\gamma>0$. Combining \eqref{eq:eiconv9}-\eqref{eq:eiconv12}, for any $\gamma>0$,
\begin{align}
\bQ(D^c_L)=	O(L^{-\gamma}).\label{eq:eiconv13}
\end{align}
Finally, noting that $p'\theta^*-p'\theta^{*,2L}$ is bounded by some constant $M>0$ due to the boundedness of $\Theta$, we have
\begin{multline}
E_{\bQ}\big[p'\theta^*-p'\theta^{*,2L}\big]=\int_{D_L}	p'\theta^*-p'\theta^{*,2L}d\bQ+\int_{D_L^c}p'\theta^*-p'\theta^{*,2L}d\bQ\\
= O((2L/\ln 2L)^{-\nu/d}(\ln 2L)^\delta)+O(2L^{-\gamma}),\label{eq:eiconv14}
\end{multline}
where the second equality follows from \eqref{eq:eiconv8} and \eqref{eq:eiconv13}. Since $\gamma>0$ can be made aribitrarily large, one may let the second term on the right hand side of \eqref{eq:eiconv14} converge to 0 faster than the first term. Therefore
\begin{align}
	E_{\bQ}\big[p'\theta^*-p'\theta^{*,2L}\big]= O((2L/\ln 2L)^{-\nu/d}(\ln 2L)^\delta),
\end{align}
which establishes the claim of the theorem for $0<\nu<\infty$. When the second condition of Assumption \ref{as:kernel} (iii) holds (i.e.,  $\nu=\infty$), the argument above holds for any $0<\nu<\infty.$
\end{proof}

\subsection{Auxiliary Lemmas for the Proof of Theorem \ref{thm:eiconv}}

Let $D_L$ be defined as in \eqref{eq:eiconv2}. The following lemma shows that on $D_L\cap E_{1L}$, $p'\theta^*$ and $p'\theta^{(\ell(\omega,L))}$ are close to each other, where we recall that $\theta^{(\ell(\omega,L))}$ is the expected improvement maximizer (but does not belong to $\cC$ for $\omega \in E_{1L}$).
\begin{lemma}\label{lem:star_ell}
Suppose Assumptions \ref{as:kernel}, \ref{as:smoothness}, and \ref{as:degeneracy} hold.	Let $\varepsilon_L$ be a positive sequence such that $\varepsilon_L\to 0$ and $r_L=o(\varepsilon_L)$.
Then, there exists a constant $M>0$ such that $\sup_{\omega\in D_L\cap E_{1L}}|p'\theta^*-p'\theta^{(\ell(\omega,L))}|/\varepsilon_{\ell(\omega,L)}\le M$  for all $L$ sufficiently large.
\end{lemma}

\begin{proof}
We show the result by contradiction.
Let $\{\omega_L\}\subset\Omega$  be a  sequence such that $\omega_L\in D_L\cap E_{1L}$ for all $L$.
First, assume that, for any $M>0$, there is a  subsequence  such that $|p'\theta^*-p'\theta^{(\ell(\omega_L,L))}|> M \varepsilon_{\ell(\omega_L,L)}$ for all $L$. This occurs if it contains a further subsequence along which, for all $L$, (i) $p'\theta^{(\ell(\omega_L,L))} - p'\theta^*>M\varepsilon_{\ell(\omega_L,L)}$ or (ii) $p'\theta^*-p'\theta^{(\ell(\omega_L,L))}>M\varepsilon_{\ell(\omega_L,L)}$.

\bigskip
\noindent
Case (i): $p'\theta^{(\ell(\omega_L,L))}-p'\theta^*>M\varepsilon_{\ell(\omega_L,L)}$ for all $L$ for some subsequence.

To simplify notation, we select a further subsequence $\{a_L\}$ of $\{L\}$  such that
for any $a_L<a_{L'}$, $\ell(\omega_{a_L},a_L)<\ell(\omega_{a_{L'}},a_{L'})$. This then  induces a sequence $\{\theta^{(\ell)}\}$ of expected improvement maximizers such that  $p'\theta^{(\ell)}-p'\theta^*>M\varepsilon_\ell$ for all $\ell,$ where each $\ell$ equals $\ell(\omega_{a_L},a_L)$ for some $a_L\in\mathbb N$. In what follows, we therefore omit the arguments of $\ell$, but this sequence's dependence on $(w_{a_L},a_L)$ should be implicitly understood.

Recall that $\cC$ defined in equation \eqref{eq:calc} is a compact set and that $\Pi_{\mathcal C}\theta^{(\ell)}=\argmin_{\theta\in\mathcal C}\|\theta^{(\ell)}-\theta\|$ denotes the projection of $\theta^{(\ell)}$ on $\mathcal C$. 
Then
\begin{align}
 p'\theta^{(\ell)}-p'\theta^*&= (p'\theta^{(\ell)}-p'\Pi_{\mathcal C}\theta^{(\ell)})+(p'\Pi_{\mathcal C}\theta^{(\ell)}-p'\theta^*)\notag\\
&\le \|p\|\|\theta^{(\ell)}-\Pi_{\mathcal C}\theta^{(\ell)}\|+(p'\Pi_{\mathcal C}\theta^{(\ell)}-p'\theta^*)\le d(\theta^{(\ell)},\mathcal C)
,\label{eq:star_ell-1}
\end{align}
where the first inequality follows from the Cauchy-Schwarz inequality, and the second inequality follows from $p'\Pi_{\mathcal C}\theta^{(\ell)}-p'\theta^*\le 0$ due to $\Pi_{\mathcal C}\theta^{(\ell)} \in \mathcal C$.
Therefore, by equation \eqref{as:pm}, for any $M>0$
\begin{align}
\bar{g}(\theta^{(\ell)})-c(\theta^{(\ell)})_+\ge C_2d(\theta^{(\ell)},\mathcal C)> C_2M\varepsilon_\ell,\label{eq:star_ell0}
\end{align}
for all $\ell$ sufficiently large, where the last inequality follows from $p'\theta^{(\ell)}-p'\theta^*>M\varepsilon_\ell$. Take $M$ such that $C_2M>1$. Then $(\bar g(\theta^{(\ell)})-c(\theta^{(\ell)}))/\varepsilon_\ell>C_2M>1$ 
for all $\ell$ sufficiently large, contradicting $\omega_L\in E_{1L}$.

\bigskip
\noindent
Case (ii): Similar to Case (i), we work with a further subsequence along which $p'\theta^*-p'\theta^{(\ell)}>M\varepsilon_\ell$ for all $\ell$. Recall that along this subsequence, $\theta^{(\ell)} \notin \mathcal C$ because $0<\bar g(\theta^{(\ell)})-c(\theta^{(\ell)})\le \varepsilon_\ell$. We will construct $\tilde\theta^{(\ell)}\in\mathcal C^{-\varepsilon_\ell}$ s.t. $\mathbb{EI}_{\ell-1}(\tilde\theta^{(\ell)})>\mathbb{EI}_{\ell-1}(\theta^{(\ell)})$, contradicting the definition of $\theta^{(\ell)}$.

By Assumption \ref{as:degeneracy},
\begin{align}
d_H(\mathcal C^{-\varepsilon_\ell},\mathcal C)\le M \varepsilon_\ell,\label{eq:star_ell1}
\end{align}
for all $\ell$ such that $\varepsilon_\ell\le \tau_1$.
By the Cauchy-Schwarz inequality, for any $\tilde\theta$,
\begin{align}
p'\theta^*-p'\tilde\theta\le \|p\|\|\theta^*-\tilde\theta\|.\label{eq:star_ell2}
\end{align}
Therefore, minimizing both sides with respect to $\tilde \theta\in \mathcal C^{-\varepsilon_\ell}$ and noting that $\|p\|=1$, we obtain
\begin{align}
p'\theta^*-\sup_{\tilde\theta\in \mathcal C^{-\varepsilon_\ell}}p'\tilde\theta \le \inf_{\tilde\theta\in \mathcal C^{-\varepsilon_\ell}}\|\theta^*-\tilde\theta\|.\label{eq:star_ell3}
\end{align}
Further, noting that $\theta^*\in\mathcal C$,
\begin{align}
\inf_{\tilde\theta\in \mathcal C^{-\varepsilon_\ell}}\|\theta^*-\tilde\theta\|\le \sup_{\theta\in\mathcal C}\inf_{\tilde\theta\in \mathcal C^{-\varepsilon_\ell}}\|\theta-\tilde\theta\|\le d_H(\mathcal C^{-\varepsilon_\ell},\mathcal C).\label{eq:star_ell4}
\end{align}
By \eqref{eq:star_ell1}-\eqref{eq:star_ell4},
\begin{align}
p'\theta^*-\sup_{\theta\in \mathcal C^{-\varepsilon_\ell}} p'\theta \le M \varepsilon_\ell,\label{eq:star_ell5}
\end{align}
for all $\ell$ sufficiently large.
Therefore, for all $\ell$ sufficiently large, one has
\begin{align}
p'\theta^*-\sup_{\theta\in \mathcal C^{-\varepsilon_\ell}} p'\theta  <p'\theta^*-p'\theta^{(\ell)}, \label{eq:star_ell6}
\end{align}
implying existence of $\tilde \theta^{(\ell)}\in\mathcal C^{-\varepsilon_\ell}$ s.t.
\begin{align}
p'\tilde\theta^{(\ell)}>p'\theta^{(\ell)}.\label{eq:star_ell7}
\end{align}
By Lemma \ref{lem:bds}, for $t(\theta)\equiv (\bar g(\theta)-c(\theta))/s_\ell(\theta)$, one can write
\begin{align}
\mathbb{EI}_{\ell-1}(\theta^{(\ell)}) &\le (p'\theta^{(\ell)}-p'\theta^{*,\ell-1})_+\Bigl(1-\Phi\Bigl(\frac{t(\theta^{(\ell)})-R}{\varsigma}\Bigr)\Bigr)\\
&\le (p'\theta^{(\ell)}-p'\theta^{*,\ell-1})_+(1-\Phi(-R/\varsigma)), \label{eq:star_ell8}
\end{align}
where the last inequality uses $t(\theta^{(\ell)})>0$. Lemma \ref{lem:bds} also yields
\begin{align}
\mathbb{EI}_{\ell-1}(\tilde\theta^{(\ell)})&\ge (p'\tilde\theta^{(\ell)}-p'\theta^{*,\ell-1})_+\Bigl(1-\Phi\Bigl(\frac{t(\tilde \theta^{(\ell)})+R}{\varsigma}\Bigr)\Bigr)\notag\\
&> (p'\theta^{(\ell)}-p'\theta^{*,\ell-1})_+\Bigl(1-\Phi\Bigl(\frac{t(\tilde \theta^{(\ell)})+R}{\varsigma}\Bigr)\Bigr)\label{eq:star_ell9}
\end{align}
for all $\ell$ sufficiently large, where the second inequality follows from \eqref{eq:star_ell7}. Next, by Assumption \ref{as:degeneracy},
\begin{align}
t(\tilde \theta^{(\ell)})=\frac{\bar g(\tilde\theta^{(\ell)})-c(\tilde\theta^{(\ell)})}{s_\ell(\tilde\theta^{(\ell)})}\le \frac{-C_1\varepsilon_\ell}{s_\ell(\tilde\theta^{(\ell)})}\label{eq:star_ell10}
\end{align}
for all $\ell$ sufficiently large.
Note that  $s_\ell(\tilde\theta^{(\ell)})=O(r_\ell)$ by \eqref{eq:eiconv3a} and $r_\ell=o(\varepsilon_\ell)$ by assumption. Hence,  $t(\tilde \theta^{(\ell)})\to -\infty$.
This in turn implies
\begin{align}
\mathbb{EI}_{\ell-1}(\tilde\theta^{(\ell)})> (p'\theta^{(\ell)}-p'\theta^{*,\ell-1})_+(1-\Phi(-R/\varsigma))\label{eq:star_ell11}
\end{align}
for all $\ell$ sufficiently large. \eqref{eq:star_ell8} and \eqref{eq:star_ell11} jointly establish the desired contradiction.
\end{proof}

The next lemma shows that on $D_L\cap E_{1L}$, $p'\theta^*$ and $p'\theta^{*,(\ell(\omega,L))}$ are close to each other, where we recall that $\theta^{*,(\ell(\omega,L))}$ is the optimum value among the available feasible points (it belongs to $\cC$).
\begin{lemma}\label{lem:star-current}
  Suppose Assumptions \ref{as:kernel}, \ref{as:smoothness}, and \ref{as:degeneracy} hold.	Let $\varepsilon_L$ be a positive sequence such that $\varepsilon_L\to 0$ and $r_L=o(\varepsilon_L)$.
  Then, there exists a constant $M>0$ such that $\sup_{\omega\in D_L\cap E_{1L}}|p'\theta^*-p'\theta^{*,\ell(\omega,L)}|/\varepsilon_{\ell(\omega,L)}\le M$ for all $L$ sufficiently large.
\end{lemma}

\begin{proof}
We show below $p'\theta^*-p'\theta^{*,\ell(\omega,L)-1}=O(\varepsilon_{\ell(\omega,L)})$ uniformly over $D_L\cap E_{1L}$ for some decreasing sequence $\varepsilon_\ell$ satisfying the assumptions of the lemma. The claim then follows by re-labeling $\varepsilon_\ell$.

Suppose by contradiction that, for any $M>0$, there is a  subsequence $\{\omega_{a_L}\}\subset \Omega$ along which $\omega_{a_L}\in D_{a_L}$ and $|p'\theta^*-p'\theta^{*,\ell(\omega_{a_L},a_L)-1}|>M\varepsilon_{\ell(\omega_{a_L},a_L)}$ for all $L$ sufficiently large. To simplify notation, we select a subsequence $\{a_L\}$ of $\{L\}$  such that
 for any $a_L<a_{L'}$, $\ell(\omega_{a_L},a_L)<\ell(\omega_{a_{L'}},a_{L'})$. This then  induces a sequence  such that    $|p'\theta^*-p'\theta^{*,\ell-1}|>M\varepsilon_\ell$ for all $\ell,$ where each $\ell$ equals $\ell(\omega_{a_L},a_L)$ for some $a_L\in\mathbb N$.
Similar to the proof of Lemma \ref{lem:star_ell}, we omit the arguments of $\ell$ below and construct a sequence of points $\tilde\theta^{(\ell)}\in\mathcal C^{-\varepsilon_{\ell}}$ such that $\mathbb{EI}_{\ell-1}(\tilde\theta^{(\ell)})>\mathbb{EI}_{\ell-1}(\theta^{(\ell)})$.

Arguing as in \eqref{eq:star_ell1}-\eqref{eq:star_ell4}, one may find a sequence of points $\tilde\theta^{(\ell)}\in \mathcal C^{-\varepsilon_\ell}$ such that
\begin{align}
p'\theta^*-p'\tilde\theta^{(\ell)}\le M_1 \varepsilon_\ell,\label{eq:star-current1}
\end{align}
for some $M_1>0$ and for all $\ell$ sufficiently large. Furthermore, by Lemma \ref{lem:star_ell},
\begin{align}
|p'\theta^*-p'\theta^{(\ell)}|\le M_2\varepsilon_\ell, \label{eq:star-current2}
\end{align}
for some $M_2>0$ and for all $\ell$ sufficiently large.
Arguing as in \eqref{eq:star_ell8},
\begin{align}
\mathbb{EI}_{\ell-1}(\theta^{(\ell)})&\le (p'\theta^{(\ell)}-p'\theta^{*,\ell-1})_+\bigl(1-\Phi(-R/\varsigma)\bigr)\notag\\
&= (p'\theta^*-p'\theta^{*,\ell-1}-(p'\theta^*-p'\theta^{(\ell)}))_+\bigl(1-\Phi(-R/\varsigma)\bigr)\notag\\
&\le (p'\theta^*-p'\theta^{*,\ell-1})\bigl(1-\Phi(-R/\varsigma)\bigr)+|p'\theta^*-p'\theta^{(\ell)}|,\label{eq:star-current3}
\end{align}
where the last inequality follows from the triangle inequality, $p'\theta^*-p'\theta^{*,\ell-1}\ge 0$, and $1-\Phi(\frac{-R}{\varsigma})\le 1.$
Similarly, by Lemma \ref{lem:bds},
\begin{align}
\mathbb{EI}_{\ell-1}(\tilde\theta^{(\ell)})&\ge (p'\tilde\theta^{(\ell)}-p'\theta^{*,\ell-1})_+\Bigl(1-\Phi\Bigl(\frac{t(\tilde\theta^{(\ell)})+R}{\varsigma}\Bigr)\Bigr)\notag\\
&=(p'\theta^*-p'\theta^{*,\ell-1}-(p'\theta^*-p'\tilde\theta^{(\ell)}))_+\Bigl(1-\Phi\Bigl(\frac{t(\tilde\theta^{(\ell)})+R}{\varsigma}\Bigr)\Bigr)\notag\\
&\ge (p'\theta^*-p'\theta^{*,\ell-1})\Bigl(1-\Phi\Bigl(\frac{t(\tilde\theta^{(\ell)})+R}{\varsigma}\Bigr)\Bigr)-(p'\theta^*-p'\tilde\theta^{(\ell)}),\label{eq:star-current4}
\end{align}
where the last inequality holds for all $\ell$ sufficiently large because  $p'\theta^*-p'\tilde\theta^{(\ell)}\in(0, M_2\varepsilon_\ell]$ and one can find a subsequence $p'\theta^*-p'\theta^{*,\ell-1}>M_2\varepsilon_\ell$ so that $p'\theta^*-p'\theta^{*,\ell-1}-(p'\theta^*-p'\tilde\theta^{(\ell)})>0$ for all $\ell$ sufficiently large.

Subtracting \eqref{eq:star-current3} from \eqref{eq:star-current4} yields
\begin{align}
&~\mathbb{EI}_{\ell-1}(\tilde\theta^{(\ell)})-\mathbb{EI}_{\ell-1}(\theta^{(\ell)})\notag\\
\ge &~ (p'\theta^*-p'\theta^{*,\ell-1})\Bigl(\Phi\Bigl(\frac{-R}{\varsigma}\Bigr)-\Phi\Bigl(\frac{t(\tilde\theta^{(\ell)})+R}{\varsigma}\Bigr)\Bigr)-(p'\theta^*-p'\tilde\theta^{(\ell)})-|p'\theta^*-p'\theta^{(\ell)}|\notag\\
\ge &~ (p'\theta^*-p'\theta^{*,\ell-1})\Bigl(\Phi\Bigl(\frac{-R}{\varsigma}\Bigr)-\Phi\Bigl(\frac{t(\tilde\theta^{(\ell)})+R}{\varsigma}\Bigr)\Bigr)- (M_1+M_2) \varepsilon_\ell,\label{eq:star-current5}
\end{align}
where the last inequality follows from \eqref{eq:star-current1} and \eqref{eq:star-current2}.
Note that there is a constant $\zeta>0$ s.t.
\begin{align}
\Phi\Bigl(\frac{-R}{\varsigma}\Bigr)-\Phi\Bigl(\frac{t(\tilde\theta^{(\ell)})+R}{\varsigma}\Bigr)>\zeta,\label{eq:star-current6}
\end{align}
due to $t(\tilde\theta^{(\ell)})\to-\infty$ by \eqref{eq:star_ell10}, \eqref{eq:eiconv3a}, and $r_\ell=o(\varepsilon_\ell)$. Therefore, for all $\ell$ sufficiently large,
\begin{align}
\mathbb{EI}_{\ell-1}(\tilde\theta^{(\ell)})-\mathbb{EI}_{\ell-1}(\theta^{(\ell)})>M\zeta\varepsilon_\ell - (M_1+M_2)\varepsilon_\ell.\label{eq:star-current7}
\end{align}
One may take $M$ large enough so that, for some  positive constant $\gamma$, $M\zeta\varepsilon_\ell - (M_1+M_2)\varepsilon_\ell>\gamma\varepsilon_\ell$ for all $\ell$ sufficiently large, which implies $\mathbb{EI}_{\ell-1}(\tilde\theta^{(\ell)})-\mathbb{EI}_{\ell-1}(\theta^{(\ell)})>0$ for all $\ell$ sufficiently large. However, this contradicts the assumption that $\theta^{(\ell)}\notin\mathcal C^{-\varepsilon_\ell}$ is the expected improvement maximizer.
\end{proof}

The next lemma shows that on $D_L\cap E_{2L}$, $p'\theta^*$ and $p'\theta^{*,(\ell(\omega,L))}$ are close to each other.
\begin{lemma}\label{lem:star-current_inner}
  Suppose Assumptions \ref{as:kernel}, \ref{as:smoothness}, and \ref{as:degeneracy} hold.	Let $\{\varepsilon_L\}$ be a positive sequence such that $\varepsilon_L\to 0$ and $r_L=o(\varepsilon_L)$.
  Then, there exists a constant $M>0$ such that $\sup_{\omega\in D_L\cap E_{2L}}|p'\theta^*-p'\theta^{*,\ell(\omega,L)}|/\varepsilon_{\ell(\omega,L)}\le M$   for all $L$ sufficiently large.
\end{lemma}

\begin{proof}
Note that, for any $L\in\mathbb N$, $\omega\in D_L\cap E_{2L}$, and $\ell=\ell(\omega,L)$, $\theta^{(\ell)}$ satisfies $\bar g(\theta^{(\ell)})-c(\theta^{(\ell)})\le 0$, hence $p'\theta^{^*,\ell}\ge p'\theta^{(\ell)}$, which in turn implies
\begin{align}
0\le p'\theta^*-p'\theta^{*,\ell}\le p'\theta^*-p'\theta^{(\ell)}.
\end{align}
Therefore, it suffices to show the existence of $M>0$ that ensures $(p'\theta^*-p'\theta^{(\ell(\omega,L))})_+\le M\varepsilon_{\ell(\omega,L)}$ uniformly over $D_L\cap E_{2L}$ for all $L$.
Suppose by contradiction that, for any $M>0$, there is a  subsequence $\{\omega_{a_L}\}\subset \Omega$ along which $\omega_{a_L}\in D_{a_L}\cap E_{2a_L}$ and $p'\theta^*-p'\theta^{(\ell(\omega_{a_L},a_L))}>M\varepsilon_{\ell(\omega_{a_L},a_L)}$ for all $L$ sufficiently large. Again, we select a subsequence $\{a_L\}$ of $\{L\}$  such that
for any $a_L<a_{L'}$, $\ell(\omega_{a_L},a_L)<\ell(\omega_{a_{L'}},a_{L'})$. This then  induces a sequence $\{\theta^{(\ell)}\}$ of expected improvement maximizers such that    $(p'\theta^*-p'\theta^{(\ell)})_+>M\varepsilon_\ell$ for all $\ell,$ where each $\ell$ equals $\ell(\omega_{a_L},a_L)$ for some $a_L\in\mathbb N$.

Similar to the proof of Lemma \ref{lem:star_ell}, we omit the arguments of $\ell$ below and prove the claim by contradiction.
Below, we assume that, for any $M>0$, there is a further subsequence along which $p'\theta^*-p'\theta^{(\ell)}>M\varepsilon_\ell$ for all $\ell$ sufficiently large.

Now let $\varepsilon_\ell'=\tilde C \varepsilon_\ell$ with $\tilde C>0$ specified below. By Assumption \ref{as:degeneracy}, for all $\tilde\theta\in\mathcal C^{-\varepsilon_\ell'}$, it holds that
\begin{align}
\bar g(\tilde\theta)-c(\tilde \theta)\le -\tilde C C_1 \varepsilon_\ell,
\end{align}
for all $\ell$ sufficiently large. Noting that $-\varepsilon_\ell\le \bar g(\theta^{(\ell)})-c(\theta^{(\ell)})$  and  taking $\tilde C$ such that $\tilde CC_1>1$, it follows that $\theta^{(\ell)}\notin \mathcal C^{-\varepsilon_\ell'}$ for all $\ell$ sufficiently large.

Arguing as in \eqref{eq:star_ell1}-\eqref{eq:star_ell4}, one may find a sequence of points $\tilde\theta^{(\ell)}\in \mathcal C^{-\varepsilon_\ell'}$ such that
\begin{align}
p'\theta^*-p'\tilde\theta^{(\ell)}\le M_1 \varepsilon'_\ell= M_1\tilde C\varepsilon_\ell,\label{eq:star-current_inner1}
\end{align}
This and the assumption that one can find a subsequence such that $p'\theta^*-p'\theta^{(\ell)}>M_1\tilde C\varepsilon_\ell$ for all $\ell$ imply
\begin{align}
p'\theta^*-p'\tilde\theta^{(\ell)}<p'\theta^*-p'\theta^{(\ell)},\label{eq:star_current_inner2}
\end{align}
for all $\ell$ sufficiently large. Now mimic the argument along \eqref{eq:star_ell8}-\eqref{eq:star_ell11} to deduce
\begin{align}
\mathbb{EI}_{\ell-1}(\tilde\theta^{(\ell)})>\mathbb{EI}_{\ell-1}(\theta^{(\ell)})
\end{align}
for all $\ell$ sufficiently large. However, this contradicts the assumption that $\theta^{(\ell)}\notin \mathcal C^{-\varepsilon_\ell'}$ is the expected improvement maximizer.
\end{proof}

The next lemma shows that on $D_L\cap E_{3L}$, $p'\theta^*$ and $p'\theta^{*,(\ell(\omega,L))}$ are close to each other.
\begin{lemma}\label{lem:case3}
  Suppose Assumptions \ref{as:kernel}, \ref{as:smoothness}, and \ref{as:degeneracy} hold.	Let $\varepsilon_L=(L/\ln L)^{-\nu/d}(\ln L)^\delta$ for $\delta\ge 1+\chi$. Let $\eta_L=\varepsilon_L/r_L=(\ln L)^{\delta-\chi}$.  Then there exists a constant $M>0$ such that $\sup_{\omega\in D_L\cap E_{3L}}|p'\theta^*-p'\theta^{*,\ell(\omega,L)}|/\exp(-M\eta_{\ell(\omega,L)})\le M$ for all $L$ sufficiently large.
\end{lemma}

\begin{proof}
Let $\{\omega_L\}\subset\Omega$ be a sequence such that $\omega_L\in D_L$ for all $L$. 
Since $\omega_L \in B_L$, there is $\ell=\ell(\omega_L,L)$ such that $L\le \ell\le 2L$ and $\theta^{(\ell)}$ is chosen by maximizing the expected improvement.
For later use, we note that, for any $\tilde M>0$,
it  can be shown that $\exp(-\tilde M\eta_{L-1})/\exp(-\tilde M\eta_L)\to 1$, which in turn implies that there exists a constant $C>1$ such that
\begin{align}
	\exp(-\tilde M\eta_{L-1})\le C\exp(-\tilde M\eta_L),\label{eq:etabound2}
\end{align}
for all $L$ sufficiently large.

 For $\theta\in\Theta$ and $L\in\mathbb N$, let $\mathbb{I}_L(\theta)\equiv (p'\theta-p'\theta^{*,L})_+1\{\bar g(\theta)\le c(\theta)\}.$
 Recall that $\theta^*$ is an optimal solution to \eqref{eq:general_problem}.  Then, for all $L$ sufficiently large,
\begin{align}
p'\theta^*-p'\theta^{*,\ell-1}&\stackrel{(1)}{=} \mathbb{I}_{\ell-1}(\theta^*)	
\stackrel{(2)}{\le} \EI_{\ell-1}(\theta^*)\bigl(1-\Phi(R/\varsigma)\bigr)^{-1} 
\stackrel{(3)}{\le} \EI_{\ell-1}(\theta^{(\ell)})\bigl(1-\Phi(R/\varsigma)\bigr)^{-1}\notag\\
&\stackrel{(4)}{\le} \Bigl(\mathbb{I}_{\ell-1}(\theta^{(\ell)})+M_1 \exp(-\tilde M\eta_{\ell-1})\Bigr)\bigl(1-\Phi(R/\varsigma)\bigr)^{-1} \notag\\
&\stackrel{(5)}{\le} \Bigl(\mathbb{I}_{\ell-1}(\theta^{(\ell)})+M_2 \exp(-\tilde M\eta_{\ell})\Bigr)\bigl(1-\Phi(R/\varsigma)\bigr)^{-1} \notag\\
  &\stackrel{(6)}{\le} \Bigl(\mathbb{I}_{\ell-1}(\theta^{*,\ell})+M_2 \exp(-\tilde M\eta_{\ell})\Bigr)\bigl(1-\Phi(R/\varsigma)\bigr)^{-1}\notag\\
  &\stackrel{(7)}{\le} \Bigl(\mathbb{EI}_{\ell-1}(\theta^{*,\ell})+2M_2 \exp(-\tilde M\eta_{\ell})\Bigr)\bigl(1-\Phi(R/\varsigma)\bigr)^{-1}\notag\\
  &\stackrel{(8)}{\le} \Bigl(\mathbb{EI}_{\ell-1}(\theta^{(\ell-1)})+2M_2 \exp(-\tilde M\eta_{\ell})\Bigr)\bigl(1-\Phi(R/\varsigma)\bigr)^{-1}\notag\\
  &\stackrel{(9)}{\le} \Bigl(\mathbb{I}_{\ell-1}(\theta^{(\ell-1)})+3M_2 \exp(-\tilde M\eta_{\ell})\Bigr)\bigl(1-\Phi(R/\varsigma)\bigr)^{-1}\notag\\
  &\stackrel{(10)}{\le} 3M_2 \exp(-\tilde M\eta_{\ell})\bigl(1-\Phi(R/\varsigma)\bigr)^{-1}\notag,
\end{align}
where (1) follows by construction, (2) follows from  Lemma \ref{lem:bds} (ii), (3) follows from $\theta^{(\ell)}$ being the maximizer of the expected improvement, (4) follows from Lemma \ref{lem:bds2},
 (5) follows from \eqref{eq:etabound2} with $M_2=CM_1$, (6) follows from $\theta^{*,\ell}=\text{argmax}_{\theta\in \mathcal C_\ell} p'\theta$, (7) follows from Lemma \ref{lem:bds2}, (8) follows from $\theta^{(\ell-1)}$ being the expected improvement maximizer, (9) follows from Lemma \ref{lem:bds2}, and (10) follows from $\mathbb{I}_{\ell-1}(\theta^{(\ell-1)})=0$
 due to the definition of $\theta^{*,\ell-1}$. This establishes the claim.
\end{proof}

For evaluation points $\theta_L$ such that $|\bar g(\theta_L)-c(\theta_L)|>\varepsilon_L$, the following lemma is an analog of Lemma 8 in \cite{Bull_Convergence_2011}, which links the expected improvement to the actual improvement achieved by a new evaluation point $\theta$.
\begin{lemma}\label{lem:bds2}Suppose $\Theta\subset\mathbb R^d$ is bounded and $p\in \mathbb S^{d-1}$.
Suppose the evaluation points $(\theta^{(1)},\cdots,\theta^{(L)})$ are drawn by Algorithm \ref{alg:evalpts} and let Assumptions \ref{as:kernel} and \ref{as:smoothness}-(ii) hold.
For $\theta\in\Theta$ and $L\in\mathbb N$, let $\mathbb{I}_L(\theta)\equiv (p'\theta-p'\theta^{*,L})_+1\{\bar g(\theta)\le c(\theta)\}.$
Let $\{\varepsilon_L\}$ be a positive sequence such that $\varepsilon_L\to 0$ and $r_L=o(\varepsilon_L)$. Let $\eta_L\equiv\varepsilon_L/r_L.$
 Then, for any sequence $\{\theta_L\}\subset\Theta$ such that $|\bar g(\theta_L)-c(\theta_L)|>\varepsilon_L$,
\begin{align}
\mathbb{I}_L(\theta_L)-\gamma_L\le \EI_L(\theta_L)\le \mathbb{I}_L(\theta_L)+\gamma_L,\label{eq:lembds1}
\end{align}
 where $\gamma_L=O(\exp(-M\eta_L))$.
\end{lemma}

\begin{proof}[\rm \textbf{Proof of Lemma \ref{lem:bds2}}]
If $s_L(\theta_L)=0$, then the posterior variance of $c(\theta_L)$ is zero. Hence, $\EI_L(\theta_L)=\mathbb{I}_L(\theta_L)$, and the claim of the lemma holds.

Suppose $s_L(\theta_L)>0$. We first show the upper bound. Let $u\equiv (\bar g(\theta_L)-c_L(\theta_L))/s_L(\theta_L)$ and $t\equiv (\bar g(\theta_L)-c(\theta_L))/s_L(\theta_L)$.
By Lemma 6 in \cite{Bull_Convergence_2011}, we have $|u-t|\le R.$ Starting from Lemma \ref{lem:bds}(i), we can write
\begin{align}
	\EI_L(\theta_L)&\le (p'\theta_L-p'\theta^{*,L})_+\Big(1-\Phi\Big(\frac{t-R}{\varsigma}\Big)\Big)\notag\\
	&= (p'\theta_L-p'\theta^{*,L})_+(1\{\bar g(\theta_L)\le c(\theta_L)\}+1\{\bar g(\theta_L)> c(\theta_L)\})\Big(1-\Phi\Big(\frac{t-R}{\varsigma}\Big)\Big)\notag\\
	&\le \mathbb{I}_L(\theta_L)+(p'\theta_L-p'\theta^{*,L})_+1\{\bar g(\theta_L)> c(\theta_L)\}\Big(1-\Phi\Big(\frac{t-R}{\varsigma}\Big)\Big),\label{eq:bds1}
\end{align}
where the last inequality used $1-\Phi(x)\le 1$ for any $x\in\mathbb R$.
Note that one may write
\begin{align}
	1\{\bar g(\theta_L)> c(\theta_L)\}\Big(1-\Phi\Big(\frac{t-R}{\varsigma}\Big)\Big)&=1\{\bar g(\theta_L)> c(\theta_L)\}\Big(1-\Phi\Big(\frac{\bar g(\theta_L)- c(\theta_L)-s_L(\theta_L)R}{\varsigma s_L(\theta_L)}\Big)\Big).\label{eq:bds2}
\end{align}
To be clear about the hyperparameter value at which  we evaluate $s_L$, we will write $s_L(\theta_L;\beta)$.
By the hypothesis that $\|c\|_{\mathcal H_{\bar\beta}}\le R$ and Lemma 4 in \cite{Bull_Convergence_2011}, we have
\begin{align}
\|c\|_{\mathcal H_{\beta_L}}\le R^2\prod_{k=1}^d(\overline{\beta}_k/\underline{\beta}_k) \equiv S.	\label{eq:eiconv3}
\end{align}
Note that there are $\lfloor\eta L\rfloor$ uniformly sampled points, and $K_\beta$ is associated with index $\nu\in (0,\infty)$.
As shown in the proof of Theorem 5 in \cite{Bull_Convergence_2011}, this ensures that
\begin{align}
	\sup_{\beta\in \prod_{k=1}^d[\underline\beta_k,\overline\beta_k]}s_L(\theta_L;\beta)=O(h_L^\nu(\ln L)^\chi)=O(r_L).\label{eq:eiconv3a}
\end{align}
Below, we simply write this result $s_L(\theta_L)=O(r_L).$
 This, together with $|\bar g(\theta_L)- c(\theta_L)|>\varepsilon_L$ and the fact that $1-\Phi(\cdot)$ is decreasing, yields
\begin{align}
1\{\bar g(\theta_L)> c(\theta_L)\}\Big(1-\Phi\Big(\frac{\bar g(\theta_L)- c(\theta_L)-s_L(\theta_L)R}{\varsigma s_L(\theta_L)}\Big)\Big)&\le 1-\Phi\Big(\frac{\varepsilon_L}{ \varsigma s_L(\theta_L)}-\frac{R}{\varsigma}\Big)\notag\\
&\le 1-\Phi(M_1\eta_L-M_2),\label{eq:bds4z}
\end{align}
for some $M_1>0$ and where $M_2=R/\varsigma$.
Note that, by the triangle inequality,
\begin{align}
   1-\Phi(M_1\eta_L-M_2)\le  1-\Phi(M_1\eta_L)+| (1-\Phi(M_1\eta_L-M_2))- (1-\Phi(M_1\eta_L))|,\label{eq:bds4a}
\end{align}
and
\begin{align}
   1-\Phi(M_1\eta_L)\le \frac{1}{M_1\eta_L} \phi(M_1\eta_L)=O(\exp(-M\eta_L)),\label{eq:bds4}
\end{align}
for some $M>0$, 
where $\phi$ is the density of the standard normal distribution, and the  inequality follows from $1-\Phi(x)\le \phi(x)/x$. The second term on the right hand side of \eqref{eq:bds4a}
can be bounded as
\begin{align}
  | (1-\Phi(M_1\eta_L-M_2))- (1-\Phi(M_1\eta_L))|\le \phi(\tilde\eta_L)M_2=O(\exp(-M\eta_L))\label{eq:bds4b}
\end{align}
by the mean value theorem, where $\tilde\eta_L$ is a point between $M_1\eta_L$ and $M_1\eta_L-M_2$.
 The claim of the lemma then follows from  \eqref{eq:bds1}, \eqref{eq:bds4z}-\eqref{eq:bds4b}, and $(p'\theta_L-p'\theta_L^{*,L})$ being bounded because $\Theta$ is bounded.

 Similarly, for the lower bound, we have
 \begin{align}
 	\EI_L(\theta_L)&
 	\ge (p'\theta_L-p'\theta^*_L)_+\Big(1-\Phi\Big(\frac{t+R}{\varsigma}\Big)\Big)\notag\\
 	&\ge (p'\theta_L-p'\theta^*_L)_+1\{\bar g(\theta_L)\le c(\theta_L)\}\Big(1-\Phi\Big(\frac{t+R}{\varsigma}\Big)\Big)\notag\\
 	&\ge \mathbb{I}_L(\theta_L)-(p'\theta_L-p'\theta^*_L)_+1\{\bar g(\theta_L)\le c(\theta_L)\}\Phi\Big(\frac{t+R}{\varsigma}\Big).\label{eq:bds5}
 \end{align}
 Note that we may write
 \begin{align}
 	1\{\bar g(\theta_L)\le c(\theta_L)\}\Phi\Big(\frac{t+R}{\varsigma}\Big)=1\{\bar g(\theta_L)<  c(\theta_L)\}\Phi\Big(\frac{\bar g(\theta_L)- c(\theta_L)+s_L(\theta_L)R}{\varsigma s_L(\theta_L)}\Big),\label{eq:bds6}
 \end{align}
by $|\bar g(\theta_L)-c(\theta_L)|>\varepsilon_L$. Arguing as in \eqref{eq:bds5} and noting that $\Phi$ is increasing, one has
 \begin{align}
 1\{\bar g(\theta_L)<  c(\theta_L)\}\Phi\Big(\frac{\bar g(\theta_L)- c(\theta_L)+s_L(\theta_L)R}{\varsigma s_L(\theta_L)}\Big)&\le \Phi\Big(\frac{-\varepsilon_L}{\varsigma s_L(\theta_L)}+M_2\Big)\notag\\
 &\le \Phi(-M_1\eta_L+M_2),\label{eq:bds7}
 \end{align}
 for some $M_1>0$ and $M_2>0$. By the triangle inequality,
 \begin{align}
 	\Phi(-M_1\eta_L+M_2)\le \Phi(-M_1\eta_L) + |\Phi(-M_1\eta_L+M_2)-\Phi(-M_1\eta_L)|,\label{eq:bds8}
 \end{align}
 where arguing as in \eqref{eq:bds4},
 \begin{align}
 	\Phi(-M_1\eta_L) =1-\Phi(M_1\eta_L)=O(\exp(-M\eta_L)).\label{eq:bds9}
 \end{align}
 The second term on the right hand side of \eqref{eq:bds8}
 can be bounded as
 \begin{multline}
 	|\Phi(-M_1\eta_L+M_2)-\Phi(-M_1\eta_L)|\\
	=|(1-\Phi(M_1\eta_L-M_2))- (1-\Phi(M_1\eta_L))|\le \phi(\tilde\eta_L)M_2=O(\exp(-M\eta_L)),\label{eq:bds10}
 \end{multline}
 by the mean value theorem, where $\tilde\eta_L$ is a point between $M_1\eta_L$ and $M_1\eta_L-M_2$.
  The claim of the lemma then follows from  \eqref{eq:bds5}-\eqref{eq:bds10}, and $(p'\theta_L-p'\theta_L^{*,L})$ being bounded because $\Theta$ is bounded.
\end{proof}

\begin{lemma}\label{lem:bds}Suppose $\Theta\subset\mathbb R^d$ is bounded and $p\in \mathbb S^{d-1}$  and let Assumptions \ref{as:kernel} and \ref{as:smoothness}-(ii) hold.
Let $t(\theta)\equiv (\bar g(\theta)-c(\theta))/s_L(\theta)$. For $\theta\in\Theta$ and $L\in\mathbb N$, let $\mathbb{I}_L(\theta)\equiv (p'\theta-p'\theta^{*,L})_+1\{\bar g(\theta)\le c(\theta)\}.$
 Then, (i) for any $L\in\mathbb N$ and $\theta\in\Theta$,
\begin{align}
(p'\theta-p'\theta^{*,L})_+\Big(1-\Phi\Big(\frac{t(\theta)+R}{\varsigma}\Big)\Big)\le \EI_L(\theta)\le (p'\theta-p'\theta^{*,L})_+\Big(1-\Phi\Big(\frac{t(\theta)-R}{\varsigma}\Big)\Big).
\end{align}
 Further, (ii) for any $L\in\mathbb N$ and $\theta\in\Theta$ such that $s_L(\theta)>0$,
\begin{align}
\mathbb{I}_L(\theta)	&\le \EI_L(\theta)\Big(1-\Phi\Big(\frac{R}{\varsigma}\Big)\Big)^{-1}.\label{eq:lembds2}
\end{align}
\end{lemma}

\begin{proof}
(i)	Let $u(\theta)\equiv (\bar g(\theta)-c_L(\theta))/s_L(\theta)$ and $t(\theta)\equiv (\bar g(\theta)-c(\theta))/s_L(\theta)$.
	By Lemma 6 in \cite{Bull_Convergence_2011}, we have $|u(\theta)-t(\theta)|\le R.$ Since $1-\Phi(\cdot)$ is decreasing, we have
	\begin{align}
		\EI_L(\theta)&=(p'\theta-p'\theta^{*,L})_+\Big(1-\Phi\Big(\frac{u(\theta)}{\varsigma }\Big)\Big)
		\le (p'\theta-p'\theta^{*,L})_+\Big(1-\Phi\Big(\frac{t(\theta)-R}{\varsigma}\Big)\Big).
		\end{align}
Similarly,
\begin{align}
		\EI_L(\theta)&=(p'\theta-p'\theta^{*,L})_+\Big(1-\Phi\Big(\frac{u(\theta)}{\varsigma }\Big)\Big)
		\ge (p'\theta-p'\theta^{*,L})_+\Big(1-\Phi\Big(\frac{t(\theta)+R}{\varsigma}\Big)\Big).
\end{align}

\noindent
(ii) For the lower bound in \eqref{eq:lembds2}, we have
\begin{align}
	\EI_L(\theta)&
	\ge (p'\theta-p'\theta^{*,L})_+\Big(1-\Phi\Big(\frac{t(\theta)+R}{\varsigma}\Big)\Big)\notag\\
	&\ge (p'\theta-p'\theta^{*,L})_+1\{\bar g(\theta)\le c(\theta)\}\Big(1-\Phi\Big(\frac{t(\theta)+R}{\varsigma}\Big)\Big)\notag\\
	&\ge \mathbb{I}_L(\theta)\bigl(1-\Phi(R/\varsigma)\bigr),
\end{align}
where the last inequality follows from $t(\theta)=(\bar g(\theta)-c(\theta))/s_L(\theta)\le 0$ and the fact that $1-\Phi(\cdot)$ is decreasing.
\end{proof}

\section{Applying the E-A-M Algorithm to Profiling}
\label{sec:EAM_for_BCS}
We describe below how to use the E-A-M procedure to compute BCS-profiling based confidence intervals.
Let $\cT\subset\mathbb R$ denote the parameter space for  $\tau=p'\theta$.
The (one-dimensional) profiling confidence region  is
\begin{align}
\Bigl\{\tau\in\cT:\inf_{\theta:p'\theta=\tau}T_n(\theta)\le c^{MR}_n(\tau)\Bigr\},\label{eq:EAM_prof1}
\end{align}
where $c^{MR}_n$ is the critical value proposed in  \cite{BCS14_subv} and $T_n$ is any test statistic that they allow for. The E-A-M algorithm can be used to compute the
endpoints of this set so that the researcher may report an interval.

For ease of exposition, we discuss below the computation of the right end point of the confidence interval, which is the optimal value of the following problem:\footnote{The left end point is the optimal value of a program that replaces $\max$ with $\min$.}
\begin{align}
\max_{\tau\in\cT}& ~\tau \label{eq:EAM_prof2}\\
\text{s.t.} &~\inf_{\theta\in\Theta:p'\theta=\tau}T_n(\theta)\le c^{MR}_n(\tau).\notag
\end{align}
We then take $c(\tau)\equiv -\inf_{\theta\in\Theta:p'\theta=\tau}T_n(\theta)+c^{MR}_n(\tau)$ as a black-box function and apply the E-A-M algorithm.\footnote{One may view \eqref{eq:EAM_prof2} as a special case of \eqref{eq:general_problem} with a scalar control variable and  a single constraint $g_1(\tau)\le c(\tau)$ with $g_1(\tau)=0$.} We include the profiled statistic in the black-box function because it involves a non-linear optimization problem, which is also relatively expensive.
The modified procedure is as follows. 
\begin{description}
\item[Initialization:]Draw randomly (uniformly) over $\cT\subset\mathbb R$ a set $(\tau^{(1)},\dots,\tau^{(k)})$ of initial evaluation points and evaluate $c(\tau^{(\ell)})$ for $\ell=1,\dots,k-1$. Initialize $L=k$.
\item[E-Step:]Evaluate $c(\tau^{(L)})$ and record the tentative optimal value
\begin{equation*}
\tau^{*,L}\equiv\max\bigl\{\tau^{\ell}:\ell\in\{1,\dots,L\},c(\tau^{(\ell)})\ge 0\bigr\}.
\end{equation*}
	\item[A-step: (Approximation)] Approximate $\tau\mapsto c(\tau)$ by a flexible auxiliary model. We again use the kriging approximation, which for a mean-zero Gaussian process $\zeta(\cdot)$ indexed by $\tau$ and with constant variance $\varsigma^2$ specifies
	\begin{align}
		\Upsilon^{(\ell)}&=\mu+\zeta(\tau^{(\ell)}), ~\ell=1,\dots, L\label{eq:EAM:gauss_prior_prof}\\
		Corr(\zeta(\tau),\zeta(\tau'))&=K_\beta(\tau-\tau'),~ \tau,\tau' \in \mathbb R\label{eq:EAM:corr_prof},
	\end{align}
	where $K_\beta$ is a kernel with a scalar parameter $\beta \in [\underline{\beta},\overline{\beta}]\subset \R_{++}$. The parameters are estimated in the same way as before.

	The (best linear) predictor of $c$ and its derivative  are then given by
	\begin{align}
		c_L(\tau)&=\hat\mu+\mathbf{r}_L(\tau)'\mathbf{R}_L^{-1}(\mathbf \Upsilon-\hat\mu \mathbf 1),\label{eq:cLdef_prof}\\
		\nabla_\tau c_L(\tau)&=\hat\mu+\mathbf{Q}_L(\tau)\mathbf{R}_L^{-1}(\mathbf \Upsilon-\hat\mu \mathbf 1),\label{eq:cLgrad_prof}
	\end{align}
	where $\mathbf r_L(\tau)$ is a vector whose $\ell$-th component is $Corr(\zeta(\tau),\zeta(\tau^{(\ell)}))$ as given above with estimated parameters, $\mathbf Q_L(\tau)=\nabla_\tau \mathbf r_L(\tau)'$, and $\mathbf R_L$ is an $L$-by-$L$ matrix whose $(\ell,\ell')$ entry is $Corr(\zeta(\tau^{(\ell)}),\zeta(\tau^{(\ell')}))$ with estimated parameters. The amount of uncertainty left in $c(\tau)$ is captured by the following variance:
\begin{align}
\hat\varsigma^2s^2_L(\tau)= \hat\varsigma^2\Big(1-\mathbf r_L(\tau)'\mathbf R_L^{-1}\mathbf r_L(\tau)+\frac{(1-\mathbf 1'\mathbf R_L^{-1}\mathbf r_L(\tau))^2}{\mathbf 1'\mathbf R_L^{-1}\mathbf 1}\Big). \label{eq:variance_prof}
\end{align}

	\item[M-step: (Maximization):] With probability $1-\epsilon,$ maximize the expected improvement function $\EI_L$ to obtain the next evaluation point, with:
	\begin{align}
		\tau^{(L+1)}\equiv\argmax_{\tau\in\cT}\EI_L(\tau)=\argmax_{\tau\in\cT} (\tau-\tau^{*,L})_+\Big(1-\Phi\Big(\frac{-c_L(\tau)}{\hat\varsigma s_L(\tau)}\Big)\Big).\label{eq:max_ei_prof}
	\end{align}
With probability $\epsilon$, draw $\tau^{(L+1)}$ randomly from a uniform distribution over $\cT$.
\end{description}

As before, $\tau^{*,L}$ is reported as end point of $CI_n$ upon convergence. In order for Theorem \ref{thm:eiconv} to apply to this algorithm, the profiled statistic $\inf_{\theta\in\Theta:p'\theta=\tau}T_n(\theta)$ and the critical value $\hat{c}_n^{MR}$ need to be sufficiently smooth. We leave derivation of sufficient conditions for this to be the case to future research.

\section{An Entry Game Model and Some Monte Carlo Simulations}
\label{sec:MC}
We evaluate the statistical and numerical performance of calibrated projection and E-A-M in comparison with BCS-profiling in a Monte Carlo experiment run on a server with two Intel Xeon X5680 processors rated at 3.33GHz with 6 cores each and with a memory capacity of 24Gb rated at 1333MHz.
The experiment simulates a two-player entry game in the Monte Carlo exercise of BCS, using their code to implement their method.\footnote{\label{ftn:BCS_code}See \url{http://qeconomics.org/ojs/index.php/qe/article/downloadSuppFile/431/1411}.}

\subsection{The General Entry Game Model}
\label{sec:entry}
We consider a two player entry game based on \cite{CilibertoTamer09}:
\begin{equation*}\small
\begin{tabular}{ccc}
& $Y_2=0$ & $Y_2=1$ \\ \cline{2-3}
$Y_1=0$ & \multicolumn{1}{|c}{$0,0$} & \multicolumn{1}{|c|}{$0,Z_2'\vartheta_1+u_{2}$} \\
\cline{2-3}
$Y_1=1$ & \multicolumn{1}{|c}{$Z_1'\vartheta_1+u_{1},0$} & \multicolumn{1}{|c|}{$Z_1'(\vartheta_1+\Delta_1)+u_{1},Z_2'(\vartheta_2+\Delta_2)+u_{2}$} \\ \cline{2-3}
\end{tabular}
\end{equation*}\normalsize
Here, $Y_\ell$, $Z_\ell$, and $u_\ell$ denote player $\ell'$s binary action, observed characteristics, and unobserved characteristics. The strategic interaction effects $Z_\ell'\Delta_\ell \le 0$ measure the impact of the opponent's entry into the market. We let $X\equiv(Y_1,Y_2,Z_1',Z_2')'$.
We generate $Z=(Z_1,Z_2)$ as an i.i.d. random vector taking values in a finite set whose distribution $p_z=P(Z=z)$ is known. We let $u=(u_1,u_2)$ be independent of $Z$ and such that $Corr(u_1,u_2)\equiv r\in[0,1]$ and $Var(u_\ell)=1,\ell=1,2$. We let $\theta\equiv(\vartheta_1',\vartheta_2',\Delta_1',\Delta_2',r)'.$ For a given set $A \subset \R^2$, we define $G_r(A)\equiv P(u\in A)$. We choose $G_r$ so that the c.d.f. of $u$ is continuous, differentiable, and has a bounded p.d.f.
The outcome $Y=(Y_1,Y_2)$ results from pure strategy Nash equilibrium play. For some value of $Z$ and $u$, the model predicts monopoly outcomes $Y=(0,1)$ and $(1,0)$ as multiple equilibria. When this occurs, we select outcome $(0,1)$ by independent Bernoulli trials with parameter $\mu\in[0,1]$. This gives rise to the following restrictions:
\begin{align}
	&E[1\{Y=(0,0)\}1\{Z=z\}]-G_r((-\infty,-z_1'\vartheta_1)\times(-\infty,-z_2'\vartheta_2))p_z=0 \label{eq:entry5}\\
	&E[1\{Y=(1,1)\}1\{Z=z\}]-G_r([-z_1'(\vartheta_1+\Delta_1),+\infty)\times [-z_2'(\vartheta_2+\Delta_2),+\infty))p_z=0\label{eq:entry6}\\
	&E[1\{Y=(0,1)\}1\{Z=z\}]-G_r((-\infty,-z_1'(\vartheta_1+\Delta_1))\times [-z_2'\vartheta_2,+\infty))p_z\le 0\label{eq:entry7}\\
	-&E[1\{Y=(0,1)\}1\{Z=z\}]
	+\Big[G_r((-\infty,-z_1'(\vartheta_1+\Delta_1))\times [-z_2'\vartheta_2,+\infty)\notag\\
 &\hspace{3.6cm}	-G_r([-z_1'\vartheta_1,-z_1'(\vartheta_1+\Delta_1))\times [-z_2'\vartheta_2,-z_2'(\vartheta_2+\Delta_2))\Big]p_z\le 0.\label{eq:entry8}
\end{align}
We show in Online Appendix \ref{sec:verify_examples} that this model satisfies Assumptions \ref{as:Lipschitz_m_over_sigma} and \ref{as:correlation}-\ref{as:correlation_pair}.\footnote{The specialization in which we compare to BCS also fulfils their assumptions. The assumptions in \cite{PakesPorterHo2011} exclude any DGP that has moment equalities.} Throughout, we analytically compute the moments' gradients and studentize them using sample analogs of their standard deviations.

\subsection{A Comparison to BCS-Profiling}
\label{sec:MC_results}
BCS specialize this model as follows. First, $u_1,u_2$ are independently uniformly distributed on $[0,1]$ and the researcher knows $r=0$. Equality \eqref{eq:entry5} disappears because $(0,0)$ is never an equilibrium. Next, $Z_1=Z_2=[1; \{W_k\}_{k=0}^{d_W}]$, where $W_k$ are observed market type indicators, $\Delta_\ell = [\delta_\ell; 0_{d_W}]$ for $\ell=1,2$, and $\vartheta_1=\vartheta_2 =\vartheta = [0;\{\vartheta^{[k]}\}_{k=0}^{d_W}]$.\footnote{This allows for market-type homogeneous fixed effects but not for player-specific covariates nor for observed heterogeneity in interaction effects.} The parameter vector is $\theta=[\delta_1;\delta_2;\vartheta]$ with parameter space $\Theta = \{\theta \in \R^{2+d_W}: (\delta_1,\delta_2)\in [0,1]^2,~\vartheta_k \in [0,\min\{\delta_1,\delta_2\}],~k=1,\dots,d_W\}$. This leaves 4 moment equalities and 8 moment inequalities (so $J = 16$); compare equation (5.1) in BCS. We set $d_W=3$, $P(W_k=1)=1/4,k=0,1,2,3$, $\theta = [0.4 ; 0.6 ;0.1  ;0.2  ;0.3]$, and $\mu=0.6$. The implied true bounds on parameters are $\delta_1 \in [0.3872,0.4239]$, $\delta_2 \in [0.5834,0.6084]$, $\vartheta^{[1]}\in [0.0996, 0.1006]$, $\vartheta^{[2]} \in [0.1994,0.2010]$, and $\vartheta^{[3]} \in [0.2992,0.3014]$.

The BCS-profiling confidence interval $CI_n^{prof}$ inverts a test of $H_0:p^\prime \theta=\tau$ over a grid for $\tau$. We do not in practice exhaust the grid but search inward from the extreme points of $\Theta$ in directions $\pm p$. At each $\tau$ that is visited, we use BCS code to compute a profiled test statistic and the corresponding critical value $\hat{c}_n^{MR}(\tau)$. The latter is a quantile of the minimum of two distinct bootstrap approximations, each of which solves a nonlinear program for each bootstrap draw. Computational cost quickly increases with grid resolution, bootstrap size, and the number of starting points used to solve the nonlinear programs.

Calibrated projection computes $\hat{c}_n(\theta)$ by solving a series of linear programs for each bootstrap draw.\footnote{We implement this step using the high-speed solver CVXGEN, available from \url{http://cvxgen.com} and described in \cite{CVXGEN}.}  It computes the extreme points of $CI_n$ by solving the nonlinear program \eqref{eq:CI} twice, a task that is much accelerated by the E-A-M algorithm. Projection of \cite{AS} operates very similarly but computes its critical value $\hat{c}^{proj}_n(\theta)$ through bootstrap simulation without any optimization.  

We align grid resolution in BCS-profiling with the E-A-M algorithm's convergence threshold of $0.005$.\footnote{This is only one of several individually necessary stopping criteria. Others include that the current optimum $\theta^{*,L}$ and the expected improvement maximizer $\theta^{L+1}$ (see equation \eqref{eq:max_ei}) satisfy $|p^\prime (\theta^{L+1} - \theta^{*,L})| \le 0.005$. See \cite{KMST_code} for the full list of convergence requirements.} We run all methods with $B=301$ bootstrap draws, and calibrated and ``uncalibrated" (i.e., based on \cite{AS}) projection also with $B=1001$.\footnote{Based on some trial runs of BCS-profiling for $\delta_1$, we estimate that running it with $B=1001$ throughout would take 3.14-times longer than the computation times reported in Table \ref{tab:mcBCSshort}. By comparison, calibrated projection takes only 1.75-times longer when implemented with $B=1001$ instead of $B=301$.} Some other choices differ: BCS-profiling is implemented with their own choice to multi-start the nonlinear programs at 3 oracle starting points, i.e. using knowledge of the true DGP; our implementation of both other methods multi-starts the nonlinear programs from 30 data dependent random points (see \cite{KMST_code} for details).

Table \ref{tab:mcBCSshort} displays results for $(\delta_1,\delta_2)$ and for 300 Monte Carlo repetitions of all three methods. All confidence intervals are conservative, reflecting the effect of GMS. As expected, uncalibrated projection is most conservative, with coverage of essentially $1$. 
Also, BCS-profiling is more conservative than calibrated projection. The most striking contrast is in computational effort. Here, uncalibrated projection is fastest -- indeed, in contrast to received wisdom, this procedure is computationally somewhat easy. This is due to our use of the E-A-M algorithm and therefore part of this paper's contribution. Next, our implementation of calibrated projection beats BCS-profiling with gridding by a factor of about $70$. This can be disentangled into the gain from using calibrated projection, with its advantage of bootstrapping linear programs, and the gain afforded by the E-A-M algorithm. It turns out that implementing BCS-profiling with the adapted E-A-M algorithm (see Appendix \ref{sec:EAM_for_BCS}) improves computation by a factor of about $4$; switching to calibrated projection leads to a further improvement by a factor of about $17$. Finally, Table \ref{tab:mcBCS} extends the analysis to all components of $\theta$ and to 1000 Monte Carlo repetitions. We were unable to compute this for BCS-profiling.

In sum, the Monte Carlo experiment on the same DGP used in BCS yields three interesting findings: (i) The E-A-M algorithm accelerates projection of the \cite{AS} confidence region to the point that this method becomes reasonably cheap; (ii) it also substantially accelerates computation of profiling intervals, and (iii) for this DGP, calibrated projection combined with the E-A-M algorithm has the most accurate size control while also being computationally attractive.

\newpage

\section*{Tables}

\begin{table}[h]
\caption{Results for empirical application, with $\alpha=0.05$, $\rho=6.6055$, $n=7882$, $\kappa_n=\sqrt{\ln n}$. ``Direct search" refers to \texttt{fmincon} performed after E-A-M and starting from feasible points discovered by E-A-M, including the E-A-M optimum.}
\label{tab:empirical}
\begin{center}
{\renewcommand{\arraystretch}{1.3}
\begin{tabular}{c|cc|ccc} 
\hline
 \hline
 & \multicolumn{2}{c} {$CI_n$}  &  \multicolumn{3}{|c}{Computational Time}          \\             
& E-A-M & Direct Search &  E-A-M  & Direct Search & Total  \\
 \hline
$\vartheta^{cons}_{LCC}$         & $[-2.0603,-0.8510]$             & $[-2.0827,-0.8492]$        & 24.73  & ${\color{white}0}32.46$       & ${\color{white}0}57.51$ \\
$\vartheta^{size}_{LCC}$         & $[0.1880,0.4029]$             & $[0.1878,0.4163]$         & 16.18  & 230.28       & 246.49 \\
$\vartheta^{pres}_{LCC}$         & $[1.7510,1.9550]$              & $[1.7426,1.9687]$         & 16.07  & 115.20       & 131.30 \\
$\vartheta^{cons}_{OA}$         & $[0.3957,0.5898]$  & $[0.3942,0.6132]$         & 27.61 & 107.33       & 137.66 \\
$\vartheta^{size}_{OA}$         & $[0.3378,0.5654]$             & $[0.3316,0.5661]$         & 11.90 & 141.73       & 153.66 \\
$\vartheta^{pres}_{OA}$         & $[0.3974,0.5808]$             & $[0.3923,0.5850]$          & 13.53 & 148.20       & 161.75 \\
$\delta_{LCC}$         & $[-1.4423,-0.1884]$            & $[-1.4433,-0.1786]$        & 15.65 & 119.50       & 135.17 \\
$\delta_{OA}$          & $[-1.4701,-0.7658]$            & $[-1.4742,-0.7477]$        & 13.06 & 114.14       & 127.23  \\
$r$         & $[0.1855,0.85]{\color{white}00}$               & $[0.1855,0.85]{\color{white}00}$           & ${\color{white}0}5.37$ & ${\color{white}0}42.38$ & ${\color{white}0}47.78$\\
\hline
\hline
\end{tabular}}
\end{center}
\end{table}

\newpage
\begin{table}[h]
\begin{center}
\small
\caption{Results for Set 1 with $n=4000$, $MCs = 300$, $B = 301$, $\rho=5.04$, $\kappa_n=\sqrt{\ln n}$.}
\label{tab:mcBCSshort}
\bigskip

{\renewcommand{\arraystretch}{1.2}\begin{tabular}{c|c|c|ccccccc}\hline \hline
\multirow{3}{*} & \multirow{3}{*} {$1-\alpha$} &  \multicolumn{8}{c}{Median CI} \\
\cline{3-10}
&         &  \multicolumn{4}{c|}{$CI_n^{prof}$} & \multicolumn{2}{c|}{$CI_n$} & \multicolumn{2}{c}{$CI_n^{proj}$}  \\
\hline
Implementation &         &  \multicolumn{2}{c}{Grid} & \multicolumn{2}{c|}{E-A-M} & \multicolumn{2}{c|}{E-A-M} & \multicolumn{2}{c}{E-A-M}  \\
\hline
\multirow{3}{*} {$\delta_1=0.4$}  & 0.95    &  \multicolumn{2}{c}{[0.330,0.495]} & \multicolumn{2}{c|}{[0.331,0.495]}  & \multicolumn{2}{c|}{[0.336,0.482]} & \multicolumn{2}{c}{[0.290,0.558]} \\
                                  & 0.90    &  \multicolumn{2}{c}{[0.340,0.485]} & \multicolumn{2}{c|}{[0.340,0.485]}  & \multicolumn{2}{c|}{[0.343,0.474]} & \multicolumn{2}{c}{[0.298,0.543]} \\
                                  & 0.85    &  \multicolumn{2}{c}{[0.345,0.475]} & \multicolumn{2}{c|}{[0.346,0.479]}  & \multicolumn{2}{c|}{[0.348,0.466]} & \multicolumn{2}{c}{[0.303,0.537]} \\
                        \hline
\multirow{3}{*} {$\delta_2=0.6$}  & 0.95    &  \multicolumn{2}{c}{[0.515,0.655]} & \multicolumn{2}{c|}{[0.514,0.655]}  & \multicolumn{2}{c|}{[0.519,0.650]} & \multicolumn{2}{c}{[0.461,0.682]}  \\
                                  & 0.90    &  \multicolumn{2}{c}{[0.525,0.647]} & \multicolumn{2}{c|}{[0.525,0.648]}  & \multicolumn{2}{c|}{[0.531,0.643]} & \multicolumn{2}{c}{[0.473,0.675]}  \\
                                  & 0.85    &  \multicolumn{2}{c}{[0.530,0.640]} & \multicolumn{2}{c|}{[0.531,0.642]}  & \multicolumn{2}{c|}{[0.539,0.639]} & \multicolumn{2}{c}{[0.481,0.671]}  \\		
\hline\hline				
\multicolumn{10}{c}{}\\
\multicolumn{10}{c}{}\\

\hline \hline
\multirow{3}{*} & \multirow{3}{*} {$1-\alpha$} &  \multicolumn{8}{c}{Coverage} \\
\cline{3-10}
&         &  \multicolumn{4}{c|}{$CI_n^{prof}$} & \multicolumn{2}{c|}{$CI_n$} & \multicolumn{2}{c}{$CI_n^{proj}$}  \\
\hline
Implementation &         &  \multicolumn{2}{c}{Grid} & \multicolumn{2}{c|}{E-A-M} & \multicolumn{2}{c|}{E-A-M} & \multicolumn{2}{c}{E-A-M}  \\
\cline{3-10}
      &         &  \multicolumn{1}{c}{Lower}        & \multicolumn{1}{c}{Upper} & \multicolumn{1}{c}{Lower}        & \multicolumn{1}{c|}{Upper}  & \multicolumn{1}{c}{Lower}         & \multicolumn{1}{c|}{Upper}  & Lower         & Upper \\
\hline
\multirow{3}{*} {$\delta_1=0.4$}  & 0.95   & \multicolumn{1}{c}{0.997} & \multicolumn{1}{c}{0.990} & \multicolumn{1}{c}{1.000} & \multicolumn{1}{c|}{0.993} & 0.993 & \multicolumn{1}{c|}{0.977} & 1.000  & 1.000 \\
                                  & 0.90   & \multicolumn{1}{c}{0.990} & \multicolumn{1}{c}{0.980} & \multicolumn{1}{c}{0.993} & \multicolumn{1}{c|}{0.977} & 0.987 & \multicolumn{1}{c|}{0.960} & 1.000  & 1.000 \\
                                  & 0.85   & \multicolumn{1}{c}{0.970} & \multicolumn{1}{c}{0.970} & \multicolumn{1}{c}{0.973} & \multicolumn{1}{c|}{0.960} & 0.957 & \multicolumn{1}{c|}{0.930} & 1.000  & 1.000 \\
                       \hline                                   
\multirow{3}{*} {$\delta_2=0.6$}  & 0.95   & \multicolumn{1}{c}{0.987} & \multicolumn{1}{c}{0.993} & \multicolumn{1}{c}{0.990} & \multicolumn{1}{c|}{0.993} & 0.973 & \multicolumn{1}{c|}{0.987} & 1.000  & 1.000  \\
                                  & 0.90   & \multicolumn{1}{c}{0.977} & \multicolumn{1}{c}{0.973} & \multicolumn{1}{c}{0.980} & \multicolumn{1}{c|}{0.977} & 0.940 & \multicolumn{1}{c|}{0.953} & 1.000  & 1.000  \\
                                  & 0.85   & \multicolumn{1}{c}{0.967} & \multicolumn{1}{c}{0.957} & \multicolumn{1}{c}{0.963} & \multicolumn{1}{c|}{0.960} & 0.943 & \multicolumn{1}{c|}{0.927} & 1.000  & 1.000  \\
\hline\hline								  
\multicolumn{10}{c}{}\\
\multicolumn{10}{c}{}\\
\hline \hline
\multirow{3}{*} & \multirow{3}{*} {$1-\alpha$} &  \multicolumn{8}{c}{Average Time} \\
\cline{3-10}
&         &  \multicolumn{4}{c|}{$CI_n^{prof}$} & \multicolumn{2}{c|}{$CI_n$} & \multicolumn{2}{c}{$CI_n^{proj}$}  \\
\hline
Implementation &         &  \multicolumn{2}{c}{Grid} & \multicolumn{2}{c|}{E-A-M} & \multicolumn{2}{c|}{E-A-M} & \multicolumn{2}{c}{E-A-M}  \\
\hline

\multirow{3}{*} {$\delta_1=0.4$}  & 0.95 &  \multicolumn{2}{c}{1858.42} & \multicolumn{2}{c|}{425.49}   & \multicolumn{2}{c|}{26.40} & \multicolumn{2}{c}{18.22} \\
                                  & 0.90 &  \multicolumn{2}{c}{1873.23} & \multicolumn{2}{c|}{424.11}   & \multicolumn{2}{c|}{25.71} & \multicolumn{2}{c}{18.55} \\
                                  & 0.85 &  \multicolumn{2}{c}{1907.84} & \multicolumn{2}{c|}{444.45}   & \multicolumn{2}{c|}{25.67} & \multicolumn{2}{c}{18.18} \\
                        \hline                                                                                                                         
\multirow{3}{*} {$\delta_2=0.6$}  & 0.95 &  \multicolumn{2}{c}{1753.54} & \multicolumn{2}{c|}{461.30}   & \multicolumn{2}{c|}{26.61} & \multicolumn{2}{c}{22.49} \\
                                  & 0.90 &  \multicolumn{2}{c}{1782.91} & \multicolumn{2}{c|}{472.55}   & \multicolumn{2}{c|}{25.79} & \multicolumn{2}{c}{21.38} \\
                                  & 0.85 &  \multicolumn{2}{c}{1809.65} & \multicolumn{2}{c|}{458.58}   & \multicolumn{2}{c|}{25.00} & \multicolumn{2}{c}{21.00} \\
\hline\hline
\end{tabular}}

 \begin{tablenotes}
      \footnotesize
      \item Notes: (1) Projections of $\Theta_I$ are: $\delta_1 \in [0.3872,0.4239]$, $\delta_2 \in [0.5834   ,0.6084]$, $\zeta_1\in [0.0996,0.1006]$, $\zeta_2 \in [0.1994,0.2010]$, $\zeta_3 \in [0.2992,0.3014]$. (2) ``Upper" coverage is for $\max_{\theta \in \Theta_I(P)}p^\prime \theta$, and similarly for ``Lower". (3) ``Average time" is computation time in seconds averaged over MC replications. (4) $CI_n^{prof}$ results from BCS-profiling, $CI_n$ is calibrated projection, and $CI_n^{proj}$ is uncalibrated projection. (5) ``Implementation" refers to the method used to compute the extreme points of the confidence interval.
    \end{tablenotes}
\end{center}
\end{table}

\vspace{0.15in}
\begin{table}[h]
\begin{center}
\small
\caption{Results for Set 1 with $n=4000$, $MCs = 1000$, $B = 999$, $\rho=5.04$, $\kappa_n=\sqrt{\ln n}$.}
\vspace{0.15in}
\label{tab:mcBCS}
\def\arraystretch{0.9}
\begin{tabular}{c|c|cc|cc|cc|cc}\hline \hline
\multirow{2}{*}
                        & \multirow{2}{*} {$1-\alpha$} &  \multicolumn{2}{c|}{Median CI}                       &  \multicolumn{2}{c|}{$CI_n$ Coverage}        &  \multicolumn{2}{c|}{$CI_n^{proj}$ Coverage}         &  \multicolumn{2}{c}{Average Time}       \\
                        &         &  $CI_n$ & $CI_n^{proj}$ & Lower        & Upper  & Lower         & Upper  & $CI_n$ & $CI_n^{proj}$   \\
                        \hline
\multirow{3}{*} {$\delta_1=0.4$}     & 0.95    & [0.333,0.478] & [0.288,0.555]  & 0.988 & 0.982 & 1      & 1     & 42.41 & 22.23 \\
                        		     & 0.90    & [0.341,0.470] & [0.296,0.542]  & 0.976 & 0.957 & 1      & 1     & 41.56 & 22.11 \\
                        		     & 0.85    & [0.346,0.464] & [0.302,0.534]  & 0.957 & 0.937 & 1      & 1     & 40.47 & 19.79 \\
                        \hline                                                                                   
\multirow{3}{*} {$\delta_2=0.6$}     & 0.95    & [0.525,0.653] & [0.466,0.683]  & 0.969 & 0.983 & 1      & 1     & 42.11 & 24.39 \\
                        		     & 0.90    & [0.538,0.646] & [0.478,0.677]  & 0.947 & 0.960 & 1      & 1     & 40.15 & 28.13 \\
                        		     & 0.85    & [0.545,0.642] & [0.485,0.672]  & 0.925 & 0.941 & 1      & 1     & 41.38 & 26.44 \\
                        \hline                                                                                   
\multirow{3}{*} {$\zeta^{[1]}=0.1$}  & 0.95    & [0.054,0.142] & [0.020,0.180]  & 0.956 & 0.958 & 1      & 1     & 40.31 & 22.53 \\
                        			 & 0.90    & [0.060,0.136] & [0.028,0.172]  & 0.911 & 0.911 & 1      & 1     & 36.80 & 24.15 \\
                        			 & 0.85    & [0.064,0.132] & [0.032,0.167]  & 0.861 & 0.860 & 0.999  & 0.999 & 39.10 & 21.81 \\
                        \hline
\multirow{3}{*} {$\zeta^{[2]}=0.2$}  & 0.95    & [0.156,0.245] & [0.121,0.281]  & 0.952 & 0.952 & 1      & 1     & 39.23 & 24.66 \\
                        			 & 0.90    & [0.162,0.238] & [0.128,0.273]  & 0.914 & 0.910 & 0.998  & 0.998 & 41.53 & 21.66 \\
                        			 & 0.85    & [0.165,0.234] & [0.133,0.268]  & 0.876 & 0.872 & 0.996  & 0.996 & 39.44 & 22.83 \\
 \hline
\multirow{3}{*} {$\zeta^{[3]}=0.3$}  & 0.95    & [0.257,0.344] & [0.222,0.379]  & 0.946 & 0.946 & 1      & 1     & 41.45 & 22.91 \\
                        			 & 0.90    & [0.263,0.338] & [0.230,0.371]  & 0.910 & 0.909 & 0.997  & 0.999 & 42.09 & 22.83 \\
                        			 & 0.85    & [0.267,0.334] & [0.235,0.366]  & 0.882 & 0.870 & 0.994  & 0.993 & 42.19 & 23.69 \\
\hline \hline
\end{tabular}

 \begin{tablenotes}
      \footnotesize
      \item \hspace{1.25in} Notes: Same DGP and conventions as in Table \ref{tab:mcBCSshort}.
    \end{tablenotes}
\end{center}
\end{table}

\end{appendix}

\numberwithin{equation}{section}
\numberwithin{figure}{section}
\numberwithin{table}{section}

 \newgeometry{
 letterpaper,
 total={210mm,297mm},
 left=20mm,
 right=20mm,
 top=30mm,
 bottom=30mm,
 }
%


\pagebreak
\onehalfspacing

\pagenumbering{arabic}

\setdisplayskipstretch{0.8}

\small

\begin{appendices}

\begin{center}\huge
Online Appendix: \\
Confidence Intervals for\linebreak
Projections of Partially Identified Parameters 
\par\end{center}\vspace{\baselineskip}

\maketitle
\etocdepthtag.toc{mtappendix}
\etocsettagdepth{mtchapter}{none}
\etocsettagdepth{mtappendix}{subsection}
\tableofcontents

\bigskip

\section*{Structure of the Appendix}
Section \ref{sec:background} states and proofs Theorem \ref{cor:eam_conv}, which establishes convergence-related results for our E-A-M algorithm. It also provides background material for the E-A-M algorithm, and details on the root-finding algorithm that we use to compute $\hat{c}_n(\theta)$. 
Section \ref{sec:AssRes} presents the assumptions under which we prove asymptotic uniform validity of coverage of our procedure.
Section \ref{sec:verify_examples} verifies some of our main assumptions for moment (in)equality models that have received much attention in the literature.
Section \ref{sec:app_A} summarizes the notation we use and the structure of the proof of Theorem \ref{thm:validity},\footnote{Section \ref{subsec:notation} provides in Table \ref{table:notation} a summary of the notation used throughout, and in Figure \ref{fig:flow_app} and Table \ref{table:flow} a flow diagram and heuristic explanation of how each lemma contributes to the proof of Theorem \ref{thm:validity}.} 
and provides a proof of Theorems \ref{thm:validity} (both under our main assumptions and under a high level assumption replacing Assumption \ref{as:correlation} and dropping the $\rho$-box constraints). 
Section \ref{app:Lemma} contains the statements and proofs of the lemmas used to establish Theorems \ref{thm:validity} and \ref{cor:eam_conv}, as well as a rigorous derivation of the almost sure representation result for the bootstrap empirical process that we use in the proof of Theorem \ref{thm:validity}. 

Throughout the Appendix we use the convention $\infty \cdot 0 = 0$. 

\setcounter{section}{3}

\section{Additional Convergence Results and Background Materials for the E-A-M algorithm and for Computation of $\hat{c}_n(\theta)$ }
\label{sec:background}
\subsection{Theorem \ref{cor:eam_conv}: An Approximating Critical Level Sequence for the E-A-M Algorithm}
\subsubsection{Assumption \ref{as:Lipschitz_m_over_sigma}: A Low Level Condition Yielding 
a Stochastic Lipschitz-Type Property for $\hat{c}_n$}
\label{app:as:Lipschitz}
In order to establish convergence of our E-A-M algorithm, we need $\hat{c}_n$ to uniformly  stochastically exhibit a Lipschitz-type property so that its mollified counterpart (see equation \eqref{eq:c_hat_mollified}) is sufficiently smooth and yields valid inference. Below we provide a low level condition under which we are able to establish the Lipschitz-type property. In Appendix \ref{sec:verify_ass_EAM} we verify the condition for the canonical examples in the moment (in)equality literature.
\begin{assumption}
\label{as:Lipschitz_m_over_sigma}
The model $\cP$ for $P$ satisfies: 
\begin{itemize}
	\item[(i)] $|\sigma_{P,j}(\theta)^{-1}m_j(x,\theta)-\sigma_{P,j}(\theta')^{-1}m_j(x,\theta')|\le \bar M(x)\|\theta-\theta'\|$ with $E_P[\bar M(X)^{2}]<M$  for all $\theta,\theta'\in\Theta$, $x\in\mathcal X$, $j=1,\cdots,J$, 
	 and there exists a function $F$ such that $|\sigma_{P,j}(\theta)^{-1}m_j(\cdot,\theta)|\le F(\cdot)$ for all $\theta\in\Theta$ 
	 and $E_P[|F(X)\bar M(X)|^2]<M$.
	\item[(ii)] $\varphi_j$ is Lipschitz continuous in $x\in\mathbb R$ for all $j=1,\dots,J.$
\end{itemize}
\end{assumption}

\subsubsection{Statement and Proof of Theorem \ref{cor:eam_conv}}
For all $\tau>0$ let $\hat{c}_{n,\tau}(\theta)$ be a mollified version of $\hat{c}_n(\theta)$, i.e.:
\begin{align}
\hat{c}_{n,\tau}(\theta)&=\int_{\R^d} \hat{c}_n(\theta-\nu)\phi_\tau(\nu)d\nu=\int_{\R^d} \hat{c}_n(\theta)\phi_\tau(\theta-\nu)d\nu,\label{eq:c_hat_mollified}
\end{align}
where the family of functions $\phi_\tau$ is a mollifier as defined in \cite[Example 7.19]{Rockafellar_Wets2005aBK}. Choose it to be a family of bounded, measurable, smooth functions such that
$\phi_\tau(z) \ge 0 ~\forall z \in \R^d$, $\int_{\R^d} \phi_\tau(z)dz=1$ and with $\mathbb{B_\tau}=\{z:\phi_\tau(z) > 0\}=\{z:\Vert z \Vert \le \tau\}$.
\begin{theorem}\label{cor:eam_conv}	
Suppose Assumptions  \ref{as:momP_AS}, \ref{as:GMS}, \ref{as:momP_KMS}, \ref{as:bcs1} and \ref{as:Lipschitz_m_over_sigma} hold. 
Let $\tau_n$ be a positive sequence such that $\tau_n=n^{-\zeta}$ with $\zeta> 1/2$.
 Let $\{\beta_n\}$ be a positive sequence such that  $\beta_n=o(1)$ and $\|\hat D_n-D_P\|_{\infty}=O_{\mathcal P}(\beta_n)$. 
Let $\varepsilon_n=\kappa_n^{-1} \sqrt n\tau_n\vee  \beta_n$.
Then, 
\begin{enumerate}
\item \label{cor:eam_conv:c_hat_Lip}
\begin{align}
\limsup_{n\to\infty}\sup_{P\in\mathcal P}P\left(\sup_{\|\theta-\theta'\|\le\tau_n}|\hat{c}_n(\theta)-\hat{c}_n(\theta^\prime)| > C\varepsilon_n\right)=0; \label{eq:c_hat_Lip}
\end{align}
\item \label{cor:eam_conv:c_mollified_Lip}
Let $\hat{c}_{n,\tau_n}$ be defined as in  \eqref{eq:c_hat_mollified} with $\tau_n$ replacing $\tau$. Then there exists $C>0$ such that
\begin{align}
\liminf_{n\to\infty}\inf_{P\in\mathcal P}P\Big(\Vert\hat{c}_n-\hat{c}_{n,\tau_n}\Vert_\infty \le C\varepsilon_n\Big)=1; \label{eq:c_mollified_Lip}
\end{align}
\item \label{cor:eam_conv:ctau_valid}
Let Assumption \ref{as:correlation} also hold. Let $\{P_n,\theta_n\}$ be a sequence such that $P_n\in\mathcal P$ and $\theta_n \in \Theta_I(P_n)$ for all $n$ and $\kappa_n^{-1}\sqrt{n}\gamma_{1,P_n,j}(\theta_n) \to \pi_{1j} \in \mathbb R_{[-\infty]},~j=1,\dots,J,$
$\Omega_{P_n} \uni \Omega,$ and
$D_{P_n}(\theta_n) \to D$. Let
\begin{align}
\hat c_{n,\rho,\tau}(\theta)\equiv \inf_{\lambda \in B^d_{n,\rho}}\hat c_{n,\tau}(\theta+\frac{\lambda\rho}{\sqrt n}).\label{eq:def_chat_nrhotau}
\end{align}	
For $c\ge 0$, let $U_n(\theta_n,c)$ be defined as in \eqref{eq:set_U_NL}. Then,
\begin{align}
\liminf_{n \to \infty} P_n\left( U_n(\theta_n,\hat{c}_{n,\rho,\tau_n})\neq \emptyset \right) \ge 1-\alpha.
\label{eq:validity_mollified_hat_c}
\end{align}
\item \label{cor:eam_conv:ctau_smooth}
Fix $P\in\cP$ and $n$. There exists  $R>0$ such that $\|\hat c_{n,\tau_n}\|_{\mathcal H_\beta}\le R$.
\end{enumerate}
\end{theorem}

\begin{proof} We establish each part of the theorem separately.

 \textbf{Part \ref{cor:eam_conv:c_hat_Lip}.}
Throughout, let $C>0$ denote a positive constant, which may be different in different appearances.
Define the event
\begin{multline}
E_n\equiv\big\{x^\infty\in\mathcal X^\infty:\|\hat D_n-D_{P}\|_{\infty}\le C \beta_n,~\sup_{\|\theta-\theta'\|\le \tau_n}\|\mathbb G_n(\theta)-\mathbb G_n(\theta'))\|\le (\ln n)^2\tau_n,\\
\sup_{\theta\in\Theta}|\eta_{n,j}(\theta)|\le C/\sqrt n,~\max_{j=1,\cdots,J}\sup_{\|\theta-\theta'\|<\tau_n}|\eta_{n,j}(\theta)-\eta_{n,j}(\theta')|\le C\tau_n\big\}.\label{eq:moll_lip1}
\end{multline}	
Note that $(\ln n)^2\tau_n/(-\tau_n\ln\tau_n)=(\ln n)^2/\zeta \ln n=\ln n/\zeta$, and hence tends to $\infty.$ By Assumption \ref{as:Lipschitz_m_over_sigma}-(i) and arguing as in the proof of Theorem 2 in  \cite{Andrews:1994aa}, condition  \eqref{eq:as_euclidean} in Lemma \ref{lem:se_rate} is satisfied with $v=d$. Also, by Lemma \ref{lem:correl_Lip}, \eqref{eq:correl_holder} in Lemma \ref{lem:se_rate} holds with $\gamma=1$. This therefore ensures the conditions of Lemma \ref{lem:se_rate}.

Similarly, by Assumption \ref{as:Lipschitz_m_over_sigma}-(i) $m^2_j(x,\theta)/\sigma_{P,j}^2(\theta)$ satisfies
\begin{align}
	\Big|\frac{m^2_j(x,\theta)}{\sigma_{P,j}^2(\theta)}-\frac{m^2_j(x,\theta)}{\sigma_{P,j}^2(\theta)}\Big|&\le \Big|\frac{m_j(x,\theta)}{\sigma_{P,j}(\theta)}+\frac{m_j(x,\theta')}{\sigma_{P,j}(\theta')}\Big|\Big|\frac{m_j(x,\theta)}{\sigma_{P,j}(\theta)}-\frac{m_j(x,\theta')}{\sigma_{P,j}(\theta')}\Big|\\
	&\le 2F(x)\bar M(x)\|\theta-\theta'\|.
\end{align}
Let $\bar F(x)\equiv 2F(x)\bar M(x)$.
By Theorem 2.7.11 in \cite{Vaart_Wellner2000aBK}, 
\begin{align}
 N_{[]}(\epsilon \|\bar F\|_{L^2_P},\mathcal M^2_{P},\|\cdot\|_{L^2_P})\le N(\epsilon,\Theta,\|\cdot\|)\le (\text{diam}(\Theta)/\epsilon)^d,
\end{align}
 where $N(\epsilon,\Theta,\|\cdot\|)$ is the covering number of $\Theta$. This ensures
\begin{align}
	\int_0^\infty \sup_{P\in\mathcal P}\sqrt{\ln N_{[]}(\epsilon \|\bar F\|_{L^2_P},\mathcal M^2_{P},\|\cdot\|_{L^2_P})}d\epsilon<\infty.
\end{align}
Further, for any $C>0$ 
\begin{align}
E_P[\bar F^2(X)1\{\bar F(X)>C\}]~\le~ E_P[\bar F^2(X)]P(\bar F(X)>C)	~\le~ 4E_P[|F(X)M(X)|^2]\frac{\|\bar F\|_{L^1_P}}{C}~\le~\frac{4M^2}{C},
\end{align}
which implies $\lim_{C\to\infty}\sup_{P\in\mathcal P}E_P[\bar F^2(X)1\{\bar F(X)>C\}]=0$. By Theorems 2.8.4 and 2.8.2 in \cite{Vaart_Wellner2000aBK}, this implies that $\mathcal S_P$ is Donsker and pre-Gaussian uniformly in $P\in\mathcal P.$ This therefore ensures the conditions of Lemma \ref{lem:eta_rate} (i). Note also that Assumption \ref{as:Lipschitz_m_over_sigma}-(i) ensures the conditions of Lemma \ref{lem:eta_rate} (ii).
Therefore, by Lemmas \ref{lem:se_rate}-\ref{lem:eta_rate} and Assumption \ref{as:momP_KMS}, for any $\eta>0$, there exists $C>0$ such that $\inf_{P\in\mathcal P}P(E_n)\ge1-\eta$ for all $n$ sufficiently large.

Let $\theta,\theta'\in \Theta$. For each $j$, we have
\begin{multline}
\left\vert\mathbb{G}_{n,j}^{b}(\theta )+\rho\hat{D}
_{n,j}(\theta)\lambda +\varphi_j(\hat{\xi}_{n,j}(\theta))-\mathbb{G}_{n,j}^{b}(\theta^\prime)-\rho\hat{D}
_{n,j}(\theta^\prime)\lambda -\varphi_j(\hat{\xi}_{n,j}(\theta^\prime))\right\vert \\\le 
|\mathbb{G}_{n,j}^{b}(\theta)-\mathbb{G}_{n,j}^{b}(\theta')|+\rho\|\hat D_{n,j}(\theta)-\hat D_{n,j}(\theta')\|\sup_{\lambda \in B^d}\|\lambda\|+|\varphi_j(\hat{\xi}_{n,j}(\theta))-\varphi_j(\hat{\xi}_{n,j}(\theta^\prime))|.\label{eq:moll_lip2}
\end{multline}
Assume that the sample path $\{X_i\}_{i=1}^\infty$ is such that the event $E_n$ holds. Conditional on $\{X_i\}_{i=1}^\infty$ and using $\mathbb{G}_{n,j}^{b}(\theta)-\mathfrak G^b_{n,j}(\theta)=\mathfrak G^b_{n,j}(\theta)\eta_{n,j}(\theta)$,
\begin{multline}
|\mathbb{G}_{n,j}^{b}(\theta)-\mathbb{G}_{n,j}^{b}(\theta')|\le |\mathfrak G^b_{n,j}(\theta)-\mathfrak G^b_{n,j}(\theta')|+2\sup_{\theta\in\Theta}|\mathfrak G^b_{n,j}(\theta)|\sup_{\theta\in\Theta}|\eta_{n,j}(\theta)|\\
\le |\mathfrak G^b_{n,j}(\theta)-\mathfrak G^b_{n,j}(\theta')|+2\sup_{\theta\in\Theta}|\mathfrak G^b_{n,j}(\theta)|\frac{C}{\sqrt n}.\label{eq:moll_lip3}
\end{multline}

Define the event $F_n\in \mathcal C$ for the bootstrap weights by
\begin{align}
F_n\equiv\big\{m_n\in Q:\sup_{\|\theta-\theta'\|\le\tau_n}\|\mathfrak G^b_{n}(\theta)-\mathfrak G^b_{n}(\theta')\|\le (\ln n)^2\tau_n,~\sup_{\theta\in\Theta}\|\mathfrak G^b_{n}(\theta)\|\le C\big\}.\label{eq:moll_lip4}
\end{align}
By Lemma \ref{lem:se_rate} (ii) and the asymptotic tightness of $\mathfrak G^b_n$, for any $\eta>0$, there exists a $C$ such that $P^*_n(F_n)\ge 1-\eta$ for all $n$ sufficiently large.
Suppose that the multinomial bootstrap weight $M_n$ is such that
$F_n$ holds. Then, the right hand side of \eqref{eq:moll_lip3} is bounded by
$(\ln n)^2\tau_n+ C/\sqrt n$ for some $C>0$.

Next, by the triangle inequality and Assumption \ref{as:momP_KMS},
\begin{multline}
\|\hat D_{n,j}(\theta)-\hat D_{n,j}(\theta')\|\le \|\hat D_{n,j}(\theta)-D_{P,j}(\theta)\|+\|D_{P,j}(\theta)-D_{P,j}(\theta')\|+\|\hat D_{n,j}(\theta')-D_{P,j}(\theta')\|\\
\le C\beta_n + C\tau_n.\label{eq:moll_lip5}
\end{multline}

Finally, note that by the Lipschitzness of $\varphi_j$,  $|\varphi_j(\hat{\xi}_{n,j}(\theta))-\varphi_j(\hat{\xi}_{n,j}(\theta^\prime))|\le C|\hat{\xi}_{n,j}(\theta)-\hat{\xi}_{n,j}(\theta^\prime)|$ and
\begin{multline}
\hat{\xi}_{n,j}(\theta)-\hat{\xi}_{n,j}(\theta^\prime)
\\
= \kappa_n^{-1}\Big[\sqrt n\Big(\frac{\bar m_{n,j}(\theta)}{ \sigma_{P,j}(\theta)}(1+\eta_{n,j}(\theta))-\frac{E_P[m_j(X,\theta)]}{ \sigma_{P,j}(\theta)}\Big)-\sqrt n\Big(\frac{\bar m_{n,j}(\theta')}{ \sigma_{P,j}(\theta')}(1+\eta_{n,j}(\theta'))-\frac{E_P[m_j(X,\theta')]}{ \sigma_{P,j}(\theta')}\Big)\Big]\\
+\kappa_n^{-1}\sqrt n\Big(\frac{E_P[m_j(X,\theta)]}{ \sigma_{P,j}(\theta)}-\frac{E_P[m_j(X,\theta')]}{ \sigma_{P,j}(\theta')}\Big).\label{eq:moll_lip6}
\end{multline}
Hence,
\begin{multline}
|\hat{\xi}_{n,j}(\theta)-\hat{\xi}_{n,j}(\theta^\prime)|\le \kappa_n^{-1}|\mathbb G_{n,j}(\theta)-\mathbb G_{n,j}(\theta')|\\
+ \kappa_n^{-1}\sqrt n\Big|\frac{\bar m_{n,j}(\theta)}{ \sigma_{P,j}(\theta)}\eta_{n,j}(\theta)-\frac{\bar m_{n,j}(\theta')}{ \sigma_{P,j}(\theta')}\eta_{n,j}(\theta')\Big|+\kappa_n^{-1}\sqrt n D_{P,j}(\bar\theta)\|\theta-\theta'\|.\label{eq:moll_lip7}
\end{multline}
By Lemma \ref{lem:se_rate}, the right hand side of \eqref{eq:moll_lip7} can be further bounded by
\begin{multline}
\kappa_n^{-1}(\ln n)^2\tau_n+\kappa_n^{-1}\sqrt n\Big|\frac{\bar m_{n,j}(\theta)}{ \sigma_{P,j}(\theta)}-\frac{\bar m_{n,j}(\theta')}{ \sigma_{P,j}(\theta')}\Big||\eta_{n,j}(\theta)|\\
+\kappa_n^{-1}\sqrt n\Big|\frac{\bar m_{n,j}(\theta')}{ \sigma_{P,j}(\theta')}\Big||\eta_{n,j}(\theta)-\eta_{n,j}(\theta')|+C\kappa_n^{-1}\sqrt n\tau_n\\
\le \kappa_n^{-1}(\ln n)^2\tau_n+\kappa_n^{-1}\sqrt n\tau_n \frac{C}{\sqrt n}+C\kappa_n^{-1}\sqrt n\tau_n+C\kappa_n^{-1}\sqrt n\tau_n,\label{eq:moll_lip8}
\end{multline}
where the last inequality follows from Condition (i) and Lemma \ref{lem:eta_rate} (ii).

Combining \eqref{eq:moll_lip2}, \eqref{eq:moll_lip3}, \eqref{eq:moll_lip5}, and \eqref{eq:moll_lip6}-\eqref{eq:moll_lip8}, we obtain
\begin{align}
\left\vert\mathbb{G}_{n,j}^{b}(\theta )+\hat{D}
_{n,j}(\theta)\lambda +\varphi_j(\hat{\xi}_{n,j}(\theta))-\mathbb{G}_{n,j}^{b}(\theta^\prime)-\hat{D}
_{n,j}(\theta^\prime)\lambda -\varphi_j(\hat{\xi}_{n,j}(\theta^\prime))\right\vert\le C\varepsilon_n.
\end{align}
In particular, if $\mathbf{1}\left(\Lambda_n^b (\theta ,\rho ,\hat{c}_n(\theta))\cap \{p^{\prime }\lambda =0\}\neq \emptyset \right)=1$, it also holds that $\mathbf{1}\big(\Lambda_n^b (\theta^\prime ,\rho ,\hat{c}_n(\theta)+C\varepsilon_n)\cap\{p^{\prime }\lambda =0\}\neq \emptyset \big)=1$ because
\begin{align*}
\mathbb{G}_{n,j}^{b}(\theta^\prime)+\hat{D}
_{n,j}(\theta^\prime)\lambda +\varphi_j(\hat{\xi}_{n,j}(\theta^\prime)) \le \mathbb{G}_{n,j}^{b}(\theta )+\hat{D}
_{n,j}(\theta)\lambda +\varphi_j(\hat{\xi}_{n,j}(\theta))+C\varepsilon_n \le \hat{c}_n(\theta)+C\varepsilon_n,
\end{align*}
Recalling that $P^*_n(F_n)\ge 1-\eta$ for all $n$ sufficiently large, we then have
\begin{multline}
P^*_n\left(\big\{\Lambda_n^b (\theta' ,\rho ,\hat{c}_n(\theta)+C\varepsilon_n)\cap \{p^{\prime }\lambda =0\}\neq \emptyset\big\}\right)\\	
\ge P^*_n\left(\big\{\Lambda_n^b (\theta' ,\rho ,\hat{c}_n(\theta)+C\varepsilon_n)\cap \{p^{\prime }\lambda =0\}\neq \emptyset\big\}\cap F_n \right)	\\
\ge
P^*_n\left(\big\{\Lambda_n^b (\theta ,\rho ,\hat{c}_n(\theta))\cap \{p^{\prime }\lambda =0\}\neq \emptyset\big\} \cap F_n\right)\ge 1-\alpha-\eta.
\end{multline}
Since $\eta$ is arbitrary, we have
\begin{align*}
\hat c_n(\theta')\le \hat c_n(\theta)+C\varepsilon_n.
\end{align*}
Reversing the roles of $\theta$ and $\theta^\prime$ and noting that $\sup_{P\in\mathcal P}P(E_n)\to 0$  yields the first claim of the lemma.

 \textbf{Part \ref{cor:eam_conv:c_mollified_Lip}.}
To obtain the result in equation~\eqref{eq:c_mollified_Lip}, we use that for any $\theta,\theta^\prime \in \Theta$ such that $\|\theta-\theta'\|\le \tau_n$, $|\hat{c}_n(\theta)-\hat{c}_n(\theta^\prime)| \le C \varepsilon_n$ with probability approaching 1 uniformly in $P\in\mathcal P$ by the result in Part \ref{cor:eam_conv:c_hat_Lip}. This implies
\begin{multline*}
\vert\hat{c}_n(\theta)-\hat{c}_{n,\tau_n}(\theta)\vert=\left\vert \int_{\R^d} \hat{c}_n(\theta-\nu)\phi_{\tau_n}(\nu)d\nu - \hat{c}_n(\theta) \right\vert\le  \int_{\R^d}\left\vert \hat{c}_n(\theta-\nu)- \hat{c}_n(\theta) \right\vert \phi_{\tau_n}(\nu)d\nu \\
=\int_{\mathbb{B}_{\tau_n}}\left\vert \hat{c}_n(\theta-\nu)- \hat{c}_n(\theta) \right\vert \phi_{\tau_n}(\nu)d\nu\le C\varepsilon_n \int_{\mathbb{B}_{\tau_n}}\phi_{\tau_n}(\nu)d\nu \le C\varepsilon_n.
\end{multline*}

 \textbf{Part \ref{cor:eam_conv:ctau_valid}.}
By Part \ref{cor:eam_conv:c_mollified_Lip} and the definition of $\hat c_{n,\rho,\tau}$ in \eqref{eq:def_chat_nrhotau}, it follows that
\begin{align}
\hat c_{n,\rho,\tau_n}(\theta_n)&\ge \hat c_{n,\rho}(\theta_n)-e_n\\
&\ge c^I_{n,\rho}(\theta_n)-e_n,\notag
\end{align}
for some $e_n=O_{\mathcal P}(\varepsilon_n)$, where the second inequality follows from the construction of $c^I_{n,\rho}$ in the proof of Lemma \ref{lem:val}.  
Note that  Lemma \ref{lem:cv_convergence} and the fact that $\varepsilon_n=o_{\mathcal P}(1)$ by Part \ref{cor:eam_conv:c_hat_Lip} imply $c^I_{n,\rho}(\theta_n)-e_n\stackrel{P_n}{\to} c^*_{\pi^*}.$ 
Replicate equation \eqref{eq:val_3_1} with $\hat{c}_{n,\rho,\tau_n}$ replacing $\hat{c}_{n,\rho}$, and mimic the argument following \eqref{eq:val_3_1} in the proof of Lemma \ref{lem:val}. Then, the conclusion of the lemma follows.

 \textbf{Part \ref{cor:eam_conv:ctau_smooth}.}
By the construction of the mollified version of the critical value, we have
$\hat{c}_{n,\tau_n}\in \mathcal C^\infty(\Theta)$ \citep[][Theorem 2.29]{Adams:2003aa}. Therefore it has derivatives of all order.
Using the multi-index notation, for any $s>0$ and $|\alpha|\le s$,
the partial derivative $\nabla^\alpha \hat{c}_{n,\tau_n}$ is bounded by some constant $M>0$ on the compact set $\Theta$, and hence 
\begin{align*}
\int_{\Theta} |\nabla^\alpha \hat{c}_{n,\tau_n}(\theta)|^2d\upsilon(\theta)\le M\upsilon(\Theta)<\infty,
\end{align*}
where $\upsilon$ denote the Lebesgue measure on $\mathbb R^d.$
This ensures $\nabla^\alpha \hat{c}_{n,\tau_n}\in L^2_\upsilon(\Theta)$ for all $|\alpha|\le s$. Hence, $\hat{c}_{n,\tau_n}$ is in the Sobolev-Hilbert space $H^{s}(\Theta^o)$ for any $s>0$. 
Note that when a Mat\'{e}rn kernel with $\nu<\infty$ is used and
 $\hat{c}_{n,\tau_n}$ is continuous,
 Lemma 3 in \cite{Bull_Convergence_2011} implies that the RKHS-norm $\| \cdot\|_{\mathcal H_{\bar\beta}}$ (in $\mathcal H_{\bar\beta}( \Theta)$)  and the Sobolev-Hilbert norm $\|\cdot\|_{H^{\nu+d/2}}$ are equivalent.
  Hence, there is  $R>0$ such that $\|\hat c_{n,\tau_n}\|_{\mathcal H_\beta}\le C\|\hat c_{n,\tau_n}\|_{H^{\nu+d/2}}\le R$.
  
\end{proof}

\subsection{The kernel of the Gaussian Process and its Associated Function Space}\label{sec:RKHS}
Following \cite{Bull_Convergence_2011}, we consider two commonly used classes of kernels. 
The first one is the Gaussian kernel, which is given by
\begin{align}
K_{\beta}(\theta-\theta')=\exp\big(-\sum_{k=1}^d|(\theta_k-\theta'_k)/\beta_k|^{2}\big), ~\beta_k\in[\underline{\beta}_k,\overline{\beta}_k], ~k=1,\cdots,d,	\label{eq:kernel}
\end{align}
where $0<\underline{\beta}_k<\overline{\beta}_k<\infty$ for all $k$.
The second one is the class of Mat\'{e}rn kernels \citep[see, e.g.,][Chapter 4]{ras:wil05} defined by
\begin{align*}
K_{\beta}(\theta-\theta')=\frac{2^{1-\nu}}{D(\nu)}\Big(\sqrt 2\nu\sum_{k=1}^d|(\theta_k-\theta'_k)/\beta_k|^{2}\Big)^\nu k_\nu\Big(\sqrt 2\nu\sum_{k=1}^d|(\theta_k-\theta'_k)/\beta_k|^{2}\Big),~\nu\in(0,\infty),~\nu\notin\mathbb N,
\end{align*}
where $D$ is the gamma function, and $k_\nu$ is the modified Bessel function of the second kind.\footnote{The requirement $\nu\notin\mathbb N$ is not essential for the convergence result. However, it simplifies some of the arguments as one can exploit the $2\nu$-H\"{o}lder continuity of $K_\beta$ at the origin without a log factor \citep[][Assumption 4]{Bull_Convergence_2011}.}
The index $\nu$ controls the smoothness of $K_\beta$. In particular, the Fourier transform $\hat K_\beta(\zeta)$ of the Mat\'{e}rn kernel is bounded from above and below by the order of $\|\zeta\|^{-2\nu-d}$ as $\|\zeta\|\to\infty$, i.e. $\hat K_\beta(\zeta)=\Theta(\|\zeta\|^{-2\nu-d})$. Similarly, the Fourier transform of the Gaussian kernel satisfies $\hat K_\beta(\zeta)=O(\|\zeta\|^{-2\nu-d})$ for any $\nu>0$. Below, we treat the Gaussian kernel as a kernel associated with $\nu=\infty.$

Each kernel is associated with a space of functions $\mathcal H_\beta(\mathbb R^d)$, called the reproducing kernel Hilbert space (RKHS). Below, we give some background on this space and refer to \cite{Steinwart:2008aa,Vaart:2008aa} for further details. 
For $D\subseteq \mathbb R^d$, let $K:D\times D\to\mathbb R$ be a symmetric and positive definite function. $K$ is said to be a reproducing kernel of a Hilbert space $\mathcal H(D)$ if $K(\cdot,\theta')\in \mathcal H(D)$ for all $\theta'\in D$, and
\begin{align*}
f(\theta)=\langle f,K(\cdot,\theta)\rangle_{\mathcal H(D)}
\end{align*}
holds for all $f\in\mathcal H(D)$ and $\theta\in D$. The space $\mathcal H(D)$ is called a reproducing kernel Hilbert space (RKHS) over $D$ if for all $\theta\in D$, the point evaluation functional $\delta_\theta:\mathcal H(D)\to\mathbb R$ defined by $\delta_\theta(f)=f(\theta)$ is continuous. When  $K(\theta,\theta')=K_\beta(\theta-\theta')$ is used as the correlation functional of the Gaussian process, we denote the associated RKHS by $\mathcal H_\beta(D)$.
Using Fourier transforms, the norm on $\mathcal H_\beta(D)$ can be written as
\begin{align}
\|f\|_{\mathcal H_\beta}\equiv \inf_{g|_D=f}\int \frac{\hat g(\zeta)}{\hat K_\beta(\zeta)}d\zeta,	
\end{align}
where  the infimum is taken over functions $g:\mathbb R^d\to\mathbb R$ whose restrictions to $D$ coincide with $f$, and we take $0/0=0$.

The RKHS has a connection to other well-known classes of functions. In particular, when $D$ is a Lipschitz domain, i.e. the boundary of $D$ is locally the graph of a Lipschitz function \citep[][]{Tartar:2007aa} and the kernel is associated with $\nu\in (0,\infty)$,
  $\mathcal H_\beta(D)$ is equivalent to the Sobolev-Hilbert space $H^{\nu+d/2}(D^o)$, which is the space of functions on $D^o$ such that
  \begin{align}
\|f\|^2_{H^{\nu+d/2}}\equiv\inf_{g|_{D^o}=f}\int \frac{\hat g(\zeta)}{(1+\|\zeta\|^2)^{\nu+d/2}}d\zeta	
  \end{align} 
  is finite, where the infimum is taken over functions $g:\mathbb R^d\to\mathbb R$ whose restrictions to $D^o$ coincide with $f$.
  Further, if $\nu=\infty$, $\mathcal H_\beta(D)$ is continuously embedded in $H^{s}(D^o)$ for all $s>0$ \citep[][Lemma 3]{Bull_Convergence_2011}.

Theorem \ref{thm:eiconv} requires that $c$ has a finite RKHS norm. 
This is to ensure that the approximation error made by the best linear predictor  $c_L$ of the Gaussian process regression is controlled uniformly \citep{Narcowich_Refined_2003}.
When a Mat\'{e}rn kernel is used, it suffices to bound the norm in the Sobolev-Hilbert space $H^{\nu+d/2}$ to bound $c$'s RKHS norm. We do so in Theorem \ref{cor:eam_conv} by introducing a mollified version of $\hat c_n$.

\subsection{A Reformulation of the M-step as a Nonlinear Program}\label{sec:Mstep_nlp}
In \eqref{eq:max_ei}, $\theta^{(L+1)}$ is defined as the maximizer of the following maximization problem
	\begin{align}
\max_{\theta\in\Theta} (p'\theta-p'\theta^*_L)_+\Big(1-\Phi\Big(\frac{\bar g(\theta)-c_L(\theta)}{\hat\varsigma s_L(\theta)}\Big)\Big),
	\end{align}
where $\bar g(\theta)=\text{max}_{j=1,\dots,J}g_j(\theta)$. Since $\Phi$ is strictly increasing, one may rewrite the objective function as
\begin{align*}
(p'\theta-p'\theta^*_L)_+\Big(1-\max_{j=1,\dots,J}\Phi\Big(\frac{ g_j(\theta)-c_L(\theta)}{\hat\varsigma s_L(\theta)}\Big)\Big)
=\min_{j=1,\dots,J}(p'\theta-p'\theta^*_L)_+\Big(1-\Phi\Big(\frac{ g_j(\theta)-c_L(\theta)}{\hat\varsigma s_L(\theta)}\Big)\Big).
\end{align*}	
Hence, $\theta^{(L+1)}$ is a solution to the maximin problem:
\begin{align*}
\max_{\theta\in\Theta}\min_{j=1,\dots,J}(p'\theta-p'\theta^*_L)_+\Big(1-\Phi\Big(\frac{ g_j(\theta)-c_L(\theta)}{\hat\varsigma s_L(\theta)}\Big)\Big),
\end{align*}
which can be solved, for example, by Matlab's \verb1fminimax1 function. It can also be rewritten as a nonlinear program:
\begin{align*}
\max_{(\theta,v)\in \Theta\times\mathbb R}&v ~~ \text{s.t. } (p'\theta-p'\theta^*_L)_+\Big(1-\Phi\Big(\frac{ g_j(\theta)-c_L(\theta)}{\hat\varsigma s_L(\theta)}\Big)\Big)\ge v, ~j=1,\dots,J,
\end{align*}
which can be solved by nonlinear optimization solvers,
	e.g. Matlab's \verb1fmincon1 or \verb1KNITRO1.
We note that the objective function and constraints together with their gradients are available in closed form.

\subsection{Root-Finding Algorithm Used to Compute $\hat{c}_n(\theta)$}\label{sec:bisection}
This section explains in detail how $\hat{c}_n(\theta)$ in equation \eqref{eq:def:c_hat} is computed.
For a given $\theta \in \Theta$, $P^*(\Lambda_n^b (\theta ,\rho ,c)\cap \{p^{\prime }\lambda =0\}\neq \emptyset)$ increases in $c$ (with $\Lambda_n^b (\theta ,\rho ,c)$ defined in \eqref{eq:Lambda_n}), and so $\hat c_n(\theta)$ can be quickly computed via a root-finding algorithm, such as the Brent-Dekker Method (BDM), see \cite{brent1971algorithm} and \cite{dekker1969finding}.  To do so, define $h_{\alpha}(c) = \frac{1}{B} \sum_{b=1}^B \psi_b(c) - (1-\alpha)$ where
\[
\psi_b(c(\theta)) = \mathbf{1}(\Lambda_n^b(\theta,\rho,c) \cap \{p'\lambda = 0\}\neq \emptyset).
\]
Let $\bar c(\theta)$ be an upper bound on $\hat c_n(\theta)$ (for example, the asymptotic Bonferroni bound $\bar c(\theta) \equiv \Phi^{-1}(1-\alpha/J))$.  It remains to find $\hat c_n(\theta)$ so that $h_{\alpha}(\hat c_n(\theta)) =  0$ if $h_{\alpha}(0)\leq 0$.  It is possible that $h_{\alpha}(0) > 0$ in which case we output $\hat c_n(\theta) = 0$.  Otherwise, we use BDM to find the unique root to $h_{\alpha}(c)$ on $[0,\bar c(\theta)]$ where, by construction, $h_{\alpha}(\bar c_n(\theta)) \geq 0$.  We propose the following algorithm:

\textbf{Step 0} (Initialize)
\begin{enumerate}[label=(\roman*)]
\item Set $\mathit{Tol}$ equal to a chosen tolerance value;
\item Set $c_L = 0$ and $c_U =\bar c(\theta)$ (values of $c$ that bracket the root $\hat c_n(\theta)$);
\item Set $c_{-1} = c_L$ and $c_{-2}=[]$ to be undefined for now (proposed values of $c$ from $1$ and $2$ iterations prior).  Also set $c_0 = c_L$ and $c_1 = c_U$.
\item Compute $\varphi_j( \hat \xi_{n,j}(\theta))$ $j = 1,\dots,J$;
\item Compute $\hat D_{P,n}(\theta)$;
\item Compute $\G^b_{n,j}$ for $b = 1,\dots,B$, $j = 1,\dots,J$;
\item Compute $\psi_b(c_L)$ and $\psi_b(c_U)$ for $b = 1,\dots,B$;
\item Compute $h_{\alpha}(c_L)$ and $h_{\alpha}(c_U)$.
\end{enumerate}

\textbf{Step 1} (Method Selection)

\begin{parindent1}
Use the BDM rule to select the updated value of $c$, say $c_2$.  The value is updated using one of three methods: Inverse Quadratic Interpolation, Secant, or Bisection.  The selection rule is based on the values of $c_{i}$, $i=-2,-1,0,1$ and the corresponding function values.
\end{parindent1}

\textbf{Step 2} (Update Value Function)

\begin{parindent1}
Update the value of $h_{\alpha}(c_2)$.  We can exploit previous computation and monotonicity function $\psi_b(c_2)$ to reduce computational time:
\begin{enumerate}
\item If $\psi_b(c_L) = \psi_b(c_U) = 0$, then $\psi_b(c_2) = 0$;
\item If $\psi_b(c_L) = \psi_b(c_U) = 1$, then $\psi_b(c_2) = 1$.
\end{enumerate}
\end{parindent1}

\textbf{Step 3} (Update)
\begin{enumerate}[label=(\roman*)]
\item If $h_{\alpha}(c_2) \geq 0$, then set $c_U = c_2$.  Otherwise set $c_L = c_2$.
\item Set $c_{-2} = c_{-1}$, $c_{-1} = c_0$, $c_0 = c_L$, and $c_1 = c_U$.
\item Update corresponding function values $h_{\alpha}(\cdot)$.
\end{enumerate}

\textbf{Step 4} (Convergence)
\begin{enumerate}[label=(\roman*)]
\item If $h_{\alpha}(c_U) \leq \mathit{Tol}$ or if $|c_U - c_L| \leq \mathit{Tol}$, then output $\hat c_n(\theta) = c_U$ and exit. Note: $h_{\alpha}(c_U) \geq 0$, so this criterion ensures that we have \emph{at least} $1-\alpha$ coverage.
\item Otherwise, return to \textbf{Step 1}.
\end{enumerate}
The computationally difficult part of the algorithm is computing $\psi_b(\cdot)$ in \textbf{Step 2}.  This is simplified for two reasons.  First, evaluation of $\psi_b(c)$ entails determining whether a constraint set comprised of $J+2d-2$ linear inequalities in $d-1$ variables is feasible.  This can be accomplished efficiently employing commonly used software.\footnote{Examples of high-speed solves for linear programs include CVXGEN, availiable from \url{http://www.cvxgen.com} and Gurobi, available from \url{http://www.gurobi.com}.}  Second, we exploit monotonicity in $\psi_b(\cdot)$, reducing the number of linear programs needed to be solved.

\section{Assumptions for Asymptotic Coverage Validity}
\label{app:all_ass}
\subsection{Main Assumptions}
\label{sec:AssRes}
We posit that $P$, the distribution of the observed data, belongs to a class of distributions denoted by $\cP$. We write stochastic order relations that hold uniformly over $P \in \cP$ using the notations $o_{\mathcal P}$ and $O_{\mathcal P}$; see Appendix \ref{subsec:notation} for the formal definitions.
Below, $\epsilon$, $\varepsilon$, $\delta$, $\omega$, $\underline{\sigma}$, $M$, $\bar{M}$ denote generic constants which may be different in different appearances but cannot depend on $P$. Given a square matrix $A$, we write $\eig(A)$ for its smallest eigenvalue.
\begin{assumption}
\label{as:momP_AS} (a) $\Theta\subset\mathbb{\ R}^{d}$ is a compact hyperrectangle with nonempty interior.

\noindent (b) All distributions $P \in \cP$ satisfy the following:
\begin{itemize}
\item[(i)] $E_P[m_j(X_i,\theta)]\le 0,~j=1,\dots,J_1$ and $E_{P}[m_j(X_i,\theta)]= 0,~j=J_1+1,\dots, J_1+J_2$ for some $\theta\in\Theta$;

\item[(ii)] $\{X_i, i\ge 1\}$ are i.i.d.;

\item[(iii)] $\sigma^2_{P,j}(\theta)\in(0,\infty)$ for $j=1,\dots, J$ for all $\theta\in\Theta$;

\item[(iv)] For some $\delta>0$ and $M\in (0,\infty)$ and for all $j$, $E_P[\sup_{\theta\in\Theta}|{m_{j}(X_i,\theta)}/{\sigma_{P,j}(\theta)}|^{2+\delta}]\le M$.
\end{itemize}
\end{assumption}
\begin{assumption}
\label{as:GMS}
The function $\varphi_j$ is continuous at all $x\ge 0$ and $\varphi_j(0)=0$; $\kappa_n \to \infty$ and $\kappa_n=o(n^{1/2})$. If Assumption \ref{as:correlation}-\ref{as:correlation_pair} is imposed, $\kappa_n=o(n^{1/4})$.
\end{assumption}

Assumption \ref{as:momP_AS}-(a) requires that $\Theta$ is a hyperrectangle, but can be replaced with the assumption that $\theta$ is defined through a finite number of nonstochastic inequality constraints smooth in $\theta$ and such that $\Theta$ is convex. Compactness is a standard assumption on $\Theta$ for extremum estimation. We additionally require convexity as we use mean value expansions of $E_{P}[m_j(X_i,\theta)]/\sigma_{P,j}(\theta)$ in $\theta$; see \eqref{eq:mean_val_exp}. 
Assumption \ref{as:momP_AS}-(b) defines our moment (in)equalities model. Assumption \ref{as:GMS} constrains the GMS function and the rate at which its tuning parameter diverges. Both \ref{as:momP_AS}-(b) and \ref{as:GMS} are based on \cite{AS} and are standard in the literature,\footnote{Continuity of $\varphi_j$ for $x\ge0$ is restrictive only for GMS function $\varphi^{(2)}$ in \cite{AS}.} although typically with $\kappa_n=o(n^{1/2})$. The slower rate $\kappa_n=o(n^{1/4})$ is satisfied for the popular choice, recommended by \cite{AS}, of $\kappa_n=\sqrt{\ln n}$. 

Next, and unlike some other papers in the literature, we impose restrictions on the correlation matrix of the moment functions. These conditions can be easily verified in practice because they are implied when the correlation matrix of the moment equality functions and the moment inequality functions  specified below have a determinant larger than a predefined constant for any $\theta \in \Theta$. 
\begin{assumption}
\label{as:correlation} 
All distributions $P \in \cP$  satisfy \textbf{one} of the following two conditions for some constants $\omega>0,\underline{\sigma}>0,\epsilon>0,\varepsilon>0,M<\infty$:
\begin{enumerate}
\item \label{as:correlation_base} 
Let $\cJ(P,\theta;\varepsilon)\equiv \left\{j\in \{1,\cdots,J_1\}:E_P[m_j(X_i,\theta)]/\sigma_{P,j}(\theta)\ge -\varepsilon
\right\}$.
Denote 
\begin{align*}
\tilde m(X_i,\theta)&\equiv \left(\{m_j(X_i,\theta)\}_{j \in \cJ(P,\theta;\varepsilon)},m_{J_1+1}(X_i,\theta),\dots,m_{J_1+J_2}(X_i,\theta)\right)',\\
\tilde\Omega_P(\theta)&\equiv Corr_P(\tilde m(X_i,\theta)).
\end{align*}
Then $\inf_{\theta \in \Theta_I(P)}\eig(\tilde\Omega_P(\theta))\ge \omega$.  

\item \label{as:correlation_pair} 
The functions $m_j(X_i,\theta)$ are defined on $\Theta^\epsilon=\{\theta \in \R^d:d(\theta,\Theta)\le \epsilon\}$. 
There exists $R_1\in\N$, $1\le R_1 \le J_1/2$, and measurable functions $t_j:\cX \times \Theta^\epsilon \rightarrow [0,M],~j\in \cR_1\equiv\{1,\dots,R_1\}$, such that for each $j\in \cR_1$,
\begin{align}
m_{j+R_1}(X_i,\theta)&=-m_j(X_i,\theta)-t_j(X_i,\theta).\label{eq:def_paired_ineq}
\end{align} 

For each $j \in \cR_1\cap\cJ(P,\theta;\varepsilon)$ and any choice $\ddot{m}_j(X_i,\theta) \in \{m_{j}(X_i,\theta),m_{j+R_1}(X_i,\theta)\}$, denoting $\tilde\Omega_P(\theta)\equiv Corr_P(\tilde m(X_i,\theta))$, where
\begin{align*}
\tilde m(X_i,\theta)&\equiv \Big(\{\ddot{m}_j(X_i,\theta)\}_{j \in \cR_1 \cap \cJ(P,\theta;\varepsilon)},\\
& \hspace{0.75cm} \{m_j(X_i,\theta)\}_{j \in \cJ(P,\theta;\varepsilon) \setminus \{1,\dots,2R_1\}},m_{J_1+1}(X_i,\theta),\dots,m_{J_1+J_2}(X_i,\theta)\Big)',
\end{align*}
one has
\begin{align}
\inf_{\theta \in \Theta_I(P)}\eig(\tilde\Omega_P(\theta))\ge \omega.\label{eq:eig_cond_paired_ineq}
\end{align}
Finally,
\begin{align}
&\inf_{\theta \in \Theta_I(P)}\sigma_{P,j}(\theta)>\underline{\sigma}~ \text{for}~j=1,\dots,R_1.\label{eq:33primeiii}
\end{align}
\end{enumerate}
\end{assumption}
Assumption \ref{as:correlation}-\ref{as:correlation_base} requires that the correlation matrix of the moment functions corresponding to close-to-binding moment conditions has
eigenvalues uniformly bounded from below. 
This assumption holds in many applications of interest, including: (i) instances when the data is collected by intervals with minimum width;\footnote{\label{fn:interval_pos_width}
Empirically relevant examples are that of: (a) the Occupational Employment Statistics (OES) program at the Bureau of Labor Statistics, which collects wage data from employers as intervals of positive width, and uses these data to construct estimates for wage and salary workers in 22 major occupational groups and 801 detailed occupations; and (b) when, due to concerns for privacy, data is reported as the number of individuals who belong to each of a finite number of cells (for example, in public use tax data).}
(ii) in treatment effect models with (uniform) overlap;
(iii) in static complete information entry games under weak solution concepts, e.g. rationality of level 1, see \cite{ALT}. 

We are aware of two examples in which Assumption \ref{as:correlation}-\ref{as:correlation_base} may fail. 
One are missing data scenarios,
e.g. scalar mean, linear regression, and best linear prediction, with a vanishing probability of missing data. 
The other example, which is extensively simulated in Section \ref{sec:MC}, is the \cite{CilibertoTamer09} entry game model when the solution concept is pure strategy Nash equilibrium. We show in Appendix \ref{sec:verify_3.3} that these examples satisfy Assumption \ref{as:correlation}-\ref{as:correlation_pair}.

\begin{remark}\label{rem:verify_4_3:separable}
Assumption \ref{as:correlation}-\ref{as:correlation_pair} weakens \ref{as:correlation}-\ref{as:correlation_base} by allowing for (drifting to) perfect correlation among moment inequalities that cannot cross. This assumption is often satisfied in moment conditions that are separable in data and parameters, i.e. for each $j=1,\dots,J$,
\begin{align}
E_P[m_j(X_i,\theta)] = E_P[h_j(X_i)]-v_j(\theta),\label{eq:def_separable_model}
\end{align}
for some measurable functions $h_j:\cX \rightarrow \R$ and $v_j:\Theta \rightarrow \R$. Models like the one in \cite{CilibertoTamer09} fall in this category, and we verify Assumption \ref{as:correlation}-\ref{as:correlation_pair} for them in Appendix \ref{sec:verify_3.3}. The argument can be generalized to other separable models. 

In Appendix \ref{sec:verify_3.3}, we also verify Assumption \ref{as:correlation}-\ref{as:correlation_pair} for some models that are not separable in the sense of equation \eqref{eq:def_separable_model}, for example best linear prediction with interval outcome data. The proof can be extended to cover (again non-separable) binary models with discrete or interval valued covariates under the assumptions of \cite{mag:mau08}.
\end{remark}  

In what follows, we refer to pairs of inequality constraints indexed by $\{j,j+R_1\}$ and satisfying \eqref{eq:def_paired_ineq} as ``paired inequalities."\label{verbal_def:paired_ineq} Their presence requires a modification of the bootstrap procedure. This modification exclusively concerns the definition of $\Lambda_n^b(\theta,\rho,c)$ in equation \eqref{eq:Lambda_n}. We explain it here for the case that the GMS function $\varphi_j$ is the hard-thresholding one in footnote \ref{fn:eq:zeta} of the main paper, and refer to Appendix \ref{app:Lemma} equations \eqref{eq:paired_bind_Lboot_smooth}-\eqref{eq:paired_bind_Uboot_smooth} for the general case. 
If
\begin{align*}
\varphi_j(\hat{\xi}_{n,j}(\theta))=0 = \varphi_j(\hat{\xi}_{n,j+R_1}(\theta)),
\end{align*}
we replace $\mathbb{G}_{n,j+R_1}^{b}(\theta )$ with $-\mathbb{G}_{n,j}^{b}(\theta )$ and $\hat{D}_{n,j+R_1}(\theta)$ with $-\hat{D}_{n,j}(\theta)$, so that inequality $\mathbb{G}_{n,j+R_1}^{b}(\theta )+\hat{D}_{n,j+R_1}(\theta)\lambda  \le c$ is replaced with $-\mathbb{G}_{n,j}^{b}(\theta )-\hat{D}_{n,j}(\theta)\lambda \le c$ in equation \eqref{eq:Lambda_n}. In words, when hard threshold GMS indicates that both paired inequalities bind, we pick one of them, treat it as an equality, and drop the other one. In the proof of Theorem \ref{thm:validity}, we show that this tightens the stochastic program.\footnote{When paired inequalities are present, in equation \eqref{eq:CI} instead of $\hat{\sigma}_{n,j}$ we use the estimator $\hat{\sigma}^M_{n,j}$ specified in \eqref{eq:def_hat_sigma_M} in Lemma \ref{lem:eta_conv} p.\pageref{eq:def_hat_sigma_M} of the Appendix for $\sigma_{P,j},j=1,\dots,2R_1$ (with $R_1 \le J_1/2$ defined in the assumption). In equation \eqref{eq:hat_xi} we use $\hat{\sigma}_{n,j}$ for all $j=1,\dots,J$. To ease notation, we do not distinguish the two unless it is needed.} 
The rest of the procedure is unchanged.

Instead of Assumption \ref{as:correlation}, BCS (Assumption 2) impose the following high-level condition: (a) The limit distribution of their profiled test statistic is continuous at its $1-\alpha$ quantile if this quantile is positive; (b) else, their test is asymptotically valid with a critical value of zero. In Appendix \ref{app:proofs_alt_cont}, we show that we can replace Assumption \ref{as:correlation} with a weaker high level condition (Assumption \ref{ass:continuity_limit_cov}) that resembles the BCS assumption but constrains the limiting coverage probability. (We do not claim that the conditions are equivalent.) The substantial amount of work required for us to show that Assumption \ref{as:correlation} implies Assumption \ref{ass:continuity_limit_cov} is suggestive of how difficult these high-level conditions can be to verify.\footnote{Assumption \ref{as:correlation} is used exclusively to obtain the conclusions of Lemma \ref{lem:empt}, \ref{lem:exist_sol} and \ref{lem:cbd}, hence any alternative assumption that delivers such results can be used.}
Moreover, in Appendix \ref{sec:counterexample} we provide a simple example that violates Assumption \ref{as:correlation} and in which all of calibrated projection, BCS-profiling, and the bootstrap procedure in \cite{PakesPorterHo2011} fail. The example leverages the fact that when binding constraints are near-perfectly correlated, the projection may be estimated superconsistently, invalidating the simple nonparametric bootstrap.\footnote{The example we provide satisfies all assumptions explicitly stated in \cite{PakesPorterHo2011}, illustrating an oversight in their Theorem 2.}

Together with imposition of the $\rho$-box constraints, Assumption \ref{as:correlation} allows us to dispense with restrictions on the local geometry of the set $\Theta_I(P)$. Restrictions of this type, which are akin to constraint qualification conditions, are imposed by BCS (Assumption A.3-(a)), \cite[Assumptions A.3-A.4]{PakesPorterHo2011}, \cite[Condition C.2]{CHT}, and elsewhere. In practice, they can be hard to verify or pre-test for. We study this matter in detail in \cite{KMS2}.

We next lay out regularity conditions on the gradients of the moments. 
\begin{assumption}
\label{as:momP_KMS} All distributions $P\in \cP$ satisfy  the following conditions:
\begin{itemize}

\item[(i)] For each $j,$ there exist $D_{P,j}(\cdot)\equiv\nabla_\theta
\{E_P[m_j(X,\cdot)]/\sigma_{P,j}(\cdot)\}$ and its estimator $\hat D_{n,j}(\cdot)$ such that $\sup_{\theta\in\Theta^\epsilon}\|\hat D_{n,j}(\theta)-D_{P,j}(\theta)\|=o_{\cP}(1)$.

\item[(ii)] There exist $M, \bar{M}<\infty$ such that for all $\theta,\tilde{\theta} \in \Theta^\epsilon$ $\max_{j=1,\dots,J} \|D_{P,j}(\theta)-D_{P,j}(\tilde{\theta})\|\le M \|\theta-\tilde{\theta}\|$ and $\max_{j=1,\dots,J} \sup_{\theta \in \Theta_I(P)}\|D_{P,j}(\theta)\|\le  \bar{M}$. 
\end{itemize}
\end{assumption}
Assumption \ref{as:momP_KMS} requires that each of the $J$ normalized population moments is differentiable, that its derivative is Lipschitz continuous, and that this derivative can be consistently estimated uniformly in $\theta$ and $P$.\footnote{The requirements are imposed on $\Theta^\epsilon$. Under Assumption \ref{as:correlation}-\ref{as:correlation_base} it suffices they hold on $\Theta$.} We require these conditions because we use a linear expansion of the population moments to obtain a first-order approximation to the nonlinear programs defining $CI_n$, and because our bootstrap procedure requires an estimator of $D_{P}$. 

A final set of assumptions is on the normalized empirical process. For this,
define the variance semimetric $\varrho_P$ by
\begin{align}
	\varrho_{P}(\theta,\tilde{\theta})\equiv \Big\|\big\{\big[Var_P\big(\sigma_{P,j}^{-1}(\theta)m_j(X,\theta)-\sigma_{P,j}^{-1}(\tilde{\theta})m_j(X,\tilde{\theta})\big)\big]^{1/2}\big\}_{j=1}^J\Big\|.
\end{align}
For each $\theta,\tilde{\theta}\in\Theta$ and $P$, let $Q_P(\theta,\tilde{\theta})$ denote a $J$-by-$J$ matrix whose $(j,k)$-th element is the covariance between 
$m_j(X_i,\theta)/\sigma_{P,j}(\theta)$ and $m_k(X_i,\tilde{\theta}))/\sigma_{P,k}(\tilde{\theta})$.
\begin{assumption}\label{as:bcs1}
All distributions $P \in \cP$ satisfy the following conditions:
\begin{itemize}
\item [(i)]	The class of functions $\{\sigma_{P,j}^{-1}(\theta)m_j(\cdot,\theta):\mathcal X\to\mathbb R,\theta\in\Theta\}$ is measurable for each $j=1,\cdots, J$.
\item [(ii)]  The empirical process $\mathbb G_{n}$ with $j$-th component $\mathbb G_{n,j}$ is uniformly asymptotically $\varrho_P$-equicontinuous. That is, for any $\epsilon>0$,
	\begin{align}
		\lim_{\delta\downarrow 0}\limsup_{n\to\infty}\sup_{P\in\mathcal P}P\left(\sup_{\varrho_P(\theta,\tilde{\theta})<\delta}\|\mathbb G_{n}(\theta)-\mathbb G_{n}(\tilde{\theta})\|>\epsilon\right)=0.
	\end{align}
\item [(iii)] $Q_P$ satisfies
\begin{align}
	\lim_{\delta\downarrow 0}\sup_{\|(\theta_1,\tilde{\theta}_1)-(\theta_2,\tilde{\theta}_2)\|<\delta}\sup_{P\in\mathcal P}\|Q_P(\theta_1,\tilde{\theta}_1)-Q_P(\theta_2,\tilde{\theta}_2)\|=0.	
\end{align}	
\end{itemize}	
\end{assumption}
Under this assumption, the class of  normalized moment functions is uniformly Donsker \citep{Bugni:2015jk}. We use this fact to show validity of our method.

\subsection{High Level Conditions Replacing Assumption \ref{as:correlation} and the $\rho$-Box Constraints}
\label{sec:alt:ass:eig}
Next, we consider two high level assumptions. The first one aims at informally mimicking Assumption A.2 in \cite{BCS14_subv} and replaces Assumption \ref{as:correlation}. The second one replaces the use of the $\rho$-box constraints. Below, for a given set $A \subset \R^d$, let $\Vert A \Vert_H=\sup_{a \in A}\Vert a \Vert$ denote its Hausdorff norm.
\begin{assumption}
\label{ass:continuity_limit_cov}
Consider any sequence $\{P_n,\theta_n\} \in \{(P,\theta):P \in \cP,\theta \in \Theta_I(P)\}$ such that 
\begin{align*}
\kappa_n^{-1}\sqrt{n}\gamma_{1,P_n,j}(\theta_n) &\to \pi_{1j} \in \R_{[-\infty]},~j=1,\dots,J,\\
\Omega_{P_n} &\uni \Omega,\\
D_{P_n}(\theta_n) &\to D.
\end{align*}
Let $\pi^*_{1j}=0$ if $\pi_{1j} = 0$ and $\pi^*_{1j}=-\infty$ if $\pi_{1j} < 0$. Let $\HH$ be a Gaussian process with covariance kernel $\Omega$. Let
\begin{align}
\mathfrak{w}_{j}(\lambda)&\equiv \HH_{j}+\rho D_{j}\lambda+\pi^*_{1,j}.\label{eq:wj_var}
\end{align}
Let
\begin{align}
\mathfrak{W}(c)&\equiv \big\{ \lambda\in \mathfrak B^d_\rho: p^\prime \lambda = 0 \cap \mathfrak w_{j}(\lambda)\le c, \: \forall j=1,\dots,J\big\} \label{eq:set_W_Lpop_var},\\
c_{\pi^*}&\equiv \inf\{c\in\mathbb R_+:\mathrm{Pr}(\mathfrak{W}(c)\ne\emptyset)\ge 1-\alpha\}.\label{eq:def:c_pi_star}
\end{align}
Then:
\begin{enumerate}
\item If $c_{\pi^*}>0$, $\mathrm{Pr}\left(\mathfrak{W}(c)\neq \emptyset \right)$ is continuous and strictly increasing at $c=c_{\pi^*}$.
\item If $c_{\pi^*}=0$, $\lim \inf_{n \to \infty} P_n(U_n(\theta_n,0)\neq \emptyset)\ge 1-\alpha$, where $U_n(\theta_n,c),~c\ge 0$ is as in \eqref{eq:set_U_NL}.
\end{enumerate}
\end{assumption}
\begin{assumption}
\label{ass:continuity_limit_cov_II} 
Consider any sequence $\{P_n,\theta_n\} \in \{(P,\theta):P \in \cP,\theta \in \Theta_I(P)\}$ as in Assumption \ref{ass:continuity_limit_cov}.
Let
$$ \bar{\mathfrak{W}}(c)\equiv \big\{ \lambda\in \mathbb R^d: p^\prime \lambda = 0 \cap \mathfrak w_{j}(\lambda)\le c, \: \forall j=1,\dots,J\big\}, $$
which differs from \eqref{eq:set_W_Lpop_var} by not constraining $\lambda$ to $\mathfrak B^d_\rho$, and let $\bar c \equiv \Phi^{-1}(1-\alpha/J)$ denote the asymptotic Bonferroni critical value. Then for every $\eta>0$ there exists $M_\eta<\infty$ s.t. $\mathrm{Pr}(\Vert \bar{\mathfrak{W}}(\bar{c}) \Vert_H > M_\eta) \leq \eta$. 
\end{assumption}

\subsection{Example of Methods Failure When Assumption \ref{as:correlation} Fails}
\label{sec:counterexample}
Consider one-sided testing with two inequality constraints in $\R^2$. The
constraints are%
\begin{eqnarray*}
\theta _{1}+\theta _{2} &\leq &E_P(X_1) \\
\theta _{1}-\theta _{2} &\leq &E_P(X_2).
\end{eqnarray*}%
The projection of $\Theta_I(P)$ in direction $p=(1,0)$ is $(-\infty,(E_P(X_1)+E_P(X_2))/2]$, the support set is $H(p,\Theta_I)=\{((E_P(X_1)+E_P(X_2))/2,(E_P(X_1)-E_P(X_2))/2)\}$, and the support function takes value $\theta_1^*=(E_P(X_1)+E_P(X_2))/2$.

The random variables $(X_1,X_2)^{\prime }$ have a mixture distribution
as follows:%
\begin{equation*}
\left[\begin{array}{c} X_1 \\ X_2 \end{array}\right] \sim \left\{\begin{array}{ll}N\left( 0,\left[\begin{array}{cc}1 & -1 \\ -1 & 1 \end{array} \right] \right) & \text{with probability }1-1/n, \\ \delta _{(1,1)}\text{ (degenerate)} & \text{otherwise},\end{array}\right. 
\end{equation*}%
hence $E_P(X_1)=E_P(X_2)=\theta_1^*=1/n$. Note in particular the implication that%
\begin{equation*}
\frac{X_1+X_2}{2}=\left\{\begin{array}{ll}0 & \text{with probability }1-1/n, \\1 & \text{otherwise}.\end{array}\right. 
\end{equation*}

The natural estimator of $\theta _{1}^{\ast }$ is $\hat{\theta }_{1}^{\ast }=(\bar{X}_{1}+\bar{X}_{2})/2$. It is distributed as $Z/n$, where $Z$ is Binomial with parameters $(1/n,n)$. For large $n$, the distribution of $Z$ is well approximated as Poisson with parameter $1$. In particular, with probability approximately $e^{-1}\approx 37\%$, every sample realization of $(X_1+X_2)/2$ equals zero. In this case, the following happens: (i) The projection of the sample analog of the identified set is $(-\infty ,0]$, so that a strictly positive critical value or level would be needed to cover the true projection. (ii) Because the empirical distribution of $(X_1+X_2)/2$ is degenerate at zero, the distribution of $(\bar{X}_{1}^{b}+\bar{X}_2^b)/2$ is as well. Hence, all of \cite{PakesPorterHo2011}, \cite{BCS14_subv}, and calibrated projection (each with either parametric or nonparametric bootstrap) compute critical values or relaxation levels of $0$.

This bounds from above the true coverage of all of these methods at $e^{-1}\approx 63\%$. Note that $(m<n)$-subsampling will encounter the same problem. Next we provide some discussion of the example.

\noindent\textbf{Violation of Assumptions.}
The example violates our Assumption \ref{as:correlation} because $Cov(X_1,X_2)\rightarrow 1$. It also violates Assumption 2 in \cite{BCS14_subv}: Their Assumption A2-(b) should apply, but the profiled test statistic on the true null concentrates at $1/n$. The example satisfies the assumptions explicitly stated in \cite{PakesPorterHo2011}, illustrating an oversight in their Theorem 2. (We here refer to the inference part of their 2011 working paper. We identified corresponding oversights in the proof of their Proposition 6.)

The example satisfies the assumptions of \cite{AS} and \cite{Andrews_Guggenberger2009bET}, and both methods work here. The reason is that both focus on the distribution of the criterion function at a fixed $\theta $ and are not affected by the irregularity of $\hat{\theta }_{1}^{\ast }$.

\noindent\textbf{Relation to \cite{Mammen92}.}
In this example, all of \cite{BCS14_subv}, \cite{PakesPorterHo2011}, and our calibrated projection method reduce to one-sided nonparametric percentile bootstrap confidence intervals for $(E_P(X_1)+E_P(X_2))/2$ estimated by $(\bar{X}_{1}+\bar{X}_{2})/2$. By \cite[Theorem 1]{Mammen92}, asymptotic normality of an appropriately standardized estimator, i.e.
\begin{equation*}
\exists \{a_{n}\}:a_{n}\left( (\bar{X}_{1}+\bar{X}_{2})-(E_P(X_1)+E_P(X_2))\right) \overset{d}{\rightarrow }N(0,1),
\end{equation*}
is \textit{necessary and} sufficient for this interval to be valid. This fails (the true limit is recentered Poisson at rate $a_n=n$), so that validity of any of the aforementioned methods would contradict the Theorem. 

\section{Verification of Assumptions for the Canonical Partial Identification Examples}\label{sec:verify_examples}

In this section we verify: (i) Assumption \ref{as:Lipschitz_m_over_sigma} which is the crucial condition in Theorem \ref{cor:eam_conv}, and (ii) Assumption \ref{as:correlation}-\ref{as:correlation_pair}, for the canonical examples in the partial identification literature:
\begin{enumerate}
\item \textbf{Mean with interval data (of which missing data is a special case)}. Here we assume that $W_0,W_1$ are two observable random variables such that $P(W_0\le W_1)=1$. The identified set is defined as
\begin{align}
\Theta_I(P)=\{\theta \in \Theta \subset \R: E_P(W_0)-\theta \le 0, \theta -E_P(W_1)\le 0\}.\label{eq:Theta_I:mean}
\end{align}
\item \textbf{Linear regression with interval outcome data and discrete regressors.} Here the modeling assumption is that $W=Z^\prime \theta +u$, where $Z=[Z_1;\dots;Z_d]$ is a $d\times 1$ random vector with $Z_1=1$. We assume that $Z$ has $k$ points of support denoted $z^1,\dots,z^k \in \R^d$ with $\max_{r=1,\dots,k}\Vert z^r \Vert<M<\infty$. The researcher observes $\{W_0,W_1,Z\}$ with $P(W_0\le W \le W_1|Z=z^r)=1,r=1,\dots,k$.
The identified set is
\begin{align}
\Theta_I(P)=\{\theta \in \Theta \subset \R^d: E_P(W_0|Z=z^r)-z^{r\prime} \theta \le 0, z^{r\prime} \theta-E_P(W_1|Z=z^r) \le 0, r=1,\dots,k\}.\label{eq:Theta_I:OLS}
\end{align}
\item \textbf{Best linear prediction with interval outcome data and discrete regressors.} Here the variables are defined as for the linear regression case.
\cite{ber:mol08} show that the identified set for the parameters of a best linear predictor of $W$ conditional on $Z$ is given by the set $\Theta_I(P)=E_P(ZZ^\prime)^{-1}E_P(Z\mathbf{W})$, where $\mathbf{W}=[W_0,W_1]$ is a random closed set and, with some abuse of notation, $E_P(Z\mathbf{W})$ denotes the Aumann expectation of $Z\mathbf{W}$. \newline
Here we go beyond the results in \cite{ber:mol08} and derive a moment inequality representation for $\Theta_I(P)$ when $Z$ has a discrete distribution. We denote by $u^r$ the vector $u^r=e^{r\prime}(M_{P}^\prime M_{P})^{-1}M_{P}^\prime E_P(ZZ^\prime)$, $r=1,\dots,k$, where $e^r$ is the $r$-th basis vector in $\R^k$ and $M_{P}$ is a $d\times K$ matrix with $r$-th column equal to $P(Z=z^r)z^r$; we let $q^r=u^r E_P(ZZ^\prime)^{-1}$. Observe that for any selection $\tilde{W}\in \mathbf{W}~a.s.$ one has $u^rE_P(ZZ^\prime)^{-1}E_P(Z\tilde{W})=e^{r\prime}[E_P(\tilde{W}|Z=z^1);\dots;E_P(\tilde{W}|Z=z^k)]$, so that the support function in direction $u^r$ is maximized/minimized by setting
$E_P(\tilde{W}|Z=z^r)$ equal to $E_P(W_1|Z=z^r)$ and $E_P(W_0|Z=z^r)$, respectively. Hence, the identified set can be written in terms of moment inequalities as
\begin{align}
\Theta_I(P)=\{\theta \in \Theta \subset \R^d: &~q^r [E_P(Z(Z^\prime\theta -W_0-\mathbf{1}(q^rZ>0)(W_1-W_0)))]\le 0 \notag\\
-&~q^r [E_P(Z(Z^\prime\theta -W_0-\mathbf{1}(q^rZ<0)(W_1-W_0)))]\le0,
 r=1,\dots,k\}.\label{eq:Theta_I:BLP}
\end{align}
The set is expressed through evaluation of its support function, given in \cite[Proposition 2]{BontempsMagnacMaurin2012E}, at directions $\pm u^r$; these are the directions orthogonal to the flat faces of $\Theta_I(P)$.
\item \textbf{Complete information entry games with pure strategy Nash equilibrium as solution concept.} Here again we assume that the vector $Z$ has $k$ points of support with bounded norm, and the identified set is
\begin{align}
\Theta_I(P)=\{\theta \in \Theta \subset \R^d:\text{ equations } \eqref{eq:entry5}, \eqref{eq:entry6}, \eqref{eq:entry7}, \eqref{eq:entry8} \text{ hold for all } Z=z^r,r=1,\dots,k\}.
\end{align}
\end{enumerate}
In the first three examples we let $X\equiv(W_0,W_1,Z)'$. In the last example we let $X\equiv(Y_1,Y_2,Z)'$. Throughout, we propose to estimate $E_P(W_\ell|Z=z^r)$ and $E_P(Y_1=s,Y_2=t|Z=z^r)$, $\ell=0,1$, $(s,t)\in \{0,1\}\times\{0,1\}$ and $r=1,\dots,k$, using
\begin{align}
\hat{E}_n(W_\ell|Z=z^r)&=\frac{\sum_{i=1}^n W_{\ell,i}\mathbf{1}(Z_i=z^r)}{\sum_{i=1}^n \mathbf{1}(Z_i=z^r)},\label{eq:kernel_reg1}\\
\hat{E}_n(Y_1=s,Y_2=t|Z=z^r)&=\frac{\sum_{i=1}^n \mathbf{1}(Y_{1,i}=s,Y_{2,i}=t,Z_i=z^r)}{\sum_{i=1}^n \mathbf{1}(Z_i=z^r)},\label{eq:kernel_reg}
\end{align}
as it is done in, e.g., \cite{CilibertoTamer09}.
We assume that for each of the four canonical examples under consideration, Assumption \ref{as:momP_AS} as well as one of the assumptions below hold. 
\begin{assumption}
\label{as:var:BM} 
The model $\cP$ for $P$ satisfies $\min_{\ell=0,1}\min_{r=1,\dots,k}Var_P(W_\ell|Z=z^r)>\underline{\sigma}>0$ and \newline $\min_{r=1,\dots,k}P(Z=z^r)>\varpi>0$.
\end{assumption}
\begin{assumption}
\label{as:overlap} 
The model $\cP$ for $P$ satisfies: (1) $\eig(M_P^\prime M_P)>\varsigma$;
(2) $\eig(E_P(ZZ^\prime))>\varsigma$;\newline (3) $\eig(Corr_P([vech(ZZ'); W_0]))>\varsigma$ and $\eig(Corr_P([vech(ZZ'); W_1]))>\varsigma$; for some $\varsigma>0$, where $vech(A)$ denotes the half-vectorization of the matrix $A$.
\end{assumption}
\begin{assumption}
\label{as:overlap_game} 
The model $\cP$ for $P$ satisfies $\min_{r=1,\dots,k,(s,t)\in\{0,1\}\times\{0,1\}}P(Y_1=s,Y_2=t,Z=z^r)>\varpi>0$.
\end{assumption}
These are simple to verify low level conditions. We note that \cite{ImbensManski04} and \cite{Stoye09} directly assume the unconditional version of \ref{as:var:BM}, while \cite{ber:mol08} assume \ref{as:var:BM} itself. 

\subsection{Verification of Assumptions \ref{as:Lipschitz_m_over_sigma} and \ref{as:smoothness}-(i)}
\label{sec:verify_ass_EAM}
We show that in each of the four examples $\frac{m_j(x,\theta)}{\sigma_{P,j}(\theta)},~j=1,\dots,J$ is Lipschitz continuous in $\theta \in \Theta$ for all $x \in \mathcal{X}$ and that $D_P$ can be estimated at rate $n^{-1/2}$. The same arguments, with small modification, deliver verification of Assumption \ref{as:smoothness}-(i) provided $\hat{\sigma}_{n,j}(\theta)>0$.

\begin{enumerate}
\item \textbf{Mean with interval data}. Here $\sigma_{P,\ell}(\theta)=\sigma_{P,\ell}$, and under Assumption \ref{as:var:BM} it is uniformly bounded from below. Then
\begin{align*}
\left|\frac{m_j(x,\theta)}{\sigma_{P,j}}-\frac{m_j(x,\theta')}{\sigma_{P,j}}\right|&=\frac{\Vert(\theta'-\theta)\Vert}{\sigma_{P,j}},~~\ell=0,1,\\
D_{P,\ell}(\theta)&=\frac{(-1)^{(1-\ell)}}{\sigma_{P,\ell}},~~\ell=0,1.
\end{align*} 
Assumption \ref{as:var:BM} then guarantees that Assumption \ref{as:Lipschitz_m_over_sigma} is satisfied.
\item \textbf{Linear regression with interval outcome data and discrete regressors.}  Here again $\sigma_{P,\ell r}(\theta)=\sigma_{P,\ell r}$, and under Assumptions \ref{as:var:BM}-\ref{as:overlap} it is uniformly bounded from below. We first consider the rescaled function $\frac{(-1)^j(W_\ell\mathbf{1}(Z=z^r)/P(Z=z^r)-z^{r\prime}\theta)}{\sigma_{P,\ell r}}$:
\begin{align*}
\left|\frac{(-1)^j(W_\ell\mathbf{1}(Z=z^r)/P(Z=z^r)-z^{r\prime}\theta)}{\sigma_{P,\ell r}}-\frac{(-1)^j(W_\ell\mathbf{1}(Z=z^r)/P(Z=z^r)-z^{r\prime}\theta')}{\sigma_{P,\ell r}}\right|&=\Vert z^r \Vert\frac{\Vert(\theta'-\theta)\Vert}{\sigma_{P,\ell r}(\theta)},~~\ell=0,1,
\end{align*} 
so that Assumption \ref{as:Lipschitz_m_over_sigma} is satisfied for these rescaled functions by  Assumptions \ref{as:var:BM}-\ref{as:overlap}. Next, we observe that
\begin{align*}
D_{P,j}=\frac{(-1)^{(1-j)}z^{r\prime}}{\sigma_{P,\ell r}}&,~~\ell=0,1,r=1,\dots,k,
\end{align*} 
and it can be estimated at rate $n^{-1/2}$ by Lemma \ref{lem:eta_rate}. Theorem \ref{cor:eam_conv} then holds observing that $|P(Z=z^r)/(\sum_{i=1}^n \mathbf{1}(Z_i=z^r)/n)-1| = O_{\cP}(n^{-1/2})$ and treating this random element similarly to how we treat $\eta_{n,j}(\cdot)$ in the proof of Theorem \ref{cor:eam_conv}.
 
\item \textbf{Best linear prediction with interval outcome data and discrete regressors.} Here 
\begin{align}
m_r(X_i,\theta)=q^r [Z_i(Z_i^\prime \theta -(W_{0,i}+\mathbf{1}(q^rZ_i>0)(W_{1,i}-W_{0,i})))]
\end{align}
hence is Lipschitz in $\theta$ with constant $Z_iZ_i^\prime$. Under Assumptions \ref{as:var:BM}-\ref{as:overlap}, $Var_P(m_r(X_i,\theta))$ is uniformly bounded from below, and Lipschitz in $\theta$ with a constant that depends on $Z_i^4$. Hence $\frac{m_r(X_i,\theta)}{\sigma_{P,r}(\theta)}$ is Lipschitz in $\theta$ with a constant that depends on powers of $Z$. Because $Z$ has bounded support, Assumption \ref{as:Lipschitz_m_over_sigma} is satisfied. A simple argument yields that $D_P$ can be estimated at rate $n^{-1/2}$.
\item \textbf{Complete information entry games with pure strategy Nash equilibrium as solution concept.} Here again $\sigma_{P,st r}(\theta)=\sigma_{P,str}$, and under Assumptions \ref{as:momP_AS} and \ref{as:overlap_game} it is uniformly bounded from below. 
The result then follows from a similar argument as the one used in Example 2 (Linear regression with interval outcome data and discrete regressors), observing that the rescaled function of interest is now
\begin{align*}
\frac{\mathbf{1}(Y_1=s,Y_2=t,Z=z^r)/P(Z=z^r)-g_{str}(\theta)}{\sigma_{P,str}}&,~~(s,t)\in \{0,1\}\times\{0,1\},r=1,\dots,k,
\end{align*}
and the gradient is
\begin{align*}
\frac{1}{\sigma_{P,str}}\nabla_\theta g_{str}(\theta) &,~~(s,t)\in \{0,1\}\times\{0,1\},r=1,\dots,k,
\end{align*}
where $g_{str}(\theta)$ are model-implied entry probabilities, and hence taking their values in $[0,1]$. The entry models typically posited assume that payoff shocks have smooth distributions (e.g., multivariate normal), yielding that $\nabla_\theta g_{str}(\theta)$ is well defined and bounded.
\end{enumerate}

\subsection{Verification of Assumption \ref{as:correlation}-\ref{as:correlation_pair}}
\label{sec:verify_3.3}
Here we verify Assumption \ref{as:correlation}-\ref{as:correlation_pair} for the canonical examples in the moment (in)equalities literature:
\begin{enumerate}
\item \textbf{Mean with interval data}. In the generalization of this example in \cite{ImbensManski04} and \cite{Stoye09}, equations \eqref{eq:def_paired_ineq}-\eqref{eq:eig_cond_paired_ineq} are satisfied by construction, equation \eqref{eq:33primeiii} is directly assumed.

\item \textbf{Linear regression with interval outcome data and discrete regressors.} Equation \eqref{eq:def_paired_ineq} is satisfied by construction. Given the estimator that we use for the population moment conditions, we verify equation \eqref{eq:33primeiii} for the variances of the limit distribution of the vector $[\sqrt{n}(\hat{E}_n(W_\ell|Z=z^r)-E_P(W_\ell|Z=z^r))]_{\ell\in \{0,1\},r=1,\dots,k}$. We then have that equation \eqref{eq:33primeiii} follows from Assumption \ref{as:var:BM}.
Concerning equation \eqref{eq:33primeiii}, this needs to be verified for the correlation matrix of the limit distribution of a $r\times 1$ random vector that for each $r=1,\dots,k$ equals any choice in $\{\sqrt{n}(\hat{E}_n(W_0|Z=z^r)-E_P(W_0|Z=z^r)),\sqrt{n}(\hat{E}_n(W_1|Z=z^r)-E_P(W_1|Z=z^r))\}$, which suffices for our results to hold. We then have that \eqref{eq:eig_cond_paired_ineq} holds because the correlation matrix is diagonal.
\item \textbf{Best linear prediction with interval outcome data and discrete regressors.} Equation \eqref{eq:def_paired_ineq} is again satisfied by construction. Equation \eqref{eq:eig_cond_paired_ineq} holds under Assumptions \ref{as:var:BM}-\ref{as:overlap}.
Equation \eqref{eq:33primeiii} is verified to hold under Assumption \ref{as:var:BM} in \cite[p. 808]{ber:mol08}. 
\item \textbf{Complete information entry games with pure strategy Nash equilibrium as solution concept.} 
In this case equations \eqref{eq:entry7} and \eqref{eq:entry8} are paired, but the corresponding moment functions differ by the model implied probability of the region of multiplicity, hence equation \eqref{eq:def_paired_ineq} is satisfied by construction. 
Given the estimator that we use for the population moment conditions, we verify equations \eqref{eq:eig_cond_paired_ineq} and \eqref{eq:33primeiii} for the variances and for the correlation matrix of the limit distribution of the vector $\sqrt{n}(\hat{E}_n(Y_1=s,Y_2=t|Z=z^r)-E_P(Y_1=s,Y_2=t|Z=z^r)_{(s,t)\in \{0,1\}\times\{0,1\},r=1,\dots,k})$, which suffices for our results to hold. 
Equation \eqref{eq:eig_cond_paired_ineq} holds provided that $|Corr(Y_{i1}(1- Y_{i2}),Y_{i1} Y_{i2})|<1-\epsilon$ for some $\epsilon>0$ and Assumption \ref{as:overlap_game} holds.\footnote{In more general instances with more than two players, it follows if the multinomial distribution of outcomes of the game (reduced by one element) has a correlation matrix with eigenvalues uniformly bounded away from zero.}
To see that equation \eqref{eq:33primeiii} also holds, note that Assumption \ref{as:overlap_game} yields that $P(Y_{i1}=1,Y_{i2}=0, Z_i=z^r)$ is uniformly bounded away from 0 and 1, thereby implying that for each 
$(s,t)\in \{0,1\}\times\{0,1\},r=1,\dots,k$, $(P(Y_1=s,Y_2=t|Z=z^r)(1-P(Y_1=s,Y_2=t|Z=z^r)))/(P(Z=z^r)(1-P(Z=z^r)))$ is uniformly bounded away from zero. 
\end{enumerate}

\clearpage
\section{Proof of Theorem \ref{thm:validity}} \label{sec:app_A}

\subsection{Notation and Structure of the Proof of Theorem \ref{thm:validity}}\label{subsec:notation}
For any sequence of random variables $\{X_n\}$ and a positive sequence $a_n$, we write $X_n=o_{\mathcal P}(a_n)$ if for any $\epsilon,\eta>0$, there is $N\in\mathbb N$ such that $\sup_{P\in\mathcal P}P(|X_n/a_n|>\epsilon)<\eta,\forall n\ge N$. We write $X_n=O_{\mathcal P}(a_n)$ if for any $\eta>0$, there is a $M\in\mathbb R_+$ and $N\in\mathbb N$ such that $\sup_{P\in\mathcal P}P(|X_n/a_n|>M)<\eta,\forall n\ge N$. 

\setcounter{table}{0} 
\begin{table}[htbp]

\advance\leftskip-1.5cm
\parbox{20cm}{\caption[Table Notation]{Important notation. Here $(P_n,\theta_n)\in \{(P,\theta):P \in \cP,\theta \in \Theta_I(P)\}$ is a subsequence as defined in \eqref{eq:cover3}-\eqref{eq:pi_def} below,  $\theta_n^\prime \in \thetnprime$, $B^d = \{x \in \R^d:|x_i| \le 1, i=1,\dots,d\}$, $B^d_{n,\rho}\equiv \frac{\sqrt n}{\rho }(\Theta-\theta_n)\cap B^d$, $\mathfrak B^d_\rho=\lim_{n\to\infty}B^d_{n,\rho}$, and $\lambda\in\R^d$.}\label{table:notation}} \smallskip

\resizebox{20cm}{!}{

\begin{tabular}{lllllllllllllllll}
\hline \hline
\\ 
$\G_{n,j}(\cdot) $ & $=$ & $ \frac{\sqrt n (\bar{m}_{n,j}(\cdot) - E_P(m_j(X_i,\cdot)))}{\sigma_{P,j}(\cdot)},~j=1,\dots,J$ & Sample empirical process.\\ 
[0.2cm]
$\G_{n,j}^b(\cdot) $ & $=$ & $ \frac{\sqrt n (\bar{m}^b_{n,j}(\cdot) - \bar{m}_{n,j}(\cdot))}{\hat{\sigma}_{n,j}(\cdot)},~j=1,\dots,J$ & Bootstrap empirical process. \\ 
[0.2cm]
$\eta_{n,j}(\cdot) $ & $=$ & $ \frac{\sigma_{P,j}(\cdot)}{\hat{\sigma}_{n,j}(\cdot)}-1,~j=1,\dots,J$ & Estimation error in sample moments' asymptotic standard deviation.\\
[0.2cm]
$D_{P,j}(\cdot) $ & $=$ & $ \nabla_\theta \left(\frac{E_P(m_j(X_i,\cdot))}{\sigma_{P,j}(\cdot)} \right),~j=1,\dots,J$ & Gradient of population moments w.r.t. $\theta$, with estimator $\hat{D}_{n,j}(\cdot) $. \\
[0.2cm]
$\gamma_{1,P_n,j}(\cdot) $ & $=$ & $ \frac{E_{P_n}(m_j(X_i,\cdot))}{\sigma_{P_n,j}(\cdot)},~j=1,\dots,J$ & Studentized population moments. \\
[0.4cm]
$\pi_{1,j}$ & $=$ & $\lim_{n \to \infty} \kappa_{n}^{-1} \sqrt {n} \gamma_{1,P_n,j}(\theta_{n}^\prime)$ & Limit of rescaled population moments, constant $\forall \theta_n^\prime \in \thetnprime$\\
 &  &  & by Lemma \ref{lem:Jstar}.\\
$\pi_{1,j}^*$ & $=$ & $\left\{ \begin{matrix}
   0, & \mathrm{if}~\pi_{1,j}=0,\\
   -\infty, & \mathrm{if}~\pi_{1,j}<0.
\end{matrix}\right.$ & ``Oracle" GMS.\\
[0.4cm]
$\hat{\xi}_{n,j}(\cdot)$ & $=$ & $\left\{\begin{array}{ll}
\kappa_n^{-1}\sqrt n\bar m_{n,j}(\cdot)/\hat\sigma_{n,j}(\cdot), & j=1,\dots,J_1 \\ 
0, & j=J_1+1,\dots,J
\end{array}\right.$  & Rescaled studentized sample moments, set to $0$ for equalities. \\
[0.5cm]
$\varphi^*_j(\xi)$& $=$ & $\begin{cases}
	\varphi_j(\xi)&\pi_{1,j}=0\\
	-\infty&\pi_{1,j}<0\\
	0&j=J_1+1,\cdots,J.
\end{cases}$ & Infeasible GMS that is less conservative than $\varphi_j$.  \\
[0.5cm]
$u_{n,j,\theta_n}(\lambda)$ & $=$ &  $\{\mathbb G_{n,j}(\theta_n+\larhon)+\rho D_{P_n,j}(\bar{\theta}_n)\lambda+\pi^*_{1,j}\}(1+\eta_{n,j}(\theta_n+\larhon))$ & Mean value expansion of nonlinear constraints with sample empirical process \\
 &  &  & and ``oracle" GMS, with $\bar{\theta}_n$ componentwise between $\theta_n$ and $\theta_n+\larhon$.\\
[0.4cm]
$U_n(\theta_n,c)$ & $=$ &  $\big\{\lambda\in B^d_{n,\rho}: p^\prime \lambda = 0 \cap u_{n,j,\theta_n}(\lambda)\le c, \: \forall j=1,\dots,J\big\}$ & Feasible set for nonlinear sample problem intersected with $p^\prime \lambda=0$.\\
[0.4cm]
$\mathfrak{w}_{j}(\lambda)$ & $=$ & $\HH_{j}+\rho D_{j}\lambda+\pi^*_{1,j}$ & Linearized constraints with a Gaussian shift and ``oracle'' GMS.   \\
[0.4cm]
$\mathfrak W(c)$ & $=$ & $\big\{\lambda\in \mathfrak B^d_\rho: p^\prime \lambda = 0 \cap \mathfrak w_{j}(\lambda)\le c, \: \forall j=1,\dots,J\big\}$ & Feasible set for linearized limit problem intersected with $p'\lambda=0$.
\\
[0.4cm]
$c_{\pi^*}$ &$=$ & $\inf\{c\in\mathbb R_+:\mathrm{Pr}(\mathfrak{W}(c)\ne\emptyset)\ge 1-\alpha\}$.& Limit problem critical level.
\\
[0.4cm]
$v^b_{n,j,\theta_n^\prime}(\lambda)$ & $=$ &  $\mathbb G^b_{n,j}(\theta_n^\prime)+\rho \hat{D}_{n,j}(\theta_n^\prime)\lambda+\varphi_j(\hat{\xi}_{n,j}(\theta_n^\prime))$ & Linearized constraints with bootstrap empirical process and sample GMS.\\
[0.4cm]
$V_n^b(\theta_n^\prime,c)$ & $=$ &  $\big\{ \lambda\in B^d_{n,\rho}: p^\prime \lambda = 0 \cap v^b_{n,j,\theta_n^\prime}(\lambda)\le c, \: \forall j=1,\dots,J\big\}$ & Feasible set for linearized bootstrap problem with sample GMS and $p^\prime \lambda=0$. &\\
[0.4cm]
$v^I_{n,j,\theta_n^\prime}(\lambda)$ & $=$ &  $\mathbb G^b_{n,j}(\theta_n^\prime)+\rho \hat{D}_{n,j}(\theta_n^\prime)\lambda+\varphi^*_j(\hat{\xi}_{n,j}(\theta_n^\prime))$ & Linearized constraints with bootstrap empirical process and infeasible sample GMS.\\
[0.4cm]
$V_n^I(\theta_n^\prime,c)$ & $=$ &  $\big\{ \lambda\in B^d_{n,\rho}: p^\prime \lambda = 0 \cap v^I_{n,j,\theta_n^\prime}(\lambda)\le c, \: \forall j=1,\dots,J\big\}$ & Feasible set for linearized bootstrap problem with infeasible sample GMS and $p^\prime \lambda=0$. &\\
[0.4cm]

$\hat c_n(\theta)$ & $=$ &  $\inf\{c\in\mathbb R_+:P^*(V_n^b(\theta,c)\neq \emptyset)\ge 1-\alpha\}$ & Bootstrap critical level.\\
[0.2cm]
$\hat c_{n,\rho}(\theta)$ & $=$ & $\inf_{\lambda \in B^d_{n,\rho}} \hat{c}_n(\theta+\frac{\lambda\rho }{\sqrt n})$ & Smallest value of the bootstrap critical level in a $B^d_{n,\rho}$ neighborhood of $\theta$.\\
[0.2cm]
$\hat \sigma_{n,j}^M(\theta)$ & $=$ & $\hat \mu_{n,j}(\theta)\hat\sigma_{n,j}(\theta)+(1-\hat \mu_{n,j}(\theta))\hat\sigma_{n,j+R_1}(\theta)$ & Weighted sum of the estimators of the standard deviations of paired inequalities \\
[0.2cm]
\hline \hline
\end{tabular}
}
\end{table}

\enlargethispage*{5\baselineskip}
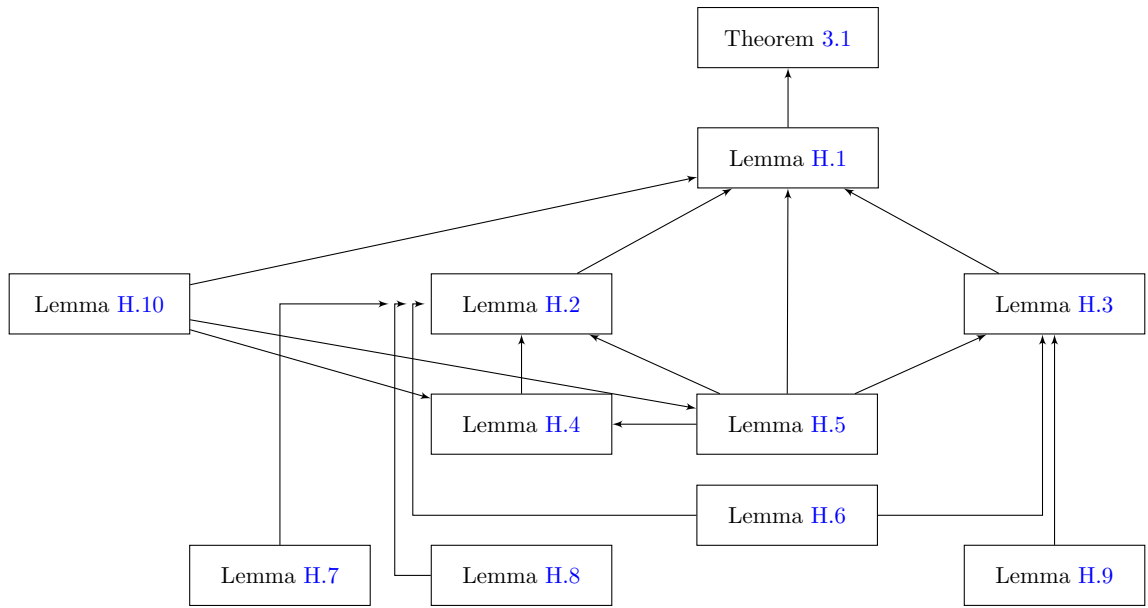
\begin{figure}[h]
	\caption{Structure of Lemmas used in the proof of Theorem \ref{thm:validity}-\ref{thm:validity:basic}.}
	\label{fig:flow_app}
\begin{center}
\begin{tikzpicture}[node distance = 2cm, scale=0.8, every node/.style={transform shape}]
\linespread{0.9}
\node [lemma] (Thm31) {Theorem \ref{thm:validity}};
 
\node [lemma, below of =Thm31] (B2) {Lemma \ref{lem:val}};
 
\node [lemma, below right =2cm of B2] (B3) {Lemma \ref{lem:cv_convergence}};
 
\node [lemma, below left =2cm of B2] (B4) {Lemma \ref{lem:UnonEmpty}};
 
\node [lemma, below  of = B4] (B7) {Lemma \ref{lem:res}};
 
\node [lemma, right = 1.4cm of B7] (B6) {Lemma \ref{lem:Jstar}};
 
\node [lemma, below = 0.5cm of B6] (B1) {Lemma \ref{lem:empt}};

\node [lemma, below = 3.5cm of B3] (B9) {Lemma \ref{lem:pair}};
  

\node [lemma, below =  1.5cm of  B7] (B5) {Lemma \ref{lem:cbd}};
 
\node [lemma, left  = 1cm of B5] (B8) {Lemma \ref{lem:exist_sol}};
 
\node [lemma, left  = 4cm of B4] (B11) {Lemma \ref{lem:eta_conv}};

\draw[line](B11) -- (B2);

\draw[line](B11) -- (B7);

\draw[line](B11) -- (B6);

 
\draw [line] (B8) |- ([xshift=-0.7cm]B4.west);
 
\draw [line] (B5)-| ([xshift=-0.6cm]B4.south west) |- ([xshift=-0.4cm]B4.west);
 
\draw [line] (B1)-| ([xshift=-0.3cm]B4.south west) |- ([xshift=-0.1cm]B4.west);
 
\draw[line](B6) -- (B7);
 
\draw[line](B6) --(B4);
 
\draw[line](B6) --(B3);
 
\draw[line](B6) --(B2);
 
\draw[line](B1) -| ([xshift=-0.2cm]B3.south);
 
\draw[line](B3) -- (B2);
 
\draw[line](B4) -- (B2);
 
\draw[line](B2)--(Thm31);
 
\draw[line](B7)--(B4);
 

\draw[line](B9) -- (B3);
 
 
 
 
 
 
 
 

\end{tikzpicture}
\end{center}
\end{figure}

\begin{table}[!htbp]
\advance\leftskip-1.5cm
\parbox{20cm}{\caption[Table Flow]{Heuristics for the role of each Lemma in the proof of Theorem \ref{thm:validity}. Notes: (i) Uniformity in Theorem \ref{thm:validity} is enforced arguing along subsequences; (ii) When needed, random variables are realized on the same probability space as shown in Lemma \ref{lem:val} and Lemma \ref{cor:asrep} (see Appendix \ref{app:asrep} for details); (iii) Here $(P_n,\theta_n)\in \{(P,\theta):P \in \cP,\theta \in \Theta_I(P)\}$ is a subsequence as defined in \eqref{eq:cover3}-\eqref{eq:pi_def} below; (iv) All results hold for any $\theta_n^\prime \in \thetnprime$.}\label{table:flow}} \smallskip

\resizebox{20cm}{!}{
\begin{tabular}{lllllllllllllllll}
\hline \hline
 \\
Theorem \ref{thm:validity} & $P_n(p'\theta_n \in CI) \geq P_n\left(U_{n}(\theta_{n},\hat{c}_{n,\rho}(\theta_{n}))\neq \emptyset\right).$\\ 
 & Coverage is conservatively estimated by the probability that $U_{n}$ is nonempty. \\ 
[0.15cm]
Lemma \ref{lem:val} & $\liminf P_n\left(U_{n}(\theta_{n},\hat{c}_{n,\rho}(\theta_{n}))\neq \emptyset\right)\geq 1-\alpha.$\\
[0.15cm]
Lemma \ref{lem:UnonEmpty} & $P_n(U(\theta_n,c^I_n(\theta_{n}))\neq \emptyset ,\mathfrak W(c_{\pi^*})=\emptyset )+P_n(U(\theta_n,c^I_n(\theta_{n}))=\emptyset ,\mathfrak W(c_{\pi^*})\neq \emptyset)=o_{\cP}(1).$ \\
 & Argued by comparing $U_n$ and its limit $\mathfrak W$ (after coupling).\\
[0.15cm]
Lemma \ref{lem:cv_convergence} & $P^*_n(V^I_n(\theta'_n,c)\ne\emptyset)-	\Pr(\mathfrak{W}(c)\ne\emptyset)\to 0$ and  $c^I_n(\theta'_n)\stackrel{P_n}{\to}c_{\pi^*}$ if $c_{\pi^*}>0$.\\
 & The bootstrap critical value that uses the less conservative GMS yileds a convergent critical value.\\
[0.15cm]
Lemma \ref{lem:res} & $\sup_{\lambda\in B^d} | \max_{j}(u_{n,j,\theta_{n}}(\lambda)-c^I_{n}(\theta_n))-\max_{j}(\mathfrak w_{j}(\lambda)-c_{\pi^*})|=o_\cP(1)$, and similarly for $\mathfrak w_{j}$ and $v^I_{n,j,\theta_{n}^\prime}$.\\
 &  The criterion functions entering $U_n$ and $\mathfrak W$ converge to each other.
\\
[0.15cm]
Lemma \ref{lem:Jstar} & Local-to-binding constraints are selected by GMS uniformly over the $\rho $-box (intuition: $\rho  n^{-1/2}=o_{\cP}(\kappa _{n}^{-1})$),\\
 & and $\Vert \hat{\xi}_n(\theta_n^\prime) - \kappa_n^{-1}\sqrt{n}\sigma_{P_n,j}^{-1}(\theta_n^\prime)E_{P_n}[m_j(X_i,\theta_n^\prime)] \Vert =o_\cP(1)$. \\
[0.15cm]
Lemma \ref{lem:empt} & $\forall \eta>0 ~\exists \delta>0,  : \Pr(\{\mathfrak W(c)\ne\emptyset\}\cap \{\mathfrak W^{-\delta}(c)=\emptyset\})<\eta$, and similarly for $V^I_n$. \\
 & It is unlikely that these sets are nonempty but become empty upon slightly tightening stochastic constraints.\\
[0.15cm]
Lemma \ref{lem:exist_sol} & Intersections of constraints whose gradients are almost linearly dependent are unlikely to realize inside $\mathfrak{W}$. \\
 & Hence, we can ignore irregularities that occur as linear dependence is approached.\\
[0.15cm]
Lemma \ref{lem:cbd} & If there are weakly more equality constraints than parameters, then $c$ is uniformly bounded away from zero.\\
 & This simplifies some arguments.\\
[0.15cm]
Lemma \ref{lem:pair} & If two paired inequalities are local to binding, then they are also asymptotically identical up to sign.\\
 & This justifies ``merging" them.\\
[0.15cm]
Lemma \ref{lem:eta_conv} & $\eta_{n,j}(\cdot)$ converges to zero uniformly in $P$ and $\theta$.\\
[0.15cm]
\hline \hline
\end{tabular}
}
\end{table}

\pagebreak

\clearpage
\subsection{Proof of Theorem \ref{thm:validity}}\label{sec:proofs:thms}

\subsubsection{Main Proofs}
\textbf{Proof of Theorem \ref{thm:validity}-\ref{thm:validity:basic}.}

Following \cite{Andrews_Guggenberger2009bET}, we index distributions by a vector of nuisance parameters relevant for the asymptotic size. For this, let $\gamma_P\equiv(\gamma_{1,P},\gamma_{2,P},\gamma_{3,P})$, where
 $\gamma_{1,P}=(\gamma_{1,P,1},\cdots,\gamma_{1,P,J})$ with 
\begin{align}
\gamma_{1,P,j}(\theta)=\sigma_{P,j}^{-1}(\theta)E_P[m_j(X_i,\theta)],
~j=1,\cdots, J,\label{eq:cover1}
\end{align}
$\gamma_{2,P}=(s(p,\Theta_I(P)),vech(\Omega_P(\theta)),vec(D_P(\theta)))$, and
$\gamma_{3,P}=P$. We proceed in steps.

\noindent \textbf{Step 1.} Let $\{P_n,\theta_n\} \in \{(P,\theta):P \in \cP,\theta \in \Theta_I(P)\}$ be a sequence such that
\begin{equation}
\liminf_{n\to\infty}\inf_{P\in\mathcal P}	\inf_{\theta \in \Theta_I(P)}P(p^\prime \theta \in CI_n)=\liminf_{n\to\infty}P_n(p^\prime \theta_n \in CI_n),\label{eq:cover2}
\end{equation}
with $CI_n=[-s(-p,\cC_n(\hat{c}_n)),s(p,\cC_n(\hat{c}_n))]$. We then let $\{l_n\}$ be a subsequence of $\{n\}$ such that 
\begin{equation}
\liminf_{n\to\infty}P_n(p^\prime \theta_n \in CI_n)=\lim_{n\to\infty}P_{l_n}(p^\prime \theta_{l_n} \in CI_{l_n}).\label{eq:cover3}
\end{equation}
Then there is a further subsequence $\{a_n\}$ of $\{l_n\}$ such that
\begin{align}
\lim_{a_n \to \infty} \kappa_{a_n}^{-1} \sqrt {a_n} \sigma_{P_{a_n},j}^{-1}(\theta_{a_n})E_{P_{a_n}}[m_j(X_i,\theta_{a_n})] = \pi_{1,j} \in \R_{[-\infty]},~j=1,\dots,J.\label{eq:pi_def}
\end{align}
To avoid multiple subscripts, with some abuse of notation we write $(P_n,\theta_n)$ to refer to $(P_{a_n},\theta_{a_n})$ throughout this Appendix. 
We let
\begin{align}
\pi_{1,j}^*&=\left\{ \begin{matrix}
   0 & \mathrm{if}~\pi_{1,j}=0,\\
   -\infty & \mathrm{if}~\pi_{1,j}<0.
\end{matrix}\label{eq:def_pi_star}
 \right.
\end{align}
The projection of $\theta_n$ is covered when
\begin{align}
& ~-s(-p,\mathcal C_n(\hat{c}_n))\le p'\theta_n  \le s(p,\mathcal C_n(\hat{c}_n))\notag \\
	 \Longleftrightarrow &~ \begin{Bmatrix} \inf p^\prime\vartheta  &   \\ \text{s.t. } \vartheta \in \Theta, & \frac{\sqrt n \bar{m}_{n,j}(\vartheta)}{\hat\sigma_{n,j}(\vartheta)}\le  \hat{c}_n(\vartheta),  \forall j \end{Bmatrix} \le  p'\theta_n  \notag \le  \begin{Bmatrix} \sup p'\vartheta  &   \\ \text{s.t. } \vartheta \in \Theta, & \frac{\sqrt n \bar{m}_{n,j}(\vartheta)}{\hat\sigma_{n,j}(\vartheta)}\le  \hat{c}_n(\vartheta), \forall j \end{Bmatrix}\\
	 \Longleftrightarrow &~  \begin{Bmatrix} \inf_{\lambda} p^\prime\lambda  &   \\ \text{s.t. } \lambda\in \frac{\sqrt n}{\rho }(\Theta-\theta_n), & \frac{\sqrt n\bar{m}_{n,j}(\theta_n+\larhon)}{\hat\sigma_{n,j}(\theta_n+\larhon)}\le  \hat{c}_n(\theta_n+\larhon), \forall j \end{Bmatrix} \leq 0\notag\\
	 &\hspace{0.5in}\leq\begin{Bmatrix} \sup_{\lambda} p'\lambda  &   \\ \text{s.t. } \lambda\in \frac{\sqrt n}{\rho }(\Theta-\theta_n), & \frac{\sqrt n\bar{m}_{n,j}(\theta_n+\larhon)}{\hat\sigma_{n,j}(\theta_n+\larhon)}\le  \hat{c}_n(\theta_n+\larhon), \forall j \end{Bmatrix} \label{eq:rho_appears}\\
	 	 \Longleftrightarrow &~ \begin{Bmatrix} \inf_{\lambda} p^\prime\lambda  &   \\ \text{s.t. } \lambda\in \frac{\sqrt n}{\rho }(\Theta-\theta_n), & \\ \{\mathbb G_{n,j}(\theta_n+\larhon)+\rho  D_{P_n,j}(\bar\theta_n)\lambda+\sqrt{n} \gamma_{1,P_n,j}(\theta_n+\larhon)\}(1+\eta_{n,j}(\theta_n+\larhon))\le  \hat{c}_n(\theta_n+\larhon), \forall j  &  \end{Bmatrix} \leq 0  \notag\\
	 &\hspace{0.5in}\leq\begin{Bmatrix} \sup_{\lambda} p'\lambda  &   \\ \text{s.t. } \lambda\in \frac{\sqrt n}{\rho }(\Theta-\theta_n), & \\ \{\mathbb G_{n,j}(\theta_n+\larhon)+\rho  D_{P_n,j}(\bar\theta_n)\lambda+\sqrt{n} \gamma_{1,P_n,j}(\theta_n)\}(1+\eta_{n,j}(\theta_n+\larhon))\le  \hat{c}_n(\theta_n+\larhon), \forall j  &  \end{Bmatrix} \label{eq:cover_part1},
	 \end{align}
	with $\eta_{n,j}(\cdot)\equiv \sigma_{P,j}(\cdot)/\hat\sigma_{n,j}(\cdot)-1$ and where we localized $\vartheta$ in a $\sqrt n/\rho$-neighborhood of $\Theta-\theta_n$ and we took a mean value expansion yielding, for all $j$,
\begin{align}
\frac{\sqrt n\bar{m}_{n,j}(\theta_n+\larhon)}{\hat\sigma_{n,j}(\theta_n+\larhon)}=\bigl\{\mathbb G_{n,j}(\theta_n+\tfrac{\lambda \rho}{\sqrt{n}})+\rho D_{P_n,j}(\bar\theta_n)\lambda+\sqrt{n} \gamma_{1,P_n,j}(\theta_n)\bigr\}\bigl(1+\eta_{n,j}(\theta_n+\tfrac{\lambda \rho}{\sqrt{n}})\bigr).\label{eq:MVE}
\end{align}	 
Denote $B^d_{n,\rho}\equiv \frac{\sqrt n}{\rho }(\Theta-\theta_n)\cap B^d$, with $B^d = \{x \in \R^d:|x_i| \le 1, i=1,\dots,d\}$.
Then the event in \eqref{eq:cover_part1} is implied by
	 \begin{align}
	 	 		 &\begin{Bmatrix} \inf_{\lambda} p^\prime\lambda  &   \\ \text{s.t. } \lambda\in B^d_{n,\rho}, & \\ \{\mathbb G_{n,j}(\theta_n+\larhon)+\rho  D_{P_n,j}(\bar\theta_n)\lambda+\sqrt{n} \gamma_{1,P_n,j}(\theta_n)\}(1+\eta_{n,j}(\theta_n+\larhon))\le  \hat{c}_n(\theta_n+\larhon), \forall j  &  \end{Bmatrix} \leq 0  \notag\\
	 &\hspace{0.5in}\leq\begin{Bmatrix} \sup_{\lambda} p'\lambda  &   \\ \text{s.t. } \lambda\in B^d_{n,\rho}, & \\ \{\mathbb G_{n,j}(\theta_n+\larhon)+\rho  D_{P_n,j}(\bar\theta_n)\lambda+\sqrt{n} \gamma_{1,P_n,j}(\theta_n)\}(1+\eta_{n,j}(\theta_n+\larhon))\le  \hat{c}_n(\theta_n+\larhon), \forall j  &  \end{Bmatrix}.\label{eq:cover_part2}
	 \end{align}

\noindent \textbf{Step 2.} This step is used only when Assumption \ref{as:correlation}-\ref{as:correlation_pair} is invoked. When this assumption is invoked, recall that
in equation \eqref{eq:CI} we use the estimator specified in Lemma \ref{lem:eta_conv} equation \eqref{eq:def_hat_sigma_M} for $\sigma_{P,j},j=1,\dots,2R_1$ (with $R_1 \le J_1/2$ defined in the statement of the assumption). In equation \eqref{eq:Lambda_n} we use the sample analog estimators of $\sigma_{P,j}$ for all $j=1,\dots,J$.
To keep notation manageable, we explicitly denote the estimator used in \eqref{eq:CI} by $\hat{\sigma}^M_j$ only in this step but in almost all other parts of this Appendix we use the generic notation $\hat{\sigma}_j$. 

For each $j=1,\dots,R_1$ such that 
\begin{align}
\pi^*_{1,j} = \pi^*_{1,j+R_1} = 0, \label{eq:cond_both_CT_binding}
\end{align}
where $\pi^*_{1}$ is defined in \eqref{eq:def_pi_star}, let 
\begin{align}
\tilde\mu_j &=\left\{\begin{matrix}
 1 & \mathrm{if}~ \gamma_{1,P_n,j}(\theta_n)=0=\gamma_{1,P_n,j+R_1}(\theta_n),\\
\frac{\gamma_{1,P_n,j+R_1}(\theta_n)(1+\eta_{n,j+R_1}(\theta_n+\larhon))}{\gamma_{1,P_n,j+R_1}(\theta_n)(1+\eta_{n,j+R_1}(\theta_n+\larhon))+\gamma_{1,P_n,j}(\theta_n)(1+\eta_{n,j}(\theta_n+\larhon))} & \mathrm{otherwise},
\end{matrix}\right. \label{eq:tilde_weight_j}\\
\tilde\mu_{j+R_1}&=\left\{\begin{matrix}
 0 & \mathrm{if}~ \gamma_{1,P_n,j}(\theta_n)=0=\gamma_{1,P_n,j+R_1}(\theta_n),\\
\frac{\gamma_{1,P_n,j}(\theta_n)(1+\eta_{n,j}(\theta_n+\larhon))}{\gamma_{1,P_n,j+R_1}(\theta_n)(1+\eta_{n,j+R_1}(\theta_n+\larhon))+\gamma_{1,P_n,j}(\theta_n)(1+\eta_{n,j}(\theta_n+\larhon))} & \mathrm{otherwise},
\end{matrix}\right. \label{eq:tilde_weight_j_pl_J11}
\end{align}
For each $j=1,\dots,R_1$, replace the constraint indexed by $j$, that is
\begin{align}
\frac{\sqrt n\bar{m}_{n,j}(\theta_n+\larhon)}{\hat\sigma^M_{n,j}(\theta_n+\tlarhon)}\le  \hat{c}_n(\theta_n+\tlarhon),\label{eq:replace_CT_LB}
\end{align}
with the following weighted sum of the paired inequalities
\begin{align}
\tilde\mu_j\frac{\sqrt n\bar{m}_{n,j}(\theta_n+\larhon)}{\hat\sigma^M_{n,j}(\theta_n+\larhon)}-\tilde\mu_{j+R_1}\frac{\sqrt n\bar{m}_{j+R_1,n}(\theta_n+\larhon)}{\hat\sigma^M_{n,j+R_1}(\theta_n+\tlarhon)}\le  \hat{c}_n(\theta_n+\tlarhon), \label{eq:replace_CT_mix_LB}
\end{align}
and for each $j=1,\dots,R_1$, replace the constraint indexed by $j+R_1$, that is
\begin{align}
\frac{\sqrt n\bar{m}_{j+R_1,n}(\theta_n+\larhon)}{\hat\sigma^M_{n,j+R_1}(\theta_n+\larhon)}\le  \hat{c}_n(\theta_n+\tlarhon), \label{eq:replace_CT_UB}
\end{align}
with
\begin{align}
-\tilde\mu_j\frac{\sqrt n\bar{m}_{n,j}(\theta_n+\larhon)}{\hat\sigma^M_{n,j}(\theta_n+\larhon)}+\tilde\mu_{j+R_1}\frac{\sqrt n\bar{m}_{j+R_1,n}(\theta_n+\larhon)}{\hat\sigma^M_{n,j+R_1}(\theta_n+\larhon)}\le  \hat{c}_n(\theta_n+\tlarhon),\label{eq:replace_CT_mix_UB}
\end{align}
It then follows from Assumption \ref{as:correlation}-\ref{as:correlation_pair} that these replacements are conservative  because 
\begin{align*}
\frac{\bar{m}_{j+R_1,n}(\theta_n+\larhon)}{\hat\sigma^M_{n,j+R_1}(\theta_n+\larhon)}\le -\frac{\bar{m}_{n,j}(\theta_n+\larhon)}{\hat\sigma^M_{n,j}(\theta_n+\larhon)},
\end{align*}
and therefore \eqref{eq:replace_CT_mix_LB} implies \eqref{eq:replace_CT_LB} and \eqref{eq:replace_CT_mix_UB} implies \eqref{eq:replace_CT_UB}.

\noindent \textbf{Step 3.} Next, we make the following comparisons:
\begin{align}
\pi_{1,j}^*&= 0 \Rightarrow \pi_{1,j}^* \ge \sqrt{n} \gamma_{1,P_n,j}(\theta_n), \label{eq:pi_star_zero}\\
\pi_{1,j}^*&= -\infty \Rightarrow \sqrt{n} \gamma_{1,P_n,j}(\theta_n) \to -\infty.\label{eq:pi_star_inf}
\end{align}
For any constraint $j$ for which $\pi_{1,j}^*= 0$, \eqref{eq:pi_star_zero} yields that replacing $\sqrt{n} \gamma_{1,P_n,j}(\theta_n)$ in \eqref{eq:cover_part2} with $\pi_{1,j}^*$ introduces a conservative distortion. Under Assumption \ref{as:correlation}-\ref{as:correlation_pair}, for any $j$ such that \eqref{eq:cond_both_CT_binding} holds, the substitutions in \eqref{eq:replace_CT_mix_LB} and \eqref{eq:replace_CT_mix_UB} yield $\tilde\mu_j\sqrt{n} \gamma_{1,P_n,j}(\theta_n)(1+\eta_{n,j}(\theta_n+\larhon))-\tilde\mu_{j+R_1}\sqrt{n} \gamma_{1,P_n,j+R_1}(\theta_n)(1+\eta_{n,j+R_1}(\theta_n+\larhon))=0$, and therefore replacing this term with $\pi_{1,j}^*=0=\pi_{1,j+R_1}^*$ is inconsequential.

 For any $j$ for which $\pi_{1,j}^* = -\infty$, \eqref{eq:pi_star_inf} yields that for $n$ large enough, $\sqrt{n} \gamma_{1,P_n,j}(\theta_n)$ can be replaced with $\pi_{1,j}^*$. To see this, note that by the Cauchy-Schwarz inequality, Assumption \ref{as:momP_KMS} (i)-(ii), and $\lambda\in B^d_{n,\rho}$, it follows that 
\begin{align}
	\rho D_{P_{n},j}(\bar\theta_{n})\lambda \le \rho  \sqrt{d} (\|D_{P_{n},j}(\bar\theta_{n})-D_{P_{n},j}(\theta_{n})\|+\|D_{P_{n},j}(\theta_{n})\|)\le \rho \sqrt{d} (\rho  M/\sqrt n + \bar M),\label{eq:val6}
\end{align}
where $\bar{M}$ and $M$ are as defined in Assumption \ref{as:momP_KMS}-(i) and (ii) respectively, and we used that $\bar{\theta}_{n}$ lies component-wise between $\theta_{n}$ and $\theta_{n}+\larhon$. Using that $\G_{n,j}$ is asymptotically tight by Assumption \ref{as:bcs1}, we have that for any $\tau>0$, there exists a $T>0$ and $N_1\in\mathbb N$ such that for all $n\ge N_1$,
\begin{align}
&P_n \left(\max_{j:\pi^*_{1,j}=-\infty}\{\mathbb G_{n,j}(\theta_n+\tlarhon)+\rho  D_{P_n,j}(\bar\theta_n)\lambda+\sqrt{n} \gamma_{1,P_n,j}(\theta_n)\}(1+\eta_{n,j}(\theta_n+\tlarhon))\le  0 ,~\forall \lambda\in B^d_{n,\rho} \right)>1-\tau/2.\label{eq:pi_star_gamma_vanishes}
\end{align}
To see this, note that $\pi_{ij}^*=-\infty$ if and only if $\lim_{n \to \infty}\frac{\sqrt n}{\kappa_n}\gamma_{1P_nj}(\theta_n)=\pi_{1j}\in [-\infty,0)$. Suppose first that $\pi_{1j}>-\infty$. Then for all $\epsilon>0$ there exists $N_2 \in \N$ such that $
\Bigl\vert \frac{\sqrt n}{\kappa_n}\gamma_{1P_nj}(\theta_n)-\pi_{1j} \Bigr\vert \le \epsilon$, for all $n \ge N_2$. 
Choose $\epsilon>0$ such that $\pi_{1j}+\epsilon <0$. 
Let $N=\max\{N_1,N_2\}$. Then we have
\begin{align}
&~P_n \left(\max_{j:\pi^*_{1,j}=-\infty}\{\mathbb G_{n,j}(\theta_n+\tlarhon)+\rho  D_{P_n,j}(\bar\theta_n)\lambda+\sqrt{n} \gamma_{1,P_n,j}(\theta_n)\}(1+\eta_{n,j}(\theta_n+\tlarhon))\le  0 ,~\forall \lambda\in B^d_{n,\rho} \right)\notag\\
\ge &~P_n\left(\max_{j:\pi^*_{1,j}=-\infty} \{T+\rho  (\bar{M} + \tfrac{\rho  M}{\sqrt{n}})+\sqrt{n} \gamma_{1,P_n,j}(\theta_n)\}(1+\eta_{n,j}(\theta_n+\tlarhon))\le  0 \cap \max_{j:\pi^*_{1,j}=-\infty}\mathbb G_{n,j}(\theta_n+\tlarhon) \le T \right)\notag\\
\ge &~P_n\left(\max_{j:\pi^*_{1,j}=-\infty} \{T+\rho  (\bar{M} + \tfrac{\rho  M}{\sqrt{n}})+\kappa_n (\pi_{1j}+\epsilon)\}(1+\eta_{n,j}(\theta_n+\tlarhon))\le  0 \cap \max_{j:\pi^*_{1,j}=-\infty}\mathbb G_{n,j}(\theta_n+\tlarhon) \le T \right)\notag\\
= &~P_n\left(\max_{j:\pi^*_{1,j}=-\infty} \left\{\frac{T}{\kappa_n}+\frac{\rho }{\kappa_n} (\bar{M} + \tfrac{\rho  M}{\sqrt{n}})+(\pi_{1j}+\epsilon)\right\}(1+\eta_{n,j}(\theta_n+\tlarhon))\le  0 \cap \max_{j:\pi^*_{1,j}=-\infty}\mathbb G_{n,j}(\theta_n+\tlarhon) \le T \right)\notag\\
= 	&~P_n \left(\max_{j:\pi^*_{1,j}=-\infty}\mathbb G_{n,j}(\theta_n+\tlarhon) \le T\right)>1-\tau/2,~\forall n\ge N.\notag
\end{align}
If $\pi_{1j}=-\infty$ the same argument applies a fortiori.
We therefore have that for $n\ge N$,
	 \begin{align}
	 &~P_n \Bigg( \begin{Bmatrix} \inf_{\lambda} p^\prime\lambda   \\ \text{s.t. } \lambda\in B^d_{n,\rho},  \\ \{\mathbb G_{n,j}(\theta_n+\larhon)+\rho D_{P_n,j}(\bar\theta_n)\lambda+\sqrt{n} \gamma_{1,P_n,j}(\theta_n)\}(1+\eta_{n,j}(\theta_n+\larhon))\le  \hat{c}_n(\theta_n+\larhon), \forall j    \end{Bmatrix} \leq 0  \notag\\
	&\hspace{0.5in}\leq\begin{Bmatrix} \sup_{\lambda} p'\lambda    \\ \text{s.t. } \lambda\in B^d_{n,\rho}, \\ \{\mathbb G_{n,j}(\theta_n+\larhon)+\rho D_{P_n,j}(\bar\theta_n)\lambda+\sqrt{n} \gamma_{1,P_n,j}(\theta_n)\}(1+\eta_{n,j}(\theta_n+\larhon))\le  \hat{c}_n(\theta_n+\larhon), \forall j    \end{Bmatrix}\Bigg) \label{eq:replace_pi_star1}\\
	\ge &~P_n \Bigg( \begin{Bmatrix} \inf_{\lambda} p^\prime\lambda     \\ \text{s.t. } \lambda\in B^d_{n,\rho},  \\ \{\mathbb G_{n,j}(\theta_n+\larhon)+\rho D_{P_n,j}(\bar\theta_n)\lambda+ \pi_{1,j}^*\}(1+\eta_{n,j}(\theta_n+\larhon))\le  \hat{c}_n(\theta_n+\larhon), \forall j    \end{Bmatrix} \leq 0  \notag\\
	&\hspace{0.5in}\leq\begin{Bmatrix} \sup_{\lambda} p'\lambda     \\ \text{s.t. } \lambda\in B^d_{n,\rho},  \\ \{\mathbb G_{n,j}(\theta_n+\larhon)+\rho D_{P_n,j}(\bar\theta_n)\lambda+\pi_{1,j}^*\}(1+\eta_{n,j}(\theta_n+\larhon))\le  \hat{c}_n(\theta_n+\larhon), \forall j   \end{Bmatrix}\Bigg)-\tau/2.\label{eq:replace_pi_star2}
	\end{align}
	Since the choice of $\tau$ is arbitrary, the limit of the term in \eqref{eq:replace_pi_star1} is not smaller than the limit of the first term in \eqref{eq:replace_pi_star2}.
Hence, we continue arguing for the event whose probability is evaluated in \eqref{eq:replace_pi_star2}. 

Finally, by definition $\hat{c}_n(\cdot) \ge 0$ and therefore $\inf_{\lambda \in B^d_{n,\rho}}\hat{c}_n(\theta_n+\larhon)$ exists. Therefore, the event whose probability is evaluated in \eqref{eq:replace_pi_star2} is implied by the event
\begin{align}
&\begin{Bmatrix} \inf_{\lambda} p^\prime\lambda     \\ \text{s.t. } \lambda\in B^d_{n,\rho},  \\ \{\mathbb G_{n,j}(\theta_n+\larhon)+\rho D_{P_n,j}(\bar\theta_n)\lambda+\pi_{1,j}^*\}(1+\eta_{n,j}(\theta_n+\larhon))\le  \inf_{\lambda \in B^d_{n,\rho}}\hat{c}_n(\theta_n+\larhon), \forall j    \end{Bmatrix} \leq 0  \notag\\
	 &\hspace{0.5in}\leq\begin{Bmatrix} \sup_{\lambda} p'\lambda     \\ \text{s.t. } \lambda\in B^d_{n,\rho},  \\ \{\mathbb G_{n,j}(\theta_n+\larhon)+\rho D_{P_n,j}(\bar\theta_n)\lambda+\pi_{1,j}^*\}(1+\eta_{n,j}(\theta_n+\larhon))\le \inf_{\lambda \in B^d_{n,\rho}}\hat{c}_n(\theta_n+\larhon), \forall j    \end{Bmatrix}
	 \label{eq:heuristic_3}
\end{align}
For each $\lambda\in\R^d$, define
\begin{align}
u_{n,j,\theta_n}(\lambda)&\equiv	\bigl\{\mathbb G_{n,j}(\theta_n+\tfrac{\lambda \rho}{\sqrt{n}})+\rho D_{P_n,j}(\bar{\theta}_n)\lambda+\pi^*_{1,j}\bigr\}\bigr(1+\eta_{n,j}(\theta_n+\tfrac{\lambda \rho}{\sqrt{n}})\bigl),\label{eq:uj}
\end{align}
where under Assumption \ref{as:correlation}-\ref{as:correlation_pair} when $\pi_{1,j}^*=0$ and $\pi_{1,j+R_1}^*=0$ the substitutions of equation \eqref{eq:replace_CT_LB} with equation \eqref{eq:replace_CT_mix_LB} and of equation \eqref{eq:replace_CT_UB} with equation \eqref{eq:replace_CT_mix_UB} have been performed. Let 
\begin{align}
U_n(\theta_n,c)\equiv	\big\{\lambda\in B^d_{n,\rho}: p^\prime \lambda = 0 \cap u_{n,j,\theta_n}(\lambda)\le c, \: \forall j=1,\dots,J\big\}, \label{eq:set_U_NL}
\end{align}
and define 
 \begin{align}
\hat c_{n,\rho} \equiv \inf_{\lambda \in B^d_{n,\rho}} \hat{c}_n(\theta+\tlarhon). \label{eq:c*}
 \end{align}
Then by \eqref{eq:heuristic_3} and the definition of $U_n$, we obtain
\begin{equation}
P_{n}(p^\prime \theta_{n} \in CI_{n}) \ge P_{n}\left(U_{n}(\theta_{n},\hat{c}_{n,\rho})\neq \emptyset\right).	\label{eq:cover3a}
\end{equation}
 By passing to a further subsequence, we may assume that
\begin{align}
	D_{P_{n}}(\theta_{n})\to D, 
	\label{eq:val3_0}
\end{align}
for some $J\times d$ matrix $D$ such that $\|D\|\le M$ and $\Omega_{P_n}\stackrel{u}{\to}\Omega$ for some correlation matrix $\Omega$. By Lemma 2 in \cite{Andrews_Guggenberger2009bET} and Assumption \ref{as:bcs1} (i), uniformly in $\lambda \in B^d$, $\mathbb G_{n}(\theta_n+\larhon)\stackrel{d}{\to} \HH$ for a normal random vector with the correlation matrix $\Omega$. 
By Lemma \ref{lem:val},
\begin{align}
\liminf_{n \to \infty}&P_{n}\left(U_{n}(\theta_{n},\hat{c}_{n,\rho})\neq \emptyset\right)\ge 1-\alpha \label{eq:cover4}.
\end{align}
The conclusion of the theorem then follows from \eqref{eq:cover2}, \eqref{eq:cover3}, \eqref{eq:cover3a}, and \eqref{eq:cover4}.
\qed
\bigskip

\noindent\textbf{Proof of Theorem \ref{thm:validity}-\ref{cor:thm:validity}.}

The result follows immediately from the same steps as in the proof of Theorem \ref{thm:validity}-\ref{thm:validity:basic}.
\qed
\bigskip

\noindent\textbf{Proof of Theorem \ref{thm:validity}-\ref{thm:nonlinear}}

The argument of proof is the same as for Theorem \ref{thm:validity}-\ref{thm:validity:basic}, with the following modification.
Take $(P_n,\theta_n)$ as defined following equation \eqref{eq:pi_def}. Then $f(\theta_n)$ is covered when
\begin{align*}
	 &~\begin{Bmatrix} \inf f(\vartheta)  &   \\ \text{s.t. } \, \vartheta \in \Theta, & \frac{\sqrt n \bar{m}_{n,j}(\vartheta)}{\hat\sigma_{n,j}(\vartheta)}\le  \hat{c}^f_n(\vartheta),  \forall j \end{Bmatrix} \le  f(\theta_n) \notag \le  \begin{Bmatrix} \sup f(\vartheta)  &   \\ \text{s.t. } \, \vartheta \in \Theta, & \frac{\sqrt n \bar{m}_{n,j}(\vartheta)}{\hat\sigma_{n,j}(\vartheta)}\le  \hat{c}^f_n(\vartheta), \forall j \end{Bmatrix}\\
	 \Longleftrightarrow &~ \begin{Bmatrix} \inf_{\lambda} \nabla f(\tilde{\theta}_n)\lambda  &   \\ \text{s.t. } \lambda\in \frac{\sqrt n}{\rho }(\Theta-\theta_n), & \frac{\sqrt n\bar{m}_{n,j}(\theta_n+\larhon)}{\hat\sigma_{n,j}(\theta_n+\larhon)}\le  \hat{c}^f_n(\theta_n+\larhon), \forall j \end{Bmatrix} \leq 0\notag\\
	 &\hspace{0.5in}\leq\begin{Bmatrix} \sup_{\lambda} \nabla f(\tilde{\theta}_n)\lambda  &   \\ \text{s.t. } \lambda\in \frac{\sqrt n}{\rho }(\Theta-\theta_n), & \frac{\sqrt n\bar{m}_{n,j}(\theta_n+\larhon)}{\hat\sigma_{n,j}(\theta_n+\larhon)}\le  \hat{c}^f_n(\theta_n+\larhon), \forall j \end{Bmatrix}, 
	 \end{align*}
where we took a mean value expansion yielding 
\begin{align}
f(\theta_n+\tlarhon)=f(\theta_n)+\frac{\rho }{\sqrt{n}} \nabla f(\tilde{\theta}_n)\lambda,\label{eq:MVE_f}
\end{align}	 
for $\tilde{\theta}_n$ a mean value that lies componentwise between $\theta_n$ and $\theta_n+\larhon$, and we used that the sign of the last term in \eqref{eq:MVE_f} is the same as the sign of $\nabla f(\tilde{\theta}_n)\lambda$. With the objective function in \eqref{eq:MVE_f} so redefined, all expression in the proof of Theorem \ref{thm:validity}-\ref{thm:validity:basic} up to \eqref{eq:uj} continue to be valid.
We can then redefine the set $U_n(\theta_n,c)$ in \eqref{eq:set_U_NL} as

\begin{align*}
U_n(\theta_n,c)\equiv	\big\{\lambda\in B^d_{n,\rho}: \|\nabla f(\tilde{\theta}_n)\|^{-1}\nabla f(\tilde{\theta}_n)\lambda = 0 \cap u_{n,j,\theta_n}(\lambda)\le c, \: \forall j=1,\dots,J\big\}.
\end{align*}
Replace $p^\prime$ with $\|\nabla f(\tilde{\theta}_n)\|^{-1}\nabla f(\tilde{\theta}_n)$ in all expressions involving the set $U_n(\theta_n,\hat{c}_{n,\rho}^{f}(\theta_n))$, and replace $p^\prime$ with $\|\nabla f(\theta_n)^\prime\|^{-1}\nabla f(\theta_n^\prime)$ in all expressions for the sets $V^I_n(\theta_n^\prime,\hat{c}^f_n(\theta_n^\prime))$, and in all the almost sure representation counterparts of these sets. Observe that we can select a convergent subsequence from $\{\|\nabla f(\theta_n)^\prime\|^{-1}\nabla f(\theta_n^\prime)\}$ that converges to some $p$ in the unit sphere, so that the form of $\mathfrak W(c_{\pi^*})$ in \eqref{eq:Wdelta} is unchanged. This yields the result, noting that by the assumption $\Vert \nabla f(\tilde{\theta}_n)-\nabla f(\theta_n^\prime)\Vert = O_\cP(\rho /\sqrt{n})$
\qed
\color{black}

\subsubsection{Proof of Theorem \ref{thm:validity}-\ref{thm:validity:basic} with High Level Assumption \ref{ass:continuity_limit_cov} Replacing Assumption \ref{as:correlation}, and Dropping the $\rho$-Box Constraints Under Assumption \ref{ass:continuity_limit_cov_II}}
\label{app:proofs_alt_cont}
\begin{lemma}
\label{lem:validity_ass_cont}
Suppose that Assumption \ref{as:momP_AS}, \ref{as:GMS}, \ref{as:momP_KMS} and \ref{as:bcs1} hold.
\begin{enumerate}[label=(\Roman*)]
\item \label{lem:validity_cont_ass} Let also Assumption \ref{ass:continuity_limit_cov} hold. Let $0<\alpha < 1/2$. Then,
\begin{align*}
\liminf_{n\to\infty}\inf_{P\in\mathcal P}\inf_{\theta\in\Theta_I(P)}P(p'\theta\in CI_n)\ge 1-\alpha.
\end{align*}	
\item \label{lem:norho}
Let also Assumption \ref{ass:continuity_limit_cov_II} and either Assumption \ref{as:correlation} or \ref{ass:continuity_limit_cov} hold. Let $\hat{c}_n=\inf\{c\in\mathbb R_+:P^*(\{\Lambda_n^b (\theta ,+\infty ,c)\cap \{p^{\prime }\lambda =0\}\}\neq \emptyset)\ge 1-\alpha\}$, where $\Lambda_n^b$ is defined in equation \eqref{eq:Lambda_n} and $CI_n\equiv [-s(-p,\mathcal C_n(\hat{c}_n)),s(p,\mathcal C_n(\hat{c}_n))]$ with $s(q,\mathcal C_n(\hat{c}_n)),q\in \{p,-p\}$ defined in equation \eqref{eq:CI}. Then
\begin{align*}
\liminf_{n\to\infty}\inf_{P\in\mathcal P}\inf_{\theta\in\Theta_I(P)}P(p'\theta\in CI_n)\ge 1-\alpha.
\end{align*}
\end{enumerate} 	
\end{lemma}
\begin{proof} We establish each part of the Lemma separately.

 \textbf{Part \ref{lem:validity_cont_ass}.}
This part of the lemma replaces Assumptions \ref{as:correlation} with Assumption \ref{ass:continuity_limit_cov}. Hence we establish the result by showing that all claims that were made under Assumption \ref{as:correlation} remain valid under Assumption \ref{ass:continuity_limit_cov}. We proceed in steps.\smallskip

\noindent \underline{Step 1. Revisiting the proof of Lemma \ref{lem:empt}, equation \eqref{eq:Wdelta_lb_c}.}

Let $\cJ^*$ be as defined in \eqref{eq:Jstar}. If $\cJ^*=\emptyset$ we immediately have that Lemma \ref{lem:empt} continues to hold. Hence we assume that $\cJ^*\neq \emptyset$. To keep the notation simple, below we argue as if all $j=1,\dots,J$ belong to $\cJ^*$.

Consider the case that $c_{\pi^*}>0$. For some $c_{\pi^*}>\delta>0$, let
	\begin{align}
\mathfrak{W}(c-\delta)
\equiv \big\{ \lambda\in \mathfrak B^d_\rho: p^\prime \lambda =0 &\cap \mathfrak{w}_{j}(\lambda)\le c-\delta , \: \forall j=1,\dots,J\big\}, \label{eq:Wdelta_ass_cont}
	\end{align}
	where we emphasize that the set $\mathfrak{W}(c-\delta)$ is obtained by a $\delta$-contraction of all constraints, including those indexed by $j=J_1+1,\dots,J$. By Assumption \ref{ass:continuity_limit_cov}, for any $\eta>0$ there exists a $\delta$ such that
	\begin{align*}
	\eta \ge \left\vert \mathrm{Pr}\left(\mathfrak{W}(c_{\pi^*})\neq \emptyset \right)-\mathrm{Pr}\left(\mathfrak{W}(c_{\pi^*}-\delta)\neq \emptyset \right)\right\vert = \mathrm{Pr}\left(\{\mathfrak{W}(c_{\pi^*})\neq \emptyset\} \cap  \{\mathfrak{W}(c_{\pi^*}-\delta)=\emptyset\}\right),\\
	\eta \ge \left\vert \mathrm{Pr}\left(\mathfrak{W}(c_{\pi^*}+\delta)\neq \emptyset \right)-\mathrm{Pr}\left(\mathfrak{W}(c_{\pi^*})\neq \emptyset \right)\right\vert = \mathrm{Pr}\left(\{\mathfrak{W}(c_{\pi^*}+\delta)\neq \emptyset\} \cap  \{\mathfrak{W}(c_{\pi^*})=\emptyset\}\right).
	\end{align*}
	The result follows.
\smallskip
	
	\noindent \underline{Step 2. Revisiting the proof of Lemma \ref{lem:UnonEmpty}.}
	
	Case 1 of Lemma \ref{lem:UnonEmpty} is unaltered. Case 2 of Lemma \ref{lem:UnonEmpty} follows from the same argument as used in Case 1 of Lemma \ref{lem:UnonEmpty}, because under Assumption \ref{ass:continuity_limit_cov} as shown in step 1 of this proof all inequalities are tightened. In Case 3 of Lemma \ref{lem:UnonEmpty} the result in \eqref{eq:cover4} holds automatically by Assumption \ref{ass:continuity_limit_cov}-(ii). (As a remark, Lemmas \ref{lem:exist_sol}-\ref{lem:cbd} are no longer needed to establish Lemma \ref{lem:UnonEmpty}.)\smallskip
	
	\noindent \underline{Step 3. Revisiting the proof of Lemma \ref{lem:cv_convergence}.}
	Under Assumption \ref{ass:continuity_limit_cov} we do not need to merge paired inequalities. Hence, part (iii) of Lemma \ref{lem:cv_convergence} holds automatically because $\varphi^*_j(\xi) \le \varphi_j(\xi)$ for any $j$ and $\xi$. We are left to establish parts (i) and (ii) of Lemma \ref{lem:cv_convergence}. These follow immediately, because Lemma \ref{lem:empt} remains valid as shown in step 1 and by Assumption \ref{ass:continuity_limit_cov}, $\Pr(\mathfrak{W}(c)\neq \emptyset)$ is strictly increasing at $c=c_{\pi^*}$ if $c_{\pi^*}>0$. (As a remark, Lemma \ref{lem:pair} is no longer needed to establish Lemma \ref{lem:cv_convergence}.)\smallskip 

In summary, the desired result follows by applying Lemma \ref{lem:val} in the proof of Theorem \ref{thm:validity}-\ref{thm:validity:basic} as Lemmas \ref{lem:UnonEmpty}, \ref{lem:cv_convergence} and \ref{lem:empt} remain valid,  Lemmas \ref{lem:res}, \ref{lem:Jstar}, \ref{lem:eta_conv} and the Lemmas in Appendix \ref{app:asrep} are unaffected, and Lemmas \ref{lem:exist_sol}, \ref{lem:cbd}, \ref{lem:pair} are no longer needed.

\textbf{Part \ref{lem:norho}.} This is established by adapting the proof of Theorem \ref{thm:validity}-\ref{thm:validity:basic} as follows:

In the main proof, we pass to an a.s. representation early on, so that $\mathfrak{W}$ realizes jointly with other random variables (we denote almost sure representations adding a superscript ``$^*$" on the original variable). At the same time, we entirely drop $\rho$. This means that algebraic expressions, e.g. in the main proof, simplify as if $\rho=1$, but it also removes any constraints along the lines of $\lambda \in B^d_{n,\rho}$ in equation \eqref{eq:cover_part2}. Indeed, \eqref{eq:cover_part2} is replaced by: 
\begin{align*}
\dots \Leftarrow	 &\begin{Bmatrix} \inf_{\lambda} p^\prime\lambda     \\ \text{s.t. } \lambda\in \bar{\mathfrak{W}}^*(\bar{c}),  \\ \{\mathbb G^*_{n,j}(\theta_n+\lambda/\sqrt{n})+D_{P_n,j}(\bar\theta_n)\lambda+\sqrt{n} \gamma_{1,P_n,j}(\theta_n)\}(1+\eta_{n,j}(\theta_n+\lambda/\sqrt{n}))\le  \hat{c}_n(\theta_n+\lambda/\sqrt{n}), \forall j    \end{Bmatrix} \leq 0 \\
	 &\hspace{0.5in}\leq\begin{Bmatrix} \sup_{\lambda} p'\lambda    \\ \text{s.t. } \lambda\in \bar{\mathfrak{W}}^*(\bar{c}),  \\ \{\mathbb G^*_{n,j}(\theta_n+\lambda/\sqrt{n})+ D_{P_n,j}(\bar\theta_n)\lambda+\sqrt{n} \gamma_{1,P_n,j}(\theta_n)\}(1+\eta_{n,j}(\theta_n+\lambda/\sqrt{n}))\le  \hat{c}_n(\theta_n+\lambda/\sqrt{n}), \forall j    \end{Bmatrix},
	 \end{align*}
yielding a new definition of the set $U_n^*$ as
\begin{align*}
U_n^*(\theta_n,c)\equiv	\big\{\lambda\in \bar{\mathfrak{W}}^*(\bar{c}): p^\prime \lambda = 0 \cap u^*_{n,j,\theta_n}(\lambda)\le c, \: \forall j=1,\dots,J\big\}.
\end{align*}
Subsequent uses of $\rho$ in the main proof use that $\Vert \lambda \Vert \leq \sqrt{d}\rho=O_\cP(1)$. For example, consider the argument following equation \eqref{eq:val6} or the argument just preceding equation \eqref{eq:cover4}, and so on. All these continue to go through because $\bar{\mathfrak{W}}^*(\bar{c})=O(1)$ by assumption.

Similar uses occur in Lemma \ref{lem:val}. The next major adaptation is that in \eqref{eq:set_W_Lpop_star} and \eqref{eq:val3}: we again drop $\rho$ but nominally introduce the constraint that $\lambda \in \bar{\mathfrak{W}}^*(\bar{c})$. However, for $c \leq \bar{c}$, this condition cannot constrain $\mathfrak{W}^*(c)$, and so we can as well drop it: The modified $\mathfrak{W}^*(c)$ equals $\bar{\mathfrak{W}}^*(c)$.

Next we argue that Lemma \ref{lem:exist_sol} continues to hold, now claimed for $\bar{\mathfrak{W}}^*$. To verify that this is the case, replace $B^d$ with $\bar{\mathfrak{W}}(\bar{c})$ throughout in Lemma \ref{lem:exist_sol}. This requires straightforward adaptation of algebra as $\bar{\mathfrak{W}}(\bar{c})$ is only stochastically and not deterministically bounded.

Finally, in Lemma \ref{lem:cv_convergence} we remove the $\rho$-constraint from $V_n^b$ and $V_n^I$ without replacement, and note that the lemma is now claimed for $\theta_n' \in \theta +\|\bar{\mathfrak{W}}(\bar c)\|_H /\sqrt{n} B^d$. Recall that in the lemma the a.s. representation of a set $A$ is denoted by $\tilde A$, and with some abuse of notation let the a.s. representation of $\bar{\mathfrak{W}}$ be denoted $\widetilde{ \bar{\mathfrak{W}}}$. Now we compare $\tilde V_n^b$ and $\tilde V_n^I$ with $\widetilde{ \bar{\mathfrak{W}}}$. To ensure that $\lambda$ is uniformly stochastically bounded in expressions like \eqref{eq:use-of-lambda}, we verify that the modified $\tilde V_n^b$ and $\tilde V_n^I$ inherit the property in Assumption \ref{ass:continuity_limit_cov_II}.
To see this, fix any unit vector $t \perp p$ and notice that any $t=\lambda /\| \lambda \|$ for $\lambda \in \widetilde{ \bar{\mathfrak{W}}}(c)$ or for $\lambda \in \tilde V_n^b(\theta_n',c)$ or for $\lambda \in \tilde V_n^I(\theta_n',c)$, $0<c\le \bar c$, satisfies this condition. By Assumption \ref{ass:continuity_limit_cov_II} and the Cauchy-Schwarz inequality, $\max_{\lambda \in \widetilde{ \bar{\mathfrak{W}}}(c)}t'\lambda=O(1)$ for any $c \le \bar{c}$. Since the value of this program is necessarily attained by a basic solution whose associated gradients span $t$, it must be the case that such solution is itself $O(1)$. 
Formally, let $C$ be the index set characterizing the solution, $\HH^C_i$ be the vector of realizations $\HH^j_i$ corresponding to $j \in C$, and $K^C(\theta_n')$ the matrix that stacks the corresponding gradients; then $(K^C(\theta_n'))^{-1}(\bar{c}\bm{1}-\HH^C_i)=O(1)$. By Lemma \ref{lem:exist_sol} and the fact that $\hat D_n(\theta_n')\stackrel{P}{\to} D$ by Assumption \ref{as:momP_KMS}, we then also have that $(\hat{K}^C(\theta_n'))^{-1}(\bar{c}\bm{1}-\mathbb{G}^b_{n,j})=O_\cP(1)$, and so for $c \leq \bar{c}$, $V^b$ is bounded in this same direction. It follows that, by similar reasoning to the preceding paragraph, the comparison between $V_n^I(\theta_n',c)$ and $\bar{\mathfrak{W}}(c)$ in Lemma \ref{lem:cv_convergence} goes through.
\end{proof}

\subsubsection{An Extension of Theorem \ref{thm:validity}}
\label{cor:math:proj}
In this subsection, we establish that, under the assumptions of Theorem \ref{thm:validity}, we actually have
\begin{equation}
\liminf_{n\to\infty}\inf_{P\in\mathcal P}	\inf_{\theta \in \Theta_I(P)} P(p'\theta \in \{p'\vartheta:\vartheta \in \cC_n(\hat{c}_n)\}) \geq 1-\alpha.
\end{equation} 
In words, the mathematical projection of $ \cC_n(\hat{c}_n)$, which will asymptotically pick up gaps in the projection of $\Theta_I$, is a uniformly asymptotically valid confidence region. This strengthens Theorem \ref{thm:validity} because $\{p'\vartheta:\vartheta \in \cC_n(\hat{c}_n)\} \subseteq CI_n$.

To prove this extension, we modify the proof of Theorem \ref{thm:validity} after \eqref{eq:def_pi_star} as follows: The projection of $\theta_n$ is covered when
\begin{align}
& ~\exists \vartheta \in \Theta: p'\vartheta=p'\theta_n, \tfrac{\sqrt n \bar{m}_{n,j}(\vartheta)}{\hat\sigma_{n,j}(\vartheta)}\le  \hat{c}_n(\vartheta),  \forall j \\
\Longleftrightarrow &~\exists \lambda \in \tfrac{\sqrt n}{\rho }(\Theta-\theta_n): p'\lambda=0, \tfrac{\sqrt n \bar{m}_{n,j}(\theta_n+\larhon)}{\hat\sigma_{n,j}(\theta_n+\larhon)}\le  \hat{c}_n(\theta_n+\tlarhon),  \forall j \\
\Longleftrightarrow &~\exists \lambda \in \tfrac{\sqrt n}{\rho }(\Theta-\theta_n): \\
 & ~ p'\lambda=0, \bigl(\mathbb G_{n,j}(\theta_n+\tlarhon)+\rho  D_{P_n,j}(\bar\theta_n)\lambda+\sqrt{n} \gamma_{1,P_n,j}(\theta_n)\bigr)\bigl(1+\eta_{n,j}(\theta_n+\tlarhon)\bigr)\le  \hat{c}_n(\theta_n+\tlarhon), \forall j \label{eq:coverage_ext}  
	 \end{align}
	 where the last line corresponds to \eqref{eq:cover_part1} and intermediate steps that are exactly analogous to the previous proof were skipped. Subsequent proof steps go through as before until, comparing \eqref{eq:set_U_NL} to \eqref{eq:coverage_ext}, we find (compare to \eqref{eq:cover3a}, noting the change from inequality to equality)
	 \begin{equation}
P_n\bigl(p'\theta_n \in \{p'\vartheta:\vartheta \in \cC_n(\hat{c}_n)\}\bigr)  = P_{n}\bigl(U_{n}(\theta_{n},\hat{c}_{n,\rho})\neq \emptyset\bigr).	\label{eq:cover3a_ext}
\end{equation}
The proof then continues as before.

\section{Auxiliary Lemmas}
\label{app:Lemma}

\subsection{Lemmas Used to Prove Theorem \ref{thm:validity}}
Throughout this Appendix, we let $(P_n,\theta_n)\in \{(P,\theta):P \in \cP,\theta \in \Theta_I(P)\}$ be a subsequence as defined in the proof of Theorem \ref{thm:validity}-\ref{thm:validity:basic}. That is, along $(P_n,\theta_n)$, one has
\begin{align}
\kappa_n^{-1}\sqrt{n}\gamma_{1,P_n,j}(\theta_n) &\to \pi_{1j} \in \R_{[-\infty]},~j=1,\dots,J,\label{eq:subseq1}\\
\Omega_{P_n} &\uni \Omega,\label{eq:subseq2}\\
D_{P_n}(\theta_n) &\to D.\label{eq:subseq3}
\end{align}
Fix $c \ge 0$.
For each $\lambda\in\R^d$ and $\theta \in \thetnprime$, let
\begin{align}
\mathfrak{w}_{j}(\lambda)&\equiv \HH_{j}+\rho D_{j}\lambda+\pi^*_{1,j},\label{eq:wj}
\end{align}
where $\pi_{1,j}^*$ is defined in \eqref{eq:def_pi_star} and we used Lemma \ref{lem:Jstar}. Under Assumption \ref{as:correlation}-\ref{as:correlation_pair} if 
\begin{align}
\pi_{1,j}^*=0=\pi_{1,j+R_1}^*,\label{eq:paired_bind}
\end{align}
we replace the constraints
\begin{align}
 \HH_{j}+\rho D_{j}\lambda &\le c, \label{eq:paired_bind_R2L}\\
\HH_{j+R_1}+\rho D_{j+R_1}\lambda &\le c, \label{eq:paired_bind_R2U}	
\end{align}
with 
\begin{align}
\mu_j(\theta)\{ \HH_{j}+\rho D_{j}\lambda\} - \mu_{j+R_1}(\theta)\{\HH_{j+R_1}+\rho D_{j+R_1}\lambda\} &\le c,\label{eq:paired_bind_R4L}\\
-\mu_j(\theta)\{ \HH_{j}+\rho D_{j}\lambda\} + \mu_{j+R_1}(\theta)\{\HH_{j+R_1}+\rho D_{j+R_1}\lambda\} &\le c,\label{eq:paired_bind_R4U}	
\end{align}
where
\begin{align}
\mu_j(\theta) &=\left\{\begin{matrix}
 1 & \mathrm{if}~ \gamma_{1,P_n,j}(\theta)=0=\gamma_{1,P_n,j+R_1}(\theta),\\
\frac{\gamma_{1,P_n,j+R_1}(\theta)}{\gamma_{1,P_n,j+R_1}(\theta)+\gamma_{1,P_n,j}(\theta)} & \mathrm{otherwise},
\end{matrix}\right. \label{eq:weight_j}\\
\mu_{j+R_1}(\theta)&=\left\{\begin{matrix}
 0 & \mathrm{if}~ \gamma_{1,P_n,j}(\theta)=0=\gamma_{1,P_n,j+R_1}(\theta),\\
\frac{\gamma_{1,P_n,j}(\theta)}{\gamma_{1,P_n,j+R_1}(\theta)+\gamma_{1,P_n,j}(\theta)} & \mathrm{otherwise},
\end{matrix}\right. \label{eq:weight_j_pl_J11}
\end{align}

When Assumption \ref{as:correlation}-\ref{as:correlation_pair} is invoked with hard-threshold GMS, replace constraints $j$ and $j+R_1$ in the definition of $\Lambda_n^b (\theta'_n, \rho  ,c),\theta'_n\in \thetnprime$ in equation \eqref{eq:Lambda_n} as described on p.\pageref{verbal_def:paired_ineq} of the paper; when it is invoked with a GMS function $\varphi$ that is smooth in its argument, replace them, respectively, with 
\begin{align}
&\hat\mu_{n,j}(\theta'_n)\{\mathbb{G}_{n,j}^{b}(\theta'_n)+\hat{D}
_{n,j}(\theta'_n)\lambda\} - \hat\mu_{n,j+R_1}(\theta'_n)\{\mathbb{G}_{n,j+R_1}^{b}(\theta'_n)+\hat{D}
_{n,j+R_1}(\theta'_n)\lambda \} + \varphi_j(\hat{\xi}_{n,j}(\theta'_n))\le c, \label{eq:paired_bind_Lboot_smooth}\\
-&\hat\mu_{n,j}(\theta'_n)\{\mathbb{G}_{n,j}^{b}(\theta'_n)+\hat{D}
_{n,j}(\theta'_n)\lambda\} + \hat\mu_{n,j+R_1}(\theta'_n)\{\mathbb{G}_{n,j+R_1}^{b}(\theta'_n)+\hat{D}
_{n,j+R_1}(\theta'_n)\lambda \} +\varphi_{j+R_1}(\hat{\xi}_{n,j+R_1}(\theta'_n))\le c, \label{eq:paired_bind_Uboot_smooth}
\end{align}
where 
\begin{align}
\hat\mu_{n,j+R_1}(\theta'_n) &=\min\left\{\max\left(0,\frac{\frac{\bar m_{n,j}(\theta'_n)}{\hat\sigma_{n,j}(\theta'_n)}}{\frac{\bar m_{n,j+R_1}(\theta'_n)}{\hat\sigma_{n,j+R_1}(\theta'_n)}+\frac{\bar m_{n,j}(\theta'_n)}{\hat\sigma_{n,j}(\theta'_n)}}\right),1\right\},\label{eq:mu_hat_j}\\
\hat\mu_{n,j}(\theta'_n)&=
1-\hat\mu_{n,j+R_1}(\theta'_n). \label{eq:mu_hat_j_pl_J11}
\end{align}
\color{black}

Let $\mathfrak B^d_\rho=\lim_{n\to\infty}B^d_{n,\rho}.$
Let the intersection of $\{\lambda\in \mathfrak B^d_\rho:~p^\prime\lambda=0\}$ with the level set associated with the so defined function  $\mathfrak w_{j}(\lambda)$ be
\begin{align}
\mathfrak W(c)\equiv \big\{ \lambda\in \mathfrak B^d_\rho: p^\prime \lambda = 0 \cap \mathfrak w_{j}(\lambda)\le c, \: \forall j=1,\dots,J\big\} \label{eq:set_W_Lpop}.
\end{align}
Due to the substitutions in equations \eqref{eq:paired_bind_R2L}-\eqref{eq:paired_bind_R4U}, the paired inequalities (i.e., inequalities for which \eqref{eq:paired_bind} holds under Assumption \ref{as:correlation}-\ref{as:correlation_pair}) are now genuine equalities relaxed by $c$. With some abuse of notation, we index them among the $j=J_1+1,\dots,J$.
With that convention, for given $\delta \in \R$, define
	\begin{align}
\mathfrak W^{\delta}(c)
\equiv \big\{ \lambda\in \mathfrak B^d_\rho: p^\prime \lambda =0 &\cap \mathfrak w_{j}(\lambda)\le c+\delta , \: \forall j=1,\dots,J_1, \notag\\
&\cap \mathfrak w_{j}(\lambda)\le c , \: \forall j=J_1+1,\dots,J\big\}. \label{eq:Wdelta}
	\end{align}
Define the $(J+2d+2) \times d$ matrix
\begin{align}
K_P(\theta,\rho ) \equiv 
\begin{bmatrix}
[\rho D_{P,j}(\theta)]_{j=1}^{J_1+J_2}\\
[-\rho D_{P,j-J_2}(\theta)]_{j=J_1+J_2+1}^J\\
I_d\\
-I_d\\
p^\prime\\
-p^\prime
\end{bmatrix}. 
\end{align}
Given a square matrix $A$, we let $\eig(A)$ denote its smallest eigenvalue. 
 In all Lemmas below, we assume $\alpha<1/2$.

\bigskip
\begin{lemma}\label{lem:val}
Let Assumptions \ref{as:momP_AS}, \ref{as:GMS}, \ref{as:correlation}, \ref{as:momP_KMS}, and \ref{as:bcs1} hold. Let $\{P_n,\theta_n\}$ be a sequence such that $P_n\in\mathcal P$ and $\theta_n \in \Theta_I(P_n)$ for all $n$ and $\kappa_n^{-1}\sqrt{n}\gamma_{1,P_n,j}(\theta_n) \to \pi_{1j} \in \R_{[-\infty]},~j=1,\dots,J,$
$\Omega_{P_n} \uni \Omega,$ and
$D_{P_n}(\theta_n) \to D$.
Then,
\begin{align}
\liminf_{n \to \infty} P_n\left( U_n(\theta_n,\hat{c}_{n,\rho})\neq \emptyset \right) \ge 1-\alpha.
\label{eq:validity}
\end{align}
\end{lemma}

\begin{proof}
We consider a subsequence along which $ \liminf_{n\to\infty}P_n(U_{n}(\theta_{n},\hat{c}_{n,\rho}\ne\emptyset)$ is achieved as a limit. For notational simplicity, we use $\{n\}$ for this subsequence below.
	
Below, we construct a sequence of critical values  such that
\begin{align}
	\hat c_n(\theta_n')\ge c^{I}_n(\theta_n')+ o_{P_n}(1),\label{eq:chat_cI}
\end{align} 
and $c^{I}_n(\theta_n')\stackrel{P_n}{\to} c_{\pi^*}$ for any $\theta'_n\in \thetnprime$. The construction is as follows. When $c_{\pi^*}=0$, let $c^{I}_n(\theta_n')=0$ for all $\theta'_n\in \thetnprime$, and hence $c^{I}_n(\theta_n')\stackrel{P_n}{\to} c_{\pi^*}$. If $c_{\pi^*}>0$, let $c^I_n(\theta_n)\equiv \inf\{c\in\R_+:P^*_n(V^I_n(\theta_n,c))\ge 1-\alpha\}$, where $V^I_n$ is defined as in Lemma \ref{lem:cv_convergence}. By Lemma \ref{lem:cv_convergence} (iii), this critical value sequence satisfies \eqref{eq:chat_cI} with probability approaching 1.
Further, by Lemma \ref{lem:cv_convergence} (ii), $c^{I}_n(\theta_n')\stackrel{P_n}{\to} c_{\pi^*}$ for any $\theta'_n\in \thetnprime$.

For each $\theta\in\Theta$,  let
\begin{align}
	c^{I}_{n,\rho}(\theta)\equiv\inf_{\lambda\in B^d_{n,\rho}}c^I_n(\theta+\tlarhon).
\end{align}
Since the $o_{P_n}(1)$ term in \eqref{eq:chat_cI} does not affect the argument below, we redefine $c^{I}_{n,\rho}(\theta_n)$ as $c^{I}_{n,\rho}(\theta_n)+o_{P_n}(1)$.
By \eqref{eq:chat_cI} and simple addition and subtraction,
\begin{align}
P_{n}\Big(U_{n}(\theta_{n},\hat{c}_{n,\rho}(\theta_{n}))\neq \emptyset\Big) &
\ge P_{n}\Big(U_{n}(\theta_{n},c^{I}_{n,\rho}(\theta_{n}))\neq \emptyset\Big)\notag\\
&=\Pr(\mathfrak{W}(c_{\pi^*})\ne\emptyset)+\Big[ P_n\Big(U_{n}(\theta_{n},c^{I}_{n,\rho}(\theta_{n}))\ne\emptyset\Big)-	\Pr\Big(\mathfrak{W}(c_{\pi^*})\ne\emptyset\Big)\Big].\label{eq:val_3_1}
\end{align}
As previously argued, $\mathbb G_{n}(\theta_n+\larhon)\stackrel{d}{\to} \HH$. Moreover, by Lemma~\ref{lem:eta_conv}, 
$\sup_{\theta\in\Theta}\|\eta_{n}(\theta)\|\stackrel{p}{\to} 0$ uniformly in $\cP$, and by Lemma \ref{lem:cv_convergence}, $c^I_{n,\rho}(\theta_n)\stackrel{p}{\to}c_{\pi^*}$.
 Therefore, uniformly in $\lambda\in B^d$, the sequence $\{(\mathbb G_{n}(\theta_n+\larhon),\eta_{n}(\theta_{n}+\larhon),c^{I}_{n,\rho}(\theta_n))\}$ satisfies 
\begin{align}
	\bigl(\mathbb G_{n}(\theta_n+\tlarhon),	\eta_{n}(\theta_{n}+\tlarhon),c_{n,\rho}^{I}(\theta_n)\bigr)\stackrel{d}{\to} (\HH,0,c_{\pi^*}).
\end{align}
In what follows, using Lemma 1.10.4 in \cite{Vaart_Wellner2000aBK} we take $(\G^*_{n}(\theta_n+\larhon),\eta^*_{n}, c^*_{n})$ to be the almost sure representation of $(\mathbb G_{n}(\theta_n+\larhon),\eta_{n}(\theta_{n}+\larhon),c_{n,\rho}^{I}(\theta_n))$ defined on some probability space $(\Omega,\mathcal F,\mathbf P)$ such that
 $(\mathbb G^*_{n}(\theta_n+\larhon),\eta^*_{n}, c^*_{n})\stackrel{a.s.}{\to}(\HH^*,0, c_{\pi^*})$, where $\HH^*\stackrel{d}{=}\HH$.

For each $\lambda\in\R^d$, we define analogs to the quantities in \eqref{eq:uj} and \eqref{eq:wj} as
\begin{align}
u^*_{n,j,\theta_n}(\lambda)&\equiv	\bigl\{\mathbb G^*_{n,j}(\theta_n+\tlarhon)+\rho D_{P_{n},j}(\bar\theta_n)\lambda+\pi^*_{1,j}\bigr\}(1+\eta^*_{n,j}),\label{eq:ujstar}\\
\mathfrak{w}^*_{j}(\lambda)&\equiv \HH^*_{j}+\rho D_{j}\lambda+\pi^*_{1,j}.\label{eq:wjstar}
\end{align}
 where we used that by Lemma \ref{lem:Jstar}, $\kappa_n^{-1}\sqrt{n} \gamma_{1,P,j}(\theta_n)-\kappa_n^{-1}\sqrt{n} \gamma_{1,P,j}(\theta_n^\prime)=o(1)$ uniformly over $\theta_n^\prime \in \thetnprime$ and therefore $\pi_{1,j}^*$ is constant over this neighborhood, and we applied a similar replacement as described in equations \eqref{eq:paired_bind_R2L}-\eqref{eq:paired_bind_R4U} for the case that $\pi_{1,j}^*=0=\pi_{1,j+R_1}^*$.
Similarly, we define analogs to the sets in \eqref{eq:set_U_NL} and \eqref{eq:set_W_Lpop} as
\begin{align}
U^*_{n}(\theta_{n},c_{n}^{*})&\equiv	\big\{\lambda\in B^d_{n,\rho}: p^\prime \lambda = 0 \cap u^*_{n,j,\theta_n}(\lambda)\le c_{n}^{*}, \: \forall j=1,\dots,J\big\}, \label{eq:set_U_NL_star}\\
\mathfrak W^*(c_{\pi^*})&\equiv \big\{ \lambda\in \mathfrak B^d_\rho: p^\prime \lambda = 0 \cap \mathfrak{w}^*_{j}(\lambda)\le c_{\pi^*},~\forall j=1,\dots,J\big\}.\label{eq:set_W_Lpop_star}
\end{align}
 It then follows that equation (\ref{eq:val_3_1}) can be rewritten as
\begin{align}
P_{n}\Big(U_{n}(\theta_{n},\hat{c}_{n,\rho}(\theta_{n}))\neq \emptyset\Big)\ge \mathbf P(\mathfrak{W}^*(c_{\pi^*})\ne\emptyset)+\Big[ \mathbf P\Big(U^*_{n}(\theta_{n},c^*_n)\ne\emptyset\Big)-	\mathbf P\Big(\mathfrak{W}^*(c_{\pi^*})\ne\emptyset\Big)\Big].\label{eq:val3}
\end{align} 
By the definition of $c_{\pi^*}$, we have $\mathbf P(\mathfrak{W}^*(c_{\pi^*})\ne\emptyset)\ge 1-\alpha$. Therefore, we are left to show that the second term on the right hand side of \eqref{eq:val3} tends to 0 as $n\to\infty$.

 Define 
 \begin{equation}
 \label{eq:Jstar}
\mathcal J^*\equiv\{j=1,\cdots,J:\pi^*_{1,j}=0\} .
 \end{equation}
\noindent \textbf{Case 1.} Suppose first that $\mathcal J^*= \emptyset$, which implies $J_2=0$ and $\pi^*_{1,j}=-\infty$ for all $j.$
Then we have
\begin{align}
U^*_{n}(\theta_{n},c^{*}_{n})=\{\lambda \in  B^d_{n,\rho}:p^\prime\lambda =0\},~~~	\mathfrak{W}^*(c_{\pi^*})=\{\lambda \in  \mathfrak B^d_\rho:p^\prime\lambda =0\},
\end{align}
with probability 1, and hence 
\begin{align}
	\mathbf{P}\Big( \{U^*_{n}(\theta_{n},c^{*}_{n})\neq \emptyset \} \cap \{ \mathfrak{W}^*(c_{\pi^*})\neq \emptyset\}  \Big)=1.\label{eq:val9}
\end{align}
This in turn implies that
\begin{align}
 \Big| \mathbf{P}\Big(U^*_{n}(\theta_{n},c^{*}_{n})\neq \emptyset \} \Big)
- \mathbf{P}\Big(\mathfrak{W}^*(c_{\pi^*})\neq \emptyset\} \Big)\Big|=0,\label{eq:val10}
\end{align}
where we used $|\mathbf P(A)-\mathbf P(B)|\le \mathbf P(A\Delta B)\le 1-\mathbf P(A\cap B)$ for any pair of events $A$ and $B.$ Hence, the term in the square brackets in \eqref{eq:val3} is 0.

\noindent \textbf{Case 2.} Now consider the case that $\cJ^*\neq \emptyset$. We show that the term in  the square brackets in \eqref{eq:val3} converges to 0.
To that end, note that for any events $A,B$,
\begin{align}
\Big\vert\mathbf{P}(A\neq \emptyset) - \mathbf{P}(B\neq \emptyset)  \Big\vert
\le \Big\vert &\mathbf{P}(\{A= \emptyset \} \cap \{ B \neq \emptyset\})+
\mathbf{P}(\{A\neq \emptyset \} \cap \{ B =\emptyset\})\Big\vert \label{eq:cover5}
\end{align}
Hence, we aim to establish that for $A=U^*_{n}(\theta_{n},c^{*}_{n})$, $B= \mathfrak{W}^*(c_{\pi^*})$, the right hand side of equation (\ref{eq:cover5}) converges to zero. But this is guaranteed by Lemma \ref{lem:UnonEmpty}. 
Therefore, the conclusion of the lemma follows.
\end{proof}

\begin{lemma}\label{lem:UnonEmpty}
Let Assumptions \ref{as:momP_AS}, \ref{as:GMS}, \ref{as:correlation}, \ref{as:momP_KMS}, and \ref{as:bcs1} hold. Let $(P_n,\theta_n)$ have the almost sure representations given in Lemma \ref{lem:val}, and let $\mathcal J^*$ be defined as in \eqref{eq:Jstar}. Assume that $\mathcal J^* \neq \emptyset$.
Then for any $\eta>0$, there exists  $N \in \N$ such that 
\begin{align}
&\mathbf{P}\Big( \{U^*_{n}(\theta_{n},c^{*}_{n})\neq\emptyset \} \cap \{ \mathfrak{W}^*(c_{\pi^*})= \emptyset \}  \Big)
\le \eta/2,\label{UW_nonempty1}\\
&\mathbf{P}\Big( \{U^*_{n}(\theta_{n},c^{*}_{n})= \emptyset \} \cap \{ \mathfrak{W}^*(c_{\pi^*})\neq \emptyset \} \Big)
\le \eta/2,\label{UW_nonempty2}
\end{align}
for all $n\ge N$, where the sets in the above expressions are defined in equations \eqref{eq:set_U_NL_star} and \eqref{eq:set_W_Lpop_star}.
\end{lemma}
\begin{proof}
We begin by observing that for $j \notin \cJ^*$, $\pi^*_{1,j}=-\infty$, and therefore the corresponding inequalities 
\begin{align*}
 \left(\mathbb G^*_{n,j}(\theta_n+\tlarhon)+\rho D_{P_{n},j}(\bar\theta_n)\lambda+\pi^*_{1,j}\right)(1+\eta^*_{n,j}) \le c_{n}^{*},\\  
 \HH^*_{j}+\rho D_{j}\lambda+\pi^*_{1,j}\le c_{\pi^*}
\end{align*}
are satisfied with probability approaching one by similar arguments as in \eqref{eq:pi_star_gamma_vanishes}. Hence, we can redefine the sets of interest as
\begin{align}
U^*_{n}(\theta_{n},c_{n}^{*})&\equiv	\big\{\lambda\in B^d_{n,\rho}: p^\prime \lambda = 0 \cap u^*_{n,j,\theta_n}(\lambda)\le c_{n}^{*}, \: \forall j \in \cJ^*\big\}, \label{eq:set_U_NL_star_new}\\
 \mathfrak W^*(c_{\pi^*})&\equiv \big\{ \lambda\in \mathfrak B^d_\rho: p^\prime \lambda = 0 \cap \mathfrak{w}^*_{j}(\lambda)\le c_{\pi^*},~\forall j\in\mathcal J^*\big\}.\label{eq:set_W_Lpop_star_new}
\end{align}
We first show \eqref{UW_nonempty1}.
For this, we start by defining the events
\begin{align}
A_{ n} \equiv \left\{\sup_{\lambda \in  B^d} \max_{j \in \cJ^*}\left\vert (u^*_{n,j,\theta_n}(\lambda)-c^*_n)-( \mathfrak{w}^*_{j}(\lambda)-c_{\pi^*}) \right\vert \ge\delta  \right\}\label{eq:int2-2_2}.
\end{align}
By Lemma \ref{lem:res}, using the assumption that $\cJ^* \neq \emptyset$, for any $\eta>0$ there exists $N \in \N$ such that
\begin{align}
\mathbf{P}(A_{n})< \eta/2,~\forall n \ge N.\label{eq:Aln}
\end{align}
Define the sets of $\lambda$s, $U_n^{*,+\delta}$ and $\mathfrak{W}^{*,+\delta}$ by relaxing the constraints shaping  $U_n^*$ and $\mathfrak{W}^{*}$ by $\delta$:
\begin{align}
U_n^{*,+\delta}(\theta_n,c) &\equiv \{\lambda\in  B^d_{n,\rho}: p^\prime \lambda =0 \cap u^*_{n,j,\theta_n}(\lambda)\le c+\delta,~j \in \cJ^*\},\\
\mathfrak W^{*,+\delta}(c)&\equiv \{\lambda\in  \mathfrak B^d_\rho: p^\prime \lambda =0 \cap \mathfrak{w}^*_{j}(\lambda)\le c+\delta,~j \in \cJ^*\}.
\end{align}
Compared to the set in equation \eqref{eq:Wdelta}, here we replace $u^*_{n,j,\theta_n}(\lambda)$ for $u_{n,j,\theta_n}(\lambda)$ and $\mathfrak{w}^*_{j}(\lambda)$ for $\mathfrak{w}_{j}(\lambda)$, we retain only constraints in $\cJ^*$, and we relax all such constraints by $\delta>0$ instead of relaxing only those in $\{1,\dots,J_1\}$.
Next, define the event $L_{ n}\equiv\{ U_n^{*}(\theta_n,c^*_n)\subset \mathfrak W^{*,+\delta}(c_{\pi^*})\}$ and note that $A^c_{n}\subseteq L_{n}$. 

We may then bound the left hand side of \eqref{UW_nonempty1} as
\begin{align}
\mathbf{P}\Big(\{ U^*_{n}(\theta_n,c^*_n)\ne \emptyset \} \cap \{\mathfrak W^{*}(c_{\pi^*})= \emptyset \}\Big)
&\le \mathbf{P}\Big( \{ U^{*}_{n}(\theta_n,c^*_n)\ne \emptyset \} \cap \{\mathfrak W^{*,+\delta}(c_{\pi^*})=\emptyset \}\Big)\notag\\
&+\mathbf{P}\Big(\{\mathfrak W^{*,+\delta}(c_{\pi^*})\ne \emptyset \}\cap \{\mathfrak W^{*}(c_{\pi^*})=\emptyset \}\Big), \label{eq:VWres1}
\end{align}
where we used $P(A\cap B)\le P(A\cap C)+P(B\cap C^c)$ for any events $A,B,$ and $C.$ The first term on the right hand side of \eqref{eq:VWres1} can further be bounded as
\begin{align}
\mathbf{P}\Big( \{ U^{*}_{n}(\theta_n,c^*_n)\ne \emptyset \} \cap \{\mathfrak W^{*,+\delta}(c_{\pi^*})=\emptyset \}\Big)
&\le \mathbf{P}\Big(\{U^{*}_{n}(\theta_n,c^*_n)\not\subseteq\mathfrak W^{*,+\delta}(c_{\pi^*}) \} \Big) \notag\\
&= \mathbf {P}(L_{n}^c)\le\mathbf{P}(A_{n})<\eta/2,~\forall n\ge N~, \label{eq:Ato0_1}
\end{align}
where the penultimate inequality follows from $A^c_{n}\subseteq L_{n}$ as argued above, and the last inequality follows from \eqref{eq:Aln}.
For the second term on the left hand side of \eqref{eq:VWres1}, by  Lemma \ref{lem:empt}, there exists $N'\in\mathbb N$ such that
\begin{align}
\mathbf{P}\Big(\{\mathfrak W^{*,+\delta}(c_{\pi^*})\ne \emptyset \}\cap \{\mathfrak W^{*}(c_{\pi^*})=\emptyset \}\Big) \le \eta/2, ~\forall n\ge N'.\label{eq:Ato0_2}
\end{align}
Hence,  \eqref{UW_nonempty1} follows from \eqref{eq:VWres1}, \eqref{eq:Ato0_1}, and \eqref{eq:Ato0_2}.

To establish \eqref{UW_nonempty2}, we distinguish three cases.

\noindent \textbf{Case 1.} Suppose first that $J_2 = 0$ (recalling that under Assumption \ref{as:correlation}-\ref{as:correlation_pair} this means that there is no $j=1,\dots,R_1$ such that $\pi_{1,j}^*=0=\pi_{1,j+R_1}^*$), and hence one has only moment inequalities. In this case, by \eqref{eq:set_U_NL_star_new} and \eqref{eq:set_W_Lpop_star_new}, one may write
\begin{align}
	U^*_n(\theta_n,c) &\equiv \big\{\lambda\in B^d_{n,\rho}: p^\prime \lambda = 0 \cap  u^*_{n,j,\theta_n}(\lambda) \le c,~j\in \cJ^*  \big\},\\
	\mathfrak W^{*,-\delta}(c) &\equiv \big\{\lambda\in \mathfrak B^d_\rho: p^\prime \lambda =0 \cap \mathfrak{w}^*_{j}(\lambda)\le c-\delta,~j\in \cJ^*  \big\},
	\label{eq:int2-6}
\end{align}
where $\mathfrak W^{*,-\delta}$, $\delta>0$, is obtained by tightening the inequality constraints shaping $\mathfrak W^{*}$. Define the event 
\begin{align}
	R_{2 n}\equiv\{\mathfrak W^{*,-\delta}(c_{\pi^*})\subset U^*_{n}(\theta_{n},c^{*}_{n})\},\label{eq:int2-7}
\end{align}
and note that $A_{n}^c \subseteq R_{2 n}$. The result in equation \eqref{UW_nonempty2} then follows by Lemma \ref{lem:empt} using again similar steps to \eqref{eq:VWres1}-\eqref{eq:Ato0_2}.

\noindent \textbf{Case 2.} Next suppose that $J_2\ge d$. In this case,  we define $\mathfrak W^{*,-\delta}$ to be the set obtained by tightening by $\delta$ the inequality constraints as well as each of the two opposing inequalities obtained from the equality constraints.
That is, 
\begin{align}
	\mathfrak W^{*,-\delta}(c_{\pi^*}) &\equiv \{\lambda\in  \mathfrak B^d_\rho: p^\prime \lambda =0 \cap \mathfrak{w}^*_{j}(\lambda)\le c-\delta,~j \in \cJ^*\},\label{eq:int2-8}
\end{align}
that is, the same set as in \eqref{eq:Wdelta_lb_c} with $\mathfrak{w}^*_{j}(\lambda)$ replacing $\mathfrak{w}_{j}(\lambda)$ and defining the set using only inequalities in $\cJ^*$.
Note that, by Lemma \ref{lem:cbd}, there exists $N \in \N$ such that for all $n \ge N$ $c^I_n(\theta)$ is bounded from below by some $\underline{c}>0$ with probability approaching one uniformly in $P \in \cP$ and $\theta \in \Theta_I(P)$. This ensures  $c_{\pi^*}$ is bounded from below by $\underline{c}>0$.
This in turn allows us to construct a non-empty tightened constraint set with probability approaching 1.
Namely, for $\delta<\underline{c}$,  $\mathfrak W^{*,-\delta}(c_{\pi^*})$ is nonempty with probability approaching 1 by Lemma \ref{lem:empt},  and hence its superset $\mathfrak W^{*}(c_{\pi^*})$  is also non-empty with probability approaching 1. However, note that $A_{n}^c \subseteq R_{2 n}$, where $R_{2 n}$ is in \eqref{eq:int2-7} now defined using the tightened constraint set $\mathfrak W^{*,-\delta}(c_{\pi^*}) $ being defined as in \eqref{eq:int2-8}, and therefore the same argument as in the previous case applies.
	
\noindent \textbf{Case 3.} Finally, suppose that $1\le J_2 <d$. Recall that, with probability 1 (under $\mathbf P$),
\begin{align}
c_{\pi^*}=\lim_{n \to \infty}c^*_{n},\label{eq:tilde_c_lim}
\end{align}
and note that by construction $c_{\pi^*}\ge0.$ 
Consider first the case that $c_{\pi^*}>0$. Then, by taking $\delta< c_{\pi^*}$, the argument in Case 2 applies.

Next consider the case that $c_{\pi^*}=0$.  
Observe that
\begin{align}
&~\mathbf{P}\Big( \{U^*_{n}(\theta_{n},c^{*}_{n})= \emptyset \} \cap \{ \mathfrak W^*(c_{\pi^*})\neq \emptyset \} \Big) \\
\le &~ \mathbf{P}\Big( \{U^*_{n}(\theta_{n},c^{*}_{n})= \emptyset \} \cap \{ \mathfrak W^{*,-\delta}(0)\neq \emptyset \} \Big) +\mathbf{P}\Big(\{\mathfrak W^{*,-\delta}(0)= \emptyset \} \cap \{ \mathfrak W^*(0)\neq \emptyset \} \Big),
\label{eq:W_lim_U1}
\end{align}
with $\mathfrak W^{*,-\delta}(0)$ defined as in \eqref{eq:Wdelta} with $c=0$ and with $\mathfrak{w}^*_{j}(\lambda)$ replacing $\mathfrak w_{j}(\lambda)$.
By Lemma \ref{lem:empt}, for any $\eta>0$ there exists $\delta>0$ and $N \in \N$ such that
\begin{align}
\mathbf{P}\Big(\{\mathfrak W^{*,-\delta}(0)= \emptyset \} \cap \{\mathfrak W^*(0)\neq \emptyset \} \Big)<\eta/3 \text{  for all } n\ge N.\label{eq:final2}
\end{align}
Therefore, the second term on the right hand side of \eqref{eq:W_lim_U1} can be made arbitrarily small.

We now consider the first term  on the right hand side of \eqref{eq:W_lim_U1}. 
Let $g$ be a $J+2d+2$ vector with
 \begin{align}
 g_{j} =\left\{ \begin{array}{lll}
   - \HH_j, &  j\in \mathcal J^*,\\
  0, & j\in\{1,\cdots,J\}\setminus \mathcal J^*,\\
  1, &  j=J+1,\dots,J+2d,\\
 0, & j=J+2d+1,J+2d+2,
 	\end{array} 
 	\right. \label{eq:g}
 \end{align}
where we used that $\pi_{1,j}^*=0$ for $j\in\cJ^*$ and where the last assignment is without loss of generality because of the considerations leading to the sets in \eqref{eq:set_U_NL_star_new}-\eqref{eq:set_W_Lpop_star_new}. 

For a given set $C \subset \{1,\dots,J+2d+2\}$, let the vector $g^C$ collect the entries of $g^C$ corresponding to indices in $C$. 
Let \begin{align}
K  \equiv 
\begin{bmatrix}
[\rho D_{j}]_{j=1}^{J_1+J_2}\\
[-\rho D_{j-J_2}]_{j=J_1+J_2+1}^J\\
I_d\\
-I_d\\
p^\prime\\
-p^\prime
\end{bmatrix}. 
\end{align}
Let the matrix $K^C$ collect the rows of $K$ corresponding to indices in $C$.

Let $\widetilde{\cC}$ collect all size $d$ subsets $C$ of $\{1,...,J+2d+2\}$ ordered
lexicographically by their smallest, then second smallest, etc. elements.
Let the random variable $\cC$  equal the first element of $\widetilde{\cC}$ s.t. $\det K^C\neq 0$ and $\lambda^{C}=( K^C)
^{-1}g^C\in \mathfrak W^{* ,-\delta }(0)$ if such an
element exists; else, let $\cC=\{J+1,...,J+d\}$ and $\lambda^{C}= \mathbf{1}_d$, where $\mathbf{1}_d$ denotes a $d$ vector with each entry equal to $1$. Recall that $\mathfrak W^{* ,-\delta }(0)$ is a (possibly empty) measurable random polyhedron in a compact subset of $\mathbb{R}^{d}$, see, e.g., \cite[Definition 1.1.1]{mo1}. Thus, if $\mathfrak W^{* ,-\delta }(0)\neq \emptyset$, then $\mathfrak W^{* ,-\delta }(0)$ has extreme points, each of which is characterized as the intersection of $d$ (not necessarily unique) linearly independent constraints interpreted
as equalities. Therefore, $\mathfrak W^{* ,-\delta }(0)\neq
\emptyset $ implies that $\lambda ^{\cC}\in \mathfrak W^{* ,-\delta }(0)$ and therefore also that $\cC\subset \cJ^* \cup \{J+1,\dots,J+2d+2\}$. Note that the associated random vector $\lambda^{\cC}$ is a measurable selection of a random closed set that equals $\mathfrak W^{*,-\delta}(0)$ if $\mathfrak W^{*,-\delta}(0)\neq \emptyset$ and equals $\mathfrak B^d_\rho$ otherwise, see, e.g., \cite[Definition 1.2.2]{mo1}.

Lemma \ref{lem:exist_sol} establishes that for any $\eta >0$, there exist $\varepsilon_\eta>0$
and $N$ s.t. $n\geq N$ implies%
\begin{align}
\mathbf{P} \left( \mathfrak W^{*,-\delta}(0) \neq
\emptyset ,\left\vert \det K^{\cC}\right\vert \leq \varepsilon_\eta\right) \leq \eta,\label{eq:det_bound}
\end{align}
 which in turn, given our definition of $\cC$,  yields that there is $M>0$ and $N$ such that
\begin{align}
\mathbf P\Big(\big|\det \left(K^{\cC}\right)^{-1} \big|\le M\Big)\ge 1-\eta,~\forall n\ge N. \label{eq:det_inv_Op1}
\end{align}

Let $g_n$ be a $J+2d+2$ vector with
\begin{align}
\label{eq:tilde_g}
g_{n,j}(\theta+\lambda/\sqrt{n})\equiv \left\{ \begin{array}{llll}
c^{*}_{n}/(1+\eta^*_{n,j})-\G^*_{n,j}(\theta+\larhon) &  \mathrm{if}~j \in \cJ^*,\\
0, & \mathrm{if}~j\in\{1,\cdots,J\}\setminus\cJ^*,\\
1, &  \mathrm{if}~j=J+1,\dots,J+2d,\\
0, & \mathrm{if}~j=J+2d+1,J+2d+2,\\
	\end{array} 
	\right. 
\end{align}
using again that $\pi_{1,j}^*=0$ for $j \in \cJ^*$. For each $P\in\mathcal P$, let
\begin{align}
K_P(\theta,\rho) \equiv 
\begin{bmatrix}
[\rho D_{P,j}(\theta)]_{j=1}^{J_1+J_2}\\
[-\rho D_{P,j-J_2}(\theta)]_{j=J_1+J_2+1}^J\\
I_d\\
-I_d\\
p^\prime\\
-p^\prime
\end{bmatrix}. \label{eq:KP}
\end{align}
For each $n$ and $\lambda \in B^d$, define the mapping $\phi_{n}:B^d\to \R^d_{[\pm \infty]}$ by
\begin{align}
\label{eq:phi_fp}
\phi_{n}(\lambda) \equiv \left(K_{P_{n}}^{\cC}(\bar\theta(\theta_{n},\lambda),\rho )\right)^{-1} g^{\cC}_{n}(\theta_{n} + \tlarhon),
\end{align}
where the notation $\bar\theta(\theta_n,\lambda)$ emphasizes that $\bar{\theta}$ depends on $\theta_n$ and $\lambda$ because it lies component-wise between $\theta_n$ and $\theta_n + \larhon$.
We show that $\phi_{n}$ is a contraction mapping and hence has a fixed point.

For any $\lambda,\lambda^\prime \in B^d$ write
\begin{align}
\Vert \phi_{n}(\lambda) - \phi_{n}(\lambda^\prime) \Vert \notag&= \Big\Vert \big(K_{P_{n}}^{\cC}(\bar\theta(\theta_{n},\lambda),\rho )\big)^{-1}g^{\cC}_{n}(\theta_{n}+\tlarhon)- \big(K_{P_{n}}^{\cC}(\bar\theta(\theta_{n},\lambda^\prime),\rho )\big)^{-1}g^{\cC}_{n}(\theta_{n}+\tfrac{\lambda^\prime \rho }{\sqrt{n}}) \Big\Vert \notag\\
&\le \Big\Vert \big(K_{P_{n}}^{\cC}(\bar\theta(\theta_{n},\lambda),\rho )\big)^{-1} \Big\Vert_2 \Big\Vert g^{\cC}_{n}(\theta_{n}+\tlarhon)- g^{\cC}_{n}(\theta_{n}+\tfrac{\lambda^\prime \rho }{\sqrt{n}}) \Big\Vert \notag\\
&\qquad+\Big\Vert \big(K_{P_{n}}^{\cC}(\bar\theta(\theta_{n},\lambda),\rho )\big)^{-1}- \big(K_{P_{n}}^{\cC}(\bar\theta(\theta_{n},\lambda^\prime),\rho )\big)^{-1}\Big\Vert_2 \Big\Vert g^{\cC}_{n}(\theta_{n}+\tfrac{\lambda^\prime \rho }{\sqrt{n}}) \Big\Vert,\label{eq:phi_decomp}
\end{align}
where $\|\cdot\|_2$ denotes the spectral norm (induced by the Euclidean norm).

By Assumption \ref{as:bcs1} (ii), for any $\eta>0$, $k>0$, there is $N\in\mathbb N$ such that
\begin{align}
&~\mathbf P\left(\left\Vert g^{\cC}_{n}(\theta_{n}+\tlarhon)- g^{\cC}_{n}(\theta_{n}+\tfrac{\lambda^\prime \rho }{\sqrt{n}}) \right\Vert\le k\|\lambda-\lambda'\|\right) \\ 
=&~\mathbf P\left(\left\|\mathbb G^{*,{\cC}}_{n}(\theta_{n}+\tlarhon)-\mathbb G^{*,{\cC}}_{n}(\theta_{n}+\tfrac{\lambda^\prime \rho }{\sqrt{n}})\right\|\le k\|\lambda-\lambda'\|\right)\ge 1-\eta,~\forall n\ge N.\label{eq:g_gprimeOrootn}
\end{align}
Moreover, by arguing as in equation \eqref{eq:pi_star_gamma_vanishes}, for any $\eta$ there exist $0<L<\infty$ and $N\in\mathbb N$ such that $\forall n\ge N$
\begin{align}
\mathbf P\left(\sup_{\lambda^\prime \in B^d}\left\Vert g^{\cC}_{n}(\theta_{n}+\tfrac{\lambda^\prime \rho }{\sqrt{n}}) \right\Vert\le L\right) \ge 1-\eta.\label{eq:g_O1}
\end{align}
For any invertible matrix $K$, $\|K^{-1}\|_2=\left(\min \{\sqrt{\alpha}:~\alpha ~\text{ is an eigenvalue of } KK^\prime \}\right)^{-1}$.
Hence, by the proof of Lemma \ref{lem:exist_sol} and the definition of $\cC$, for any $\eta>0$, there exist $0<L<\infty$ and $N\in\mathbb N$ such that
\begin{align}
\mathbf P\big(\big\Vert \big(K^{\cC} \big)^{-1} \big\Vert_2 \le L \big)\ge  1-\eta,~\forall n\ge N,\label{eq:boundKn_inv2}
\end{align}
By \cite[ch. 5.8]{HornJohnson}, for any invertible matrices $K,\tilde K$ such that $\|\tilde K^{-1}(K-\tilde K)\|_2<1$, 
\begin{align}
	\| K^{-1}-\tilde K^{-1}\|_2\le \frac{\|\tilde K^{-1}(K-\tilde K)\|_2}{1-\|\tilde K^{-1}(K-\tilde K)\|_2}\|\tilde K^{-1}\|_2.
	\label{eq:boundKn_inv2a}
\end{align}
By the assumption that $D_{P_n}(\theta_n)\to D$ and Assumption \ref{as:momP_KMS}, for any $\eta>0$, there exists $N\in\mathbb N$ such that
\begin{align}
\sup_{\lambda\in B^d}\|K_{P_{n}}^{\cC}(\bar\theta(\theta_{n},\lambda),\rho )-K^{\cC}\|_2 \le \eta,~\forall n\ge N.
\end{align}  
By \eqref{eq:boundKn_inv2a}, the definition of the spectral norm, and the triangle inequality, for any $\eta>0$, there exist $0<L_1,L_2<\infty$ and $N\in\mathbb N$ such that
\begin{align}
&~\mathbf P \big(\sup_{\lambda\in B^d}\big\Vert \big(K_{P_{n}}^{\cC}(\bar\theta(\theta_{n},\lambda),\rho ) \big)^{-1} \big\Vert_2 \le 2L_1\big)\notag\\
\ge&~ \mathbf P\big(\big\Vert \big(K^{\cC} \big)^{-1} \big\Vert_2+\sup_{\lambda\in B^d}\|K_{P_{n}}^{\cC}(\bar\theta(\theta_{n},\lambda),\rho )^{-1}-(K^{\cC})^{-1}\|_2 \le 2L_1 \big)\notag	\\
\ge &~\mathbf P\Bigg(\big\Vert \big(K^{\cC} \big)^{-1} \big\Vert_2\le L_1 , \frac{\Vert \big( K^{\cC}\big)^{-1}\Vert_2^2}{1-\|\big( K^{\cC}\big)^{-1}(K_{P_{n}}^{\cC}(\bar\theta(\theta_{n},\lambda),\rho )-K^{\cC})\|_2}\le L_2,~\sup_{\lambda\in B^d}\|K_{P_{n}}^{\cC}(\bar\theta(\theta_{n},\lambda),\rho )-K^{\cC}\|_2 \le \frac{L_1}{L_2} \Bigg)\notag\\
\ge & ~ 1-2\eta,~\forall n\ge N,	\label{eq:boundKn_inv2b}
\end{align}
Again by applying \eqref{eq:boundKn_inv2a}, for any $k>0$, there exists $N\in\mathbb N$ such that
\begin{align}
	&~\mathbf P\big(\big\Vert \big(K_{P_{n}}^{\cC}(\bar\theta(\theta_{n},\lambda))\big)^{-1}- \big(K_{P_{n}}^{\cC}(\bar\theta(\theta_{n},\lambda^\prime))\big)^{-1}\big\Vert_2\le k\|\lambda-\lambda'\|\big)
	\\
	\ge&~ \mathbf P\Big(\sup_{\lambda\in B^d}\big\Vert \big(K_{P_{n}}^{\cC}(\bar\theta(\theta_{n},\lambda)) \big)^{-1} \big\Vert_2^2 M \rho \|\bar\theta(\theta_{n},\lambda)-\bar\theta(\theta_{n},\lambda^\prime)\|\le k\|\lambda-\lambda'\| \Big)\ge 1-\eta,~\forall n\ge N,\label{eq:boundKn_inv2c}
\end{align}
where the first inequality follows from $\|K_{P_{n}}^{\cC}(\bar\theta(\theta_{n},\lambda))-K_{P_{n}}^{\cC}(\bar\theta(\theta_{n},\lambda'))\|_2\le M\rho \|\bar\theta(\theta_{n},\lambda)-\bar\theta(\theta_{n},\lambda^\prime)\|\le M\rho ^2/\sqrt n\|\lambda-\lambda'\|$ by Assumption \ref{as:momP_KMS} (ii), and the last inequality follows from \eqref{eq:boundKn_inv2b}.

By \eqref{eq:phi_decomp}-\eqref{eq:g_O1} and \eqref{eq:boundKn_inv2b}-\eqref{eq:boundKn_inv2c}, it then follows that
 there exists $\beta \in [0,1)$ such that for any $\eta>0$, there exists $N \in \N$ such that
\begin{align}
\mathbf{P}\left(\vert \phi_{n}(\lambda) - \phi_{n}(\lambda^\prime) \vert \le \beta \Vert \lambda - \lambda^\prime \Vert, ~~\forall \lambda,\lambda^\prime \in B^d\right) \ge 1-\eta, ~\forall n\ge N. \label{eq:contraction}
\end{align}
This implies that with probability approaching 1, each $\phi_{n}(\cdot)$ is a contraction, and therefore by the Contraction Mapping Theorem it has a fixed point (e.g., \cite[Theorem 1.3]{Pata}). This in turn implies that for any $\eta>0$ there exists a $N \in \N$ such that
\begin{align}
\mathbf{P}\left(\exists \lambda^f_{n}: \lambda^f_{n} = \phi_{n}(\lambda^f_{n})  \right) \ge 1-\eta, ~\forall n\ge N. \label{eq:fixed_point}
\end{align}

Next, define the mapping
\begin{equation}
\psi_{n}(\lambda) \equiv \left( K^{\cC} \right)^{-1}g^{\cC}. \label{eq:limit_map}
\end{equation}
This map is constant in $\lambda$ and hence is uniformly continuous and a contraction with Lipschitz constant equal to zero. It therefore has $\lambda^{\cC}_{n}$ as its fixed point. Moreover, by \eqref{eq:phi_fp} and \eqref{eq:limit_map} arguing as in \eqref{eq:phi_decomp}, it follows that for any $\lambda\in B^d$,
\begin{align}
\Vert \psi_{n}(\lambda) -\phi_{n}(\lambda)\Vert&\le 	\Big\Vert \big(K_{P_{n}}^{\cC}(\bar\theta(\theta_{n},\lambda),\rho )\big)^{-1} \Big\Vert_2 \Big\Vert g^{\cC}- g^{\cC}_{n}(\theta_{n}+\tlarhon) \Big\Vert 
+\Big\Vert \big(K^{\cC}\big)^{-1}- \big(K_{P_{n}}^{\cC}(\bar\theta(\theta_{n},\lambda),\rho )\big)^{-1}\Big\Vert_2 \big\Vert g^{\cC} \big\Vert.
\end{align}
By \eqref{eq:g} and \eqref{eq:tilde_g} 
\begin{align}
 \Big\Vert g^{\cC}- g^{\cC}_{n}(\theta_{n}+\tlarhon) \Big\Vert
 &\le \max_{j \in \cJ^*}|- \HH_j^*-c^{*}_{n}/(1+\eta^*_{n,j})+\G^*_{n,j}(\theta_{n}+\tlarhon)|\notag\\
 &\le \max_{j \in \cJ^*}|\HH_j^* -\G^*_{n,j}(\theta_{n}+\tlarhon)|+\max_{j \in \cJ^*}|c^{*}_{n}/(1+\eta^*_{n,j})|.\label{eq:psiphi1}
\end{align}
We note that when Assumption \ref{as:correlation}-\ref{as:correlation_pair} is used, for each $j=1,\dots,R_1$ such that $\pi^*_{1,j}=0=\pi^*_{1,j+R_1}$ we have that $\vert \tilde{\mu}_j -\mu_j \vert = o_\cP(1)$ because $\sup_{\theta \in \Theta}\vert\eta_j(\theta)\vert=o_\cP(1)$, where $\tilde{\mu}_j $ and $\mu_j$ were defined in \eqref{eq:tilde_weight_j}-\eqref{eq:tilde_weight_j_pl_J11} and \eqref{eq:weight_j}-\eqref{eq:weight_j_pl_J11} respectively.
Moreover,  $\G^*_{n,j}(\theta_{n}+\larhon)\stackrel{a.s.}{\to}\mathbb \HH^*$  and \eqref{eq:tilde_c_lim} implies $c^{*}_{n} \to 0$ so that we have 
\begin{align}
\sup_{\lambda\in B^d} \Bigl\Vert g^{\cC}- g^{\cC}_{n}(\theta_{n}+\tlarhon) \Bigr\Vert\stackrel{a.s.}{\to}0.	\label{eq:psiphi2}
\end{align}
Further, by \eqref{eq:boundKn_inv2a}, $D_{P_n}\to D$ and, Assumption \ref{as:momP_KMS}-(ii), for any $\eta>0$, there exists $N\in\mathbb N$ such that
\begin{align}
\sup_{\lambda\in B^d}	\Big\Vert \big(K^{\cC}\big)^{-1}- \big(K_{P_{n}}^{\cC}(\bar\theta(\theta_{n},\lambda),\rho )\big)^{-1}\Big\Vert_2 \le \eta,~\forall n\ge N.\label{eq:psiphi3}
\end{align}
In sum, by \eqref{eq:g_O1}, \eqref{eq:boundKn_inv2b}, and \eqref{eq:psiphi1}-\eqref{eq:psiphi3}, for any $\eta,\nu>0$, there exists $N\ge\mathbb N$ such that
\begin{align}
\mathbf{P}\left(\sup_{\lambda \in B^d}\Vert \psi_{n}(\lambda) -\phi_{n}(\lambda)\Vert<\nu\right)\ge 1-\eta,~\forall n\ge \mathbb N.\label{eq:unif_conv_phi}
\end{align}  
Hence, for a specific choice of $\nu=\kappa(1-\beta)$, where $\beta$ is defined in equation \eqref{eq:contraction}, we have that $\sup_{\lambda \in B^d}\Vert \psi_{n}(\lambda) -\phi_{n}(\lambda)\Vert<\kappa(1-\beta)$ implies
\begin{align}
\Vert \lambda^{\cC}_{n} - \lambda^f_{n} \Vert &= \Vert \psi_{n}(\lambda^{\cC}_{n}) -\phi_{n}(\lambda^f_{n})\Vert\notag\\
&\le \Vert \psi_{n}(\lambda^{\cC}_{n}) -\phi_{n}(\lambda^{\cC}_{n})\Vert+\Vert \phi_{n}(\lambda^{\cC}_{n}) -\phi_{n}(\lambda^f_{n})\Vert\notag\\
&\le \kappa(1-\beta) + \beta \Vert \lambda^{\cC}_{n} - \lambda^f_{n} \Vert
\end{align}
Rearranging terms, we obtain $\Vert \lambda^{\cC}_{n} - \lambda^f_{n} \Vert \le \kappa$. 
Note that by Assumptions \ref{as:momP_KMS} (i) and \ref{as:bcs1} (i),  for any $\delta>0$, there exists $\kappa_\delta>0$ and $N\in\mathbb N$ such that 
\begin{align}
\mathbf P\Big(\sup_{\|\lambda-\lambda'\|\le \kappa_\delta}	|u^*_{n,j,\theta_{n}}(\lambda)-u^*_{n,j,\theta_{n}}(\lambda')|<\delta\Big)\ge 1-\eta,~\forall n\ge \mathbb N.	\label{eq:se_u}
\end{align}
For $\lambda^{\cC}_{n}\in \mathfrak W^{*,-\delta}(0)$, one has
\begin{align}
	\mathfrak w^*_{j}(\lambda^{\cC}_{n})+\delta\le 0,~j \in \{1,\cdots, J_1\} \cap \cJ^*.\label{eq:wconst}
\end{align}
Hence, by \eqref{eq:Aln}, \eqref{eq:tilde_c_lim}, and \eqref{eq:se_u}-\eqref{eq:wconst}, $\Vert \lambda^{\cC}_{n} - \lambda^f_{n} \Vert \le \kappa_{\delta/4}$, for each $j \in \{1,\cdots, J_1\} \cap \cJ^*$ we have
\begin{align}
	u^*_{n,j,\theta_{n}}(\lambda^f_{n})-c^{*}_{n}(\theta_n)\le u^*_{n,j,\theta_{n}}(\lambda^{\cC}_{n})-c^{*}_{n}(\theta_n)+\delta/4 \le \mathfrak w^*_{j}(\lambda^{\cC}_{n})+\delta/2 \le 0.
\end{align}
For $j \in \{J_1+1,\cdots, 2J_2\} \cap \cJ^*$, the inequalities hold by construction given the definition of ${\cC}$.

 In sum, for any $\eta>0$ there exists $\delta > 0$ and  $N \in \N$ such that for all $n \ge N$ we have
 \begin{align}
\mathbf{P}\Big( \{U^*_{n}(\theta_{n},c^{*}_{n})= \emptyset \} \cap \{\mathfrak W^{*,-\delta}(0)\neq \emptyset \} \Big)
\le \mathbf{P}	\Big(\nexists \lambda^f_{n}\in U^*_{n}(\theta_{n},c^{*}_{n}), \exists \lambda^{\cC}_{n}\in \mathfrak W^{*,-\delta}(0) \Big)& \notag \\
\le \mathbf{P}	\left(\left\{\sup_{\lambda \in  B^d}\Vert \psi_{n}(\lambda) -\phi_{n}(\lambda)\Vert<\kappa_\delta(1-\beta) \cap A_{n} \right\}^c \right)\le \eta/3,&\label{eq:final1}
 \end{align}
where $A^c$ denotes the complement of the set $A$, and the last inequality follows from \eqref{eq:Aln} and \eqref{eq:unif_conv_phi}. 
\end{proof}

\begin{lemma}\label{lem:cv_convergence}
Suppose Assumptions \ref{as:momP_AS}, \ref{as:GMS},  \ref{as:correlation}, \ref{as:momP_KMS}, and \ref{as:bcs1} hold.
Let	$\{P_n,\theta_n\} \in \{(P,\theta):P \in \cP,\theta \in \Theta_I(P)\}$ be a sequence satisfying \eqref{eq:subseq1}-\eqref{eq:subseq3}. 
For each $j$, let
\begin{align}
v^I_{n,j,\theta_n}(\lambda)&\equiv	\mathbb G^b_{n,j}(\theta_n)+\rho\hat{D}_{n,j}(\theta_n)\lambda+\varphi^*_j(\hat{\xi}_{n,j}(\theta_n)),\label{eq:vIj_var}\\
\mathfrak{w}_{j}(\lambda)&\equiv \HH_{j}+\rho D_{j}\lambda+\pi^*_{1,j},
\end{align}
where 
\begin{align}
\varphi^*_j(\xi)=\begin{cases}
	\varphi_j(\xi)&\pi_{1,j}=0\\
	-\infty&\pi_{1,j}<0\\
	0&j=J_1+1,\cdots,J.
\end{cases}\label{eq:greedy_gms}	
\end{align}
For each $c\ge0$, define
\begin{align}
V^I_n(\theta_n,c)&\equiv \{\lambda\in B^d_{n,\rho}:p'\lambda=0\cap v^I_{n,j,\theta_n}(\lambda)\le c,j=1,\cdots,J\},\label{eq:VIn}\\
\mathfrak{W}(c)&\equiv \big\{ \lambda\in \mathfrak B^d_\rho: p^\prime \lambda = 0 \cap \mathfrak w_{j}(\lambda)\le c, \: \forall j=1,\dots,J\big\}.
\end{align}
We then let $c^I_n(\theta_n)\equiv \inf\{c\in\R_+:P^*_n(V^I_n(\theta_n,c)\neq \emptyset)\ge 1-\alpha\}$ and $c_{\pi^*}\equiv \inf\{c\in\mathbb R_+:\Pr(\mathfrak{W}(c)\ne\emptyset)\ge 1-\alpha\}$.

Then, 
(i) for any $c>0$ and $\{\theta'_n\}\subset\Theta$ such that $\theta'_n\in \thetnprime$ for all $n$, 
\begin{align}
P^*_n(V^I_n(\theta'_n,c)\ne\emptyset)-	\Pr(\mathfrak{W}(c)\ne\emptyset)\to 0,
\end{align}	
with probability approaching 1; 

\noindent
(ii)  If $c_{\pi^*}>0$, $c^I_n(\theta'_n)\stackrel{P_n}{\to}c_{\pi^*};$ 

\noindent
(iii) For any $\{\theta'_n\}\subset\Theta$ such that $\theta'_n\in \thetnprime$ for all $n$,
\begin{align}
\hat c_n(\theta'_n)\ge c^I_n(\theta'_n)+o_{P_n}(1).	
\end{align}
\end{lemma}

\begin{proof}
	Throughout, let $c>0$ and let $\{\theta'_n\}\subset\Theta$ be a sequence such that $\theta'_n\in \thetnprime$ for all $n$.
By Lemma \ref{lem:boot_cons}, 
 in $l^\infty(\Theta)$ uniformly in $\cP$ conditional on $\{X_i\}_{i=1}^\infty$, and by Assumption \ref{as:momP_KMS} $\|\hat D_n(\theta_n')-D_{P_n}(\theta_n)\|\stackrel{p}{\to} 0$.
 Further, by Lemma \ref{lem:Jstar}, $\hat{\xi}_{n,j}(\theta_n')\stackrel{P_n}{\to}\pi_{1,j}.$ Therefore,
\begin{align}
(\mathbb G^b_{n}(\theta'_n),\hat D_{n}(\theta_{n}'), \hat\xi_{n}(\theta_n'))|\{X_i\}_{i=1}^\infty\stackrel{d}{\to}(\HH,D, \pi_1).\label{eq:asrep_invoke}
\end{align}
for almost all sample paths $\{X_i\}_{i=1}^\infty$. By Lemma \ref{cor:asrep}, conditional on the sample path, there exists an almost sure representation $(\tilde{\mathbb G}^{b}_{n}(\theta'_n),\tilde D_n,\tilde{\xi}_{n})$ of $(\mathbb G^b_{n}(\theta'_n),\hat{D}_{n}(\theta'_n),\hat{\xi}_{n}(\theta'_n))$ defined on another probability space $(\tilde \Omega,\tilde{\mathcal F},\tilde{\mathbf P})$ such that $(\tilde{\mathbb G}^{b}_{n}(\theta'_n),\tilde D_n,\tilde{\xi}_{n})\stackrel{d}{=}(\mathbb G^b_{n}(\theta'_n),\hat{D}_{n}(\theta'_n),\hat{\xi}_{n}(\theta'_n))$ conditional on the sample path. In particular, conditional on the sample, $(\hat{D}_{n}(\theta'_n),\hat{\xi}_{n}(\theta'_n))$ are non-stochastic. Therefore, we set $(\tilde D_n,\tilde{\xi}_{n})\stackrel{}{=}(\hat{D}_{n}(\theta'_n),\hat{\xi}_{n}(\theta'_n))$, $\tilde{\mathbf P}-a.s$. 
The almost sure representation satisfies
$(\tilde{\mathbb G}^b_{n}(\theta'_n),\tilde D_{n},\tilde{\xi}_{n,j})\stackrel{a.s.}{\to} (\tilde\HH,D, \pi_1)$ for almost all sample paths, where $\tilde\HH\stackrel{d}{=}\HH$. The almost sure representation $(\tilde{\mathbb G}^b_{n},\tilde D_n,\tilde{\xi}_{n})$ is defined for each sample path $x^\infty=\{x_i\}_{i=1}^\infty$, but we suppress its dependence on $x^\infty$ for notational simplicity (see Appendix \ref{app:asrep} for details).
Using this representation, define
\begin{align}
\tilde v^I_{n,j,\theta'_n}(\lambda)\equiv	\tilde{\mathbb G}^b_{n,j}(\theta'_n)+\rho\tilde D_n\lambda+\varphi^*_j(\tilde{\xi}_{n,j}),\label{cv_conv1}
\end{align}and
\begin{align}
\tilde{\mathfrak{w}}_{j}(\lambda)&\equiv \tilde \HH_j+ \rho D_{j}\lambda+\pi^*_{1,j},\label{cv_conv2}
\end{align}
where $\tilde{\mathbb Z}\stackrel{d}{=}\mathbb Z$, and $\tilde{\mathbb G}^b_{n}(\theta'_n)\to \tilde{\mathbb Z}, \tilde{\mathbf P}-a.s.$ conditional on $\{X_i\}_{i=1}^\infty$.
With this construction, one may write
\begin{multline}
|P^*_n(V^I_n(\theta'_n,c)\ne \emptyset)-	\Pr(\mathfrak{W}(c)\ne\emptyset)|=|\tilde{\mathbf P}(\tilde V^I_n(\theta'_n,c)\ne\emptyset)-	\tilde{\mathbf P}(\tilde{\mathfrak{W}}(c)\ne\emptyset)|\label{eq:cv_conv3}\\
\le |\tilde{\mathbf P}(\tilde V^I_n(\theta'_n,c)=\emptyset\cap\tilde{\mathfrak{W}}(c)\ne\emptyset)+\tilde{\mathbf P}(\tilde V^I_n(\theta'_n,c)\ne \emptyset\cap\tilde{\mathfrak{W}}(c)=\emptyset)|,
\end{multline}	
where the inequality is due to \eqref{eq:cover5}.
First,  we  bound  the first term on the right hand side of \eqref{eq:cv_conv3}. Note that
\begin{align}
\tilde{\mathbf P}(\tilde V^I_n(\theta'_n,c)=\emptyset\cap\tilde{\mathfrak{W}}(c)\ne\emptyset)
\le 	\tilde{\mathbf P}(\tilde V^{I,+\delta}_n(\theta'_n,c)=\emptyset\cap\tilde{\mathfrak{W}}(c)\ne\emptyset)+\tilde{\mathbf P}(\tilde V^{I,+\delta}_n(\theta'_n,c)\ne\emptyset\cap\tilde V^I_n(\theta'_n,c)=\emptyset),\label{eq:cv_conv4}
\end{align}
where $\tilde V^{I,+\delta}_n$ is defined as
\begin{align}
\tilde V^{I,+\delta}_n\equiv \Big\{\lambda\in B^d_{n,\rho}:p'\lambda=0\cap\tilde v^I_{n,j,\theta'_n}(\lambda)\le c+\delta,j\in\mathcal J^*\Big\}.
\end{align}
Let 
\begin{align}
A_n\equiv\Big\{\tilde\omega\in\tilde\Omega:\sup_{\lambda\in B^d}\max_{j\in\mathcal J^*}|\tilde v^I_{n,j,\theta'_n}(\lambda)-\tilde{\mathfrak{w}}_{j}(\lambda)|\ge\delta\Big\}.\label{eq:cv_conv5}
\end{align}
Let 
\begin{align}
E\equiv \{\{x_i\}_{i=1}^\infty:\|\hat D_{n}(\theta'_n)-D\|<\eta,\max_{j\in\mathcal J^*}|\varphi_j^*(\hat{\xi}_{n,j}(\theta'_n))-\pi^*_{1,j}|<\eta\}.\label{eq:cv_conv6}	
\end{align}
Note that, $P_n(E)\ge 1-\eta$ for all $n$ sufficiently large by Assumption \ref{as:momP_KMS} and Lemma \ref{lem:Jstar}. On $E$, we therefore have $\|\tilde D_n-D\|<\eta$ and $\max_{j\in\mathcal J^*}|\tilde{\xi}_{n,j}-\pi^*_{1,j}|<\eta$, $\tilde{\mathbf P}-a.s.$ 
Below, we condition on  $\{X_i\}_{i=1}^\infty\in E$. For any $j\in\mathcal J^*$,
\begin{align}
|\tilde v^I_{n,j,\theta'_n}(\lambda)-\tilde{\mathfrak{w}}_{j}(\lambda)|\le |\tilde{\mathbb G}^b_{n,j}(\theta'_n)-\tilde \HH_j|+\rho\|\tilde D_{j,n} -D_{j}\|\|\lambda\|+|\varphi_j^*(\tilde{\xi}_{n,j})-\pi^*_{1,j}|\le (2+\rho)\eta,\label{eq:use-of-lambda}
\end{align}
uniformly in $\lambda\in B^d$, where we used $\tilde{\mathbb G}^b_{n}\to \tilde{\mathbb Z}, \tilde{\mathbf P}-a.s.$
Since $\eta$ can be chosen arbitrarily small, this in turn implies
\begin{align*}
\tilde{\mathbf P}\big(A_n\big)<\eta/2,
\end{align*}
for all $n$ sufficiently large. Note also that
$\sup_{\lambda\in B^d}\max_{j\in\mathcal J^*}|\tilde v^I_{n,j,\theta'_n}(\lambda)-\tilde{\mathfrak{w}}_{j}(\lambda)|<\delta$ implies $\tilde{\mathfrak{W}}(c)\subseteq\tilde V^{I,+\delta}_n(\theta'_n,c)$, and hence
 $A^c_n$ is a subset of
\begin{align}
L_n\equiv\Big\{\tilde\omega\in\tilde\Omega:\tilde{\mathfrak{W}}(c)\subseteq\tilde V^{I,+\delta}_n(\theta'_n,c)\Big\}.\label{eq:cv_conv7}
\end{align}
Using this,
\begin{align}
\tilde{\mathbf P}(\tilde V^{I,+\delta}_n(\theta'_n,c)=\emptyset\cap\tilde{\mathfrak{W}}(c)\ne\emptyset)\le \tilde{\mathbf P}(\tilde{\mathfrak{W}}(c)\not\subseteq\tilde V^{I,+\delta}_n(\theta'_n,c))
=\tilde{\mathbf P}(L^c_n)\le \tilde{\mathbf P}(A_n)<\eta/2,\label{eq:cv_conv8}
\end{align}
for all $n$ sufficiently large.
Also, by Lemma \ref{lem:empt},
\begin{align}
\tilde{\mathbf P}(\tilde V^{I,+\delta}_n(\theta'_n,c)\ne\emptyset\cap\tilde V^I_n(\theta'_n,c)=\emptyset)<\eta/2,\label{eq:cv_conv9}
\end{align}
for all $n$ sufficiently large.

Combining \eqref{eq:cv_conv4}, \eqref{eq:cv_conv5}, \eqref{eq:cv_conv8}, \eqref{eq:cv_conv9}, and using $P_n(E)\ge 1-\eta$ for all $n$, we have
\begin{equation}
\int_{E}\tilde{\mathbf P}(\tilde V^I_n(\theta'_n,c)=\emptyset\cap\tilde{\mathfrak{W}}(c)\ne\emptyset)dP_n+\int_{E^c}\tilde{\mathbf P}(\tilde V^I_n(\theta'_n,c)=\emptyset\cap\tilde{\mathfrak{W}}(c)\ne\emptyset)dP_n\le \eta(1-\eta)+\eta\le 2\eta.
\end{equation}
The second term of the right hand side of \eqref{eq:cv_conv3} can be bounded similarly. Therefore, $|P^*(V^I_n(\theta'_n,c)\ne \emptyset)-	\Pr(\mathfrak{W}(c)\ne\emptyset)|\to 0$ with probability (under $P_n$) approaching 1. This establishes the first claim.

(ii) 
By Part (i), for $c>0$, we have
\begin{align}
P^*_n(V^I_n(\theta'_n,c)\ne \emptyset)-	\mathrm{Pr}(\mathfrak{W}(c)\ne\emptyset)\to 0.	\label{eq:cv_conv2_1}
\end{align}
Fix $c>0$, and set
 \begin{align}
 g_{j} =\left\{ \begin{array}{lll}
  c - \HH_j, &  j=1,\dots,J,\\
  1, &  j=J+1,\dots,J+2d,\\
 0, & j=J+2d+1,J+2d+2.
 	\end{array} 
 	\right. \label{eq:cv_conv2_2}
 \end{align}
Mimic the argument following \eqref{eq:int1}. Then, this yields
\begin{align}
 \left\vert \mathrm{Pr}\left(\mathfrak{W}(c)\neq \emptyset \right)-\mathrm{Pr}\left(\mathfrak{W}(c-\delta)\neq \emptyset \right)\right\vert = \mathrm{Pr}\left(\{\mathfrak{W}(c)\neq \emptyset\} \cap  \{\mathfrak{W}(c-\delta)=\emptyset\}\right)\le\eta,\label{eq:cv_conv2_3a}\\
 \left\vert \mathrm{Pr}\left(\mathfrak{W}(c+\delta)\neq \emptyset \right)-\mathrm{Pr}\left(\mathfrak{W}(c)\neq \emptyset \right)\right\vert = \mathrm{Pr}\left(\{\mathfrak{W}(c+\delta)\neq \emptyset\} \cap  \{\mathfrak{W}(c)=\emptyset\}\right)\le \eta,\label{eq:cv_conv2_3b}
\end{align}
which therefore ensures that $c\mapsto \mathrm{Pr}(\mathfrak{W}(c)\ne\emptyset)$ is continuous at $c>0$.

Next, we show $c\mapsto\mathrm{Pr}\left(\mathfrak{W}(c)\neq \emptyset \right)$ is strictly increasing at any $c>0$.
For this, consider $c>0$ and $c-\delta>0$ for $\delta>0$. 
Define the $J$ vector $e$ to have elements $e_j = c - \HH_j,~j=1,\dots,J$. Suppose for simplicity that $\cJ^*$ contains the first $J^*$ inequality constraints. Let $e^{[1:J^*]}$ denote the subvector of $e$ that only contains elements corresponding to $j \in \cJ^*$, define $D^{[1:J^*,:]}$ correspondingly, and write 
\begin{align}
K=\left[ 
\begin{array}{c}
D^{[1:J^*,:]} \\ 
I_{d} \\ 
-I_{d} \\ 
p^{\prime } \\ 
-p^{\prime }%
\end{array}%
\right] ,~g=\left[ 
\begin{array}{c}
e^{[1:J^*]} \\ 
\rho \cdot \mathbf{1}_{d} \\ 
\rho \cdot \mathbf{1}_{d} \\ 
0 \\ 
0%
\end{array}%
\right] ,~\tau =\left[ 
\begin{array}{c}
\mathbf{1}_{J^*} \\ 
\mathbf{0}_{d} \\ 
\mathbf{0}_{d} \\ 
0 \\ 
0%
\end{array}%
\right] .\label{eq:Kgtau}
\end{align}%

By Farkas' lemma \citep[][Theorem 22.1]{Rockafellar1970a}  and arguing as in \eqref{eq:int4},
\begin{align}
\mathrm{Pr}\left(\{\mathfrak{W}(c)\neq \emptyset\} \cap  \{\mathfrak{W}(c-\delta)=\emptyset\}\right)=
\mathrm{Pr}\left(\{\mu'g\ge0,\forall\mu\in\mathcal M\}\cap \{\mu'(g-\delta\tau)<0,\exists\mu\in\mathcal M\}\right),\label{eq:cv_conv3_1}
\end{align}
where $\mathcal M=\{\mu\in\mathbb R^{J^*+2d+2}_+:\mu'K=0\}.$
By Minkowski-Weyl's theorem \citep[][Theorem 3.52]{Rockafellar_Wets2005aBK}, there exists $\{\nu^t\in\mathcal M,t=1,\cdots,T\}$, for which one may write
\begin{align}
\mathcal M=\{\mu:\mu=b\sum_{t=1}^Ta_t\nu^t,b>0,a_t\ge0,\sum_{t=1}^Ta_t=1\}.\label{eq:cv_conv3_2}
\end{align}
This implies
\begin{align}
\mu'g\ge0,~\forall\mu\in\mathcal M&~\Leftrightarrow~\nu^{t}{}'g\ge0, ~\forall t\in\{1,\cdots,T\}\label{eq:cv_conv3_3a}\\
\mu'(g-\delta\tau)<0,~\exists\mu\in\mathcal M&~\Leftrightarrow~\nu^{t}{}'g<\delta\nu^{t}{}'\tau,~\exists t\in\{1,\cdots,T\}.\label{eq:cv_conv3_3b}
\end{align}
Hence,
\begin{align}
\mathrm{Pr}\left(\{\mu'g\ge0,\forall\mu\in\mathcal M\}\cap \{\mu'(g-\delta\tau)<0,\exists\mu\in\mathcal M\}\right)&=\textrm{Pr}\left(0\le \nu^{s}{}'g,~0\le \nu^{t}{}'g<\delta\nu^{t}{}'\tau,~\forall s,\exists t\right)\label{eq:cv_conv3_4}
\end{align}
Note that by \eqref{eq:Kgtau}, for each $s\in\{1,\cdots,T\},$
\begin{align}
\nu^{s}{}'g&=\nu^{s,[1:J^*]}{}'(c1_{\mathcal J^*}-\HH_{\mathcal J^*})+\rho\sum_{j=J^*+1}^{J^*+2d}\nu^{s,[j]},\label{eq:cv_conv3_5a}	\\
\nu^{s}{}'\tau&=\sum_{j=1}^{J^*}\nu^{s,[j]}.\label{eq:cv_conv3_5b}
\end{align}
For each $s\in\{1,\cdots,T\}$, let
\begin{align}
h^U_s &\equiv c\sum_{j=1}^{J^*}\nu^{s,[j]}+\rho\sum_{j=J^*+1}^{J^*+2d}\nu^{s,[j]}\label{eq:cv_conv3_6a}\\
h^L_s &\equiv (c-\delta)\sum_{j=1}^{J^*}\nu^{s,[j]},\label{eq:cv_conv3_6b}
\end{align}
where $0\le h_s^L<h_s^U$ for all $s\in\{1,\cdots,T\}$ due to $0<c-\delta<c$ and $\nu^s\in\mathbb R^{J^*+2d+2}_+$.
One may therefore rewrite the probability on the right hand side of \eqref{eq:cv_conv3_4} as
\begin{equation}
\textrm{Pr}\left(0\le \nu^{s}{}'g,~0\le \nu^{t}{}'g<\delta\nu^{t}{}'\tau,~\forall s,\exists t\right)
=\textrm{Pr}\left(\nu^{s,[1:J^*]}{}'\HH_{\mathcal J^*}\le h^U_s, h_t^L<\nu^{t,[1:J^*]}{}'\HH_{\mathcal J^*} \le h_t^U~\forall s,\exists t\right)>0,\label{eq:cv_conv3_7}
\end{equation}
where the last inequality follows because $\HH_{\mathcal J^*}$'s correlation matrix $\Omega$ has an eigenvalue bounded away from 0 by Assumption \ref{as:correlation}.
By \eqref{eq:cv_conv3_1}, \eqref{eq:cv_conv3_4}, and \eqref{eq:cv_conv3_7}, $c\mapsto\mathrm{Pr}\left(\mathfrak{W}(c)\neq \emptyset \right)$ is strictly increasing at any $c>0$.

Suppose that $c_{\pi^*}>0$, then arguing as in Lemma 5.(i) of \cite{Andrews:2010aa}, we obtain $c^I_n(\theta'_n)\stackrel{P_n}{\to}c_{\pi^*}$.

(iii) Begin with observing that one can equivalently express $\hat{c}_n$ (originally defined in \eqref{eq:def:c_hat}) as  $\hat{c}_n(\theta)= \inf\{c\in\R_+:P^*_n(V^b_n(\theta,c)\neq \emptyset) \ge 1-\alpha\}$. 

Suppose first that Assumption \ref{as:correlation}-\ref{as:correlation_base} holds. In this case, there are no paired inequalities, and $V^I_n$ differs from $V_n^b$ only in terms of the function $\varphi^*_j$ in \eqref{eq:greedy_gms} used in place of the GMS function $\varphi_j$. In particular, $\varphi^*_j(\xi)\le \varphi_j(\xi)$ for any $j$ and $\xi$, and therefore $\hat c_n(\theta_n)\ge c^{I}_n(\theta_n)$ by construction.

Next, suppose  Assumption \ref{as:correlation}-\ref{as:correlation_pair} holds and $V_n^I(\theta'_n,c)$ is defined with hard threshold GMS, i.e. with GMS function $\varphi^1$ in AS. The only case that might create concern is one in which
\begin{align}
\pi_{1,j}\in [-1,0)	~&\mathrm{and}~ \pi_{1,j+R_1}=0.
\end{align} 
In this case, only the $j+R_1$-th inequality binds in the limit, but with probability approaching $1$, GMS selects both of the pair. Therefore, we have
\begin{align}
\pi^*_{1,j}=-\infty, ~&\mathrm{and}~
\pi^*_{1,j+R_1}=0,\\
\varphi_j(\hat{\xi}_{n,j}(\theta'_n))=0, ~&\mathrm{and}~
\varphi_{j+R_1}(\hat{\xi}_{n,j+R_1}(\theta'_n))=0,
\end{align}
so that in $V_n^I(\theta'_n,c)$, inequality $j+R_1$, which is
\begin{align}
\mathbb G^b_{n,j+R_1}(\theta'_n)+\rho \hat D_{n,j+R_1}(\theta'_n)\lambda \le c,
\end{align}
is replaced with inequality
\begin{align}
-\mathbb G^b_{n,j}(\theta'_n)-\rho \hat D_{n,j}(\theta'_n)\lambda \le c,
\end{align}
as explained in Section \ref{sec:AssRes}. In this case, $\hat c_n(\theta_n)\ge c^{I}_n(\theta_n)$ is not guaranteed in finite sample. However, let $v^{IP}_n$ be as in \eqref{eq:vIj_var} but replacing  $j+R_1$-th component $\mathbb G^b_{n,j+R_1}(\theta_n)+\hat{D}_{n,j+R_1}(\theta_n)\lambda+\varphi^*_{j+R_1}(\hat{\xi}_{n,j+R_1}(\theta_n))$ with $-\mathbb G^b_{n,j}(\theta_n)-\hat{D}_{n,j}(\theta_n)\lambda-\varphi^*_j(\hat{\xi}_{n,j}(\theta_n)).$ Define $V^{IP}_n$ as in \eqref{eq:VIn} but replacing $v^I_n$ with $v^{IP}_n$. Define $c^{IP}_n(\theta_n)\equiv \inf\{c\in\R_+:P^*(V^{IP}_n(\theta_n,c))\ge 1-\alpha\}$. By construction, $\hat c_n(\theta'_n)\ge c^{IP}_n(\theta'_n)$ for any $\theta'_n\in \thetnprime$. Therefore, it suffices to show that $c_n^{IP}(\theta'_n)-c_n^I(\theta'_n)\stackrel{P_n}{\to}0$.
For this, note that Lemma \ref{lem:pair}-(3) establishes 
\begin{align}
\sup_{\lambda\in B^d_{n,\rho}}\| \mathbb G^b_{n,j+R_1}(\theta'_n)+\rho \hat D_{n,j+R_1}(\theta'_n)\lambda+\mathbb G^b_{n,j}(\theta'_n)+\rho \hat D_{n,j}(\theta'_n)\lambda \| =o_{P^*}(1),\label{eq:dif_to_zeroCT}
\end{align}
for almost all sample paths $\{X_i\}_{i=1}^\infty.$  Therefore, replacing the $j+R_1$-th inequality with the $j$-th inequality in $V^{IP}_n$ is asymptotically negligible. Mimicking the arguments in Parts (i) and (ii) then yields
\begin{align}
	c^{IP}_n(\theta'_n)\stackrel{P_n}{\to}c_{\pi^*}.
\end{align}
This therefore ensures $c_n^{IP}(\theta'_n)-c_n^I(\theta'_n)\stackrel{P_n}{\to}0$.

If the set $V^I_n(\theta'_n,c)$ is defined with a GMS function satisfying Assumption \ref{as:GMS} and continuous in its argument, we can mimic the above argument using the replacements in \eqref{eq:paired_bind_Lboot_smooth}-\eqref{eq:paired_bind_Uboot_smooth} with $\hat\mu_{n,j+R_1}$ as defined in \eqref{eq:mu_hat_j} and $\hat\mu_{n,j}(\theta'_n)$ as in \eqref{eq:mu_hat_j_pl_J11}. Then when both $\pi_j \in (-\infty,0]$ and $\pi_{j+R_1}\in (-\infty,0]$ we have:
\begin{align}
\Delta(\mu,\hat{\mu})\equiv\Big\Vert  &\hat\mu_{n,j}(\theta'_n)\{\mathbb{G}_{n,j}^{b}(\theta'_n)+\rho \hat{D}
_{n,j}(\theta'_n)\lambda\} - \hat\mu_{n,j+R_1}(\theta'_n)\{\mathbb{G}_{n,j+R_1}^{b}(\theta'_n)+\rho \hat{D}
_{n,j+R_1}(\theta'_n)\lambda\}  \notag\\
-&\mu_j(\theta'_n)\{\mathbb{G}_{n,j}^{b}(\theta'_n)+\rho \hat{D}
_{n,j}(\theta'_n)\lambda\} + \mu_{j+R_1}(\theta'_n)\{\mathbb{G}_{n,j+R_1}^{b}(\theta'_n)+\rho \hat{D}
_{n,j+R_1}(\theta'_n)\lambda \}\Big\Vert = o_{\cP}(1),\notag
\end{align}
where $\mu_j,\mu_{j+R_1}$ are defined in equations \eqref{eq:weight_j}-\eqref{eq:weight_j_pl_J11} for $\theta \in \theta_n + \thetnprime$.
Replacing $\hat\mu_{n,j}=1-\hat\mu_{n,j+R_1}$ and $\mu_j=1-\mu_{j+R_1}$ in the definition of $\Delta(\mu,\hat{\mu})$, we have
\begin{align}
\Delta(\mu,\hat{\mu})
\le \big\vert  &\hat\mu_{n,j+R_1}(\theta'_n) - \mu_{j+R_1}(\theta'_n)\Big\vert \big\Vert\{\mathbb{G}_{n,j+R_1}^{b}(\theta'_n)+\rho \hat{D}
_{n,j+R_1}(\theta'_n)\lambda\} +\{\mathbb{G}_{n,j}^{b}(\theta'_n)+\rho \hat{D}
_{n,j}(\theta'_n)\lambda\} \Big\Vert .\label{eq:const_muhat_equal_mu}
\end{align}
If both $\pi_j \in (-\infty,0],\pi_{j+R_1}\in (-\infty,0]$, the result follows by the fact that $\lambda \in B^d_{n,\rho}$ and $\hat\mu_{n,j},\hat\mu_{n,j+R_1},\mu_j,\mu_{j+R_1}$ are bounded in $[0,1]$, by Lemma \ref{lem:pair}-(3)-(4), and by Assumption \ref{as:momP_KMS}-(i). The rest of the argument follows similarly as for the case of hard-threshold GMS.
\color{black}
\end{proof}

\begin{lemma}\label{lem:res}
Let Assumptions \ref{as:momP_AS}, \ref{as:GMS}, \ref{as:momP_KMS}, and \ref{as:bcs1} hold. 
Let $(P_n,\theta_n)$ be the sequence satisfying  \eqref{eq:subseq1}-\eqref{eq:subseq3}, let $\mathcal J^*$ be defined as in \eqref{eq:Jstar}, and assume that $\cJ^* \neq \emptyset$. 
	Then, for any 
	$\varepsilon,\eta>0$ and $\theta'_n\in \thetnprime$, there exists $N'\in\mathbb N$ and $N^{''} \in\mathbb N$ such that for all $n \ge \max\{N',N^{''}\}$,
\begin{align}
	\mathbf{P}\Bigg(\sup_{\lambda\in  B^d } \left\vert \max_{j=1,\cdots,J}(u^*_{n,j,\theta_{n}}(\lambda)-c_{n}^{*})-\max_{j=1,\cdots,J}(\mathfrak{w}_{j}^*(\lambda)-c_{\pi^*}) \right\vert\ge\varepsilon\Bigg)<\eta,\label{eq:res_u_w}\\
	\mathbf{\tilde{P}}\Bigg(\sup_{\lambda\in  B^d} \left\vert \max_{j=1,\cdots,J}\tilde{\mathfrak{w}}_{j}(\lambda)-\max_{j=1,\cdots,J}\tilde v^I_{n,j,\theta'_{n}(\lambda)} \right\vert\ge\varepsilon \Bigg)<\eta,~w.p. 1,\label{eq:res_w_v}
\end{align}
where the functions $u^*_n,\mathfrak w^*,\tilde v_n,\tilde{\mathfrak w}$ are defined in equations \eqref{eq:ujstar},\eqref{eq:wjstar}, \eqref{cv_conv1}, and \eqref{cv_conv2}. 
\end{lemma}

\begin{proof}
We first establish \eqref{eq:res_u_w}. By definition,
 $\pi_{1,j}^*= -\infty$ for all $j\notin\mathcal J^*$ and therefore
\begin{align}
	&~\mathbf{P}\Big(\sup_{\lambda\in  B^d }|\max_{j=1,\cdots,J}(u^*_{n,j,\theta_{n}}(\lambda)-c^{*}_{n})-\max_{j=1,\cdots,J}(\mathfrak{w}_{j}^*(\lambda)-c_{\pi^*})|\ge\varepsilon \Big)\\
	=&~\mathbf{P}\Big(\sup_{\lambda\in  B^d }|\max_{j\in\mathcal J^*}(u^*_{n,j,\theta_{n}}(\lambda)-c^{*}_{n})-\max_{j\in\mathcal J^*}(\mathfrak{w}_{j}^*(\lambda)-c_{\pi^*})|\ge\varepsilon \Big).\label{eq:res4}
\end{align}
Hence, for the conclusion of the lemma, it suffices to show, for any $\varepsilon>0$,
\[\lim_{n\to\infty}\mathbf{P}\Big(\sup_{\lambda\in  B^d }|\max_{j\in\mathcal J^*}(u^*_{n,j,\theta_{n}}(\lambda)-c^{*}_{n})-\max_{j\in\mathcal J^*}(\mathfrak{w}_{j}^*(\lambda)-c_{\pi^*})|\ge\varepsilon \Big)=0.\]

For each $\lambda\in\mathbb R^d$, define $r_{n,j,\theta_{n}}(\lambda)\equiv (u^*_{n,j,\theta_{n}}(\lambda)-c^{*}_{n})-(\mathfrak{w}_{j}^*(\lambda)-c_{n})$.
Using 
the fact that $\pi_{1,j}^*=0$ for $j \in \cJ^*$, and the triangle and Cauchy-Schwarz inequalities, for any $\lambda\in B^d\cap \frac{\sqrt{n}}{\rho }(\Theta-\theta_{n})$ and $j \in \cJ^*$, we have
\begin{align}
 |r_{n,j,\theta_{n}}(\lambda)|
 &\le |\mathbb G^*_{n,j}(\theta_{n}+\tlarhon)-\HH^*_j| +\rho \|D_{P_{n},j}(\bar\theta_{n})-D_j\|\|\lambda\|+|\mathbb G^*_{n,j}(\theta_{n}+\tlarhon)+ \rho  D_{P_{n},j}(\bar\theta_{n})\lambda|~\eta_{n,j}^* +|c^{*}_{n}-c_{\pi^*}|\notag\\
&= |\mathbb G^*_{n,j}(\theta_{n}+\tlarhon)-\HH^*_j|+o(1)+\{O_{\mathcal P}(1)+O(1)\})\eta_{n,j}^*+o_{\cP}(1)\notag\\
&=o_{\cP}(1)\label{eq:res5}
\end{align}
where the first equality follows from $\|\lambda\|\le\sqrt d$, $D_{P_n}(\bar\theta_n)\to D$ due to $D_{P_n}(\theta_n)\to D$, Assumption \ref{as:momP_KMS}-(ii), and $\bar\theta_{n}$ being a mean value between $\theta_{n}$ and $\theta_{n}+\lambda \rho /\sqrt{n}$. We also note that $\|\G_{n,j}(\theta+\lambda/\sqrt n)\|=O_{\cP}(1)$, $\|D_{P,j}(\theta)\|$ being uniformly bounded for $\theta\in\Theta_I(P)$ (Assumption \ref{as:momP_KMS}-(i)), and $c^*_n\stackrel{a.s.}{\to}c_{\pi^*}.$
The last equality follows from $\mathbb G^*_{n,j}(\theta_{n}+\larhon)-\HH^*_j\stackrel{a.s.}{\to}0$ and $\sup_{\theta\in\Theta}|\eta_{n,j}(\theta)|=o_{\mathcal P}(1)$ by Lemma \ref{lem:eta_conv}.

We note that when paired inequalities are merged, for each $j=1,\dots,R_1$ such that $\pi^*_{1,j}=0=\pi^*_{1,j+R_1}$ we have that $\vert \tilde{\mu}_j -\mu_j \vert = o_\cP(1)$ because $\sup_{\theta \in \Theta}\vert\eta_j(\theta)\vert=o_\cP(1)$, where $\tilde{\mu}_j $ and $\mu_j$ were defined in \eqref{eq:tilde_weight_j}-\eqref{eq:tilde_weight_j_pl_J11} and \eqref{eq:weight_j}-\eqref{eq:weight_j_pl_J11} respectively.

By \eqref{eq:res5} and the fact that $j \in \cJ^*$, we have
\begin{align}
	\sup_{\lambda\in  B^d }|\max_{j\in\mathcal J^*}(u^*_{n,j,\theta_{n}}(\lambda)-c^{*}_{n})-\max_{j\in\mathcal J^*}(\mathfrak{w}_{j}^*(\lambda)-c_{\pi^*})|
	\le \sup_{\lambda\in  B^d }\max_{j\in\mathcal J^*}|r_{n,j,\theta_{n}}(\lambda)|=o_{\mathcal P}(1).\label{eq:res6}
\end{align}
The conclusion of the lemma then follows from \eqref{eq:res4} and \eqref{eq:res6}.

The result in \eqref{eq:res_w_v} follows from similar arguments.
\end{proof}

\begin{lemma}\label{lem:Jstar}
Let Assumptions \ref{as:momP_AS}, \ref{as:GMS}, \ref{as:momP_KMS}, and \ref{as:bcs1} hold. 
Given a sequence $\{Q_n,\vartheta_n\} \in \{(P,\theta):P \in \cP,\theta \in \Theta_I(P)\}$ such that $\lim_{n \to \infty} \kappa_n^{-1}\sqrt{n}\gamma_{1,Q_n,j}(\vartheta_n)$ exists for each $j=1,\dots,J$, let $\chi_j(\{Q_n,\vartheta_n\})$ be a function of the sequence $\{Q_n,\vartheta_n\}$ defined as
\begin{align}
\label{eq:chi}
\chi_j(\{Q_n,\vartheta_n\}) \equiv \left\{ \begin{array}{ll}
	0, & \mathrm{if}~ \lim_{n \to \infty} \kappa_n^{-1}\sqrt{n}\gamma_{1,Q_n,j}(\vartheta_n) = 0, \\ 
	-\infty, & \mathrm{if}~ \lim_{n \to \infty} \kappa_n^{-1}\sqrt{n}\gamma_{1,Q_n,j}(\vartheta_n) < 0.
	\end{array} 
	\right. 
\end{align} 
Then for any $\theta^\prime_n \in \theta_n +\frac{\rho }{\sqrt n}B^d$ for all $n$, one has:  (i) $\kappa_n^{-1}\sqrt{n}\gamma_{1,P_n,j}(\theta_n)-\kappa_n^{-1}\sqrt{n}\gamma_{1,P_n,j}(\theta_n^\prime) =o(1)$; (ii) $\chi(\{P_n,\theta_n\})=\chi(\{P_n,\theta_n^\prime\})=\pi^*_{1,j}$; and (iii) $\kappa_n^{-1}\frac{\sqrt n \bar{m}_{n,j}(\theta^\prime_{n})}{\hat{\sigma}_{n,j}(\theta^\prime_{n})}-\kappa_n^{-1}\frac{\sqrt nE_{P_n}[m_j(X_i,\theta^\prime_{n})]}{\sigma_{P_n,j}(\theta^\prime_{n})} =o_\cP(1)$.
\end{lemma}

\begin{proof}
For (i), the mean value theorem yields 
\begin{multline}
\sup_{P \in \cP}\sup_{\theta \in \Theta_I(P),\theta'\in\theta+\rho /\sqrt nB^d}\Bigg\vert \frac{\sqrt{n}E_P(m_j(X,\theta))}{\kappa_n \sigma_{P,j}(\theta)}-\frac{\sqrt{n}E_P(m_j(X,\theta^\prime))}{\kappa_n \sigma_{P,j}(\theta^\prime)} \Bigg\vert \\
\le \sup_{P \in \cP}\sup_{\theta \in \Theta_I(P),\theta'\in\theta+\rho /\sqrt nB^d} \frac{\sqrt{n}\Vert D_{P,j}(\tilde{\theta})\Vert \Vert \theta^\prime-\theta \Vert}{\kappa_n} =o(1),\label{eq:pop_Jstar_1}
\end{multline}
where $\tilde{\theta}$ represents a mean value that lies componentwise between $\theta$ and $\theta'$ and where we used the fact that $D_{P,j}(\theta)$ is Lipschitz continuous and $\sup_{P \in \cP}\sup_{\theta \in \Theta_I(P)}\Vert D_{P,j}(\theta) \Vert \le \bar{M}$. Result (ii) then follows immediately from \eqref{eq:chi}. For (iii), note that 
\begin{multline}
\sup_{\theta'_n\in\theta_n+\rho /\sqrt nB^d}\Big|\kappa_n^{-1}\frac{\sqrt n \bar{m}_{n,j}(\theta_n^\prime)}{\hat{\sigma}_{n,j}(\theta_n^\prime)}-\kappa_n^{-1}\frac{\sqrt nE_{P_n}[m_j(X_i,\theta_n^\prime)]}{\sigma_{P_n,j}(\theta_n^\prime)}\Big|\\
\le \sup_{\theta_n'\in\theta_n+\rho /\sqrt nB^d} \Big|\kappa_n^{-1}\frac{\sqrt n (\bar{m}_{n,j}(\theta_n^\prime) - E_{P_n}[m_j(X_i,\theta_n^\prime)])}{\sigma_{n,j}(\theta_n^\prime)} (1+\eta_{n,j}(\theta_n^\prime))
+\kappa_n^{-1}\frac{\sqrt nE_{P_n}[m_j(X_i,\theta_n^\prime)]}{\sigma_{P_n,j}(\theta_n^\prime)}\eta_{n,j}(\theta_n^\prime) \Big|\\
\le \sup_{\theta_n'\in\theta_n+\rho /\sqrt nB^d}|\kappa_n^{-1}\mathbb G_n(\theta_n')(1+\eta_{n,j}(\theta_n^\prime))| + \Big| \frac{\sqrt nE_{P_n}[m_j(X_i,\theta_n^\prime)]}{\kappa_n\sigma_{P_n,j}(\theta_n^\prime)}\eta_{n,j}(\theta_n^\prime) \Big| =o_{\mathcal P}(1) ,	\label{eq:pop_Jstar_2}
\end{multline}
where the last equality follows from $\sup_{\theta\in\Theta}|\mathbb G_{n}(\theta)|=O_{\mathcal P}(1)$ due to asymptotic tightness of $\{\mathbb G_{n}\}$ (uniformly in $P$) by Lemma D.1 in \cite{BCS14_misp}, Theorem 3.6.1 and Lemma 1.3.8 in \cite{Vaart_Wellner2000aBK},  and $\sup_{\theta\in\Theta}|\eta_{n,j}(\theta)|=o_{\mathcal P}(1)$ by Lemma \ref{lem:eta_conv}-(i).
\end{proof}

\begin{lemma}\label{lem:empt}
Let Assumptions \ref{as:momP_AS}, \ref{as:GMS}, \ref{as:correlation}, \ref{as:momP_KMS}, and \ref{as:bcs1} hold.  
	 For any $\theta_n^\prime \in \thetnprime$,
	\begin{itemize}
	\item[(i)]
For any $\eta>0$, there exist $\delta>0$   such that
	\begin{align}
		\sup_{c \ge 0}\Pr(\{\mathfrak W(c)\ne\emptyset\}\cap \{\mathfrak W^{-\delta}(c)=\emptyset\})<\eta.\label{eq:emptiness1}
		\end{align}
Moreover, for any $\eta>0$, there exist $\delta>0$ and $N\in\mathbb N$  such that
	\begin{align}
		\sup_{c \ge 0}P_n^*(\{V^I_n(\theta_n^\prime,c)\ne\emptyset\}\cap \{V_n^{I,-\delta}(\theta_n^\prime,c)=\emptyset\})<\eta,~\forall n\ge N.\label{eq:emptiness1_boot}
		\end{align}
	\item[(ii)] Fix $\underline{c}>0$ and redefine
	\begin{align}
\mathfrak W^{-\delta}(c)
\equiv \big\{ \lambda\in \mathfrak B^d_\rho: p^\prime \lambda =0 &\cap \mathfrak w_j(\lambda)\le c-\delta , \: \forall j=1,\dots,J\big\}, \label{eq:Wdelta_lb_c}
	\end{align}
and
	\begin{align}
V_n^{I,-\delta}(\theta_n^\prime,c)
\equiv \big\{ \lambda\in B^d_{n,\rho}: p^\prime \lambda =0 &\cap v^I_{n,j,\theta_n^\prime}(\lambda)\le c-\delta, \: \forall j=1,\dots,J\big\}. \label{eq:Vdelta_lb_c}
	\end{align}
	Then for any $\eta>0$, there exists $\delta>0$   such that
	\begin{align}
		\sup_{c \ge \underline{c}}\Pr(\{\mathfrak W(c)\ne\emptyset\}\cap \{\mathfrak W^{-\delta}(c)=\emptyset\})<\eta.\label{eq:emptiness1_lb_c}
		\end{align}
		with $\mathfrak W^{-\delta}(c)$ defined in \eqref{eq:Wdelta_lb_c}.
Moreover, for any $\eta>0$, there exist $\delta>0$ and $N\in\mathbb N$  such that
	\begin{align}
		\sup_{c \ge \underline{c}}P_n^*(\{V^I_n(\theta_n^\prime,c)\ne\emptyset\}\cap \{V_n^{I,-\delta}(\theta_n^\prime,c)=\emptyset\})<\eta,~\forall n\ge N,\label{eq:emptiness1_boot_lb_c}
		\end{align}
		with $V_n^{-\delta }(\theta_n^\prime,c)$ defined in \eqref{eq:Vdelta_lb_c}.
\end{itemize}	 
\end{lemma}

\begin{proof}
We first show \eqref{eq:emptiness1}. If $\cJ^*=\emptyset$, with $\cJ^*$ as defined in \eqref{eq:Jstar}, then the result is immediate. Assume then that $\cJ^* \neq \emptyset$. Any inequality indexed by $j \notin \cJ^*$ is satisfied with probability approaching one by similar arguments as in \eqref{eq:pi_star_gamma_vanishes} (both with $c$ and with $c-\delta$). Hence, one could argue for sets $\mathfrak W(c),\mathfrak W^{-\delta}(c)$ defined as in equations \eqref{eq:set_W_Lpop} and \eqref{eq:Wdelta} but with $j \in \cJ^*$. To keep the notation simple, below we argue as if all $j=1,\dots,J$ belong to $\cJ^*$.
Let $c \ge 0$ be given. Let $g$ be a $J+2d+2$ vector with entries
\begin{align}
\label{eq:int1}
g_j =\left\{ \begin{array}{lll}
 c - \HH_j, &  j=1,\dots,J,\\
1, &  j=J+1,\dots,J+2d,\\
0, & j=J+2d+1,J+2d+2,
	\end{array} 
	\right. 
\end{align}
recalling that $\pi_{1,j}^*=0$ for $j=J_1+1,\cdots,J$. Let $\tau$ be a $(J+2d+2)$ vector with entries
	\begin{align}
	\label{eq:tau}
	\tau_j = \left\{ \begin{array}{ll}
	1, & j =1,\dots,J_1, \\ 
	0, & j =J_1+1,\dots,J+2d+2.
	\end{array} 
	\right. 
	\end{align}
	Then we can express the sets of interest as
\begin{align}
\mathfrak W(c) &= \{\lambda: K\lambda \le g\},\label{eq:W:matrix}\\
\mathfrak W^{-\delta}(c)&= \{\lambda: K\lambda \le g-\delta \tau\}.\label{eq:W_delta:matrix}
\end{align}
By Farkas' Lemma, e.g. \cite[Theorem 22.1]{Rockafellar1970a}, a solution to the system of linear inequalities in \eqref{eq:W:matrix} exists if and only if for all $\mu \in \R^{J+2d+2}_+$ such that $\mu^{\prime}K=0$, one has $\mu^{\prime}g \ge 0$. Similarly, a solution to the system of linear inequalities in \eqref{eq:W_delta:matrix}
exists if and only if for all $\mu \in \R^{J+2d+2}$ such that $\mu^{\prime}K=0$, one has $\mu^{\prime}(g-\delta \tau) \ge 0$. Define
\begin{align}
	\mathcal M\equiv \{ \mu\in\mathbb R^{J+2d+2}_+:\mu' K=0\}.\label{eq:int3}
\end{align}
Then, one may write
\begin{align}
	&~\Pr(\{\mathfrak W(c)\ne \emptyset \} \cap \{W^{-\delta}(\theta_n^\prime,c)= \emptyset \} )\notag\\
=&~\Pr(\{\mu'g\ge 0, \forall\mu\in \mathcal M\}\cap \{\mu'(g-\delta \tau)< 0, \exists\mu\in \mathcal M\})\notag \\
= &~\Pr(\{\mu'g\ge 0, \forall\mu\in \mathcal M\}\cap\{ \mu'g< \delta\mu'\tau, \exists\mu\in \mathcal M\})
\label{eq:int4}.
\end{align}
Note that the set $\mathcal M$ is a non-stochastic polyhedral cone which may change with $n$. By Minkowski-Weyl's theorem (see, e.g. \cite[Theorem 3.52]{Rockafellar_Wets2005aBK}), for each $n$ there exist $\{\nu^t\in \mathcal M,t=1,\cdots, T\}$, with $T<\infty$ a constant that depends only on $J$ and $d$, such that
any $\mu\in\mathcal M$ can be represented as


\begin{align}
\mu=b \sum_{t=1}^T a_t\nu^t,\label{eq:int5}
\end{align}
where $b>0$ and $a_t\ge 0, ~t=1,\dots,T, ~\sum_{t=1}^T a_t=1$. Hence, if $\mu\in\mathcal M$ satisfies $ \mu'g< \delta\mu'\tau$, denoting $\nu^{t \prime}$ the transpose of vector $\nu^t$, we have
\begin{align}
\sum_{t=1}^T	a_t\nu^{t \prime}g<\delta\sum_{t=1}^T	a_t\nu^{t \prime}\tau.\label{eq:int6}
\end{align}
However, due to $a_t\ge 0,\forall t$ and $\nu^t\in\mathcal M$, this means  $\nu^{t\prime}g< \delta\nu^{t \prime}\tau$ for some $t\in\{1,\dots,T\}.$ Furthermore, since $\nu^t\in\mathcal M$, we have $0\le \nu^{t\prime}g$.
Therefore,
\begin{multline}
{\Pr}\left(\{\mu'g\ge 0, \forall\mu\in \mathcal M\}\cap\{\mu'g< \delta \mu^\prime \tau, \exists\mu\in \mathcal M\}\right)\\
	\le {\Pr}\left(0\le\nu^{t\prime}g< \delta\nu^{t\prime}\tau,\exists t\in\{1,\cdots,T\}\right)
	\le \sum_{t=1}^T {\Pr}\left(0\le\nu^{t\prime}g< \delta\nu^{t\prime}\tau \right).\label{eq:int7}
	\end{multline}

\noindent \textbf{Case 1.} Consider first any $t =1,\dots,T$ such that $\nu^t$ assigns positive weight only to constraints in $\{J+1,\dots,J+2d+2\}$. Then
\begin{align*}
\nu^{t \prime} g&=\sum_{j=J+1}^{J+2d}\nu^t_{j},\\
\delta\nu^{t \prime} \tau &=\delta\sum_{j=J+1}^{J+2d+2}\nu^t_{j} \tau_j=0,
\end{align*}
where the last equality follows by \eqref{eq:tau}.
Therefore 
${\Pr}\left(0\le\nu^{t\prime}g< \delta\nu^{t\prime}\tau \right)=0$.

\noindent \textbf{Case 2.} Consider now any $t = 1,\dots,T$ such that $\nu^t$ assigns positive weight also to constraints in $\{1,\dots,J\}$. 
Recall that indices $j=J_1+1,\dots,J_1+2J_2$ correspond to moment equalities, each of which is written as two moment inequalities, therefore yielding a total of $2J_2$ inequalities with $D_{j+J_2} = - D_{j}$ for $j =J_1+1,\dots,J_1+J_2$, and:
\begin{equation}
g = \left\{\begin{array}{ll}
 c - \HH_j & j =J_1+1,\dots,J_1+J_2, \\
 c + \HH_{j-J_2} & j =J_1+J_2+1,\dots,J.
\end{array} \label{eq:int8}
\right.
\end{equation}
For each $\nu^t$, \eqref{eq:int8} implies
\begin{align}
\sum_{j=J_1+1}^{J_1+2J_2} \nu^t_j g_j &=c\sum_{j=J_1+1}^{J_1+2J_2} \nu^t_j + \sum_{j=J_1+1}^{J_1+J_2} (\nu^t_j-\nu^t_{j+J_2})\HH_j.\label{eq:int9}
\end{align}
For each $j=1,\cdots,J_1+J_2$, define
\begin{align}
	\tilde\nu^t_j\equiv\begin{cases}
	\nu^t_j&j=1,\cdots, J_1\\
	\nu^t_j-\nu^t_{j+J_2}&j=J_1+1,\cdots,J_1+J_2.
	\end{cases}.\label{eq:int10}
\end{align}
We then let $\tilde\nu^t\equiv(\tilde\nu^t_{n,1},\cdots,\tilde \nu^t_{n,J_1+J_2})'$ and have
\begin{align}
\nu^{t \prime}g=\sum_{j=1}^{J_1+J_2}\tilde\nu^t_j\HH_j+c\sum_{j=1}^J\nu^t_j+\sum_{j=J+1}^{J+2d}\nu^t_j.\label{eq:int11}
\end{align}

\noindent \textbf{Case 2-a.} Suppose  $\tilde\nu^t\ne 0$. Then, by \eqref{eq:int11}, $\frac{\nu^{t \prime}g}{\nu^{t \prime}\tau}$ is a normal random variable  with variance $(\tilde\nu^{t \prime}\tau)^{-2}\tilde\nu'^t\Omega\tilde\nu^t$.   By Assumption \ref{as:correlation}, there exists a constant $\omega>0$ such that the smallest eigenvalue of $\Omega$ is bounded from below by $\omega$ for all $\theta_n^\prime$.
Hence, letting $\|\cdot\|_p$ denote the $p$-norm in $\mathbb R^{J+2d+2}$, we have
\begin{align}
	\frac{\tilde\nu'^t\Omega\tilde\nu^t}{(\tilde\nu^{t \prime}\tau)^2}\ge
	 \frac{\omega\|\tilde\nu^t\|_2^2}{(J+2d+2)^2\|\tilde\nu^t\|_2^2}\ge  \frac{\omega}{(J+2d+2)^2}.
\end{align}
Therefore, the variance of the normal random variable in \eqref{eq:int7} is uniformly bounded away from 0, which in turn allows one to find $\delta>0$ such that
${\Pr}(0\le\frac{\nu^{t \prime}g}{\nu^{t \prime}\tau} < \delta)\le \eta/T$.

\noindent \textbf{Case 2-b.} Next, consider the case $\tilde\nu^t=0$. Because we are in the case that $\nu^t$ assigns positive weight also to constraints in $\{1,\dots,J\}$, this must be because  $\nu^t_j=0$ for all $j=1,\cdots,J_1$ and $\nu^t_j=\nu^t_{j+J_2}$ for all $j=J_1+1,\cdots,J_1+J_2$, while $\nu^t_j\neq 0$ for some $j=J_1+1,\cdots,J_1+J_2$. Then we have $\sum_{j=1}^J \nu^t_j g\ge 0$, and 
$\sum_{j=1}^J\nu_j^t\tau_j=0$ because $\tau_j=0$ for each $j=J_1+1,\dots,J$. Hence, the argument for the case that $\nu^t$ assigns positive weight only to constraints in $\{J+1,\dots,J+2d+2\}$ applies and again ${\Pr}\left(0\le\nu^{t\prime}g< \delta\nu^{t\prime}\tau \right)=0$. This establishes equation \eqref{eq:emptiness1}.

To see why equation \eqref{eq:emptiness1_boot} holds, observe that the bootstrap distribution is conditional on $X_1,\dots,X_n$. Therefore, the matrix $\hat{K}_{n}$, defined as the matrix in equation (\ref{eq:KP}) but with $\hat{D}_{n}$ replacing $D_{P}$, can be treated as nonstochastic. This implies that the set $\hat{\cM}_n$, defined as the set in equation (\ref{eq:int3}) but with $\hat{K}_{n}$ replacing $K$, can be treated as nonstochastic as well. 

By an application of Lemma D.2.8 in \cite{BCS14_misp} together with Lemma \ref{cor:asrep} (through an argument similar to that following equation \eqref{eq:asrep_invoke}), $\mathbb G^b_{n}\stackrel{d}{\to}\mathbb G_P$ in $l^\infty(\Theta)$ uniformly in $\cP$ conditional on $\{X_1,\cdots,X_n\}$, and by Assumption \ref{as:momP_KMS} $\hat D_n(\theta_n')\stackrel{P_n}{\to} D$, for almost all sample paths.
\color{black}
Set
\begin{align}
g_{P_n,j}(\theta'_n) =\left\{ \begin{array}{lll}
 c -\varphi^*_j(\xi_{n,j}(\theta'_n))- \G^b_{n,j}(\theta'_n), &  j=1,\dots,J,\\
 1, &  j=J+1,\dots,J+2d,\\
0, & j=J+2d+1,J+2d+2,
	\end{array} 
	\right. 
\end{align}
and note that $|\varphi^*_j(\xi_{n,j}(\theta'_n))|<\eta$   for all $j\in\mathcal J^*$, and $\G^b_{n,j}(\theta'_n)|\{X_i\}_{i=1}^\infty\stackrel{d}{\to} N(0,\Omega)$.
Then one can mimic the argument following \eqref{eq:int1} to conclude \eqref{eq:emptiness1_boot}.

The results in \eqref{eq:emptiness1_lb_c}-\eqref{eq:emptiness1_boot_lb_c} 
follow by similar arguments, with proper redefinition of $\tau$ in equation \eqref{eq:tau}.
\end{proof}

\begin{lemma}
\label{lem:exist_sol}
Let Assumptions \ref{as:correlation} and \ref{as:bcs1} hold. 
Let $(P_n,\theta_n)$ have the almost sure representations given in Lemma \ref{lem:val}, let $\mathcal J^*$ be defined as in \eqref{eq:Jstar}, and assume that $\cJ^* \neq \emptyset$.
Let $\widetilde{\cC}$ collect all size $d$ subsets $C$ of $\{1,...,J+2d+2\}$ ordered
lexicographically by their smallest, then second smallest, etc. elements.
Let the random variable $\cC$  equal the first element of $\widetilde{\cC}$ s.t. $\det K^C\neq 0$ and $\lambda^{C}=( K^C)
^{-1}g^C\in \mathfrak W^{* ,-\delta }(0)$ if such an
element exists; else, let $\cC=\{J+1,...,J+d\}$ and $\lambda^{C}= \mathbf{1}_d$, where $\mathbf{1}_d$ denotes a $d$ vector with each entry equal to $1$, and $K$, $g$ and $ \mathfrak W^{* ,-\delta }$ are as defined in Lemma \ref{lem:UnonEmpty}. 
Then, for any $\eta >0$, there exist $0<\varepsilon_\eta<\infty$ and $N \in \N$ s.t. $n\geq N$ implies%
\begin{equation}
\mathbf{P} \left( \mathfrak W^{*,-\delta}(0) \neq 
\emptyset ,\left\vert \det K^{\cC}\right\vert \le \varepsilon_\eta\right) \leq \eta .\label{eq:lem_det_bound}
\end{equation}
\end{lemma}
\begin{proof}
We bound the probability in \eqref{eq:lem_det_bound} as follows:
\begin{eqnarray}
\mathbf{P} \left( \mathfrak W^{*,-\delta}(0) \neq \emptyset
,\left\vert \det K^{\cC}\right\vert \le \varepsilon_\eta\right) 
&\leq &\mathbf{P} \left(\exists C\in \widetilde{\mathcal{C}}:\lambda
^{C}\in B^{d},\left\vert \det K^{C}\right\vert \le \varepsilon_\eta\right)  \\
&\leq &\sum_{\substack{ C\in \widetilde{\mathcal{C}}:  \left\vert
\det K^C\right\vert \le \varepsilon_\eta}}\mathbf{P} \left( \lambda
^{C}\in B^{d}\right) \\
&\leq &\sum_{\substack{ C\in \widetilde{\mathcal{C}}:  \vert
\alpha^{C}\vert \le \varepsilon_\eta^{2/d}}}\mathbf{P} \left( \lambda
^{C}\in B^{d}\right), \label{eq:compare_Bd}
\end{eqnarray}
where $\alpha^{C}$ denote the smallest eigenvalue of $K^{C}K^{C\prime}$.
Here, the first inequality holds because $\mathfrak W^{*,-\delta}\subseteq B^{d}$ and so the event in the first probability
implies the event in the next one; the second inequality is Boolean algebra; the last inequality follows because $|\det K^C| \ge |\alpha^{C}|^{d/2}$.
Noting that $\widetilde{\mathcal{C}}$ has $\tbinom{J+2d+2}{d%
}$ elements, it suffices to show that%
\begin{equation*}
\left\vert 
\alpha^{C} \right\vert \le \varepsilon_\eta^{2/d}\Longrightarrow \mathbf{P} \left( \lambda ^{C}\in B^{d}\right) \leq \overline{%
\eta } \equiv \frac{\eta }{\tbinom{J+2d+2}{d}}.
\end{equation*}%
Thus, fix $C \in \widetilde{\mathcal{C}}$. Let $q^C$ denote the eigenvector associated with $\alpha^{C}$ and recall
that because $K^C K^{C\prime}$ is symmetric, $\Vert q^C\Vert=1$. Thus the claim is equivalent to:
\begin{equation}
\label{eq:det_bd1}
\vert q^{C \prime} K^C K^{C \prime} q^C \vert \le \varepsilon_\eta^{2/d} \implies \mathbf{P}{((K^C)^{-1} g^C \in \mathfrak B^d_\rho)} \le\overline{\eta}.
\end{equation}
Now, if $\vert q^{C \prime} K^C K^{C \prime} q^C \vert \le \varepsilon_\eta^{2/d}$ and $(K^C)^{-1} g^C \in \mathfrak B^d_\rho$, then the Cauchy-Schwarz inequality yields 
\begin{equation}
\left\vert q^{C\prime }g_{{P_n}}^{C}\right\vert =\bigl\vert q^{C\prime}K^{C}\left( K^{C}\right) ^{-1}g^{C}\bigr\vert < \sqrt{d}\varepsilon_\eta^{1/d},
\end{equation}
hence
\begin{equation}
\mathbf{P}{((K^C)^{-1} g^C \in \mathfrak B^d_\rho)} \le \mathbf{P}{\left(\vert q^{C \prime} g^C \vert < \sqrt{d}\varepsilon_\eta^{1/d}\right)}.
\end{equation}
If $q^C$ assigns non-zero weight only to non-stochastic constraints, the result follows immediately. If $q^C$ assigns non-zero weight also to stochastic constraints, Assumptions \ref{as:correlation} and \ref{as:bcs1} (iii) yield 
\begin{align}
\eig(\tilde\Omega)&\ge \omega  \notag\\
\implies Var_{\mathbf{P}}(q^{C\prime}g^C) &\ge \omega  \notag\\
\implies \mathbf{P}\left(\vert q^{C \prime} g^C \vert < \sqrt{d}\varepsilon_\eta^{1/d}\right) &= \mathbf{P}\left(- \sqrt{d}\varepsilon_\eta^{1/d} < q^{C \prime} g^C < \sqrt{d}\varepsilon_\eta^{1/d}\right) \notag\\
&< \frac{2 \sqrt{d}\varepsilon_\eta^{1/d}}{\sqrt{2 \omega \pi}},\label{eq:eig_bound}
\end{align}
where the result in \eqref{eq:eig_bound} uses that the density of a normal r.v. is maximized at the expected value. The result follows by choosing
\begin{equation*}
\varepsilon_\eta = \left(\frac{\overline{\eta}\sqrt{2 \omega \pi}}{2 \sqrt{d}}\right)^d.
\end{equation*}
\end{proof}

\begin{lemma}\label{lem:cbd}
Let Assumptions \ref{as:momP_AS}, \ref{as:GMS}, \ref{as:correlation}, \ref{as:momP_KMS}, and \ref{as:bcs1} hold.  If $J_2 \ge d$, then $\exists \underline{c} >0$ s.t.
\begin{align*}
\liminf_{n \to\infty}\inf_{P\in\mathcal P}\inf_{\theta\in \Theta_I(P)}P(c^I_{n}(\theta)\ge \underline c)=1.  
\end{align*}
\end{lemma}
\begin{proof}
Fix any $c \geq 0$ and restrict attention to constraints $\{J_1+1,...,J_1+d,J_1+J_2+1,...,J_1+J_2+d\}$, i.e. the inequalities that jointly correspond to the first $d$ equalities. We separately analyze the case when (i) the corresponding estimated gradients $\{\hat{D}_{n,j}(\theta):j = J_1+1,...,J_1+d\}$ are linearly independent and (ii) they are not. If $\{\hat{D}_{n,j}(\theta):j = J_1+1,...,J_1+d\}$ converge to linearly independent limits, then only the former case occurs infinitely often; else, both may occur infinitely often, and we conduct the argument along two separate subsequences if necessary.

For the remainder of this proof, because the sequence $\{\theta_n\}$ is fixed and plays no direct role in the proof, we suppress dependence of $\hat{D}_{n,j}(\theta)$ and $\G^b_{n,j}(\theta)$ on $\theta$. Also, if $C$ is an index set picking certain constraints, then $\hat{D}_{n}^C$ is the matrix collecting the corresponding estimated gradients, and similarly for $\G_{n}^{b,C}$.

Suppose now case (i), then there exists an index set ${\bar C}\subset \{J_1+1,...,J_1+d,J_1+J_2+1,\dots,J_1+J_2+d\}$ picking one direction of each constraint s.t. $p$ is a positive linear combination of the rows of $\hat{D}_{P}^{\bar C}$. (This choice ensures that a Karush-Kuhn-Tucker condition holds, justifying the step from \eqref{eq:cbd_s1} to \eqref{eq:cbd_s3} below.) Then the  coverage probability $P^*(V^I_n(\theta,c)\ne\emptyset)$ is asymptotically bounded above by
\begin{align}
P^* \Big( \sup_{\lambda \in \rho B^d_{n,\rho}} \Big\{ p^{\prime }\lambda:\hat{D}_{n,j}\lambda \leq c-\G_{n,j}^b,j \in \cJ^{*}\Big\} \geq 0 \Big) \leq &~ P^* \Big( \sup_{\lambda \in \mathbb{R}^d} \Big\{ p^{\prime }\lambda:\hat{D}_{n,j}\lambda \leq c-\G_{n,j}^b,j \in \bar C\Big\} \geq 0 \Big) \label{eq:cbd_s1}\\
=&~ P^* \Big(p^{\prime} (\hat{D}_{n}^{\bar C})^{-1}(c\mathbf{1}_d-\G^{b,\bar C}_n) \geq 0 \Big) \label{eq:cbd_s3}\\
=&~ P^* \bigg( \frac{p^{\prime} (\hat{D}_{n}^{\bar C})^{-1}(c\mathbf{1}_d-\G^{b,\bar C}_n)}{\sqrt{p^\prime(\hat{D}_{n}^{\bar C})^{-1}\Omega_P^C (\hat{D}_{n}^{\bar C})^{-1}p}} \geq 0 \bigg)\label{eq:cbd_s4} \\
=&~ P^* \bigg( \frac{p^{\prime} adj(\hat{D}_{n}^{\bar C})(c\mathbf{1}_d-\G^{b,\bar C}_n)}{\sqrt{p^\prime(adj(\hat{D}_{n}^{\bar C})\Omega_P^C adj(\hat{D}_{n}^{\bar C})p}} \geq 0 \bigg)\label{eq:cbd_s4b} \\
=&~ \Phi \bigg(\frac{p^\prime adj(\hat{D}_{n}^{\bar C})c\mathbf{1}_d}{\sqrt{p^\prime(adj(\hat{D}_{n}^{\bar C})\Omega_P^C adj(\hat{D}_{n}^{\bar C})p}} \bigg) + o_{\cP}(1) \label{eq:cbd_s5}\\
\leq &~ \Phi (d \omega^{-1/2} c) + o_{\cP}(1).\label{eq:cbd_s6}
\end{align}
Here, \eqref{eq:cbd_s1} removes constraints and hence enlarges the feasible set; \eqref{eq:cbd_s3} solves in closed form; \eqref{eq:cbd_s4} divides through by a positive scalar; \eqref{eq:cbd_s4b} eliminates the determinant of $\hat{D}_{n}^{\bar C}$, using that rows of $\hat{D}_{n}^{\bar C}$ can always be rearranged so that the determinant is positive; \eqref{eq:cbd_s5} follows by Assumption \ref{as:bcs1}, using that the term multiplying $\G^{b,\bar C}_{n}$ is $O_{\cP}(1)$; and \eqref{eq:cbd_s6} uses that by Assumption \ref{as:correlation}, there exists a constant $\omega>0$ that does not depend on $\theta$ such that the smallest eigenvalue of $\Omega_P$ is bounded from below by $\omega$. The result follows for any choice of $\underline{c} \in (0,\Phi
	^{-1}(1-\alpha )\times  \omega^{1/2}/d)$.
	
In case (ii), there exists an index set $\bar C \subset \{J_1+2,...,J_1+d,J_1+J_2+2,...,J_1+J_2+d\}$ collecting $d-1$ or fewer linearly independent constraints s.t. $\hat{D}_{n,J_1+1}$ is a positive linear combination of the rows of $\hat{D}_{P}^{\bar C}$. (Note that $\bar C$ cannot contain $J_1+1$ or $J_1+J_2+1$.) One can then write
\begin{align}
&~P^* \Big( \sup_{\lambda \in \rho B^d_{n,\rho}} \Big\{ p^{\prime }\lambda:\hat{D}_{n,j}\lambda \leq c-\G_{n,j}^b,j \in \bar C \cup \{J_1+J_2+1\} \Big\} \geq 0 \Big) \label{eq:cbd_s7}\\
\leq &~P^* \left( \exists \lambda: \hat{D}_{n,j}\lambda \leq c-\G_{n,j}^b,j \in \bar C \cup \{J_1+J_2+1\} \right) \label{eq:cbd_s8}\\
\leq &~P^* \Big(\sup_{\lambda \in \rho B^d_{n,\rho}} \left\{ \hat{D}_{n,J_1+1}\lambda: \hat{D}_{n,j}\lambda \leq c-\G_{n,j}^b,j \in \bar C \right\} \geq \inf_{\lambda \in \rho B^d_{n,\rho}} \left\{ \hat{D}_{n,J_1+1}\lambda: \hat{D}_{n,J_1+J_2+1}\lambda \leq c-\G_{n,J_1+J_2+1}^b\right\} \Big)\label{eq:cbd_s9} \\
= &~P^* \left( \hat{D}_{n,J_1+1} \hat{D}^{\bar C \prime}_{n} (\hat{D}^{\bar C}_{n} \hat{D}^{\bar C \prime}_{n})^{-1} (c\mathbf{1}_{\bar d}-\G^{b,\bar C}_n) \geq -c+\G_{n,J_1+J_2+1}^b \right). \label{eq:cbd_s10} 
\end{align}
Here, the reasoning from \eqref{eq:cbd_s7} to \eqref{eq:cbd_s9} holds because we evaluate the probability of increasingly larger events; in particular, if the event in \eqref{eq:cbd_s9} fails, then the constraint sets corresponding to the $\sup$ and $\inf$ can be separated by a hyperplane with gradient $\hat{D}_{n,J_1+1}$ and so cannot intersect. The last step solves the optimization problems in closed form, using (for the $\sup$) that a Karush-Kuhn-Tucker condition again holds by construction and (for the $\inf$) that $\hat{D}_{n,J_1+J_2+1}=-\hat{D}_{n,J_1+1}$. Expression \eqref{eq:cbd_s10} resembles \eqref{eq:cbd_s4}, and the argument can be concluded in analogy to \eqref{eq:cbd_s4b}-\eqref{eq:cbd_s6}.
\end{proof}


\begin{lemma} \label{lem:pair}
Let Assumptions \ref{as:momP_AS}, \ref{as:GMS}, \ref{as:correlation}-\ref{as:correlation_pair}, \ref{as:momP_KMS}, and \ref{as:bcs1} hold.  
Suppose that both $\pi_{1,j}$ and $\pi_{1,j+R_1}$ are finite, with $\pi_{1,j},~j=1,\dots,J$, defined in \eqref{eq:pi_def}. Let $(P_n,\theta_n)$ be the sequence satisfying the conditions of Lemma \ref{lem:cv_convergence}.
Then for any $\theta^\prime_n \in \thetnprime$,
\begin{itemize}
	\item[(1)] $\sigma^2_{P_{n},j}(\theta^\prime_n)/\sigma^2_{P_{n},{j+R_1}}(\theta^\prime_n) \to 1$ for $j=1,\cdots,R_1.$
	\item[(2)] $Corr_{P_{n}}(m_j(X_i,\theta^\prime_n),m_{j+R_1}(X_i,\theta^\prime_n))\to -1$ for $j=1,\cdots,R_1.$
	\item[(3)]
	$|\mathbb G_{n,j}(\theta^\prime_n)+\mathbb G_{n,j+{R_1}}(\theta^\prime_n)|\stackrel{P_n}{\to}0,$ and $|\mathbb G_{n,j}^b(\theta^\prime_n)+\mathbb G_{n,j+{R_1}}^b(\theta^\prime_n)|\stackrel{P^*_n}{\to}0$ for almost all $\{X_i\}_{i=1}^\infty$.
	\item[(4)]
 $\rho \|D_{P_{n},j+R_1}(\theta^\prime_n)+D_{P_{n},j}(\theta^\prime_n)\| \to 0$.
 \end{itemize}
\end{lemma}
\begin{proof}
By Lemma \ref{lem:Jstar}, for each $j$, $\lim_{n \to \infty}\kappa_n^{-1}\frac{\sqrt nE_{P_n}[m_j(X_i,\theta^\prime_{n})]}{\sigma_{P_n,j}(\theta^\prime_{n})} = \pi_{1,j} $, and hence the condition that $\pi_{1,j},\pi_{1,j+R_1}$ are finite is inherited by the limit of the corresponding sequences $\kappa_n^{-1}\frac{\sqrt nE_{P_n}[m_j(X_i,\theta^\prime_{n})]}{\sigma_{P_n,j}(\theta^\prime_{n})}$ and $\kappa_n^{-1}\frac{\sqrt nE_{P_n}[m_{j+J{11}}(X_i,\theta^\prime_{n})]}{\sigma_{P_n,j+J{11}}(\theta^\prime_{n})}$.

We first establish Claims 1 and 2. 
We consider two cases.

\noindent \textbf{Case 1.}
\begin{align}
\lim_{n \to \infty} \frac{\kappa_n}{\sqrt n}\sigma_{P_n,j}(\theta^\prime_n) > 0,
\end{align}
which implies that $\sigma_{P_n,j}(\theta^\prime_n) \to \infty$ at rate $\sqrt n/\kappa_n$ or faster.
Claim 1 then holds because
\begin{align}
\frac{\sigma^2_{P_n,j+R_1}(\theta^\prime_n)}{\sigma^2_{P_n,j}(\theta^\prime_n)}&=\frac{\sigma^2_{P_n,j}(\theta^\prime_n) + Var_{P_n}(t_j(X_i,\theta^\prime_n))+2Cov_{P_n}(m_j(X_i,\theta^\prime_n),t_j(X_i,\theta^\prime_n))}{\sigma^2_{P_n,j}(\theta^\prime_n)}\to 1,\label{decompose}
\end{align}
where the convergence follows because $Var_{P_n}(t_j(X_i,\theta^\prime_n))$ is bounded due to Assumption \ref{as:correlation}-\ref{as:correlation_pair}, 
\[|Cov_{P_n}(m_j(X_i,\theta^\prime_n),t_j(X_i,\theta^\prime_n))/\sigma^2_{P_n,j}(\theta^\prime_n)| \le (Var_{P_n}(t_j(X_i,\theta^\prime_n)))^{1/2}/\sigma_{P_n,j}(\theta^\prime_n),\]
 and the fact that $\sigma_{P_n,j}(\theta^\prime_n) \to \infty$.
A similar argument yields Claim 2.

\noindent \textbf{Case 2.}
\begin{align}
\lim_{n \to \infty}\frac{\kappa_n}{\sqrt n}\sigma_{P_n,j}(\theta^\prime_n) =0.\label{pair3}
\end{align}
In this case, $\pi_{1,j}$ being finite implies that $E_{P_n}m_j(X_i,\theta^\prime_n) \to 0.$ Again using the upper bound on $t_j(X_i,\theta^\prime_n)$ similarly to \eqref{decompose}, it also follows that 
\begin{align}
\lim_{n \to \infty}\frac{\kappa_n}{\sqrt n}\sigma_{P_n,j+R_1}(\theta^\prime_n) =0,
\end{align}
and hence that $E_{P_n}(t_j(X_i,\theta^\prime_n)) \to 0.$
We then have, using Assumption \ref{as:correlation}-\ref{as:correlation_pair} again,
\begin{align}
	Var_{P_n}(t_j(X_i,\theta^\prime_n))&=\int t_j(x,\theta^\prime_n)^2dP_n(x)-E_{P_n}[t_j(X_i,\theta^\prime_n)]^2 \notag\\
	&\le M\int t_j(x,\theta^\prime_n)dP_n(x) -E_{P_n}[t_j(X_i,\theta^\prime_n)]^2\to 0.\label{eq:pair6}
\end{align}
Hence,
\begin{align}
\frac{\sigma^2_{P_n,j+R_1}(\theta^\prime_n)}{\sigma^2_{P_n,j}(\theta^\prime_n)}&=\frac{\sigma^2_{P_n,j}(\theta^\prime_n) + Var_{P_n}(t_j(X_i,\theta^\prime_n))+2Cov_{P_n}(m_j(X_i,\theta^\prime_n),t_j(X_i,\theta^\prime_n))}{\sigma^2_{P_n,j}(\theta^\prime_n)} \notag\\
&\le \frac{\sigma^2_{P_n,j}(\theta^\prime_n) + Var_{P_n}(t_j(X_i,\theta^\prime_n))}{\sigma^2_{P_n,j}(\theta^\prime_n)}+\frac{2(Var_{P_n}(t_j(X_i,\theta^\prime_n)))^{1/2}}{\sigma_{P_n,j}(\theta^\prime_n)} \notag\\
&\to 1,\label{eq:pair6_1}
\end{align}
and the first claim follows.     

To obtain claim 2, note that 
\begin{align}
Corr_{P_n}(m_j(X_i,\theta^\prime_n),m_{j+R_1}(X_i,\theta^\prime_n))&=\frac{-\sigma^2_{P_n,j}(\theta^\prime_n)-Cov_{P_n}(m_{j}(X_i,\theta^\prime_n),t_j(X_i,\theta^\prime_n))}{\sigma_{P_n,j}(\theta^\prime_n)\sigma_{P_n,j+R_1}(\theta^\prime_n)}\notag \\
&\to -1,
\end{align}
where the result follows from \eqref{eq:pair6} and \eqref{eq:pair6_1}.

To establish Claim 3, consider $\mathbb G_{n}$ below. Note that, for $j=1,\cdots,R_1$,
\begin{align}
	\begin{bmatrix}
		\mathbb G_{n,j}(\theta^\prime_n)\\
			\mathbb G_{n,j+{R_1}}(\theta^\prime_n)
	\end{bmatrix}
=\begin{bmatrix}
		\frac{1}{\sqrt n}\frac{\sum_{i=1}^n (m_{j}(X_i,\theta^\prime_n)-E_{P_n}[m_{j}(X_i,\theta^\prime_n)])}{\sigma_{P_n,j}(\theta^\prime_n)}\\
		-\frac{1}{\sqrt n}\frac{\sum_{i=1}^n (m_{j}(X_i,\theta^\prime_n)-E_{P_n}[m_{j}(X_i,\theta^\prime_n)])+\frac{1}{\sqrt n}\sum_{i=1}^n (t_{j}(X_i,\theta^\prime_n)-E_{P_n}[t_{j}(X_i,\theta^\prime_n)])}{\sigma_{P_n,j+R_1}(\theta^\prime_n)}\label{eq:pair8}
	\end{bmatrix}.
\end{align}
Under the conditions of Case 1 above, we immediately obtain
\begin{align}
	|\mathbb G_{n,j}(\theta^\prime_n)+\mathbb G_{n,j+{R_1}}(\theta^\prime_n)|\stackrel{P_n}{\to}0.\label{eq:pair9}
\end{align}
Under the conditions in Case 2 above, $\frac{1}{\sqrt n}\sum_{i=1}^n (t_{j}(X_i,\theta^\prime_n)-E_{P_n}[t_{j}(X_i,\theta^\prime_n)]=o_{\cP}(1)$ due to the variance of this term being equal to $Var_{P_n}(t_{j}(X_i,\theta^\prime_n))\to 0$ and Chebyshev's inequality. Therefore, \eqref{eq:pair9} obtains again.
These results imply that $\HH_j+\HH_{j+R_1}=0, a.s.$ By Lemma \ref{lem:boot_cons}, $\{\mathbb G^b_n\}$ converges in law to the same limit as $\{\mathbb G_n\}$ for almost all sample paths $\{X_i\}_{i=1}^\infty$. This and \eqref{eq:pair9} then imply the second half of Claim 3.

To establish Claim 4, finiteness of $\pi_{1,j}$ and $\pi_{1,j+R_1}$ implies that 
\begin{align}
E_{P_n}\left( \frac{m_j(X,\theta^\prime_n)}{\sigma_{P_n,j}(\theta^\prime_n)} + \frac{m_{j+R_1}(X,\theta^\prime_n)}{\sigma_{P_n,j+R_1}(\theta^\prime_n)}\right) = O_\cP\left(\frac{\kappa_n}{\sqrt n}\right).
\end{align}
Define the $1 \times d$ vector
\begin{align}
q_n \equiv D_{P_n,j+R_1}(\theta^\prime_n)+D_{P_n,j}(\theta^\prime_n).
\end{align}
Suppose by contradiction that
\begin{align*}
\rho q_n \to \varsigma \neq 0,
\end{align*}
where $\Vert \varsigma \Vert$ might be infinite.
Write
\begin{align}
\tilde{r}_n = \frac{q_n^\prime}{\Vert q_n \Vert}.
\end{align}
Let
\begin{align}
r_n = \tilde{r}_n \rho  \kappa_n^2/\sqrt n.
\end{align}
Using a mean value expansion, where $\bar{\theta}_n$ and $\tilde{\theta}_n$ in the expressions below are two potentially different vectors that lie component-wise between $\theta^\prime_n$ and $\theta^\prime_n+r_n$, we obtain
\begin{align}
&~E_{P_n}\left( \frac{m_j(X,\theta^\prime_n+r_n)}{\sigma_{P_n,j}(\theta^\prime_n+r_n)} + \frac{m_{j+R_1}(X,\theta^\prime_n+r_n)}{\sigma_{P_n,j+R_1}(\theta^\prime_n+r_n)}\right) \notag \\
=&~E_{P_n}\left( \frac{m_j(X,\theta^\prime_n)}{\sigma_{P_n,j}(\theta^\prime_n)} + \frac{m_{j+R_1}(X,\theta^\prime_n)}{\sigma_{P_n,j+R_1}(\theta^\prime_n)}\right) +\bigl(D_{P_n,j}(\bar{\theta}_n)+D_{P_n,j+R_1}(\tilde{\theta}_n)\bigr)r_n \notag\\
 =&~ O_\cP(\frac{\kappa_n}{\sqrt n})+\left(D_{P_n,j}(\theta^\prime_n)+D_{P_n,j+R_1}(\theta^\prime_n)\right)r_n+\left(D_{P_n,j}(\bar{\theta}_n)-D_{P_n,j}(\theta^\prime_n)\right)r_n+\bigl(D_{P_n,j+R_1}(\tilde{\theta}_n)-D_{P_n,j+R_1}(\theta^\prime_n)\bigr)r_n \notag \\
 =&~O_\cP(\frac{\kappa_n}{\sqrt n})+ \frac{\rho \kappa_n^2}{\sqrt n} + O_\cP(\frac{\rho ^2\kappa_n^4}{n}).\label{eq:grad_conc}
\end{align}
It then follows that there exists $N \in \N$ such that for all $n \ge N$, the right hand side in \eqref{eq:grad_conc} is strictly greater than zero.

Next, observe that
\begin{align}
&~E_{P_n}\left( \frac{m_j(X,\theta^\prime_n+r_n)}{\sigma_{P_n,j}(\theta^\prime_n+r_n)} + \frac{m_{j+R_1}(X,\theta^\prime_n+r_n)}{\sigma_{P_n,j+R_1}(\theta^\prime_n+r_n)}\right) \notag\\
= &~E_{P_n}\left( \frac{m_j(X,\theta^\prime_n+r_n)}{\sigma_{P_n,j}(\theta^\prime_n+r_n)} + \frac{m_{j+R_1}(X,\theta^\prime_n+r_n)}{\sigma_{P_n,j}(\theta^\prime_n+r_n)}\right) - \left(\frac{\sigma_{P_n,j+R_1}(\theta^\prime_n+r_n)}{\sigma_{P_n,j}(\theta^\prime_n+r_n)}-1 \right) \frac{E_{P_n}(m_{j+R_1}(X,\theta^\prime_n+r_n))}{\sigma_{P_n,j+R_1}(\theta^\prime_n+r_n)}\notag\\
= &~E_{P_n}\left( \frac{m_j(X,\theta^\prime_n+r_n)}{\sigma_{P_n,j}(\theta^\prime_n+r_n)} + \frac{m_{j+R_1}(X,\theta^\prime_n+r_n)}{\sigma_{P_n,j}(\theta^\prime_n+r_n)}\right) - o_\cP(\frac{\rho \kappa_n^2}{\sqrt n}).\label{eq:grad_lhs_neg}
\end{align}
Here, the last step is established as follows.
First, using that $\sigma_{P_n,j}(\theta^\prime_n+r_n)$ is bounded away from zero for $n$ large enough by the continuity of $\sigma(\cdot)$ and Assumption \ref{as:correlation}-\ref{as:correlation_pair}, we have
\begin{align}
\frac{\sigma_{P_n,j+R_1}(\theta^\prime_n+r_n)}{\sigma_{P_n,j}(\theta^\prime_n+r_n)}-1=\frac{\sigma_{P_n,j+R_1}(\theta^\prime_n)}{\sigma_{P_n,j}(\theta^\prime_n)}-1 +o_\cP(1)=o_\cP(1),\label{eq:pair_prod1}
\end{align}
where we used Claim 1.
Second, using Assumption \ref{as:momP_KMS}, we have that
\begin{align}
\frac{E_{P_n}(m_{j+R_1}(X,\theta^\prime_n+r_n))}{\sigma_{P_n,j+R_1}(\theta^\prime_n+r_n)} = \frac{E_{P_n}(m_{j+R_1}(X,\theta^\prime_n))}{\sigma_{P_n,j+R_1}(\theta^\prime_n)}+D_{P_n,j+R_1}(\tilde{\theta}_n) r_n = O_\cP(\frac{\kappa_n}{\sqrt n})+O_\cP(\frac{\rho \kappa_n^2}{\sqrt n}).\label{eq:pair_prod2}
\end{align}
The product of \eqref{eq:pair_prod1} and \eqref{eq:pair_prod2} is therefore $o_\cP(\frac{\rho \kappa_n^2}{\sqrt n})$ and \eqref{eq:grad_lhs_neg} follows.

To conclude the argument, note that for $n$ large enough, $m_{j+R_1}(X,\theta^\prime_n+r_n) \le - m_j(X,\theta^\prime_n+r_n)$ $a.s.$ because for any $\theta_n \in \Theta_I(P_n)$ and $\theta^\prime_n \in \thetnprime$ for $n$ large enough, $\theta^\prime_n+r_n \in \Theta^\epsilon$ and Assumption \ref{as:correlation}-\ref{as:correlation_pair} applies. Therefore, there exists $N \in \N$ such that for all $n \ge N$, the left hand side in \eqref{eq:grad_conc} is strictly less than the right hand side, yielding a contradiction.
\end{proof}

Below, we let $\mathcal R_1=\{1,\cdots,R_1\}$ and $\mathcal R_2=\{R_1+1,\cdots,2R_1\}.$
\begin{lemma}\label{lem:eta_conv}
	Suppose Assumptions \ref{as:momP_AS}, \ref{as:GMS}, and \ref{as:bcs1} hold. For each $\theta\in\Theta$, let $\eta_{n,j}(\theta)=\sigma_{P,j}(\theta)/\hat\sigma_{n,j}(\theta)-1$. Then, (i) for each $j=1,\dots,J_1+J_2$
\begin{align}
	\inf_{P\in\mathcal P}P
	\Big(\sup_{\theta\in\Theta}|\eta_{n,j}(\theta)|
	\to 0\Big)=1.
\end{align}	
(ii)
For any $j=1,\dots,R_1$ 
let
\begin{align}
\hat \sigma_{n,j}^M(\theta)=\hat \sigma_{n,j+R_1}^M(\theta)\equiv\hat \mu_{n,j}(\theta)\hat\sigma_{n,j}(\theta)+(1-\hat \mu_{n,j}(\theta))\hat\sigma_{n,j+R_1}(\theta). \label{eq:def_hat_sigma_M}
\end{align} 
Let $(P_n,\theta_n)$ be a sequence such that $P_n\in\mathcal P$, $\theta_n\in\Theta$ for all $n$, and $\kappa_n^{-1}\sqrt{n}\gamma_{1,P_n,j}(\theta_n) \to \pi_{1j} \in \R_{[-\infty]}$. Let $\mathcal J^*$ be defined as in \eqref{eq:Jstar}.
Then, for any $\eta>0$, there exists $N\in\mathbb N$ such that
\begin{align}
	P_n\Big(\max_{j\in(\mathcal R_1\cup \mathcal R_2)\cap\mathcal J^*}\Big|\frac{\sigma_{P_n,j}(\theta_n)}{\hat\sigma^M_{n,j}(\theta_n)}-1\Big|>\eta\Big)<\eta
\end{align}
for all  $n\ge N$.
\end{lemma}

\begin{proof}
We first show that, for any $\epsilon>0$ and for any $j=1,\dots,J_1+J_2$,
\begin{align}
\inf_{P\in\mathcal P}P
\Big(\sup_{m\ge n}\sup_{\theta\in\Theta}\Big|\frac{\hat\sigma_{n,j}(\theta)}{\sigma_{P,j}(\theta)}-1\Big|\le \epsilon\Big)\to 1.	\label{eq:ec0}
\end{align}
For this, define the following sets:
	\begin{align}
		\dsM_j&\equiv\{m_j(\cdot,\theta)/\sigma_{P,j}(\theta):\theta\in\Theta,P\in\mathcal P\}\label{eq:ec1}\\
		\gcM_j&\equiv\{(m_j(\cdot,\theta)/\sigma_{P,j}(\theta))^2:\theta\in\Theta,P\in\mathcal P\}.\label{eq:ec2}
	\end{align}
By Assumptions \ref{as:momP_AS}-(a), \ref{as:momP_AS} (iv), \ref{as:bcs1} (i), (iii),  and arguing as in the proof of Lemma D.2.2 (and D.2.1) in \cite{BCS14_misp}, it follows that $\gcM_j$ and $\dsM_j$ are Glivenko-Cantelli (GC) classes uniformly in $P\in\mathcal P$ \citep[in the sense of][page 167]{Vaart_Wellner2000aBK}.   

Therefore, for any $\epsilon>0$, 
\begin{align}
\inf_{P\in\mathcal P}&P\Big(	\sup_{m\ge n}\sup_{\theta\in\Theta}\Big|\frac{n^{-1}\sum_{i=1}^n m_{j}(X_i,\theta)^2}{\sigma^2_{P,j}(\theta)}-\frac{E_P[m_j(X,\theta)^2]}{\sigma^2_{P,j}(\theta)}\Big|\le \epsilon\Big)\to 1\label{eq:ec3}\\	
\inf_{P\in\mathcal P}&P\Big(\sup_{m\ge n}\sup_{\theta\in\Theta}\Big|\frac{\bar m_{n,j}(\theta)-E_P[m_j(X,\theta)]}{\sigma_{P,j}(\theta)}\Big|\le \epsilon\Big)\to 1.\label{eq:ec4}
\end{align}
Note that, by  Assumption \ref{as:momP_AS} (iv), $|E_P[m_j(X,\theta)]/\sigma_{P,j}(\theta)|\le M$  for some constant $M>0$ that does not depend on $P$ and $(x^2-y^2)\le |x+y||x-y|\le 2M|x-y|$ for all $x,y\in [-M,M]$. By \eqref{eq:ec4},   for any $\epsilon>0$, it follows that
\begin{align}
	\inf_{P\in\mathcal P}P\Big(\sup_{m\ge n}\sup_{\theta\in\Theta}\Big|\frac{\bar m_{n,j}(\theta)^2-E_P[m_j(X,\theta)]^2}{\sigma^2_{P,j}(\theta)}\Big|\le \epsilon\Big)\to 1.\label{eq:ec5}
\end{align}

By the uniform continuity of $x\mapsto \sqrt x$ on $\mathbb R_+$, for any $\epsilon>0$, there is a constant $\eta>0$ such that
\begin{align}
\Big|\frac{\hat\sigma^2_{n,j}(\theta)}{\sigma^2_{P,j}(\theta)}-1\Big|\le \eta~\Rightarrow \Big|\frac{\hat\sigma_{n,j}(\theta)}{\sigma_{P,j}(\theta)}-1\Big|\le \epsilon.\label{eq:ec6}
\end{align}	
By the definition of $\sigma^2_{P,j}(\theta)$ and the triangle inequality,
\begin{align}
\Big|\frac{\hat\sigma^2_{n,j}(\theta)}{\sigma^2_{P,j}(\theta)}-1\Big|
\le \Big|\frac{n^{-1}\sum_{i=1}^n m(X_i,\theta)^2-E[m_j(X_i,\theta)^2]}{\sigma^2_{P,j}(\theta)}	\Big|+\Big|\frac{\bar m_{n,j}(\theta)^2-E[m_j(X_i,\theta)]^2}{\sigma^2_{P,j}(\theta)}\Big|.\label{eq:ec7}
\end{align}
By \eqref{eq:ec6}-\eqref{eq:ec7}, bounding each of the terms on the right hand side of \eqref{eq:ec7} by $\eta/2$ implies $|\hat\sigma_{n,j}(\theta)/\sigma_{P,j}(\theta)-1|\le \epsilon$.
This, together with \eqref{eq:ec3} and \eqref{eq:ec5}, ensures that, for any $\epsilon>0$, \eqref{eq:ec0} holds.

Note that $|\hat\sigma_{n,j}(\theta)/\sigma_{P,j}(\theta)-1|\le \epsilon$ implies $\hat\sigma_{n,j}(\theta)>0$, and argue as in the proof of Lemma D.2.4 in \cite{BCS14_misp}  to conclude that
\begin{align}
	\inf_{P\in\mathcal P}P
	\Big(\sup_{m\ge n}\sup_{\theta\in\Theta}\Big|\frac{\sigma_{P,j}(\theta)}{\hat\sigma_{n,j}(\theta)}-1\Big|\le \epsilon\Big)\to 1.\label{eq:ec8}
\end{align}
Finally, recall that $\eta_{n,j}(\theta)=\sigma_{P,j}(\theta)/\hat\sigma_{n,j}(\theta)-1$ and note that for any $\epsilon>0$, 
\begin{align}
1&=\lim_{n\to\infty}\inf_{P\in \mathcal P}P
\Big(\sup_{m\ge n}\sup_{\theta\in\Theta}|\eta_{n,j}(\theta)|\le \epsilon\Big)\notag\\
&\le 	\inf_{P\in \mathcal P}\lim_{n\to\infty}P
\Big(\bigcap_{m\ge n}\big\{\sup_{\theta\in\Theta}|\eta_{n,j}(\theta)|\le \epsilon\big\}\Big)\notag\\
&=\inf_{P\in \mathcal P}P
\Big(\lim_{n\to\infty}\bigcap_{m\ge n}\{\sup_{\theta\in\Theta}|\eta_{n,j}(\theta)|\le \epsilon\big\}\Big)\notag\\
&=\inf_{P\in\mathcal P}P
\Big(\sup_{\theta\in\Theta}|\eta_{n,j}(\theta)|\le \epsilon, \text{ for almost all } n\Big),\label{eq:ec9}
\end{align}
where the second equality is due to the continuity of probability with respect to monotone sequences.
 Therefore, the first conclusion of the lemma follows.
 
 \vspace{0.1in}
 \noindent

 (ii) We first give the limit of $\hat \mu_{n,j}(\theta_n)$. 
 Recall the definitions of $\hat\mu_{n,j+R_1}$ and $\hat\mu_{n,j}(\theta_n )$ in \eqref{eq:mu_hat_j}-\eqref{eq:mu_hat_j_pl_J11}.

Note that 
\begin{multline}
\sup_{\theta'_n\in\theta_n+\rho /\sqrt nB^d}\Big|\kappa_n^{-1}\frac{\sqrt n \bar{m}_{n,j}(\theta_n^\prime)}{\hat{\sigma}_{n,j}(\theta_n^\prime)}-\kappa_n^{-1}\frac{\sqrt nE_{P_n}[m_j(X_i,\theta_n^\prime)]}{\sigma_{P_n,j}(\theta_n^\prime)}\Big|\\
\le \sup_{\theta_n'\in\theta_n+\rho /\sqrt nB^d} \Big|\kappa_n^{-1}\frac{\sqrt n (\bar{m}_{n,j}(\theta_n^\prime) - E_{P_n}[m_j(X_i,\theta_n^\prime)])}{\sigma_{n,j}(\theta_n^\prime)} (1+\eta_{n,j}(\theta_n^\prime))
+\kappa_n^{-1}\frac{\sqrt nE_{P_n}[m_j(X_i,\theta_n^\prime)]}{\sigma_{P_n,j}(\theta_n^\prime)}\eta_{n,j}(\theta_n^\prime) \Big|\\
\le \sup_{\theta_n'\in\theta_n+\rho /\sqrt nB^d}|\kappa_n^{-1}\mathbb G_n(\theta_n')(1+\eta_{n,j}(\theta_n^\prime))| + \Big| \frac{\sqrt nE_{P_n}[m_j(X_i,\theta_n^\prime)]}{\kappa_n\sigma_{P_n,j}(\theta_n^\prime)}\eta_{n,j}(\theta_n^\prime) \Big| =o_{\mathcal P}(1) ,
\end{multline}
where the last equality follows from $\sup_{\theta\in\Theta}|\mathbb G_{n}(\theta)|=O_{\mathcal P}(1)$ due to asymptotic tightness of $\{\mathbb G_{n}\}$ (uniformly in $P$) by Lemma D.1 in \cite{BCS14_misp}, Theorem 3.6.1 and Lemma 1.3.8 in \cite{Vaart_Wellner2000aBK},  and $\sup_{\theta\in\Theta}|\eta_{n,j}(\theta)|=o_{\mathcal P}(1)$ by part (i) of this Lemma.
\color{black}
Hence, 
\begin{align}
	\hat\mu_{n,j}(\theta_n) \stackrel{P_n}{\to}1-\min\Big\{\max(0,\frac{\pi_{1,j}}{\pi_{1,j+R_1}+\pi_{1,j}}),1\Big\},
\end{align}
unless $\pi_{1,j+R_1}+\pi_{1,j}=0$ (this case is considered later). 
This implies that if $\pi_{1,j}\in(-\infty,0]$ and $\pi_{1,j+R_1}=-\infty$, one has
 \begin{align}
\hat\mu_{n,j}(\theta_n) \stackrel{P_n}{\to} 1.\label{eq:mulim1}
 \end{align}
Similarly, if $\pi_{1,j}=-\infty$ and $\pi_{1,j+R_1}\in(-\infty,0]$, one has
 \begin{align}
\hat\mu_{n,j+R_1}(\theta_n) \stackrel{P_n}{\to} 1.\label{eq:mulim2}
 \end{align}
 Now, one may write
 \begin{align}
 \frac{\sigma_{P_n,j}(\theta_n)}{\hat\sigma^M_{n,j}(\theta_n)}-1=\frac{\sigma_{P_n,j}(\theta_n)}{\hat\sigma_{n,j}(\theta_n)}\Big(\frac{\hat\sigma_{n,j}(\theta_n)}{\hat\sigma^M_{n,j}(\theta_n)}-1\Big)+\Big(\frac{\sigma_{P_n,j}(\theta_n)}{\hat\sigma_{n,j}(\theta_n)}-1\Big)=O_{P_n}(1)\Big(\frac{\hat\sigma_{n,j}(\theta_n)}{\hat\sigma^M_{n,j}(\theta_n)}-1\Big)+o_{P_n}(1),	\label{eq:sigconv1}
 \end{align}
 where the second equality follows from the first conclusion of the lemma. Hence, for the second conclusion of the lemma, it suffices to show $\hat\sigma_{n,j}(\theta_n)/\hat\sigma^M_{n,j}(\theta_n)-1=o_\cP(1).$ For this, we consider three cases. 
 
Suppose first $j\in\mathcal R_1\cap \mathcal J^*$ and $j+R_1\notin \mathcal J^*$. Then, $\pi^*_{1,j}=0$ and $\pi^*_{1,j+R_1}=-\infty$. Then, 
\begin{align}
	\hat\sigma^M_{n,j}(\theta_n)&=\hat \mu_{n,j}(\theta_n)\hat\sigma_{n,j}(\theta_n)+(1-\hat \mu_{n,j}(\theta_n))\hat\sigma_{n,j+R_1}(\theta_n)\\
	&=(1+o_{ P_n}(1))\hat\sigma_{n,j}(\theta_n)+(1-\hat \mu_{n,j}(\theta_n))O_{P_n}(\hat\sigma_{n,j}(\theta_n)),\label{eq:sigconv2}
\end{align}
where the second equality follows from \eqref{eq:mulim1} and the fact that
\begin{multline}
	\hat\sigma_{n,j+R_1}(\theta_n)=\Big(\hat\sigma_{n,j}^2(\theta_n)+2\widehat{Cov}_n(m_j(X_i,\theta),t_j(X_i,\theta))+\widehat{Var}_{n}(t_j(X_i,\theta))\Big)^{1/2}\\
=\Big(\hat\sigma_{n,j}^2(\theta_n)+O_{P_n}(\hat\sigma_{n,j}(\theta_n))+O_{P_n}(1)\Big)^{1/2}=O_{P_n}(\hat\sigma_{n,j}(\theta_n)),
\end{multline}
where the second equality follows from,  $Var_{P_n}(t_j(X_i,\theta))$ being bounded by Assumption \ref{as:correlation}-(II) and
\begin{align}
	&\widehat{Var}_{n}(t_j(X_i,\theta))=Var_{P_n}(t_j(X_i,\theta))+o_{P_n}(1)\\
	&\widehat{Cov}_n(m_j(X_i,\theta),t_j(X_i,\theta))\le \hat \sigma_{n,j}(\theta_n)\widehat{Var}_{n}(t_j(X_i,\theta))^{1/2},
\end{align}
where the last inequality is due to the Cauchy-Schwarz inequality.

 Therefore,
\begin{align}
\frac{\hat\sigma_{n,j}(\theta_n)}{\hat\sigma^M_{n,j}(\theta_n)}-1=\frac{\hat\sigma_{n,j}(\theta_n)-\hat\sigma^M_{n,j}(\theta_n)}{\hat \sigma^M_{n,j}(\theta_n)}=
\frac{(1-\hat \mu_{n,j}(\theta_n))O_{P_n}(\hat\sigma_{n,j}(\theta_n))}{(1+o_{ P_n}(1))\hat \sigma_{n,j}(\theta_n)+(1-\hat \mu_{n,j}(\theta_n))O_{P_n}(\hat\sigma_{n,j}(\theta_n))}
=o_{P_n}(1),	\label{eq:sigconv3}
\end{align}
where we used $\hat\sigma_{n,j}^{-1}(\theta_n)=O_{P_n}(1)$ by equation \eqref{eq:33primeiii} and part (i) of the lemma. By \eqref{eq:sigconv1} and \eqref{eq:sigconv3}, $\sigma_{P_n,j}(\theta_n)/\hat\sigma^M_{n,j}(\theta_n)-1=o_{P_n}(1)$.
Using a similar argument, the same conclusion follows when $j\in\mathcal R_1,j\notin \mathcal J^*$, but $j+R_1\in\mathcal R_2\cap \mathcal J^*.$  

Now consider the case $j\in\mathcal R_1\cap\mathcal J^*$ and $j+R_1\in\mathcal R_2\cap\mathcal J^*.$ Then, $\pi^*_{1,j}=0$ and $\pi^*_{1,j+R_1}=0.$ In this case, $\hat \mu_{n,j}(\theta_n)\in[0,1]$ for all $n$ and by Lemma \ref{lem:pair} (1),
\begin{align}
	\Big|\frac{\sigma_{P_{n},j}(\theta_n)}{\sigma_{P_{n},{j+R_1}}(\theta_n)} - 1\Big|=o_{P_n}(1),~ \text{ for } j=1,\cdots,R_1,\label{eq:sigconv4}
\end{align}
and therefore,
\begin{align}
\frac{\sigma_{P_n,j}(\theta_n)}{\hat\sigma^M_{n,j}(\theta_n)}-1&=\frac{\sigma_{P_n,j}(\theta_n)-\hat\sigma^M_{n,j}(\theta_n)}{\hat\sigma^M_{n,j}(\theta_n)}\notag\\
&=\frac{[\hat \mu_{n,j}(\theta_n)+(1-\hat \mu_{n,j}(\theta_n))]\sigma_{P_n,j}(\theta_n)-[\hat \mu_{n,j}(\theta_n)\hat\sigma_{n,j}(\theta_n)+(1-\hat \mu_{n,j}(\theta_n))\hat\sigma_{n,j+R_1}(\theta_n)]}{\hat\sigma^M_{n,j}(\theta_n)}\notag\\
&=\frac{\hat \mu_{n,j}(\theta_n)[\sigma_{P_n,j}(\theta_n)-\hat\sigma_{n,j}(\theta_n)]}{\hat\sigma^M_{n,j}(\theta_n)}+\frac{(1-\hat \mu_{n,j}(\theta_n))[\sigma_{P_n,j+R_1}(\theta_n)-\hat\sigma_{n,j+R_1}(\theta_n)+o_{P_n}(1)]}{\hat\sigma^M_{n,j}(\theta_n)},\label{eq:sigconv5}
\end{align}
where the second equality follows from the definition of $\hat \sigma^M_{n,j}(\theta_n)$, and the third equality follows from \eqref{eq:sigconv4} and $\sigma_{P_n,j+R_1}$ bounded away from 0 due to \eqref{eq:33primeiii}. Note that
\begin{align}
	\frac{\hat \mu_{n,j}(\theta_n)[\sigma_{P_n,j}(\theta_n)-\hat\sigma_{n,j}(\theta_n)]}{\hat\sigma^M_{n,j}(\theta_n)}=\hat \mu_{n,j}(\theta_n)\frac{\hat\sigma_{n,j}(\theta_n)}{\hat\sigma^M_{n,j}(\theta_n)}\Big(\frac{\sigma_{P_n,j}(\theta_n)}{\hat\sigma_{n,j}(\theta_n)}-1\Big)=o_{P_n}(1),\label{eq:sigconv6}
\end{align}
where the second equality follows from the first conclusion of the lemma. Similarly,
\begin{multline}
	\frac{(1-\hat \mu_{n,j}(\theta_n))[\sigma_{P_n,j+R_1}(\theta_n)-\hat\sigma_{n,j+R_1}(\theta_n)+o_{P_n}(1)]}{\hat\sigma^M_{n,j}(\theta_n)}
	\\
	=	(1-\hat \mu_{n,j}(\theta_n))\frac{\hat\sigma_{n,j+R_1}(\theta_n)}{\hat\sigma^M_{n,j}(\theta_n)}\Big(\frac{\sigma_{P_n,j+R_1}(\theta_n)}{\hat\sigma_{n,j+R_1}(\theta_n)}-1+o_{P_n}(1)\Big)=o_{P_n}(1).\label{eq:sigconv7}
\end{multline}
By \eqref{eq:sigconv5}-\eqref{eq:sigconv7}, it follows that $\sigma_{P_n,j}(\theta_n)/\hat\sigma^M_{n,j}(\theta_n)-1=o_{P_n}(1)$. Therefore, the second conclusion holds for all subcases.
 \end{proof}
 
\subsection{Lemmas Used to Prove Theorem \ref{cor:eam_conv}}\label{sec:Lemmas_for_EAM}
Let $\{X_i^b\}_{i=1}^n$ denote a bootstrap sample drawn randomly from the empirical distribution.
Define
\begin{align}
	\mathfrak G^b_{n,j}(\theta)&\equiv \frac{1}{\sqrt n}\sum_{i=1}^n \left( m_j(X_i^b,\theta)-\bar m_n(\theta) \right)/\sigma_{P,j}(\theta)\notag\\
	&=\frac{1}{\sqrt n}\sum_{i=1}^n(M_{n,i}-1)m_j(X_i,\theta)/\sigma_{P,j}(\theta),
\end{align}
where $\{M_{n,i}\}_{i=1}^n$ denotes the multinomial weights on the original sample, and we let
 $P^*_n$ denote the conditional distribution of $\{M_{n,i}\}_{i=1}^n$ given the sample path $\{X_i\}_{i=1}^\infty$ (see Appendix \ref{app:asrep} for details on the construction of the bootstrapped empirical process).

\begin{lemma}\label{lem:se_rate}
	(i) Let $\mathcal M_P\equiv\{f:\mathcal X\to\mathbb R:f(\cdot)=\sigma_{P,j}(\theta)^{-1}m_j(\cdot,\theta),\theta\in\Theta,j=1,\cdots,J\}$ and let $F$ be its envelope. Suppose that (i) there exist constants $K,v>0$ that do not depend on $P$ such that
	\begin{align}
	\sup_Q N(\epsilon\|F\|_{L^2_Q},\mathcal M_P,L^2_Q)\le K\epsilon^{-v}	,~0<\epsilon<1,\label{eq:as_euclidean}
	\end{align}
	where the supremum is taken over all discrete distributions; (ii) There exists a positive constant $\gamma>0$ such that 
	\begin{align}
	\|(\theta_1,\tilde\theta_1)-(\theta_2,\tilde\theta_2)\|\le\delta~~\Rightarrow~~\sup_{P\in\mathcal P}\|Q_P(\theta_1,\tilde\theta_1)-Q_P(\theta_2,\tilde\theta_2)\|\le M\delta^\gamma.\label{eq:correl_holder}
	\end{align}
	Let $\delta_n$ be a positive sequence tending to 0 and let $\epsilon_n$ be a positive sequence such that $\epsilon_n/|\delta_n^\gamma\ln\delta_n|\to\infty$ as $n\to\infty$.
		Then, \begin{align}
\sup_{P\in\mathcal P}P\left(\sup_{\|\theta-\theta'\|\le \delta_n}\|\mathbb G_n(\theta)-\mathbb G_n(\theta'))\|>\epsilon_n \right)=o(1).\label{eq:sr1}
\end{align}
Further,
\begin{align}
\lim_{n\to\infty}P^*_n \left(\sup_{\|\theta-\theta'\|\le \delta_n}\|\mathfrak G^b_{n}(\theta)-\mathfrak G^b_{n}(\theta'))\|>\epsilon_n|\{X_i\}_{i=1}^\infty \right)=0.	
\end{align}
for almost all sample paths $\{X_i\}_{i=1}^\infty$ uniformly in $P\in \mathcal P$.
\end{lemma}

\begin{proof}
For the first conclusion of the lemma, it suffices to show that there is a sequence $\{\epsilon_n\}$ such that, uniformly in $P$:
\begin{align}
P\left(\sup_{\|\theta-\theta'\|\le \delta_n}\max_{j=1,\cdots,J}|\mathbb G_{n,j}(\theta)-\mathbb G_{n,j}(\theta')|>\epsilon_n \right)=o(1).
\end{align}
For this purpose, we mostly mimic the argument required to show the stochastic equicontinuity of empirical processes \citep[see e.g.][Ch.2.5]{Vaart_Wellner2000aBK}. 
Before doing so, note that, arguing as in the proof of Lemma D.1 (Part 1) in \cite{BCS14_misp}, one has
\begin{align}
\|\theta-\theta'\|\le \delta_n~~\Rightarrow~~  \varrho_P(\theta,\theta')\le \tilde\delta_n,	\label{eq:sr2}
\end{align}
where $\tilde\delta_n=O(\delta_n^\gamma)$ by assumption.
Define
\begin{align}
\mathcal M_{P,\tilde \delta_n}=\{\sigma_{P,j}(\theta)^{-1}m_j(\cdot,\theta)-\sigma_{P,j}(\theta')^{-1}m_j(\cdot,\theta')| \theta,\theta'\in\Theta,\varrho_P(\theta,\tilde\theta)<\tilde\delta_n,j=1,\cdots,J\}.	\label{eq:sr3}
\end{align}
Define $Z_n(\tilde\delta_n)\equiv \sup_{f\in\mathcal M_{\tilde\delta_n}}|\sqrt n(\mathbb P_n-P)f|.$ Then, by \eqref{eq:sr2}, one has
\begin{align}
P\Biggl(\sup_{\|\theta-\theta'\|\le \delta_n}\max_{j=1,\cdots,J}|\mathbb G_{n,j}(\theta)-\mathbb G_{n,j}(\theta'))|>\epsilon_n)\le P(Z_n(\tilde\delta_n)>\epsilon_n \Biggr).\label{eq:sr4}	
\end{align}
From here, we deal with the supremum of  empirical processes though symmetrization and an application of a maximal inequality.
By Markov's inequality and Lemma 2.3.1 (symmetrization lemma) in \cite{Vaart_Wellner2000aBK}, one has
\begin{align}
P(Z_n(\tilde\delta_n)>\epsilon_n)\le \frac{2}{\epsilon_n}E_{P\times P^W}\left[\sup_{f\in\mathcal M_{P,\tilde\delta_n}}\Bigl\vert\frac{1}{\sqrt n}\sum_{i=1}^nW_i f(X_i)\Bigr\vert\right],\label{eq:sr5}
\end{align}
where $\{W_i\}_{i=1}^n$ are i.i.d. Rademacher random variables independent of $\{X_i\}_{i=1}^\infty$ whose law is denoted by $P^W$. Now, fix the sample path $\{X_i\}_{i=1}^n$, and let $\hat P_n$ be the empirical distribution. By Hoeffding's inequality, the stochastic process $f\mapsto\{n^{-1/2}\sum_{i=1}^nW_if(X_i)\}$ is sub-Gaussian for the $L^2_{\hat P_n}$ seminorm $\|f\|_{L^2_{\hat P_n}}=(n^{-1}\sum_{i=1}^n f(X_i)^2)^{1/2}.$ By the maximal inequality (Corollary 2.2.8) and arguing as in the proof of Theorem 2.5.2 in in \cite{Vaart_Wellner2000aBK}, one then has
\begin{align}
	E_{P^W}\left[\sup_{f\in\mathcal M_{\tilde\delta_n}}\Bigl\vert\frac{1}{\sqrt n}\sum_{i=1}^nW_i f(X_i)\Bigr\vert \right]&\le K\int_0^{\tilde\delta_n}\sqrt{\ln N(\epsilon,\mathcal M_{P,\tilde\delta_n},L^2_{\hat P_n})}d\epsilon\notag\\
	&\le K\int_0^{\tilde\delta_n/\|F\|_{L^2_Q}}\sup_Q\sqrt{\ln N(\epsilon \|F\|_{L^2_Q},\mathcal M_P,L^2_Q)}d\epsilon\notag\\
	&\le K'\int_0^{\tilde\delta_n/\|F\|_{L^2_Q}}\sqrt{-v\ln\epsilon}d\epsilon,\label{eq:sr6}
\end{align}
for some $K'>0$, 
where the last inequality follows from \eqref{eq:as_euclidean}.
Note that $\sqrt{-\ln \epsilon}\le -\ln \epsilon$ for $\epsilon\le \tilde\delta_n/\|F\|_{L^2_Q}$ with $n$ sufficiently large. Hence, 
\begin{align}
E_{P^W}\left[\sup_{f\in\mathcal M_{\tilde\delta_n}}\Bigl\vert\frac{1}{\sqrt n}\sum_{i=1}^nW_i f(X_i)\Bigr\vert\right]&\le K'v^{1/2}\int_0^{\tilde\delta_n/\|F\|_{L^2_Q}} (-\ln\epsilon) d\epsilon 	=K'v^{1/2}(\tilde\delta_n-\tilde\delta_n\ln(\tilde\delta_n)).\label{eq:sr7}
\end{align}
By \eqref{eq:sr5} and taking expectations with respect to $P$ in \eqref{eq:sr7}, it follows that
\begin{align}
P(Z_n(\tilde\delta_n)>\epsilon_n)\le 2K'v^{1/2}(\tilde\delta_n-\tilde\delta_n\ln(\tilde\delta_n))/\epsilon_n=O(\delta_n^\gamma/\epsilon_n)+O(|\delta_n^\gamma\ln(\delta_n)|/\epsilon_n)=o(1),	\label{eq:sr8}
\end{align}
where the last equality follows from the rate condition on $\epsilon_n$. By \eqref{eq:sr4} and \eqref{eq:sr8}, conclude that the first claim of the lemma holds.

For the second claim, define $Z^*_n(\tilde\delta_n)\equiv \sup_{f\in\mathcal M_{\tilde\delta_n}}|\sqrt n(\hat P^*_n-\hat P_n)f|,$ where $\hat P^*_n$ is the empirical distribution of $\{X_i^b\}_{i=1}^n$. Then, by \eqref{eq:sr2}, one has
\begin{align}
P^*_n\left(\sup_{\|\theta-\theta'\|\le \delta_n}\max_{j=1,\cdots,J}|\mathfrak G^b_{n,j}(\theta)-\mathfrak G^b_{n,j}(\theta')|>\epsilon_n\Big|\{X_i\}_{i=1}^\infty\right)\le P^*_n\big(Z^*_n(\tilde\delta_n)>\epsilon_n\big|\{X_i\}_{i=1}^\infty\big).\label{eq:sr9}	
\end{align}
By Markov's inequality and Lemma 2.3.1 (symmetrization lemma) in \cite{Vaart_Wellner2000aBK}, one has
\begin{align}
P^*_n\big(Z^*_n(\tilde\delta_n)>\epsilon_n\big|\{X_i\}_{i=1}^\infty\big)&\le \frac{2}{\epsilon_n}E_{P^*_n\times P^W}\left[\sup_{f\in\mathcal M_{P,\tilde\delta_n}}\Big|\frac{1}{\sqrt n}\sum_{i=1}^nW_i f(X^b_i)\Big| \Bigg|\{X_i\}_{i=1}^\infty\right]\\
&=\frac{2}{\epsilon_n}E_{P^*_n}\left[E_{P^W}\left[\sup_{f\in\mathcal M_{P,\tilde\delta_n}}\Big| \frac{1}{\sqrt n}\sum_{i=1}^nW_i f(X^b_i)\Big| \Bigg|\{X^b_i\},\{X_i\}_{i=1}^\infty\right]\Bigg|\{X_i\}_{i=1}^\infty\right],\label{eq:sr10}
\end{align}
where $\{W_i\}_{i=1}^n$ are i.i.d. Rademacher random variables independent of $\{X_i\}_{i=1}^\infty$ and $\{M_{n,i}\}_{i=1}^n$.
Argue as in \eqref{eq:sr5}-\eqref{eq:sr8}. Then, it follows that
\begin{align*}
P^*_n(Z^*_n(\tilde\delta_n)>\epsilon_n|\{X_i\}_{i=1}^\infty)=O(\delta_n^\gamma/\epsilon_n)+O(-\delta_n^\gamma\ln(\delta_n)/\epsilon_n)=o(1),
\end{align*}
 for almost all sample paths. Hence, the second claim of the lemma follows.
\end{proof}

\begin{lemma}\label{lem:eta_rate}
Suppose Assumptions  \ref{as:momP_AS}, \ref{as:GMS}, and \ref{as:bcs1}  hold.
Let $\mathcal S_P\equiv\{f:\mathcal X\to\mathbb R:f(\cdot)=\sigma_{P,j}(\theta)^{-2}m_j^2(\cdot,\theta),\theta\in\Theta,j=1,\cdots,J\}$ and let $F$ be its envelope. 
(i) If $\mathcal S_P$ is Donsker and pre-Gaussian uniformly in $P\in\mathcal P$, then
\begin{align}
	\sup_{\theta\in\Theta}|\eta_{n,j}(\theta)|^*=O_{\mathcal P}(1/\sqrt n);
\end{align}	
(ii) If $|\sigma_{P,j}(\theta)^{-1}m_j(x,\theta)-\sigma_{P,j}(\theta')^{-1}m_j(x,\theta')|\le \bar M(x)\|\theta-\theta'\|$ with $E_P[\bar M(X)^{2}]<M$ for all $\theta,\theta'\in\Theta$, $x\in\mathcal X$, $j=1,\cdots,J$, and $P\in\mathcal P$,  then, for any $\eta>0$, there exists a constant $C>0$ such that
\begin{align}
\limsup_{n\to\infty}	\sup_{P\in\mathcal P}P\Big(\max_{j=1,\cdots,J}\sup_{\|\theta-\theta'\|<\delta}|\eta_{n,j}(\theta)-\eta_{n,j}(\theta')|>C\delta\Big)<\eta.
\end{align}
\end{lemma}

\begin{proof}
We show the claim by first showing that, for any $\delta>0$, there exist $M>0$ and $N\in\mathbb N$ such that
\begin{align}
\inf_{P\in\mathcal P}P^\infty\Big(\sup_{\theta\in\Theta}\Big|\frac{\hat\sigma_{n,j}(\theta)}{\sigma_{P,j}(\theta)}-1\Big|\le M/\sqrt n\Big)\ge 1-\delta,~\forall n\ge N.\label{eq:ec10}
\end{align}
By Assumptions  \ref{as:momP_AS} (iv), \ref{as:bcs1} and Theorem 2.8.2 in \cite{Vaart_Wellner2000aBK},  $\dsM_P$ is a  Donsker class uniformly in $P\in\mathcal P$. 
By hypothesis, $\mathcal S_P$ is a  Donsker class uniformly in $P\in\mathcal P$. 

 Therefore, by the continuous mapping theorem, for any $\epsilon>0$, 
\begin{align}
&\Big|P\Big(	\sqrt n\sup_{\theta\in\Theta}\Big|\frac{n^{-1}\sum_{i=1}^n m_{j}(X_i,\theta)^2}{\sigma_{P,j}^2(\theta)}-\frac{E_P[m_j(X,\theta)^2]}{\sigma_{P,j}^2(\theta)}\Big|\le C_1\Big)-\text{Pr}(\sup_{\theta\in\Theta}|\mathbb H_{P,j}(\theta)|\le C_1)\Big|\le \epsilon \label{eq:ec11}\\	
&\Big|P\Big(\sqrt n\sup_{\theta\in\Theta}\Big|\frac{\bar m_{n,j}(\theta)-E_P[m_j(X,\theta)]}{\sigma_{P,j}(\theta)}\Big|\le C_2\Big)-\text{Pr}(\sup_{\theta\in\Theta}|\mathbb G_{P,j}(\theta)|\le C_2)\Big|\le \epsilon .\label{eq:ec12}
\end{align}
for $n$ sufficiently large uniformly in $P\in\mathcal P$, where $\mathbb H_{P,j}$ and $\mathbb G_{P,j}$ are tight Gaussian processes, and $C_1$ and $C_2$ are the continuity points of the distributions of $\sup_{\theta\in\Theta}|\mathbb H_{P,j}(\theta)|$ and $\sup_{\theta\in\Theta}|\mathbb G_{P,j}(\theta)|$ respectively.  As in the proof of Lemma \ref{lem:eta_conv} (i), bounding each term of the right hand side of \eqref{eq:ec7} by $C_1/\sqrt n$ and $C_2/\sqrt n$ implies that  $\sup_{\theta\in\Theta}\Big|\frac{\hat\sigma_{n,j}^2(\theta)}{\sigma_{P,j}^2(\theta)}-1\Big|\le C/\sqrt n$ for some constant $C>0$.  Now choose $C_1>0$ and $C_2>0$ so that
\begin{align}
	\text{Pr}(\sup_{\theta\in\Theta}|\mathbb H_{P,j}(\theta)|\le C_1)> 1-\delta/3~\text{ and }~
	\text{Pr}(\sup_{\theta\in\Theta}|\mathbb G_{P,j}(\theta)|\le C_2)>1-\delta/3\label{eq:ec13}
\end{align}
and set $\epsilon>0$ sufficiently small so that $1-2\delta/3-2\epsilon\ge 1-\delta$. The existence of such continuity points $C_1,C_2>0$   is due to Theorem 11.1 in \cite{Davydov:aa} applied to $\sup_{\theta\in\Theta}|\mathbb H_{P,j}(\theta)|$ and $\sup_{\theta\in\Theta}|\mathbb G_{P,j}(\theta)|$ respectively.
Then, for sufficiently large $n$,
\begin{align}
	1-\delta&\le P\Big(	\sqrt n\sup_{\theta\in\Theta}\Big|\frac{n^{-1}\sum_{i=1}^n m_{j}(X_i,\theta)^2}{\sigma_{P,j}^2(\theta)}-\frac{E_P[m_j(X,\theta)^2]}{\sigma_{P,j}^2(\theta)}\Big|\le C_1,\sqrt n\sup_{\theta\in\Theta}\Big|\frac{\bar m_{n,j}(\theta)-E_P[m_j(X,\theta)]}{\sigma_{P,j}(\theta)}\Big|\le C_2\Big)\notag\\
	&\le P\Big(\sup_{\theta\in\Theta}\Big|\frac{\hat\sigma_{n,j}^2(\theta)}{\sigma_{P,j}^2(\theta)}-1\Big|\le C/\sqrt n\Big), \label{eq:ec14}
\end{align}
uniformly in $P\in\mathcal P.$ 

Next, note that, for $x>0$ and $0<\eta<1$, $|x^2-1|\le \eta$ implies $|x-1|\le 1-(1-\eta)^{1/2}\le\eta$, and hence by \eqref{eq:ec14}, for sufficiently large $n$,
\begin{align}
	1-\delta\le  P\Big(\sup_{\theta\in\Theta}\Big|\frac{\hat\sigma_{n,j}(\theta)}{\sigma_{P,j}(\theta)}-1\Big|\le C/\sqrt n\Big),\label{eq:ec15}
\end{align} 
uniformly in $P\in\mathcal P.$ 
Finally, note again that $|\hat\sigma_{n,j}(\theta)/\sigma_{P,j}(\theta)-1|\le \epsilon$ implies $\hat\sigma_{n,j}(\theta)>0$, and
by the local Lipshitz continuity of $x\mapsto 1/x$ on a neighborhood around 1, there is a constant $C'$ such that
\begin{align}
P\Big(\sup_{\theta\in\Theta}|\eta_{n,j}(\theta)|\le C'/\sqrt n\Big)	\ge 1-\delta,\label{eq:ec16}
\end{align}
uniformly in $P\in\mathcal P$ for all $n$ sufficiently large.
This establishes the first claim of the lemma.	

(ii) First, consider
\begin{align}
\frac{\hat\sigma^2_{n,j}(\theta)}{\sigma^2_{P,j}(\theta)}=n^{-1}\sum_{i=1}^n\left(\frac{m(X_i,\theta)}{\sigma_{P,j}(\theta)}\right)^2-\left(n^{-1}\sum_{i=1}^n\frac{m(X_i,\theta)}{\sigma_{P,j}(\theta)}\right)^2.\label{eq:etaLip1}
\end{align}
We claim that this function is Lipschitz with probability approaching 1. To see this, note that, for any $\theta,\theta'\in\Theta$,
\begin{align}
&~ \Bigg|n^{-1}\sum_{i=1}^n\Bigg(\frac{m(X_i,\theta)}{\sigma_{P,j}(\theta)}\Bigg)^2 -n^{-1}\sum_{i=1}^n\Bigg(\frac{m(X_i,\theta')}{\sigma_{P,j}(\theta')}\Bigg)^2\Bigg|\notag\\
=&~\Bigg|n^{-1}\sum_{i=1}^n \Bigg(\frac{m(X_i,\theta)}{\sigma_{P,j}(\theta)}+\frac{m(X_i,\theta')}{\sigma_{P,j}(\theta')}\Bigg)\Bigg(\frac{m(X_i,\theta)}{\sigma_{P,j}(\theta)}-\frac{m(X_i,\theta')}{\sigma_{P,j}(\theta')}\Bigg)\Bigg|	\notag\\
\le&~ n^{-1}\sum_{i=1}^n 2\sup_{\theta\in\Theta}\Big|\frac{m(X_i,\theta)}{\sigma_{P,j}(\theta)}\Big|\bar M(X_i)\|\theta-\theta'\|.\label{eq:etaLip2}
\end{align}
Define $B_n\equiv n^{-1}\sum_{i=1}^n 2\sup_{\theta\in\Theta}\Big|\frac{m(X_i,\theta)}{\sigma_{P,j}(\theta)}\Big|\bar M(X_i)$. By Markov and Cauchy-Schwarz inequalities,
\begin{align}
	P(B_n>C)\le \frac{E[B_n]}{C}\le\frac{2E_P\left[\sup_{\theta\in\Theta}\Big|\frac{m(X_i,\theta)}{\sigma_{P,j}(\theta)}\Big|^2 \right]^{1/2}E_P\Big[\bar M(X_i)^2 \Big]^{1/2}}{C} \le \frac{2M}{C},\label{eq:etaLip3}
\end{align}
where the third inequality is due to Assumptions \ref{as:momP_AS} (iv) and the assumption on $\bar M$. Hence, for any $\eta>0$, one may find $C>0$ such that $\sup_{P\in\mathcal P}P(B_n>C)<\eta$ for all $n$.

Similarly, for any $\theta,\theta'\in\Theta$,
\begin{align}
&~\Bigg|\Bigg(n^{-1}\sum_{i=1}^n\frac{m(X_i,\theta)}{\sigma_{P,j}(\theta)}\Bigg)^2-\Bigg(n^{-1}\sum_{i=1}^n\frac{m(X_i,\theta')}{\sigma_{P,j}(\theta')}\Bigg)^2	\Bigg|\notag\\
=&~\Bigg|n^{-1}\sum_{i=1}^n\frac{m(X_i,\theta)}{\sigma_{P,j}(\theta)}+n^{-1}\sum_{i=1}^n\frac{m(X_i,\theta')}{\sigma_{P,j}(\theta')}\Bigg|\Bigg|n^{-1}\sum_{i=1}^n\frac{m(X_i,\theta)}{\sigma_{P,j}(\theta)}-n^{-1}\sum_{i=1}^n\frac{m(X_i,\theta')}{\sigma_{P,j}(\theta')}\Bigg|\notag\\
\le&~ n^{-1}\sum_{i=1}^n 2\sup_{\theta\in\Theta}\Bigg|\frac{m(X_i,\theta)}{\sigma_{P,j}(\theta)}\Bigg|n^{-1}\sum_{i=1}^n\bar M(X_i)\|\theta-\theta'\|.\label{eq:etaLip4}
\end{align}
Define $\tilde B_n\equiv n^{-1}\sum_{i=1}^n 2\sup_{\theta\in\Theta}\Big|\frac{m(X_i,\theta)}{\sigma_{P,j}(\theta)}\Big|n^{-1}\sum_{i=1}^n\bar M(X_i)$. By Markov, Cauchy-Schwarz, and Jensen's inequalities,
\begin{multline}
	P(\tilde B_n>C)\le \frac{E[\tilde B_n]}{C}\le \frac{2E_P\Big[\Big(n^{-1}\sum\sup_{\theta\in\Theta}\Big|\frac{m(X_i,\theta)}{\sigma_{P,j}(\theta)}\Big|\Big)^2\Big]^{1/2}E_P\Big[\Big(n^{-1}\sum\bar M(X_i)\Big)^2\Big]^{1/2}}{C} \\
\le \frac{2E_P\Big[\sup_{\theta\in\Theta}\big|\frac{m(X_i,\theta)}{\sigma_{P,j}(\theta)}\big|^2\Big]^{1/2}E_P[\bar M(X_i)^2]^{1/2}}{C}	\le \frac{2M}{C},\label{eq:etaLip5}
\end{multline}
where the last inequality is due to Assumptions \ref{as:momP_AS} (iv) and the assumption on $\bar M$.  Hence, for any $\eta>0$, one may find $C>0$ such that $\sup_{P\in\mathcal P}P(\tilde B_n>C)<\eta$ for all $n$.

Finally, let $g(y)\equiv y^{-1/2}-1$ and note that $|g(y)-g(y')|\le \frac{1}{2}\sup_{\bar y\in (1-\epsilon,1+\epsilon)}|\bar y|^{-3/2}|y-y'|$ on $(1-\epsilon,1+\epsilon).$ As shown in \eqref{eq:ec15}, $\hat\sigma^2_{n,j}(\theta)/\sigma^2_{P,j}(\theta)$ converges to 1 in probability, and $g$ is locally Lipschitz on a neighborhood of 1. Combining this with 
\eqref{eq:etaLip1}-\eqref{eq:etaLip5} yields the desired result.
\end{proof}

\begin{lemma}\label{lem:correl_Lip}
Suppose Assumption  \ref{as:momP_AS} holds.
Suppose further that $|\sigma_{P,j}(\theta)^{-1}m_j(x,\theta)-\sigma_{P,j}(\theta')^{-1}m_j(x,\theta')|\le \bar M(x)\|\theta-\theta'\|$ with $E_P[\bar M(X)^{2}]<M$  for all $\theta,\theta'\in\Theta$, $x\in\mathcal X$, $j=1,\cdots,J$, and $P\in\mathcal P$.

Then, 	
\begin{align}
\sup_{P\in\mathcal P}\|Q_P(\theta_1,\tilde\theta_1)-Q_P(\theta_2,\tilde\theta_2)\|\le M\|(\theta_1,\tilde\theta_1)-(\theta_2,\tilde\theta_2)\|,	\label{eq:correl_Lip0}
\end{align}
 for some $M>0$ and for all $\theta_1,\tilde\theta_1,\theta_2,\tilde\theta_2\in \Theta$.
\end{lemma}

\begin{proof}
Recall that 
\begin{align}
	[Q_P(\theta_1,\tilde\theta_1)]_{j,k}=E_P\Big[\frac{m_j(X_i,\theta_1)}{\sigma_{P,j}(\theta_1)}\frac{m_k(X_i,\tilde \theta_1)}{\sigma_{P,k}(\tilde \theta_1)}\Big]-E_P\Big[\frac{m_j(X_i,\theta_1)}{\sigma_{P,j}(\theta_1)}\Big]E_P\Big[\frac{m_k(X_i,\tilde\theta_1)}{\sigma_{P,k}(\tilde \theta_1)}\Big].\label{eq:correl_Lip1}
\end{align}	
For any $\theta_1,\tilde\theta_1,\theta_2,\tilde\theta_2\in \Theta$,
\begin{align}
&~\Big|E_P\Big[\frac{m_j(X_i,\theta_1)}{\sigma_{P,j}(\theta_1)}\frac{m_k(X_i,\tilde \theta_1)}{\sigma_{P,k}(\tilde \theta_1)}\Big]-E_P\Big[\frac{m_j(X_i,\theta_2)}{\sigma_{P,j}(\theta_2)}\frac{m_k(X_i,\tilde \theta_2)}{\sigma_{P,k}(\tilde \theta_2)}\Big]\Big|\notag\\
\le&~ 	\Big|E_P\Big[\Big(\frac{m_j(X_i,\theta_1)}{\sigma_{P,j}(\theta_1)}-\frac{m_j(X_i,\theta_2)}{\sigma_{P,j}(\theta_2)}\Big)\frac{m_k(X_i,\tilde \theta_2)}{\sigma_{P,k}(\tilde \theta_2)}\Big]\Big|+\Big|E_P\Big[\frac{m_j(X_i,\theta_1)}{\sigma_{P,j}(\theta_1)}\Big(\frac{m_k(X_i,\tilde \theta_1)}{\sigma_{P,k}(\tilde \theta_1)}-\frac{m_k(X_i,\tilde \theta_2)}{\sigma_{P,k}(\tilde \theta_2)}\Big)\Big]\Big|\notag\\
\le&~ E_P\Big[\sup_{\theta\in\Theta}\Big|\frac{m_k(X_i, \theta)}{\sigma_{P,k}( \theta)}\Big|\bar M(X_i)\Big]\|\theta_1-\theta_2\|+E_P\Big[\sup_{\theta\in\Theta}\Big|\frac{m_j(X_i, \theta)}{\sigma_{P,j}( \theta)}\Big|\bar M(X_i)\Big]\|\tilde\theta_1-\tilde\theta_2\|\notag\\
\le&~ M(\|\theta_1-\theta_2\|+\|\tilde\theta_1-\tilde\theta_2\|),\label{eq:correl_Lip2}
\end{align}
where the last inequality is due to the Cauchy-Schwarz inequality, Assumption \ref{as:momP_AS} (iv), and the assumption on $\bar M$.

Similarly, for any $\theta_1,\tilde\theta_1,\theta_2,\tilde\theta_2\in \Theta$,
\begin{align}
&~	\Big|E_P\Big[\frac{m_j(X_i,\theta_1)}{\sigma_{P,j}(\theta_1)}\Big]E_P\Big[\frac{m_k(X_i,\tilde\theta_1)}{\sigma_{P,k}(\tilde \theta_1)}\Big]-E_P\Big[\frac{m_j(X_i,\theta_2)}{\sigma_{P,j}(\theta_2)}\Big]E_P\Big[\frac{m_k(X_i,\tilde\theta_2)}{\sigma_{P,k}(\tilde \theta_2)}\Big]\Big|\notag\\
\le &~	\Big|E_P\Big[\frac{m_j(X_i,\theta_1)}{\sigma_{P,j}(\theta_1)}-\frac{m_j(X_i,\theta_2)}{\sigma_{P,j}(\theta_2)}\Big]\Big|\Big|E_P\Big[\frac{m_k(X_i,\tilde \theta_2)}{\sigma_{P,k}(\tilde \theta_2)}\Big]\Big|+\Big|E_P\Big[\frac{m_j(X_i,\theta_1)}{\sigma_{P,j}(\theta_1)}\Big]\Big|\Big|E_P\Big[\frac{m_k(X_i,\tilde \theta_1)}{\sigma_{P,k}(\tilde \theta_1)}-\frac{m_k(X_i,\tilde \theta_2)}{\sigma_{P,k}(\tilde \theta_2)}\Big]\Big|\notag\\
\le &~E_P\Big[\sup_{\theta\in\Theta}\Big|\frac{m_k(X_i, \theta)}{\sigma_{P,k}( \theta)}\Big|\Big]E_P[\bar M(X_i)]\|\theta_1-\theta_2\|+E_P\Big[\sup_{\theta\in\Theta}\Big|\frac{m_j(X_i, \theta)}{\sigma_{P,j}( \theta)}\Big|\Big]E_P[\bar M(X_i)]\|\tilde\theta_1-\tilde\theta_2\|\notag\\
\le &~ M(\|\theta_1-\theta_2\|+\|\tilde\theta_1-\tilde\theta_2\|),\label{eq:correl_Lip3}
\end{align}
where the last inequality is due to the Cauchy-Schwarz inequality, Assumption \ref{as:momP_AS} (iv), and the assumption on $\bar M$.
The conclusion of the lemma then follows from \eqref{eq:correl_Lip1}-\eqref{eq:correl_Lip3}.
\end{proof}

\subsection{Almost Sure Representation Lemma and Related Results}
\label{app:asrep}
In this appendix, we provide details on the almost sure representation used in Lemmas \ref{lem:cv_convergence}, \ref{lem:res}, \ref{lem:empt}, and \ref{lem:pair}. 
We start with stating a uniform version of the bootstrap consistency in \cite{Vaart_Wellner2000aBK}.
For this, we define the original sample $X^\infty=(X_1,X_2,\cdots)$ and a $n$-dimensional multinomial vector $M_n$  on a common probability space $(\mathcal X^\infty,\mathcal A^\infty,P^\infty)\times (\mathcal Z,\mathcal C,Q)$. We then view $X^{\infty}$ as the coordinate projection on the first $\infty$ coordinates of the probability space above. Similarly, we view $M_n$ as the coordinate projection on $\mathcal Z$. Here, $M_n$ follows a multinomial distribution with parameter $(n;1/n,\cdots,1/n)$ and is independent of $X^\infty.$
We then let $E_M[\cdot|X^\infty=x^\infty]$ denote the conditional expectation of $M_n$ given $X^\infty=x^\infty.$ Throughout, we let $\ell^\infty(\Theta,\mathbb R^J)$ denote uniformly bounded $\mathbb R^J$-valued functions on $\Theta$. We simply write $\ell^\infty(\Theta)$ when $J=1$. 

Using the multinomial weight, we rewrite the empirical bootstrap process as
\begin{align}
	\mathbb G^b_{n,j}(\cdot)=g_j(X^\infty,M_n)\equiv\frac{1}{\sqrt n}\sum_{i=1}^n(M_{n,i}-1)m_j(X_i,\cdot)/\hat\sigma_{n,j}(\cdot),~j=1,\cdots,J,\label{eq:defg}
\end{align}
where $g_j:\mathcal X^\infty\times \mathcal Z\to \ell^\infty(\Theta)$ is a function that maps the sample path and the multinomial weight $(X^\infty,M_n)$ to the empirical bootstrap process $\mathbb G^b_{n,j}$. 
We then let $g:\mathcal X^\infty\times \mathcal Z\to \ell^\infty(\Theta,\mathbb R^J)$ be defined by $g=(g_1,\cdots,g_J)'.$
 For any function $f:\ell^\infty(\Theta,\mathbb R^J)\to\mathbb R$, 
the conditional expectation of $f(\mathbb G^b_{n})$ given the sample path $X^\infty$ is
\begin{align}
	E_M[f(\mathbb G^b_{n})|X^\infty=x^\infty]=\int f\circ g(x^\infty,m_n)dQ(m_n),
\end{align}
where, with a slight abuse of notation, we use $Q$ for the induced law of $M_n$. 

Let $\mathcal F$ be the function space  $\{f(\cdot)=(m_1(\cdot,\theta)/\sigma_{P,1}(\theta),\cdots,m_J(\cdot,\theta)/\sigma_{P,J}(\theta)),\theta\in\Theta,P\in\mathcal P\}$. For each $j$, define a bootstrapped empirical process standardized by $\sigma_{P,j}$ as follows:
\begin{align}
	\mathfrak G^b_{n,j}(\theta)&\equiv \frac{1}{\sqrt n}\sum_{i=1}^n \left(m_j(X_i^b,\theta)-\bar m_n(\theta)\right)/\sigma_{P,j}(\theta)\notag\\
	&=\frac{1}{\sqrt n}\sum_{i=1}^n(M_{n,i}-1)m_j(X_i,\theta)/\sigma_{P,j}(\theta).
\end{align}

The following result was shown in the proof of Lemma D.2.8 in \cite{BCS14_misp}, which is a uniform version of (a part of) Theorem 3.6.2 in \cite{Vaart_Wellner2000aBK}.\ For the definition of a uniform version of Donskerness and pre-Gaussianity, we refer to \cite{Vaart_Wellner2000aBK} pages 168-169. Below, we let $P^*$ denote the outer probability of $P$ and let $T^*$ denote the minimal measurable majorant of any (not necessarily measurable) random  element $T$.
\begin{lemma}\label{lem:bootcons_standardized}
	Let $\mathcal F$ be a class of measurable functions with finite envelope function. Suppose $\mathcal F$ is such that (i) $\mathcal F$ is Donsker and pre-Gaussian uniformly in $P\in\mathcal P$; and (ii) $\sup_{P\in\mathcal P}P^*\|f-Pf\|^2_{\mathcal F}<\infty.$ Then,
\begin{align}
\sup_{h\in BL_1}|E_M[h(\mathfrak G^b_{n})|X^\infty]-E[h(\mathbb G_{P})]|\stackrel{as*}{\to}0,	
\end{align} 
uniformly in $P\in\mathcal P$.
\end{lemma}

The result above gives uniform consistency of the standardized bootstrap process $\mathfrak G^b_{n}$. We now extend this to the studentized bootstrap process $\mathbb G^b_{n}$. 
\begin{lemma}\label{lem:boot_cons}
Suppose Assumptions \ref{as:momP_AS}, \ref{as:GMS}, and \ref{as:bcs1} hold. Then,
\begin{align}
\sup_{h\in BL_1}|E_M[h(\mathbb G^b_{n})|X^\infty]-E[h(\mathbb G_{P})]|\stackrel{as*}{\to}0,	
\end{align} 
uniformly in $P\in\mathcal P$.
\end{lemma}

\begin{proof}
By Assumptions \ref{as:momP_AS} (iv) and \ref{as:bcs1}, Assumptions A.1-A.4 in \cite{BCS14_misp} hold, which in turn implies that, by their Lemma D.1.2, $\mathcal F$ is Donsker and pre-Gaussian uniformly in $P\in\mathcal P$. Further, by Assumption \ref{as:momP_AS} (iv) again, $\sup_{P\in\mathcal P}P^*\|f-Pf\|_{\mathcal F}<\infty$. Hence, by Lemma \ref{lem:bootcons_standardized},
\begin{align}
	\inf_{P\in\mathcal P}P^\infty\Big(\sup_{h\in BL_1}|E_M[h(\mathfrak G^b_{n})|X^\infty]-E[h(\mathbb G_{P})]|^*\to 0\Big)=1.\label{eq:bootcon1}
\end{align}
For later use, we define the following set of sample paths, which has probability 1 uniformly in $P\in\mathcal P$.
\begin{align}
	A\equiv \Big\{x^\infty\in\mathcal X^\infty:\sup_{h\in BL_1}|E_M[h(\mathfrak G^b_{n})|X^\infty=x^\infty]-E[h(\mathbb G_P)]|^*\to 0\Big\}.
\end{align}
Note that $\mathbb G^b_{n,j}$ and $\mathfrak G^b_{n,j}$  are related to each other by the following relationship:
\begin{align}
	\mathbb G^b_{n,j}(\theta)-\mathfrak G^b_{n,j}(\theta)=\mathfrak G^b_{n,j}(\theta)\left(\frac{\sigma_{P,j}(\theta)}{\hat\sigma_{n,j}(\theta)}-1\right)=\mathfrak G^b_{n,j}(\theta)\eta_{n,j}(\theta),~\theta\in\Theta.\label{eq:Gdiff}
\end{align}
By Assumptions \ref{as:momP_AS}, \ref{as:GMS}, and \ref{as:bcs1}, Lemma \ref{lem:eta_conv} applies. Hence,
\begin{align}
	\inf_{P\in\mathcal P}P^\infty\Big(\sup_{\theta\in\Theta}|\eta_{n,j}(\theta)|^*\to 0\Big)=1.\label{eq:etaconv_inboot}
\end{align}
Define the following set of sample paths:
\begin{align}
	B\equiv\Big\{x^\infty\in\mathcal X^\infty:\sup_{\theta\in\Theta}|\eta_{n,j}(\theta)|^*\to 0,\forall j=1,\cdots,J\Big\}.
\end{align}
For any $x^\infty\in A\cap B$, it then follows that
\begin{align}
	\sup_{h\in BL_1}\left|E_M[h(\mathbb G^b_{n})|X^\infty=x^\infty]-E[h(\mathbb G_P)]\right|^*\to 0,\label{eq:bootcon2}
\end{align}
due to \eqref{eq:bootcon1} and \eqref{eq:Gdiff}, $h$ being Lipschitz, $\mathfrak G^b_{n,j}$ being bounded (given $x^\infty$), and  $\sup_{\theta\in\Theta}|\eta_{n,j}(\theta)|^*\to 0$ for all $j$. Finally, note that $\inf_{P\in\mathcal P}P^\infty(A\cap B)=1$ due to \eqref{eq:bootcon1}, \eqref{eq:etaconv_inboot}, and De Morgan's law.
This establishes the conclusion of the lemma.	
\end{proof}

The following lemma shows that, for almost all sample path $x^\infty$, one can find an almost sure representation of the bootstrapped empirical process that is convergent.
\begin{lemma}\label{lem:asrep}
Suppose Assumptions \ref{as:momP_AS}, \ref{as:GMS}, and \ref{as:bcs1} hold. Then, for each $x^\infty\in\mathcal X^\infty$, there exists a sequence $\{\tilde G_{n,x^\infty}\in \ell(\Theta,\mathbb R^J),n\ge 1\}$ and a random element $\tilde G_{P,x^\infty}\in \ell(\Theta,\mathbb R^J)$ defined on some probability space $(\tilde\Omega,\tilde{\mathcal A},\tilde{\mathbf P})$ such that
\begin{align}
	\int h\circ g(x^\infty,m_n)dQ(m_n)&=\int h(\tilde G_{n,x^\infty}(\tilde \omega))d\tilde{\mathbf P}^*(\tilde\omega),~\forall h\in BL_1\\
	\int h(\mathbb G_P(\omega))dP(\omega)&=\int h(\tilde G_{P,x^\infty}(\tilde \omega))d\tilde{\mathbf P}^*(\tilde\omega),~\forall h\in BL_1,
\end{align}
for all $x^\infty\in C$ for some set $C\subset \mathcal X^\infty$ such that $\inf_{P\in\mathcal P}P^\infty(C)=1$ and
\begin{align}
	\inf_{P\in\mathcal P}P^\infty\Big(\big\{x^\infty\in \mathcal X^\infty:\tilde G_{n,x^\infty}\stackrel{\tilde{\mathbf P}-as*}{\to}\tilde G_{P,x^\infty}\big\}\Big)=1.\label{eq:asconv1}
\end{align}
\end{lemma}

\begin{proof}
Define the following set of sample paths:
\begin{align}
	C\equiv \Big\{x^\infty\in\mathcal X^\infty:\sup_{h\in BL_1}|E_M[h(\mathbb G^b_{n,j})|X^\infty=x^\infty]-E[h(\mathbb G_P)]|^*\to 0\Big\}.
\end{align}
By Lemma \ref{lem:boot_cons}, 	$\inf_{P\in\mathcal P}P^\infty(C)=1$.
	
For each fixed sample path $x^\infty\in C$, consider the bootstrap empirical process $g(x^\infty,M_n)$ in \eqref{eq:defg}. This is a random element in $\ell^\infty(\Theta,\mathbb R^J)$ with a law governed by $Q$. 
For each $x^\infty\in C$, by Lemma \ref{lem:boot_cons}, 
\begin{align}
\sup_{h\in BL_1}\left|\int h\circ g(x^\infty,m_n)dQ(m_n)-E[h(\mathbb G_P)]\right|^*\to 0.
\end{align}  
Hence, by Theorem 1.10.4 in \cite{Vaart_Wellner2000aBK}, for each $x^\infty\in C$, one may find an almost sure representation $\tilde G_{n,x^\infty}$ of $g(x^\infty,M_n)$ on some probability space $(\tilde\Omega,\tilde{\mathcal A},\tilde{\mathbf P})$ such that
\begin{align}
	\int h\circ g(x^\infty,m_n)dQ(m_n)=\int h(\tilde G_{n,x^\infty}(\tilde \omega))d\tilde{\mathbf P}^*(\tilde\omega),~\forall h\in BL_1.
\end{align}
In particular, the proof of Theorem 1.10.4 in \cite{Vaart_Wellner2000aBK} (see also Addendum 1.10.5) allows us to take $\tilde G_{n,x^\infty}$ to be defined for each $\tilde\omega\in\tilde\Omega$ as
\begin{align}
\tilde G_{n,x^\infty}(\tilde \omega)=g(x^\infty,M_n(\phi_n(\tilde \omega)))	,
\end{align}
for some perfect map $\phi_n:\tilde\Omega\to \mathcal Z$ (see the construction of $\phi_\alpha$ in the middle of page 61 in VW). One may define $\tilde G_{n,x^\infty}$ arbitrarily for any $x^\infty\notin C.$
The almost sure representation $\tilde G_{P,x^\infty}$ of $\mathbb G_{P,j}$ is defined similarly.
 
By Theorem  1.10.4 in \cite{Vaart_Wellner2000aBK}, Eq. \eqref{eq:bootcon2}, and $\inf_{P\in\mathcal P}P(C)=1$, it follows that
\begin{align}
	\inf_{P\in\mathcal P}P^\infty\Big(\big\{x^\infty\in \mathcal X^\infty:\tilde G_{n,x^\infty}\stackrel{\tilde{\mathbf P}-as*}{\to}\tilde G_{P,x^\infty}\big\}\Big)=1.
\end{align}
This establishes the claim of the lemma.
\end{proof}

\begin{lemma}\label{cor:asrep}
Suppose Assumptions \ref{as:momP_AS}, \ref{as:GMS}, and \ref{as:bcs1} hold.  Let $W_n\equiv (\mathbb G^b_n,Y_n)$ be a sequence in $\mathcal W\equiv\ell(\Theta,\mathbb R^J)\times \mathbb R^{d_Y}$  such that $Y_n=\tilde g(X^\infty, M_n)$ for some map $\tilde g:\mathcal X^\infty\times \mathcal Z\to\mathbb R^{d_Y}$ and
\begin{align}
\inf_{P\in \mathcal P} P^\infty\big(\sup_{h\in BL_1}|E_M[h(W_n)|X^\infty=x^\infty]-E[h(W)]|^*\to 0\big)=1,\label{eq:Wn_cons}
\end{align}
where $W=(\mathbb G,Y)$ is a Borel measurable random element in $\mathcal W$.  

Then, for each $x^\infty\in\mathcal X^\infty$, there exists a sequence $\{W^*_{n,x^\infty}\in \mathcal W,n\ge 1\}$ and a  random element $W^*_{x^\infty}\in\mathcal W$ defined on some probability space $(\tilde\Omega,\tilde{\mathcal A},\tilde{\mathbf P})$ such that
\begin{align}
	E_M[h(W_n)|X^\infty=x^\infty]&=\int h(W^*_{n,x^\infty}(\tilde \omega))d\tilde{\mathbf P}^*(\tilde\omega),~\forall h\in BL_1\\
	E[h(W)]&=\int h(W^*_{x^\infty}(\tilde\omega))d\tilde{\mathbf P}^*(\tilde\omega),~\forall h\in BL_1,
\end{align}
for all $x^\infty\in C$ for some set $C\subset \mathcal X^\infty$ such that $\inf_{P\in\mathcal P}P^\infty(C)=1$, and
\begin{align}
	\inf_{P\in\mathcal P}P^\infty\Big(\big\{x^\infty\in \mathcal X^\infty:W^*_{n,x^\infty}\stackrel{\tilde{\mathbf P}-as*}{\to}\tilde W^*_{x^\infty}\big\}\Big)=1.\label{eq:asconv2}
\end{align}
\end{lemma}

\begin{proof}
Let $C\equiv\{x^\infty:\sup_{h\in BL_1}|E_M[h(W_n)|X^\infty=x^\infty]-E[h(W)]|^*\stackrel{}{\to} 0\}$.
The rest of the proof is the same as the one for Lemma \ref{lem:asrep} and is therefore omitted.
\end{proof}

\begin{remark}
When called by the Lemmas in Appendix \ref{app:Lemma}, Lemma \ref{cor:asrep} is applied, for example, with $Y_n=(vec(\hat D_n(\theta_n')),\hat \xi_n(\theta_n'))$ and $Y=(vec(D), \pi_1)$.
\end{remark}
\end{appendices}
\newpage

\ifx\undefined\BySame
\newcommand{\BySame}{\leavevmode\rule[.5ex]{3em}{.5pt}\ }
\fi
\ifx\undefined\textsc
\newcommand{\textsc}[1]{{\sc #1}}
\newcommand{\emph}[1]{{\em #1\/}}
\let\tmpsmall\small
\renewcommand{\small}{\tmpsmall\sc}
\fi

\end{document}